\documentclass[reqno,9pt]{amsart}
\usepackage[margin=1in]{geometry}

\usepackage{graphicx}
\usepackage{amsmath,amssymb}

\newtheorem{theorem}{Theorem}

\newtheorem{definition}{Definition}
\newtheorem{lemma}{Lemma}
\newtheorem{proposition}[theorem]{Proposition}

\usepackage{enumerate}
\usepackage{ amssymb }

\usepackage[makeroom]{cancel}

\makeatletter
\renewcommand*\env@matrix[1][*\c@MaxMatrixCols c]{%
  \hskip -\arraycolsep
  \let\@ifnextchar\new@ifnextchar
  \array{#1}}
\makeatother

\let\e=\varepsilon

\let\p=\partial

\let\O=\Omega

\let\o=\omega
\let\g=\gamma

\let\b=\beta

\numberwithin{equation}{section}

\let\hide\iffalse
\let\unhide\fi

\newcommand{\R}{\mathbb{R}}

\renewcommand{\S}{\mathbb{S}}

\newcommand{\be}{\begin{equation}}
\newcommand{\bm}{\begin{multline}}
\newcommand{\ee}{\end{equation}}
\newcommand{\dd}{\mathrm{d}}

\newcommand{\xb}{x_{\mathbf{b}}}

\newcommand{\tb}{t_{\mathbf{b}}}
\newcommand{\vb}{v_{\mathbf{b}}}

\newcommand{\xf}{x_{\mathbf{f}}}
\newcommand{\tf}{t_{\mathbf{f}}}

\newcommand{\tbpm}{t_{\mathbf{b},\iota}}

\newcommand{\vbpm}{v_{\mathbf{b},\iota}}

\newcommand{\xbpm}{x_{\mathbf{b},\iota}}

\newcommand{\Bes}{\begin{eqnarray*}}
\newcommand{\Ees}{\end{eqnarray*}}
\newcommand{\Be}{\begin{equation} }
\newcommand{\Ee}{\end{equation}}

\def\p{\partial}

\def\O{\Omega}
\def\R{\mathbb{R}}

\def\B{\begin{equation}}
\def\E{\end{equation}}
\def\BN{\begin{eqnarray*}}
\def\EN{\end{eqnarray*}}

\begin{document}

\title{A note on two species collisional plasma in bounded domains}
\author{Yunbai Cao}

 \address{Department of Mathematics, University of Wisconsin, Madison, WI 53706 USA}
\email{ycao35@wisc.edu}
%

 \begin{abstract}
We construct a unique global-in-time solution to the two species Vlasov-Poisson-Boltzmann system in convex domains with the diffuse boundary condition, which can be viewed as one of the ideal scattering boundary model. The construction follows a new $L^{2}$-$L^{\infty}$ framework in \cite{VPBKim}. In our knowledge this result is the first construction of strong solutions for \textit{two species} plasma models with \textit{self-consistent field} in general bounded domains.
  \end{abstract}

  \maketitle
  
  \tableofcontents


\section{Introduction}
One of the fundamental models for dynamics of dilute charged particles (e.g., electrons and ions) is the Vlasov-Maxwell-Boltzmann (VMB) system, in which particles interact with themselves through collisions and with their self-consistent electromagnetic field:
\Be \label{VMBsystem} \begin{split}
 & \p_t F_+ + v\cdot \nabla_x F_+ + \frac{e_+}{m_+} (E + \frac{v}{c} \times B) \cdot \nabla_v F_+ = Q(F_+,F_+) + Q(F_+,F_-),
\\ &\p_t F_- + v\cdot \nabla_x F_- - \frac{e_-}{m_-} (E + \frac{v}{c} \times B) \cdot \nabla_v F_- = Q(F_-,F_+) + Q(F_-,F_-).
\end{split} \Ee
Here $F_\pm(t,x,v) \ge 0 $ are the density functions for the ions $(+)$ and electrons $(-)$ respectively, and $e_\pm$, $m_\pm$ the magnitude of their charges and masses, $c$ the speed of light.
The self-consistent electromagnetic field $E(t,x)$, $B(t,x)$ in \eqref{VMBsystem} is coupled with $F(t,x,v)$ through the Maxwell system (see \cite{Guo_M}).
Previous studies for the VMB system, for example the existence of global in time classical solution, uniqueness, and asymptotic behavior without boundaries, can be found in \cite{Guo_M}, \cite{DS}.

Now formally as the speed of light $c \to \infty$, one can derive the so-called two species Vlasov-Poisson-Boltzmann (VPB) system, where $B(t,x) =0$. And the field $E$, that we are interested in, is associated with an electrostatic potential $\phi$ as
\Be\label{Field}
E(t,x): =  - \nabla_{x} \phi(t,x),
\Ee
where the potential is determined by the Poisson equation:
\Be \label{phibigF}
-\Delta_x \phi(t,x) = \int_{\mathbb R^3 } (F_+ - F_-)  dv : = \rho.
\Ee
In this paper we consider the zero Neumann boundary condition for $\phi$:
\Be \label{0signcondition}
\, \frac{\p \phi}{\p n } = 0 \text{ for } x \in \p \O.
\Ee

It turns out that the presence of all the physical constants does not create essential mathematical difficulties. Therefore, for simplicity we normalize all constants in \eqref{VMBsystem} to be one, and the VPB system takes the form:
\Be \label{2FVPB}
\begin{split}
\p_t F_+ + v \cdot \nabla_x F_+ +E \cdot \nabla_v F_+ = Q(F_+,F_+) + Q(F_+,F_- ),
\\ \p_t F_- + v \cdot \nabla_x F_- -E \cdot \nabla_v F_+ = Q(F_-,F_+) + Q(F_-,F_- ).
\end{split} \Ee
The collision operator between particles measures ``the change rate'' in binary hard sphere collisions and takes the form of
\Be\begin{split}\label{Q}
Q(F_{1},F_{2}) (v)&: = Q_\mathrm{gain}(F_1,F_2)-Q_\mathrm{loss}(F_1,F_2)\\
&: =\int_{\R^3} \int_{\S^2} 
|(v-u) \cdot \omega| [F_1 (v') F_2 (u') - F_1 (v) F_2 (u)]
 \dd \omega \dd u,
\end{split}\Ee  
where $u^\prime = u - [(u-v) \cdot \omega] \omega$ and $v^\prime = v + [(u-v) \cdot \omega] \omega$. The collision operator enjoys a collision invariance: for any measurable $G_1, G_2$,  
\Be\label{collison_invariance}
\int_{\R^{3}} \begin{bmatrix}1 & v & \frac{|v|^{2}-3}{2}\end{bmatrix} Q(G_1,G_1) \dd v = \begin{bmatrix}0 & 0 & 0 \end{bmatrix}, \quad \int_{\R^3} Q(G_1,G_2 ) = 0.
\Ee 
It is well-known that a global Maxwellian $\mu$ 
satisfies $Q(\cdot,\cdot)=0$ where
\Be\label{Maxwellian}
\mu(v):= \frac{1}{(2\pi)^{3/2}} \exp\bigg(
 - \frac{|v |^{2}}{2 }
 \bigg).
\Ee
Throughout this paper, let's use the notation
\Be \label{iota} \begin{split}
 \iota = +  \text{ or } -, \text{ and denote }
 -\iota = \begin{cases}
- &, \text{ if } \iota = + \\ +&, \text{ if } \iota = -.
\end{cases}
\end{split} \Ee


Being an important equation in both theoretic and application aspects, the Boltzmann equation has drawn attentions and
there have been a lot of research activities in analytic study of the equation.
Notably the nonlinear energy method has led to solutions of many open problems \cite{Guo_P, Guo_M} including global strong solution of both the VMB system and the VPB system, when the initial data are close to the
Maxwellian $\mu$.
One thing to note is that these results deal with idealized periodic domains or
whole space, in which the solutions can remain bounded in $H^k$ for large $k$. 

In many important physical applications, e.g. semiconductor and tokamak, the charged dilute
gas is confined within a container, and its interaction with the boundary often plays a crucial role both in physics and mathematics. So it's natural to consider the equation \eqref{2FVPB} in a bounded domain $\O$, and the interaction of the gas with the boundary is described by suitable boundary conditions \cite{CIP, Maxwell}. In this paper we consider one of the physical conditions, a so-called diffuse boundary condition:

\Be \label{diffuseF}
F_{\iota} (t,x,v) = c_\mu \mu (v) \int_{n(x) \cdot u > 0 } F_{\iota} (t,x,u) (n(x) \cdot u ) du \text{ for } (x,v) \in \gamma_-.
\Ee
Here, $\gamma_-: = \{(x,v) \in \p\O \times \R^3: n(x) \cdot v<0\}$ and $n(x)$ is the outward unit normal at a boundary point $x$. A number $c_\mu$ is chosen to be $\sqrt{2\pi}$ so that $c_\mu  \int_{n(x) \cdot u > 0 } \mu(u) (n(x) \cdot u ) du = 1$. Due to this normalization the distrubution of \eqref{diffuseF} enjoys a null flux condition at the boundary:
\Be \label{null_flux}
\int_{\mathbb R^3 }F_{\iota} (t,x,v)  (n(x) \cdot v ) du = 0 \text{ for } x \in \p \Omega.
\Ee
%
One can view this boundary condition as one of the ideal scattering model.

However, in general, higher regularity may not be expected for solutions of the Boltzmann equation in physical bounded domains. Such a drastic difference of solutions with boundaries had been
demonstrated as the formation and propagation of discontinuity in non-convex domains \cite{Kim11, EGKM}, and a non-existence of some second order derivatives at the boundary in convex domains \cite{GKTT1}. Evidently the nonlinear energy method is not generally available to the boundary problems.
In order to overcome such critical difficulty, Guo developed a $L^2$-$L^\infty$ framework in \cite{Guo10} to study global solutions of the Boltzmann equation with various boundary conditions. The core of the method lays in a direct approach (without taking derivatives) to achieve a pointwise bound using trajectory of the transport operator, which leads substantial development in various directions including \cite{EGKM2, EGKM, GKTT1, GKTT2}. There are also studies on different type of collisional plasma models such as a Fokker-Planck equation with some boundary conditions (for example, see \cite{HwangFP} and reference therein).


%
The main goal of the paper is to study the 2 species VPB system coupled of (\ref{2FVPB}) with (\ref{Field}) and (\ref{phibigF}), which describes the dynamics of electrons in the absence of a magnetic field.
From (\ref{collison_invariance}) and (\ref{null_flux}), a smooth solution of VPB with the diffuse BC (\ref{diffuseF}) preserves total mass:
\Be\label{conservation_mass}
\iint_{\O\times\R^{3}} F_\iota(t,x,v) \dd v \dd x \equiv  \iint_{\O\times\R^{3}} F_\iota(0,x,v) \dd v \dd x
\ \ \text{for all } \  t\geq 0.
\Ee
We assume that initially $F_0(x,v)$ satisfies
\Be\label{neutral_condition}
\iint_{\O \times \mathbb{R}^{3}} (F_+(0,x,v ) -F_-(0,x,v ))  \dd v \dd x = 0
. \ \ \   \text{(a neutral condition)}
\Ee
Then $\int_\O \left\{\int_{\mathbb{R}^{3}}( F_+(t,x,v)  - F_-(t,x,v))\dd v \right\} \dd x =0$ for all $t>0$ from (\ref{conservation_mass}). This zero-mean condition guarantees a solvability of the Poisson equation (\ref{phibigF}) with the Neumann boundary condition (\ref{0signcondition}).

There are some previous studies for the one-species VPB system (which is obtained by letting $F_- = 0 $) with physical boundary conditions. For example the time asymptotics of a solution to the VPB system is studied \cite{DD} under some a priori assumption on the solutions. In \cite{Michler} renormalized solutions (no uniqueness) were constructed for the VPB system with diffuse boundary condition. Recently in \cite{VPBKim} the authors constructed a unique global strong solution to the VPB system with diffuse boundary condition. They also had a weighted $W^{1,p}$, $3 < p < 6$ estimate for the solution of such system. This regularity result was later improved in \cite{CK} where the author obtained a weighted $W^{1,\infty}$ estimate for the solution under the appearance of an external field with a favorable sign condition $E \cdot n > 0$ on the boundary which will be explained later.

We consider a perturbation around $\mu$:
\Be
F_{\iota} =\mu + \sqrt \mu f_{\iota}.
\Ee
Then the corresponding problem is given by
\begin{eqnarray}
 \label{2fVPB}
\p_t f_+  + v \cdot \nabla_v f_+ - \nabla \phi \cdot \nabla_v f_+  + \frac{v}{2} \cdot \nabla \phi f_+  
  - \frac{2}{\sqrt \mu } Q(\sqrt \mu f_+ , \,u ) - \frac{1}{\sqrt \mu } Q(\mu,\sqrt \mu f_+) - \frac{1}{\sqrt \mu } Q(\mu , \sqrt \mu f_- )  
  \\ \notag = \Gamma(f_+,f_+ + f_-) - v \cdot \nabla \phi \sqrt \mu,
\\ \notag \p_t f_-   + v \cdot \nabla_v f_- + \nabla \phi \cdot \nabla_v f_-  - \frac{v}{2} \cdot \nabla \phi f_- 
  - \frac{2}{\sqrt \mu } Q(\sqrt \mu f_- , \,u ) - \frac{1}{\sqrt \mu } Q(\mu,\sqrt \mu f_-) - \frac{1}{\sqrt \mu } Q(\mu , \sqrt \mu f_+)  
  \\= \notag \Gamma(f_-,f_+ + f_-)  + v \cdot \nabla \phi \sqrt \mu,
  \\ \notag f(0,x,v) = f_0(x,v),
\\ \label{smallfphi}
-\Delta_x \phi(t,x) = \int_{\mathbb R^3 } \sqrt \mu (f_+ - f_-)  dv, \, \frac{\p \phi}{\p n } = 0 \text{ for } x \in \p \O,
\\
 \label{diffusef}
f_{\iota} (t,x,v) = c_\mu \sqrt \mu (v) \int_{n(x) \cdot u > 0 } \sqrt \mu(u) f_{\iota} (t,x,u) (n(x) \cdot u ) du \text{ for } (x,v) \in \gamma_-.
\end{eqnarray}

For $g = \begin{bmatrix}  g_1 \\ g_2 \end{bmatrix}, h=\begin{bmatrix}  h_1 \\ h_2 \end{bmatrix}$, let
\Be \label{L_decomposition}
Lg :=  -\frac{1}{\sqrt \mu } \begin{bmatrix}  2 Q(\sqrt \mu g_1, \mu ) + Q (\mu, \sqrt \mu ( g_1 + g_2 ) )
\\   2 Q(\sqrt \mu g_2, \mu ) + Q (\mu, \sqrt \mu ( g_1 + g_2 ) )
 \end{bmatrix} : = \nu (v) g - K g.
 \Ee
 Here the collision frequency is defined as
 \Be  \label{collision_frequency}
 \nu (v) := \frac{2}{\sqrt \mu } Q_{\mathrm{loss}} (\sqrt \mu , \mu ) : = 2 \int_{\mathbb{S}^{2}}\int_{%
\mathbb{R}^{3}}|(v-u) \cdot \omega| \mu(u)\mathrm{d}%
u\mathrm{d}\omega \sim \langle v\rangle, \Ee
It is well-known that for hard-sphere case,
\Be \begin{split} \notag
& \frac{1}{\sqrt \mu (v) } Q_{\mathrm{gain}} (\sqrt \mu g_1, \mu ) = \frac{1}{\sqrt \mu (v) } Q_{ \mathrm{gain}} ( \mu, \sqrt \mu g_1 ) = \int_{\mathbb R^3 } \mathbf k_2 (v,u ) g_1(u) du,
\\ & \frac{1}{\sqrt \mu (v) } Q_{ \mathrm{loss}} ( \mu, \sqrt \mu g_1 ) = \int_{\mathbb R^3 } \mathbf k_1 (v,u ) g_1(u) du,
\end{split} \Ee
with
\begin{equation}\label{k_estimate}
\begin{split}
 \mathbf{k}_{1}(v,u)=  & \pi  |v-u|   e^{-\frac{|v|^{2} +|u|^{2}}{4}} ,
\\
\mathbf{k}_{2}(v,u) =& \pi
|v-u|^{-1} e^{- \frac{|v-u|^2}{8}} 
e^{-   \frac{  | |v|^2- |u|^2   |^2}{8|v-u|^2}}.
\end{split}
\end{equation}%
Thus
\Be \begin{split} 
Kg & : = \begin{bmatrix} \frac{2}{\sqrt \mu } Q_{\mathrm{gain} } (\sqrt \mu g_1, \mu ) + Q(\mu, \sqrt \mu ( g_1 + g_2) )  \\   \frac{2}{\sqrt \mu } Q_{\mathrm{gain} } (\sqrt \mu g_2, \mu ) + Q(\mu, \sqrt \mu ( g_1 + g_2) )\end{bmatrix} 
\\ & := \begin{bmatrix} \int_{\mathbb R^3 } \mathbf{k}_2 (v,u) (3g_1(u) + g_2(u) ) du - \int_{\mathbb R^3} \mathbf{k}_1 (v,u) (g_1(u) + g_2(u) ) du \\   \int_{\mathbb R^3 } \mathbf{k}_2 (v,u) (3g_2(u) + g_1(u) ) du - \int_{\mathbb R^3} \mathbf{k}_1 (v,u) (g_1(u) + g_2(u) ) du \end{bmatrix}.
\end{split} \Ee
The nonlinear operator is defined as
\Be \label{Gamma_def}
\Gamma(g,h): =  : = \Gamma_\text{gain} (g,h ) - \Gamma_\text{loss} (g,h ):= \frac{1}{\sqrt \mu }  \begin{bmatrix}   Q_\text{gain}(\sqrt \mu g_1, \sqrt{\mu} ( h_1+h_2) - Q_\text{loss}(\sqrt \mu g_1, \sqrt{\mu} ( h_1+h_2) ) 
\\    Q_\text{gain}(\sqrt \mu g_2, \sqrt{\mu} ( h_1+h_2) - Q_\text{loss}(\sqrt \mu g_2, \sqrt{\mu} ( h_1+h_2) ) 
 \end{bmatrix}.
\Ee
Then for $f = \begin{bmatrix} f_+ \\ f_- \end{bmatrix} $, \eqref{2fVPB} becomes
\Be \label{systemf}
\p_t f + v \cdot \nabla_x f - q \nabla \phi \cdot \nabla_v f + q \frac{v}{2} \cdot \nabla \phi f + Lf = \Gamma(f,f) - q_1v \cdot \nabla \phi \sqrt \mu,
\Ee
where $q = \begin{bmatrix} 1 & 0 \\ 0 & -1  \end{bmatrix} $, and $q_1 = \begin{bmatrix} 1  \\ -1  \end{bmatrix} $.

 Let's clarify some notations.
 We denote  
\Be
\label{weight}
w_\vartheta(v) =  e^{\vartheta|v|^2}.
\Ee
The boundary of the phase space $
\gamma := \{ (x,v) \in \partial \Omega \times \mathbb R^3 \}$ can be decomposed as 
\begin{equation} \begin{split}
\gamma_- = \{ (x,v) \in \partial \Omega \times \mathbb R^3 : n(x) \cdot v < 0 \}, &\quad (\text{the incoming set}),
\\ \gamma_+ = \{ (x,v) \in \partial \Omega \times \mathbb R^3 : n(x) \cdot v > 0 \}, &\quad (\text{the outcoming set}),
\\ \gamma_0 = \{ (x,v) \in \partial \Omega \times \mathbb R^3 : n(x) \cdot v = 0 \}, &\quad (\text{the grazing set}).
\end{split} \end{equation}
For any function $z(x,v) : \bar \Omega \times \mathbb R^3 \to \mathbb R$, denote
\Be \notag
|z|_{2,+}^2 = \int_{\g_+} z^2 d\g, \quad |z|_{2,-}^2 = \int_{\g_-} z^2 d\g, \quad |z|_{\gamma, 2}^2 = \iint_{\p \Omega \times \mathbb R^3}  z^2 |n(x) \cdot v |dv dx 
\Ee
Now for any vector-valued function $f,g :  \Omega \times \mathbb R^3 \to \mathbb R^2 $, with $f = \begin{bmatrix} f_+ \\ f_- \end{bmatrix} $, and $g = \begin{bmatrix} g_+ \\ g_- \end{bmatrix} $, let's clarify the following notations:
\begin{eqnarray}
\\ \notag
|f| = |f_+| + |f_- |, \, f \cdot g = f_+ g_+ + f_- g_-, \text{ and } \langle f, g \rangle = \iint_{ \Omega \times \mathbb R^3} f \cdot g \, dvdx =\iint_{ \Omega \times \mathbb R^3} ( f_+ g_+ + f_-g_-  )\, dv dx,
\\ \notag f^p := \begin{bmatrix} f_+^p \\ f_-^p \end{bmatrix}, \, \int f = \begin{bmatrix} \int f_+ \\ \int f_- \end{bmatrix}, \, \partial f = \begin{bmatrix} \p f_+ \\ \p f_- \end{bmatrix},
  | f |_{p,+}^p : =\int_{\gamma_+ } | f|^p d \gamma \sim \int_{\gamma_+ } ( |f_+|^p + |f_-|^p) d\gamma, 
\\ \notag | f |_{p,-}^p : = \int_{\gamma_- } |f|^p d\gamma \sim \int_{\gamma_- } (| f_+|^p + |f_-|^p) d\gamma,  \text{ and } |f|_{\gamma,p}^p = \iint_{\p \Omega \times \mathbb R^3 } |f|^p |n(x) \cdot v |dv dx,
\\ \notag
\| f(t)\|_p^p :=\iint_{\Omega \times \mathbb R^3 } |f|^p dv dx \sim  \iint_{\Omega \times \mathbb R^3 } ( |f_+(t)|^p + |f_-(t) |^ p ) dv dx , \quad \| f (t)\|_\infty = \sup_{(x,v) \in \Omega \times \mathbb R^3 }  |f_+(t) | + | f_-(t) |.
\end{eqnarray}

\subsection{A New Distance Function}

Throughout this paper we extend $ \phi_f$ for a \textit{negative time}. Let 
\Be\label{negative_t_extension}
\phi_f(s,x,v) :=\phi_{f_0}(x,v)    \ \ \text{for} \ - \infty<s<0,
\Ee
where $ \phi_{f_0}(x,v)$ satisfies $- \Delta \phi_{f_0}(x,v) = \int_\R^3 (f_{0,+} - f_{0,-} )\sqrt \mu dv$.

The \textit{characteristics (trajectory)} is determined by the Hamilton ODEs for $f_+$ and $f_-$ separately
\Be\label{hamilton_ODE}
\frac{d}{ds} \left[ \begin{matrix}X_\iota^f(s;t,x,v)\\ V_\iota^f(s;t,x,v)\end{matrix} \right] = \left[ \begin{matrix}V_\iota^f(s;t,x,v)\\ 
{-\iota} \nabla_x \phi_f
(s, X_\iota^f(s;t,x,v))\end{matrix} \right]  \ \ \text{for}   - \infty< s ,  t < \infty  ,
\Ee
with $(X_\iota^f(t;t,x,v), V_\iota^f(t;t,x,v)) =  (x,v)$.

For $(t,x,v) \in \R  \times  \O \times \R^3$, we define \textit{the backward exit time} $\tbpm^f(t,x,v)$ as   
\Be\label{tb}
\tbpm^f (t,x,v) := \sup \{s \geq 0 : X_\iota^f(\tau;t,x,v) \in \O \ \ \text{for all } \tau \in (t-s,t) \}.
\Ee
Furthermore, we define $\xbpm^f (t,x,v) := X_\iota^f(t-\tbpm(t,x,v);t,x,v)$ and $\vbpm^f (t,x,v) := V_\iota^f(t-\tbpm(t,x,v);t,x,v)$.

\begin{definition}[Distance Function] For $\e>0$, for $\iota = +$ or $-$ as in \eqref{iota}, define
\Be\label{alphaweight}\begin{split}
\alpha_{f,\e,\iota}(t,x,v) : =& \  
\chi \Big(\frac{t-\tbpm^{f}(t,x,v)+\e}{\e}\Big)
|n(\xbpm^{f}(t,x,v)) \cdot \vbpm^{f}(t,x,v)| \\
&+ \Big[1- \chi \Big(\frac{t-\tbpm^{f}(t,x,v) +\e}{\e}\Big)\Big].
\end{split}\Ee
Here we use a smooth function $\chi: \R \rightarrow [0,1]$ satisfying
\Be\label{chi}
\begin{split}
\chi(\tau)  =0,  \     \tau\leq 0, \ \text{and} \  \ 
\chi(\tau)  = 1    ,  \  \tau\geq 1, \\ 
 \frac{d}{d\tau}\chi(\tau)  \in [0,4] \ \   \text{for all }   \tau \in \R.
\end{split}
\Ee
\end{definition}

\hide
\Be\label{weight}
\alpha(t,x,v) = 
\begin{cases} \ 
\mathbf{1}_{\tb(t,x,v) \leq t+1}
|n(\xb(t,x,v)) \cdot \vb(t,x,v)| 
 \ \ \ \text{for}  \ \tb(t,x,v)< \infty
 ,\\
 \ \ \  \ \ \ 
  \ \ \ \ \  \ \ \   \ |\vb(t,x,v)|  \  \ \  \ \ \  \ \ \ \ \ \  \ \ \  \ \ \  \ \ \  \text{for}  \ \tb(t,x,v)= \infty
.
\end{cases}
\Ee

\unhide

Note that $\alpha_{f,\e,\iota}(0,x,v)\equiv \alpha_{{f_0},\e,\iota}(0,x,v)$ is determined by $f_0$ and its extension (\ref{negative_t_extension}). For the sake of simplicity, we could drop the superscription $^f$ in $X_\iota^f, V_\iota^f, \tbpm^f, \xbpm^f, \vbpm^f$ unless they could cause any confusion.

Also, denote
\Be \label{matrixalpha}
\alpha_{f,\e}(t,x,v) := \begin{bmatrix} \alpha_{f, \e, +}(t,x,v) & 0 \\ 0 & \alpha_{f, \e, -}(t,x,v) \end{bmatrix},
\Ee
and let $|\alpha_{f,\e}(t,x,v) | :=  | \alpha_{f, \e, +}(t,x,v) | + | \alpha_{f, \e, -}(t,x,v)| $.

One of the crucial properties of the new distance function in (\ref{alphaweight}) is an invariance under the Vlasov operator: 
\Be\label{alpha_invariant}
\big[\p_t + v\cdot \nabla_x -\iota \nabla_x \phi_f \cdot \nabla_v \big] \alpha_{f,\e,\iota}(t,x,v) =0.
\Ee
This is due to the fact that the characteristics solves a deterministic system (\ref{hamilton_ODE}) (See the proof in the appendix). This crucial invariant property under the Vlasov operator is one of the key points in our approach.

It is important to note that a different version of the distance function which has been used in the author's previous paper \cite{CK} to establish the regularity of the one specie VPB system is not applicable here. In \cite{CK}, the weight $\tilde \alpha$ took the form
\Be \label{alphatilde}
\tilde \alpha(t,x,v) = \bigg[ |v \cdot \nabla \xi (x)| ^2 + \xi (x)^2 - 2 (v \cdot \nabla^2 \xi(x) \cdot v ) \xi(x) - 2(E(t,\overline x ) \cdot \nabla \xi (\overline x ) )\xi(x) \bigg]^{1/2}
\Ee
for $x \in \Omega$ close to boundary, where $\overline x := \{ \bar x \in \p \Omega :  d(x,\bar x ) = d(x, \partial \Omega) \}$ is uniquely defined. And $\xi$ was assumed to be a $C^3$ function $\xi : \mathbb R^3 \to \mathbb R$ such that $\Omega = \{ x \in \mathbb R^3: \xi(x) < 0 \}$, $\partial \Omega = \{ x\in \mathbb R^3 : \xi(x) = 0 \}$, and $\nabla \xi(x) \neq 0  \text{ when } |\xi(x) | \ll 1$. And the domain was assumed to be strictly convex:
\[
\sum_{i,j} \partial_{ij} \xi(x) \zeta_i \zeta_j \ge C_\xi |\zeta|^2 \, \text{ for all } \, \zeta \in \mathbb R^3 \text{ and for all } x\in \bar \Omega = \Omega \cup \partial \Omega.
\]
One of the crucial property this $\tilde \alpha$ enjoys is the velocity lemma:
\Be \label{velalphatilde}
 |\{  \p_t + v\cdot \nabla_x + E \cdot \nabla_v \} {\tilde \alpha}(t,x,v) | \lesssim |v| \tilde \alpha,
\Ee
when under the sign condition 
\Be \label{signcondition}
E\cdot n > \delta >  0, \text{ on } \p \O, 
\Ee where $n$ is the outward normal vector.
This can be seen by direct computation:
\Be \label{transderivbeta1}
 |\{  \p_t + v\cdot \nabla_x + E \cdot \nabla_v \} {\tilde \alpha}^2(t,x,v) | \sim |v| {\tilde \alpha}^2 +C_\xi(E,\nabla_x E ,\p_t E )  |v| \xi (x),
\Ee
for some bounded function $C_\xi$. Now under \eqref{signcondition}, we get an extra stronger control for $\xi(x)$ from $\tilde \alpha^2$, and therefore the second term on the right-hand side of \eqref{transderivbeta1} can be bounded by:
\Be \label{boundalphaxi}
C_\xi |v| \xi (x) \le \frac{C_\xi}{ \inf_{y \in \p \O} E(t,y) \cdot \nabla \xi (x) } |v| (E(t,\overline x ) \cdot \nabla \xi (\overline x ) ) \xi (x)  \le \frac{C_\xi}{ \delta} \tilde \alpha^2(t,x,v).
\Ee
Thus combing \eqref{transderivbeta1} and \eqref{boundalphaxi} we obtain \eqref{velalphatilde}. This means $\tilde \alpha(t,x,v)$ retains its full power under the transport operator, which is crucially used for establishing the theories in \cite{CK}.

Thus it's clear that without the last term in \eqref{alphatilde}, i.e. in the case $E\cdot \nabla \xi = 0$ on $\p \O$, in order to have the $\xi(x)$ control from the second term on the right hand side of \eqref{transderivbeta1}, we can only obtain 
\Be
 |\{  \p_t + v\cdot \nabla_x + E \cdot \nabla_v \} {\tilde \alpha}^2(t,x,v) | \lesssim |v| \tilde \alpha(t,x,v).
 \Ee
 Therefore $\tilde \alpha (t,x,v)$ suffers a loss of power under the transport operator, and would result it's been inapplicable for the situation here.


Therefore the previous distance function $\tilde \alpha$ would work only under a crucial favorable sign condition \eqref{signcondition}. But for the two species VPB system, it's clear from the equation \eqref{2FVPB} that if one requires the sign condition for the field for $F_+$, i.e. $-\nabla \phi \cdot n > 0$, then inevitably one would have $+\nabla \phi \cdot n < 0$, so the field for $F_-$ would fail to satisfy the sign condition. We note that the similar $\tilde \alpha$ has also been used by \cite{Guo_V}, \cite{Hwang} in the study of one-species problem of Vlasov equation.

Thus one of the major benefit for this new distance function $\alpha$ is that it only requires the zero-Neuuman boundary condition $E \cdot n = 0$ (see Lemma \ref{cannot_graze}, Proposition \ref{prop_int_alpha}), and therefore with $ \pm \nabla \phi \cdot n = 0 $ from \eqref{0signcondition}, we can apply this distance function to the two species VPB system \eqref{2FVPB}. 

 \subsection{Main Theorem}
 
 The main goal of this paper is the construction of a unique global \textit{strong} solution of the two species VPB system with the diffuse boundary condition when the domain is $C^3$ and \textit{convex.} Moreover an asymptotic stability of the global Maxwellian $\mu$ is studied. 

Here a $C^{3}$ domain means that for any ${p} \in \partial{\Omega}$, there exists sufficiently small $\delta_{1}>0, \delta_{2}>0$, and an one-to-one and onto $C^{3}$-map
	\begin{equation}\label{eta}
	\begin{split}
	\eta_{{p}}:  \{ x_{{ \parallel}} \in \mathbb{R}^{2}: |x_{ \parallel}| < \delta_1  \}  \ &\rightarrow  \ \p\Omega \cap B({p}, \delta_{2}),\\
	x_{{ \parallel}}=(x_{ \parallel,1},x_{\parallel,2} )	 \ &\mapsto \    \eta_{{p}}  (x_{ \parallel,1},x_{\parallel,2} ).
	\end{split}
	\end{equation} 
	A \textit{convex} domain means that there exists $C_\O>0$ such that for all $p \in \p\O$ and  $\eta_p$ and for all $x_\parallel$ in (\ref{eta}) 
\begin{equation}\label{convexity_eta}
\begin{split}
\sum_{i,j=1}^{2} \zeta_{i} \zeta_{j}\p_{i} \p_{j} \eta _{{p}}   ( x_{\parallel }  )\cdot  
n ( x_{\parallel } )
  \leq    - C_{\Omega} |\zeta|^{2}  \ 
  \text{ for all}   \ \zeta \in \mathbb{R}^{2}.
\end{split}
\end{equation}

\begin{theorem} 
\label{main_existence}
Assume a bounded open $C^3$ domain $\O \subset\R^3$ is convex (\ref{convexity_eta}). Let $0< \tilde{\vartheta}< \vartheta\ll1$. Assume the neutral condition \eqref{neutral_condition} and the compatibility condition 
\Be\label{compatibility_condition}
f_{0,\iota} (x,v) = c_\mu \sqrt{\mu(v)} \int_{n(x) \cdot u>0} f_{0,\iota} (x,u)\sqrt{\mu(u)} \{n(x) \cdot u\} \dd u   \ \ \text{on} \ \gamma_-.
\Ee
Then there exists a small constant $0< \e_0 \ll 1$ such that for all $0< \e \leq \e_0$ if an initial datum $F_0 = \mu+ \sqrt{\mu}f_0\geq 0$ satisfies
\Be\label{small_initial_stronger}
 \|w_\vartheta f_0 \|_{L^\infty(\bar{\O} \times \R^3)}< \e,\Ee
and, recall the matrix definition of $\alpha$ in \eqref{matrixalpha},
\Be\label{W1p_initial}
\begin{split}
 \| w_{\tilde{\vartheta}} \alpha_{f_0, \e }^\beta \nabla_{x,v } f_0 \|_{ {L}^{p } ( {\O} \times \R^3)}
 <\e
\ \
\text{for}  \  \ 3< p < 6, \ \ 
1-\frac{2}{p }
 < \beta<
\frac{2}{3}
,\end{split}
\Ee
and 
\Be\label{nabla_v_f_0_bounded}
 \|   w_{\tilde{\vartheta}}   \nabla_{v } f_0 \|_{ {L}^{3 } ( {\O} \times \R^3)}< \infty,
\Ee
\hide
\Be\label{W1p_initial}
 \| w_{\tilde{\vartheta}} \alpha_{f, \e }^\beta \nabla_{x,v} f_0 \|_{ {L}^{p} ( {\O} \times \R^3)}<\e
\ \ \ \text{for} \ \ 3<p<6, \ 
1-\frac{2}{p}
 < \beta<
\frac{3}{2}
  ,
\Ee\unhide
%
%
then there exists a unique global-in-time solution $(f, \phi_f)$ to (\ref{2fVPB}), (\ref{smallfphi}), (\ref{diffusef}) such that $F(t)= \mu+ \sqrt{\mu} f(t) \geq 0$. Moreover there exists $\lambda_{\infty} > 0$ such that 
\Be\begin{split}\label{main_Linfty}
 \sup_{ t \geq0}e^{\lambda_{\infty} t} \| w_\vartheta f(t)\|_{L^\infty(\bar{\O} \times \R^3)}+ 
 \sup_{ t \geq0}e^{\lambda_{\infty} t} \| \phi_f(t)  \|_{C^{2}(\O)}  \lesssim 1,
\end{split}\Ee
and, for some $C>0$,
\Be\label{W1p_main}
 \| w_{\tilde{\vartheta}} \alpha_{f, \e }^{\beta } \nabla_{x,v} f(t)  \|_{L^{ p} ( {\O} \times \R^3)} 
 \lesssim e^{Ct} \ \ \text{for all } t \geq 0
,
\Ee
and, for $0< \delta= \delta(p,\beta) \ll1$,
\Be\label{nabla_v f_31}
\| \nabla_v f (t) \|_{L^3_x (\O) L^{1+\delta }_v (\R^3)} \lesssim_t 1  \ \ \text{for all } \  t\geq 0.
\Ee

Furthermore, if $(f, \phi_f)$ and $(g, \phi_g)$ are both solutions to (\ref{2fVPB}), (\ref{smallfphi}), (\ref{diffusef})
then 
\Be\label{stability_1+}
\| f(t) - g(t) \|_{L^{1+\delta} (\O \times \R^3)} \lesssim_t \| f(0) - g(0) \|_{L^{1+\delta} (\O \times \R^3)} \ \ \text{for all } \  t\geq 0.
\Ee 

\end{theorem}

The proof of Theorem \ref{main_existence} devotes a nontrivial extension of the argument of \cite{VPBKim} now for the two species VPB system. One of the major difference here is the $L^2$ coercivity estimate.

We now illustrate the main ideas in the proof of Theorem \ref{main_existence} which largely follows the framework in \cite{VPBKim}. In the energy-type estimate of $\nabla_{x,v}f$ in $\alpha_{f,\e}^\beta$-weighted $L^p$-norm, the operator $v\cdot \nabla_x$ causes a boundary term to be controlled:
$
 \int^t_0 \int_{\p\O} \int_{n \cdot v\leq0} 
|\alpha_{f,\e}^\beta \nabla_{x,v} f| ^p
|n \cdot v|
\dd v \dd S_x\dd s.
$
It turns out this integrand is integrable if
\Be\label{beta_lower_intro}
\beta> \frac{p-2}{p}    \ \ \text{so that} \ \   |n \cdot v|^{p \beta - p + 1} \in L^1_{loc}(\R^3).
\Ee
%
On the other hand to control the terms in the bulk we need \textit{a bound of} $\phi_f (t)$ in $C^2_x.$ 
A key observation is that
%
\Be\label{bound_wp}
\left\|\int_{\R^3}\nabla_x f \sqrt{\mu} \dd v \right\|_{L^p_x(\O)}\lesssim  \sup_{x} \sum_{\iota = \pm } \left\|  \frac{ \sqrt{\mu}}{\alpha_{f,\e,\iota}^{ \beta}} \right\|_{L^{p^*} (\R^3)}    \left\|  \alpha_{f,\e}^\beta  \nabla_x f \right\|_{L^p(\O \times \R^3)}, \ \ \text{for} \ \frac{1}{p}+ \frac{1}{p^*}=1,
\Ee
which leads $C^{2,0+}$-bound of $\phi_f$ by the Morrey inequality for $p>3$ as long as 
\Be\label{alpha_integrable}
\alpha_{f,\e,\iota}^{- \beta p^*} \in L^1_{loc}(\R^3) \ \   \text{for some } \beta p^*> \frac{p-2}{p-1}.
\Ee
The proof of (\ref{alpha_integrable}) can be found in \cite{VPBKim}, where the authors employ a change of variables 
$
v  \mapsto (\xb^f(t,x,v) , \tb^f (t,x,v)),
$
and carefully compute and bound the determinant of the Jacobian matrix to get
\Be\label{alpha_bounded_intro}
\int_{|v| \lesssim 1} \alpha_{f,\e}^{- \beta p^*} \dd v \lesssim \int_{\text{boundary}} \frac{|(x- \xb^f) \cdot n(\xb^f)|^{1- \beta p^*}}{|x-\xb^f|^{3- \beta p^*}} \dd \xb^f
+ \text{good terms}< \infty, 
\Ee
which turns to be bounded as long as $\beta p^*<1$. 

In order to run the $L^2$-$L^\infty$ bootstrap argument we need to prove the $L^2$ coercivity property of the solution $f$ (Proposition \ref{l2coercivity}).  
This is one of the major difference from \cite{VPBKim}, as here for the two species VPB system,
the null space of the linear operator $L$ in \eqref{L_decomposition} is a six-dimensional subspace of $L^2_v(\mathbb R^3; \mathbb R^2 )$ spanned by orthonormal vectors
\Be 
\left\{ \begin{bmatrix} \sqrt \mu \\ 0  \end{bmatrix}, \begin{bmatrix} 0  \\ \sqrt \mu \end{bmatrix}, \begin{bmatrix} \frac{v_i}{\sqrt 2 } \sqrt \mu \\ \frac{v_i}{\sqrt 2 } \sqrt \mu   \end{bmatrix}, \begin{bmatrix} \frac{|v|^2 - 3}{2\sqrt 2} \sqrt \mu \\ \frac{|v|^2 - 3}{2\sqrt 2} \sqrt \mu   \end{bmatrix}
 \right\}, \, i = 1,2,3,
\Ee
(see Lemma 1 from \cite{Guo_M} for the proof). And the projection of $f$ onto the null space $N(L)$ can be denoted by
\Be
\mathbf Pf(t,x,v) := \left\{ a_+(t,x) \begin{bmatrix} \sqrt \mu \\ 0  \end{bmatrix} + a_-(t,x) \begin{bmatrix} 0  \\ \sqrt \mu \end{bmatrix} + b(t,x)  \cdot \frac{v}{\sqrt 2 } \begin{bmatrix} \sqrt \mu \\ \sqrt \mu  \end{bmatrix} + c(t,x)  \frac{|v|^2 - 3}{2\sqrt 2}\begin{bmatrix} \sqrt \mu \\ \sqrt \mu  \end{bmatrix}
\right\}.
\Ee
Using the standard $L^2$ energy estimate of the equation, it is well-known (See \cite{Guo_M}) that $L$ is degenerate: $\langle Lf, f \rangle  \gtrsim \| \nu^{1/2} (I - \mathbf P) f \|_{L^2_{(\O \times \mathbb R^3) }}$. Thus it's clear that in order to control the $L^2$ norm of $f(t)$, we need a way to bound the missing $  \| \mathbf P(t) \|_{L^2} $ term.

From there we adopt the ideas from \cite{EGKM} and apply it to our setting (two species system). By using weak formulation of the equation \eqref{systemf}, we properly choose a set of test functions:
		\Be \label{tests}
		\begin{split}
			\psi_{a}  &\equiv  \begin{bmatrix} - (|v|^{2}-\beta_{a} )\sqrt{\mu }v\cdot\nabla_x\varphi _{a_+} \\  - (|v|^{2}-\beta_{a} )\sqrt{\mu }v\cdot\nabla_x\varphi _{a_-} \end{bmatrix} ,  \\
			\psi^{i,j}_{b,1} &\equiv  \begin{bmatrix} (v_{i}^{2}-\beta_ b)\sqrt{\mu }\partial _{j}\varphi _{b}^{j} \\  (v_{i}^{2}-\beta_ b)\sqrt{\mu }\partial _{j}\varphi _{b}^{j} \end{bmatrix}, \quad i,j=1,2,3,   \\
			\psi^{i,j}_{b,2} &\equiv \begin{bmatrix} |v|^{2}v_{i}v_{j}\sqrt{\mu }\partial _{j}\varphi _{b}^{i}(x) \\ |v|^{2}v_{i}v_{j}\sqrt{\mu }\partial _{j}\varphi _{b}^{i}(x) \end{bmatrix},\quad i\neq j,  \\
			\psi_c &\equiv \begin{bmatrix} (|v|^{2}-\beta_c )\sqrt{\mu }v \cdot \nabla_x \varphi_{c}\\ (|v|^{2}-\beta_c )\sqrt{\mu }v \cdot \nabla_x \varphi_{c} \end{bmatrix},  \\
		\end{split}
		\Ee
		where $\varphi_{a_{\pm}}(t,x)$, $\varphi_{b}(t,x)$, and $\varphi_{c}(t,x)$ solve
		\Be\begin{split}\label{phi_abc}
			- \Delta \varphi_{a_\pm} &= a_\pm (t,x),  \quad \p_{n}\varphi_{a_\pm} \vert_{\p\O} = 0,  \\
			- \Delta \varphi_b^j &= b_j(t,x),   \ \ \varphi^j_b|_{\p\O} =0,\  \text{and}  \ - \Delta \varphi_c  = c(t,x),   \ \ \varphi_c|_{\p\O} =0,
		\end{split}
		\Ee
		and carefully choose $\beta_a=10$, $\beta_b=1$, and $\beta_c= 5$ to satisfy \eqref{defbeta}. Integrating against those test functions $\int_0^t \langle \phi , \eqref{systemf} \rangle $, we can nicely extract the $L^2$ norms of the $N(L)$ projections of $f$: $\| a_\pm(t) \|_{L^2}^2, \| b(t) \|_{L^2}^2, \| c(t) \|_{L^2}^2 $ through the term $ \langle v \cdot \nabla_x f ,\phi \rangle$. And therefore we recover the bound
for the missing $ \| \mathbf Pf(t) \|_{L^2}^2$ term from the $L^2$ energy estimate of $f$.

Finally we use $L^2$-$L^\infty$ bootstrap argument to derive an exponential decay of $f$ in $L^\infty$. The main idea here is to control $f_+$ and $f_-$ separately along their trajectories $(X_+(s), V_+(s)  ) $ and $(X_-(s), V_-(s)  ) $ by using the double Duhamel expansion, and then use change of variables to get the $L^2$ bound. But here as we are working with the two species system, it's important to note that in the process of the double Duhamel expansion, a mix of trajectories would occur \eqref{double_teration_double}. That is if we start with either $\iota = +$ or $-$,  both the $f_+$ and $f_-$ terms would appear in the first Duhamel expansion of $f_\iota$. From there we perform the second Duhamel expansion by expanding $f_+$ along $(X_+(s), V_+(s)  )$, and expanding $f_-$ along $(X_-(s), V_-(s)  ) $. And then we treat them using two different change of variables
\Be
u \mapsto X_+(s^\prime;s,X_\iota(s;t,x,v), u), \quad
		 u \mapsto X_-(s^\prime;s,X_\iota(s;t,x,v), u)
		\Ee
accordingly to get the bound with $\| f_+\|_{L^2} + \|f_- \|_{L^2 }  $ in the bulk. But thanks to the $L^2$ coercivity (Proposition \ref{l2coercivity}) which gives control to the whole $\|f\|_{L^2} $, we can take the sum $\sum_{\iota = \pm } |f_\iota | $ and close the estimates.

%
%

\section{preliminary}

In this section, we give some basic estimates of initial-boundary problems of the transport equation in the presence of a time-dependent field $E(t,x)$, and $f$ here is assumed to be a scalar valued function $f(t,x,v): [0,\infty) \times \Omega \times \mathbb R^3 \to \mathbb R $ satisfies
\Be\label{transport_E}
\p_t f + v\cdot \nabla_x f + E \cdot \nabla_v f + \psi f = H,
\Ee
where $H=H(t,x,v)$ and $\psi= \psi(t,x,v)\geq 0$. We assume that $E$ is defined for all $t \in \R$. Throughout this section $(X(s;t,x,v),V(s;t,x,v))$ denotes the characteristic which is determined by (\ref{hamilton_ODE}) with replacing $- \iota \nabla_x \phi_f$ by $E$. 

\begin{lemma}\label{cannot_graze}Assume that $\O$ is convex (\ref{convexity_eta}). 
Suppose that $\sup_t\| E(t) \|_{C^1_x} < \infty$ and 
	\Be\label{nE=0}
	n(x) \cdot E(t,x) =0 \ \ \text{for } x \in \p\O \ \text{and for all  } t.
	\Ee
	Assume $(t,x,v) \in \R_+ \times \bar{\O} \times \R^3$ and $t+1 \geq \tb(t,x,v)$. If $x \in \p\O$ then we further assume that $n(x) \cdot v > 0$. Then we have 
	\Be
	n(\xb(t,x,v)) \cdot \vb(t,x,v) <0.\label{no_graze}
	\Ee
\end{lemma}

\begin{proof}
The proof is the same as that of Lemma 1 in \cite{VPBKim}. But since we are going to use some of the argument for later purpose, let's present the proof here.
\hide  \textit{Step 1.}  We claim that for all $(t,x) \in [ 0, \infty) \times \bar{\O}$ as $N \rightarrow \infty$
	\Be\label{conv_NN}
	\mathbf{1}_{\tb(t,x,u)< N}
	\mathbf{1}_{n(\xb(t,x,u)) \cdot \vb(t,x,u) < -\frac{1}{N}}
	\nearrow \mathbf{1}_{\tb(t,x,u)< \infty} \ \ \text{almost every } u \in \R^3.
	\Ee
	First we prove that, for fixed $N \in \mathbb{N}$, as $M \rightarrow \infty$
	\Be\label{conv_NM}
	\mathbf{1}_{\tb(t,x,u)< N}
	\mathbf{1}_{n(\xb(t,x,u)) \cdot \vb(t,x,u) < -\frac{1}{M}}
	\nearrow \mathbf{1}_{\tb(t,x,u)< N} \ \ \text{almost every } u \in \R^3.
	\Ee
	Since $\mathbf{1}_{\tb(t,x,u)< N}$ converges to $\mathbf{1}_{\tb(t,x,u)< \infty}$ as $N \rightarrow \infty$ we can apply Cantor's diagonal argument to conclude (\ref{conv_NN}) from (\ref{conv_NM}).\unhide
	

	\textit{Step 1.} Note that locally we can parametrize the trajectory (see Lemma 15 in \cite{GKTT1} or \cite{KL2} for details). We consider local parametrization (\ref{eta}). We drop the subscript $p$ for the sake of simplicity. If $X(s;t,x,v)$ is near the boundary then we can define $(X_n, X_\parallel)$ to satisfy 
	\Be\label{X_local}
	X(s;t,x,v)  =   \eta (X_\parallel (s;t,x,v)) + X_n(s;t,x,v) [- n(X_\parallel(s;t,x,v))].
	\Ee
	
	\hide $\eta$ near $\p\O$ such that 
	\Be
	\eta : (x,y,z)\in B(0,r_{1})\cap \R^{3}_{+} \mapsto B(p,r_{2})\cap\O,
	\Ee
	where $\eta(0)=p\in\p\O$ and $\eta(x,y,0) \in \p\O$ for some $r_{1}, r_{2} > 0$. Then for $X(s;t,x,v)$, there exist a unique $x_{*}\in B(p,r_{2})\cap \p\O$ such that
	\Be \label{X def}
	|X(s;t,x,v)-x_{*}| \leq \sup_{x\in B(p,r_{2})\cap \p\O} |X(s;t,x,v)-x|,
	\Ee 
	since $\eta$ is bijective. We define 
	\[
	(X_{\parallel},0) = \eta^{-1}(x_{*}) \quad \text{and} \quad  X_{n} := |X(s;t,x,v)-x_{*}|.
	\]\unhide
	For the normal velocity 
	we define
	\Be\label{def_V_n}
	V_{n}(s;t,x,v) := V(s;t,x,v)\cdot [-n(X_{\parallel}(s;t,x,v))].
	\Ee
	We define $V_{\parallel}$ tangential to the level set $  \big( \eta(X_{\parallel}) + X_{n}(-n(X_{\parallel})) \big)$ for fixed $X_{n}$. Note that 
	\[
	\frac{\p   \big( \eta(x_{\parallel} ) + x_{n}(-n(x_{\parallel})) \big) }{\p {x_{\parallel, i}}}  \perp n(x_{\parallel}) \ \  \text{for} \ i=1,2.
	\]
	We define $(V_{\parallel,1}, V_{\parallel,2})$ as 
	\Be\label{def_V_parallel}
	V_{\parallel, i} :=\Big( V - V_{n}[-n(X_{\parallel})]\Big) \cdot   \Big( 
	\frac{\p   \eta(X_{\parallel} )}{\p {x_{\parallel, i}} } + X_{n}\Big[- \frac{\p n(X_{\parallel})}{\p x_{\parallel, i}}\Big] \Big)  .
	\Ee
	\hide \Be\notag
	\begin{split} 
		&\sum_{i=1,2} V_{\parallel,i} \p_{i} \big( \eta(X_{\parallel},0) + X_{n}(-n(X_{\parallel})) \big)  \\
		&= \nabla_{\parallel} \big( \eta(X_{\parallel},0) + X_{n}(-n(X_{\parallel})) \big) V_{\parallel} = V - V_{n}(-n(X_{\parallel}))  .
	\end{split}
	\Ee\unhide
	Therefore we obtain
	\Be \label{V_local}
	V(s;t,x,u)   = V_n [- n(X_\parallel)] + V_\parallel \cdot \nabla_{x_\parallel} \eta (X_\parallel ) 
	- X_n V_\parallel \cdot \nabla_{x_\parallel} n (X_\parallel).
	\Ee  
	
	Directly we have 
	\Be \begin{split}\notag
		\dot{X}(s;t,x,u) &=\dot{X}_{\parallel} \cdot \nabla_{x_\parallel}\eta (X_\parallel) + \dot{X}_n [- n(X_\parallel)] - X_{n}\dot{X}_{\parallel}  \cdot \nabla_{x_\parallel} n(X_{\parallel}) .
	\end{split}\Ee
	Comparing coefficients of normal and tangential components, we obtain that 
	\Be\label{dot_Xn_Vn}
	\dot{X}_{n}(s;t,x,v) = V_{n}(s;t,x,v) , \ \  \dot{X}_{\parallel}(s;t,x,v) = V_{\parallel}(s;t,x,v).
	\Ee
	
	On the other hand, from (\ref{V_local}),
	\Be \begin{split} \label{Vdotn}
		\dot{V} (s) &=  \dot{V}_{n} [-n(X_{\parallel})] - V_{n} \nabla_{x_\parallel} n(X_{\parallel})\dot{X}_{\parallel} + V_{\parallel}\cdot\nabla^{2}_{x_\parallel}\eta(X_{\parallel}) \dot{X}_{\parallel} + \dot{V}_{\parallel} \cdot \nabla_{x_\parallel}\eta(X_{\parallel})  \\
		&\quad - \dot{X}_{n}\nabla_{x_\parallel} n(X_{\parallel})V_{\parallel} - X_{n}\nabla_{x_\parallel} n(X_{\parallel})\dot{V}_{\parallel} - X_{n} V_{\parallel}\cdot\nabla_{x_\parallel}^{2}n(X_{\parallel})\dot{X}_{\parallel}. 
	\end{split}\Ee
	From $(\ref{Vdotn})\cdot [-n(X_{\parallel})]$, (\ref{dot_Xn_Vn}), and $\dot{V}=E$, we obtain that 
	\Be \begin{split}\label{hamilton_ODE_perp}
		\dot{V}_n (s) 
		&=  [V_\parallel (s)\cdot \nabla^2 \eta (X_\parallel(s)) \cdot V_\parallel(s) ] \cdot n(X_\parallel(s))  
		+  E (s , X (s ) ) \cdot [-n(X_\parallel(s)) ] \\
		&\quad - X_n (s) [V_\parallel(s) \cdot \nabla^2 n (X_\parallel(s)) \cdot V_\parallel(s)]  \cdot n(X_\parallel(s)) .
	\end{split}\Ee
	
	\vspace{4pt}
	
	\textit{Step 2.} We prove (\ref{no_graze}) by the contradiction argument. Assume we choose $(t,x,v)$ satisfying the assumptions of Lemma \ref{cannot_graze}. Let us assume
	\Be\label{initial_00}
	X_n (t-\tb;t,x,v)  +V_n (t-\tb;t,x,v) =0.
	\Ee
	
	First we choose $0<\e \ll 1$ such that $X_n(s;t,x,v) \ll 1$ and 
	\Be\label{Vn_positive}
	V_n (s;t,x,v) \geq0 \ \ \text{for} \  t- \tb(t,x,v)<s<t-\tb(t,x,v) + \e.
	\Ee
	The sole case that we cannot choose such $\e>0$ is when there exists $0< \delta\ll1$ such that $V_n(s;t,x,v)<0$ for all $s \in ( t-\tb(t,x,v), t-\tb(t,x,v) + \delta)$. But from (\ref{dot_Xn_Vn}) for $s \in ( t-\tb(t,x,v), t-\tb(t,x,v) + \delta)$,
	$$
	0 \leq X_n(s;t,x,v)   =  X_n(t-\tb(t,x,v);t,x,v)  +  \int^s_{t-\tb(t,x,v)} V_n (\tau; t,x,v) \dd \tau <  0.$$
	
	Now with $\e>0$ in (\ref{Vn_positive}), temporarily we define that $t_* := t-\tb(t,x,v) + \e$, $x_* = X(t-\tb(t,x,v) + \e; t,x,v),$ and $v_* = V(t-\tb(t,x,v) + \e; t,x,v)$. Then $(X_n(s;t,x,v), X_\parallel (s;t,x,v)) = (X_n(s; t_*, x_*, v_*), X_\parallel (s; t_*, x_*, v_*))$ and \\$(V_n(s;t,x,v), V_\parallel (s;t,x,v)) = (V_n(s; t_*, x_*, v_*), V_\parallel (s; t_*, x_*, v_*))$.
	\hide
	From (\ref{hamilton_ODE}), for $(X_n(s), X_\parallel (s)) = (X_n(s; t_*, x_*, u_*), X_\parallel (s; t_*, x_*, u_*))$ and $(V_n(s), V_\parallel (s)) = (V_n(s; t_*, x_*, u_*), V_\parallel (s; t_*, x_*, u_*))$, 
	\Be \begin{split}\label{hamilton_ODE_perp}
		\dot{X}_n (s)  =& V_n (s),\\
		\dot{V}_n (s)  =& 
		[V_\parallel (s)\cdot \nabla^2 \eta (X_\parallel(s)) \cdot V_\parallel(s) ] \cdot n(X_\parallel(s))  
		+   \nabla\phi (s , X (s ) ) \cdot n(X_\parallel(s))  \\
		&+   X_n (s) [V_\parallel(s) \cdot \nabla^2 n (X_\parallel(s)) \cdot V_\parallel(s)]  \cdot n(X_\parallel(s)) .
	\end{split}\Ee\unhide
	
	Now we consider the RHS of (\ref{hamilton_ODE_perp}). From (\ref{convexity_eta}), the first term $[V_\parallel(s) \cdot \nabla^2 \eta (X_\parallel(s)) \cdot V_\parallel(s) ] \cdot n(X_\parallel(s))\leq 0$. By an expansion and (\ref{nE=0}) we can bound the second term 
	\Be\begin{split}\label{expansion_E}
		&E (s , X(s )) \cdot n(X_\parallel(s ) )\\
		=&   \ E (s , X_n(s ), X_\parallel(s ) ) \cdot n(X_\parallel (s )) \\
		=& \  E (s , 0, X_\parallel(s ) ) \cdot n(X_\parallel (s )) 
		+ \| E (s) \|_{C_x^1}  O( |X_n(s )| )\\
		=&  \ \| E (s) \|_{C_x^1}  O( |X_n(s )| ).
	\end{split}\Ee
	From (\ref{hamilton_ODE}) and assumptions of Lemma \ref{cannot_graze},
	$$|V_\parallel (s;t,x,v )|\leq |v| + \tb(t,x,v) \| E \|_\infty  \leq |v| +  (1+t) \| E \|_\infty.$$ 
	Combining the above results with (\ref{hamilton_ODE_perp}), we conclude that 
	\Be\notag
	\dot{V}_n(s;t_*,x_*,v_*) \lesssim  ( |v| + (1+ t) \| E\|_\infty  )^2X_n(s;t_*,x_*,v_*) ,
	\Ee
	and hence from (\ref{dot_Xn_Vn}) for $t-\tb(t,x,v)\leq s \leq t_*$,
	\Be\label{ODE_X+V}
	\begin{split}
		&\frac{d}{ds} [X_n (s;t_*,x_*,v_* )  +V_n (s ;t_*,x_*,v_*) ]\\
		\lesssim & \ ( |v| + (1+ t) \| E\|_\infty  )^2  [X_n (s;t_*,x_*,v_* )  +V_n (s;t_*,x_*,v_* ) ].\end{split}
	\Ee
	By the Gronwall inequality and (\ref{initial_00}), for $t-\tb(t,x,v)\leq s \leq t_*$,
	\Be \begin{split}\notag
		& [X_n (s;t_*,x_*,v_*)  +V_n (s;t_*,x_*,v_*) ]  \\
		\lesssim & \   [X_n (t-\tb(t,x,u))  +V_n (t-\tb(t,x,u)) ] e^{C \e ( |v| + (1+ t) \| E\|_\infty  )^2) }\\
		=&  \ 0.
	\end{split}\Ee 
	
	From (\ref{Vn_positive}) we conclude that $X_n (s;t,x,v) \equiv 0$ and $V_n (s;t,x,v) \equiv 0$ for all $s \in [t-\tb(t,x,u), t-\tb(t,x,u) + \e]$. We can continue this argument successively to deduce that $X_n (s;t,x,v) \equiv 0$ and $V_n (s;t,x,v) \equiv 0$ for all $s \in [t-\tb(t,x,v), t]$. Therefore $x_n =0 = v_n$ which implies $x \in \p\O$ and $n(x) \cdot v =0$. This is a contradiction since we chose $n(x) \cdot v>0$ if $x \in \p\O$.\end{proof}

\begin{lemma} \label{le:ukai} 
	Assume 
	that, for $\Lambda_1>0$, $\delta_1>0$, 
	\begin{equation}
	\label{decay_E}
	\sup_{t \geq 0} e^{\Lambda_1 t} \| E(t) \|_{\infty} \leq \delta_1 \ll1 .
	\end{equation}
	We also assume $\frac{1}{C}\langle v\rangle \leq \psi(t,x,v)\leq C \langle v\rangle$ for some $C>0$. For $\varepsilon$ satisfying
	\Be\label{lower_bound_e}
	\e> \frac{2\delta_1}{  \Lambda_1}>0,
	\Ee
	there exists a
	constant $C_{\delta_1, \Lambda_1 ,\Omega }>0$ 
	such that, for all $t\geq 0$,
	\begin{equation} \label{case:decay}
	\begin{split}
	&\int_{0}^{t}\int_{\gamma _{+}\setminus \gamma _{+}^{\varepsilon }}|h|\mathrm{%
		d}\gamma \mathrm{d}s\\
	&\leq C_{\delta_1, \Lambda_1 ,\Omega }\left\{  
	||h_{0}||_{1}+\int_{0}^{t}  \| h(s)\|_{1}+\big{\Vert} 
	[
	\partial
	_{t}+v\cdot \nabla _{x}+E \cdot \nabla_v + \psi ]h(s)\big{\Vert} _{1} \mathrm{d}s  \right\}.\end{split}
	\end{equation}
	If $E \in L^\infty$ does not decay but
	\begin{equation} \label{nondecay}
	\| E (t) \|_{\infty} \leq \delta,
	\end{equation}
	then for $\varepsilon > 0$,
	\begin{equation} \label{case:nondecay} 
	\begin{split}
	&\int_{0}^{t}\int_{\gamma _{+}\setminus \gamma _{+}^{\varepsilon }}|h|\mathrm{%
		d}\gamma \mathrm{d}s\\
	&\leq C_{\delta, t, \varepsilon, \Omega }\left\{ \
	\|h_{0}\|_{1}+\int_{0}^{t}  \| h(s)\|_{1}+\big{\Vert} 
	[
	\partial
	_{t}+v\cdot \nabla _{x}+E\cdot \nabla_v + \psi ]h(s)\big{\Vert} _{1} \mathrm{d}s\ \right\},
	\end{split}
	\end{equation}
	where we have time-dependent constant $C_{\delta,t, \varepsilon, \O}>0$. \hide Furthermore, for any $(t,x,v)$ in $[0, \infty)\times \Omega \times \mathbb{R}^{3}$
	the function $
	h
	(t+s^{\prime },
	X(t+s^\prime;t,x,v)
	,
	V(t+s^\prime;t,x,v))
	$ is absolutely continuous in
	$s^{\prime } \in (-\min \{t_{\mathbf{b}}(t,x,v),t\}, t_{%
		\mathbf{f}}(t,x, v) )$.\unhide
\end{lemma}
\begin{proof}
See the proof of Lemma 6 in \cite{VPBKim}.
\end{proof}

\begin{lemma} [Green's identity]  \label{lem_Green}
	For $p \in [1, \infty)$, we assume $f \in L^{p}_{loc} (\R_+ \times \Omega \times \mathbb R^3 )$ satisfies
	\Be\notag
	\partial_t f + v \cdot \nabla_x f + E \cdot \nabla_v f  \in L^p_{loc} (\R_+; L^p (\Omega \times \mathbb R^3 ) ),  \ \
	f   \in L^p_{loc} (\R_+; L^p (\gamma_+ ) ).	
	\Ee
	Then $f \in C^0_{loc}( \R_+ ; L^p (\Omega \times \mathbb R^3 ) )$ and $f  \in L^p_{loc} (\R_+; L^p (\gamma_-) )$. 
	
	Moreover
	\Be \label{Greedid}
	\begin{split}
		\| f(T)\|_p ^p + \int_0^{T} |f|_{ p,+ } ^p &= \| f(0) \|_p^p + \int_0^{T} |f|_{ p,- }^p  \\
		&  +p \int_0^{T} \iint _{\Omega \times \mathbb R^3 }   \{ \partial_t + v \cdot \nabla_x f + E \cdot \nabla_v f \} |f|^{p-2} f.
	\end{split} 
	\Ee
\end{lemma}

\begin{proof}
See the proof of Lemma 5 in \cite{VPBKim}.
\end{proof}

\begin{proposition} \label{inflowprop}
	Assume the compatibility condition
	\Be\label{compatibility_inflow}
	f_0(x,v) = g(0,x,v)\quad  \text{for} \quad (x,v) \in \gamma_- .
	\Ee
	Let $p \in [1, \infty )$ and $0 < \vartheta < 1/4$. \hide
	
	Define $D^\epsilon : = \{ (x,v) \in \Omega \times \mathbb R^3 : (\xb, \vb) \in \gamma_- \setminus \gamma_-^\epsilon \} .$ \unhide
	Assume
	\Be \begin{split}\label{assumption_inflow}
		\nabla_x f_0 , \nabla_v f_0    \in L^p (\Omega \times \mathbb R^3 ),
		\\ \nabla_{x,v} \tb  \partial_t g, \nabla_{x,v} \vb \nabla_v g, \nabla_{x,v} \xb \partial_{\xb } g , \nabla_{x,v} \tb \psi g \in L^p ( [0, T] \times \gamma_- ) ,
		\\ \hide\frac{ n }{ n \cdot \vb } \Big \{ \partial_t g + \sum_{i =1}^2 ( \vb \cdot \tau_i ) \partial_{\tau_i } g  + \nu g - H +  E \cdot \nabla_v g \Big \} 
		\\ +  \frac{ n \cdot \iint \partial_x E }{ n \cdot \vb } \Big \{ \partial_t g + \sum_{i =1}^2 ( \vb \cdot \tau_i) \partial_{ \tau_i } g   - \nu g + H \Big \} \in L^p ([0, T ] \times \gamma_- ),
		\\ \unhide  \nabla_x H ,\nabla_v H \in L^p  ([0, T ] \times \Omega \times \mathbb R^3 ),
		\\ 
		e^{- \vartheta |v|^2 } \nabla_x \psi, e^{-\vartheta |v|^2 } \nabla_v \psi \in L^p ([ 0, T ] \times \Omega \times \mathbb R^3 ),
		\\ e^{\vartheta |v|^2 } f_0 \in L^\infty ( \Omega \times \mathbb R^3 ) , e^{\vartheta |v|^2 } g \in L^\infty ( [0, T] \times \gamma_- ),
		\\ e^{\vartheta |v|^2 } H \in L^\infty ([0, T] \times \Omega \times \mathbb R^3 ).
	\end{split} \Ee
	Then for any $T > 0$, there exists a unique solution $f$ to (\ref{transport_E}) such that $\nabla_{x,v} f\in  C^0 ([0,T] ; L^p (\O \times \R^3)) \cap L^1((0,T); L^p (\gamma))$.
	
	\begin{proof}
See the proof of Proposition 2 in \cite{VPBKim}.
\end{proof}
	
	\hide$f$ to 
	\[ \{ \partial_t + v \cdot \nabla_x + E \cdot \nabla_v + \nu \} f = H \]
	such that $f, \partial_t, \nabla_x f ,\nabla_v f \in C^0( [ 0, T ] ; L^p (\O \times \R^3) ) $ and their traces satisfy
	\[ \begin{split}
	\partial_t f|_{\gamma_-
	}  = \partial_t g, \nabla_v f |_{\gamma
	} = \nabla_v g, \nabla_x f|_{\gamma_- 
	} = \nabla_x g, \quad \text{on} \quad \gamma_- 
	,
	\\ \nabla_x f(0,x,v) = \nabla_x f_0, \nabla_v f(0,x,v) = \nabla_v f_0, \quad \text{in} \quad\O \times \R^3,
	\\ \partial_t f(0,x,v) = \partial_t f_0, \quad \text{in} \quad \O \times \R^3.
	\end{split} \]
	
	Moreover
	\Be\label{inflow_energy}
	\begin{split}
		&\| \nabla_{x,v} f(t) \|_p^p + \int^t_0 | \nabla_{x,v} f |_{+,p}^p \\
		= & \  \| \nabla_{x,v} f _0\|_p^p + \int^t_0 | \nabla_{x,v} g |_{-,p}^p\\
		&+ p \int^t_0 \iint_{\O \times \R^3} \{ \nabla_{ x,v} H - [\nabla_{x,v} v] \nabla_x f - [\nabla_{x,v} \psi] f \}
		|\nabla_{ x,v} f|^{p-2} \nabla_{ x,v}f.
	\end{split}
	\Ee

	\unhide\end{proposition}

\begin{lemma}\label{est_X_v}
	Assume $E(t,x) \in C^1_x$ is given and (\ref{decay_E})
	and
	\Be\label{decay_phi_2}
	\sup_{t\geq 0}e^{\Lambda_2 t} \| \nabla_x E(t)\|_{\infty} \leq \delta_2 \ll1,
	\Ee
	with $\Lambda_2+ \delta_2+ \e \leq 1$.
	\hide
	\Be\begin{split}\notag
		\frac{d}{ds} 
		\begin{bmatrix}
			X(s;t,x,v)\\
			V(s;t,x,v)
		\end{bmatrix} = 
		\begin{bmatrix}
			V(s;t,x,v)\\
			- \nabla_x \phi (s,X(s;t,x,v))
		\end{bmatrix}.
	\end{split}\Ee\unhide
	Then there exists $C>0$ such that 
	\Be \label{result_X_v}
		|\nabla_vX(s;t,x,v)|
	 \leq Ce^{C \delta_2 (\Lambda_{2})^{-2}   } 
		|t-s|
	, \ \ \text{for all}   \ \max(t - \tb(t,x,v), - \e )\leq s \leq t.
	\Ee  
\end{lemma}

\begin{proof}
See the proof of Lemma 9 in \cite{VPBKim}.
\end{proof}

\section{$L^\infty$ estimate}

Let $\iota = + $ or $-$ as in \eqref{iota}. We set $F^{0}_\iota(t,x,v) \equiv \mu$ and $\phi^0\equiv 0$. We then apply proposition \ref{inflowprop} for $\ell = 0,1,2...$ to get a sequence $F^\ell$ such that 
		\Be
		\begin{split}\label{Fell}
			\p_t F_\iota^{\ell+1} + v\cdot \nabla_x F_\iota^{\ell+1} -\iota \nabla {\phi^\ell} \cdot \nabla_v F_\iota^{\ell+1}= Q_{\text{gain}} (F_\iota^\ell, F_\iota^\ell + F^\ell_{-\iota}) - Q_{\text{loss}} (F_\iota^{\ell+1} ,F_\iota^\ell + F^\ell_{-\iota} ),\\
			- \Delta \phi^\ell= \int_{\R^3} F_+^\ell - F_-^\ell \dd v, \ \ \int_\O \phi^\ell \dd x =0, \ \ \frac{\p \phi^\ell}{\p n}\Big|_{\p\O }=0,
		\end{split}\Ee
		and, on $(x,v) \in \gamma_-$,
		\Be\label{F_ell_BC}
		F_\iota^{\ell+1}(t,x,v) = c_\mu \mu \int_{n(x) \cdot v>0} F_\iota^\ell(t,x,u) \{n(x) \cdot u\} \dd u,
		\Ee
		and $F_\iota^{\ell+1} (0,x,v) = {F_0}_\iota (x,v)$.

				Then $f^{\ell+1}$ solves 
		\Be\label{fell_local}
		\begin{split}
			&[\p_t + v\cdot\nabla_x -  q \nabla_x \phi^{\ell} \cdot \nabla_v + \nu +  q \frac{v}{2} \cdot\nabla \phi^\ell  
			] f^{\ell+1}\\
			=&  \  K f^{\ell}-  q_1 v\cdot \nabla \phi^\ell  \sqrt{\mu} + \Gamma_{\text{gain}} (f^\ell, f^\ell) -  \Gamma_{\text{loss}} (f^{\ell+1},f^\ell),\\
			& - \Delta \phi^\ell = \int_{\R^3}( f^\ell_+ - f^\ell_-) \sqrt{\mu} \dd v , \ \ \int_\O \phi^\ell \dd x =0, \ \ \frac{\p \phi^\ell}{\p n}\Big|_{\p\O }=0,
		\end{split}
		\Ee 
		
		Denote the characteristics $(X_\iota^\ell, V_\iota^\ell)$ which solves
		\Be\label{XV_ell}
		\begin{split}
			\frac{d}{ds}X_\iota^\ell(s;t,x,v) &= V_\iota^\ell(s;t,x,v),\\
			\frac{d}{ds} V_\iota^\ell(s;t,x,v) &= {-\iota} \nabla \phi^\ell (s, X_\iota^\ell(s;t,x,v)).\end{split}
		\Ee
		
		\Be
		\begin{split}\label{cycle}
			t^{\ell}_{1,\iota} (t,x,v)&:= 
			\sup\{ s<t:
			X_\iota^\ell(s;t,x,v) \in \p\O
			\}
			,\\
			x^\ell_{1,\iota} (t,x,v ) &:= X_\iota^\ell (t^{\ell}_{1,\iota} (t,x,v);t,x,v)
			,\\
			t^{\ell-1}_{2,\iota} (t,x,v, v_1) &:= \sup\{ s<t^\ell_{1,\iota}:
			X^{\ell-1}_\iota (s;t^{\ell}_{1,\iota} (t,x,v),x^{\ell}_{1,\iota} (t,x,v),v_1) \in \p\O
			\}
			,\\
			x^{\ell-1}_{2,\iota} (t,x,v, v_1 ) &:= X^{\ell-1}_\iota (t^{\ell-1}_{2,\iota} (t,x,v,v_1);t^\ell_{1,\iota}(t,x,v),x^\ell_{1,\iota}(t,x,v),v_1)
			,\\
		\end{split}
		\Ee
		and inductively 
		\Be
		\begin{split}\label{cycle_ell}
			& t^{\ell-(k-1)}_{k,\iota} (t,x,v, v_1, \cdots, v_{k-1})  \\
			&:=     \sup\big\{ s<t^{\ell-(k-2)}_{k-1,\iota} 
			:
			X_\iota^{\ell-(k-1)}(s;t_{k-1,\iota}^{\ell - (k-2)}  , x_{k-1,\iota}^{\ell - (k-2)} ,v_{k-1}) \in \p\O
			\big\},\\
			& x_{k,\iota}^{\ell - (k-1)} (t,x,v, v_1, \cdots, v_{k-1})\\
			&:= X_\iota^{\ell- (k-1)} (t_{k,\iota}^{\ell- (k-1)}; t_{k-1,\iota}^{\ell- (k-2)},x_{k-1,\iota}^{\ell- (k-2)} , v_{k-1})
			.
		\end{split}
		\Ee
		Here,
		\Be\begin{split}\notag
			t^{\ell-(i-1)}_{i,\iota } &:= t^{\ell-(i-1)}_{i,\iota }
			(t,x,v,v_1, \cdots, v_{i-1}),\\
			x^{\ell-(i-1)}_{i,\iota} &:= x^{\ell-(i-1)}_{i ,\iota}
			(t,x,v,v_1, \cdots, v_{i-1}).\end{split}\Ee

\begin{proposition}
 Assume that for sufficiently small $M>0$, such that 
 \Be \label{wf0small}
 \|w_\vartheta f_0 \|_\infty < \frac{M}{2},
 \Ee
  then there exits $T^*(M) > 0 $ such that 
		\Be\label{uniform_h_ell}
		\sup_{0 \leq t \leq T^*} \max_{\ell} \| w_{\vartheta}  f^\ell (t) \|_\infty \leq M.
		\Ee
		\end{proposition}
\begin{proof}
		 We define 
		\Be
		\label{h}
		h^{\ell} (t,x,v) : = w_{\vartheta}(v) f^\ell (t,x,v).
		\Ee
		By an induction hypothesis we assume 
		\Be \label{hypAssump}
			\sup_{0 \leq t \leq T^*}\| h^{\ell}(t) \|_\infty\leq M.
		\Ee
		Then $h^{\ell+1}$ solves 
		\Be\label{hell_local}
		\begin{split}
			&[\p_t + v\cdot\nabla_x - q \nabla_x \phi^{\ell} \cdot \nabla_v + \nu + q \frac{v}{2} \cdot\nabla \phi^\ell  
			- q \frac{\nabla_x \phi^\ell \cdot \nabla_v w_{\vartheta}}{w_{\vartheta}}
			] h^{\ell+1}\\
			=&  \  K_{w_{\vartheta}} h^{\ell}- q_1 v\cdot \nabla \phi^\ell w_{\vartheta} \sqrt{\mu} + w_{\vartheta}\Gamma_{\text{gain}} (\frac{h^\ell}{w_{\vartheta}}, \frac{h^\ell}{w_{\vartheta}}) -   w_{\vartheta}\Gamma_{\text{loss}} (\frac{h^{\ell+1}}{w_{\vartheta}}, \frac{h^\ell}{w_{\vartheta}}),\\
			& - \Delta \phi^\ell = \int_{\R^3} (f_+^\ell- f_-^\ell) \sqrt{\mu} \dd v , \ \ \int_\O \phi^\ell \dd x =0, \ \ \frac{\p \phi^\ell}{\p n}\Big|_{\p\O }=0,
		\end{split}
		\Ee 
		where $K_{w_{\vartheta} }( \ \cdot \ )=w_{\vartheta}  K( \frac {1}{w_{\vartheta} } \ \cdot)$. The boundary condition is
		\Be\label{bdry_local}
		h_\iota^{\ell+1} |_{\gamma_-} = c_\mu w_{\vartheta} \sqrt{\mu} \int_{n \cdot u>0}h_\iota^\ell w_{\vartheta}^{-1}\sqrt{\mu}   \{n \cdot u\} \dd u. 
		\Ee

		\hide
		The perturbation of $F^{\ell+1} = \mu + \sqrt{\mu} f^{\ell+1}$ solves
		\Be\label{fell_local}
		\begin{split}
			&[\p_t + v\cdot\nabla_x - \nabla_x \phi^{\ell} \cdot \nabla_v + \nu + \frac{v}{2} \cdot\nabla \phi^\ell  ] f^{\ell+1}\\
			=&  \ Kf^{\ell}- v\cdot \nabla \phi^\ell \sqrt{\mu} + \Gamma_{\text{gain}} (f^\ell, f^\ell) -   \Gamma_{\text{loss}} (f^\ell, f^{\ell+1}),\\
			& - \Delta \phi^\ell = \int_{\R^3} f^\ell \sqrt{\mu} \dd v , \ \ \int_\O \phi^\ell \dd x =0, \ \ \frac{\p \phi^\ell}{\p n}\Big|_{\p\O }=0,
		\end{split}
		\Ee 
		and, on $(x,v) \in \gamma_-$,
		\Be\label{bdry_local}
		f^{\ell+1}(t,x,v) = c_\mu \sqrt{\mu} \int_{n(x) \cdot v>0} f^\ell (t,x,u)\sqrt{\mu(u)} \{n(x) \cdot u\} \dd u,
		\Ee
		and $f^{\ell+1} (0,x,v) = f_0 (x,v)$.

		\unhide
		
		We define
		\Be\label{nu_ell} \begin{split}
		\nu^\ell (t,x,v) &: = \begin{bmatrix} \nu^\ell_+ (t,x,v) & 0 \\ 0 & \nu^\ell_- (t,x,v) \end{bmatrix}  
		: = \begin{bmatrix} \nu(v) + \frac{v}{2} \cdot \nabla \phi^\ell - \frac{\nabla_x \phi^\ell \cdot \nabla_v w_{\vartheta}}{w_{\vartheta}}  & 0 
		\\ 0 & \nu(v) - \frac{v}{2} \cdot \nabla \phi^\ell + \frac{\nabla_x \phi^\ell \cdot \nabla_v w_{\vartheta}}{w_{\vartheta}} )
		\end{bmatrix}.
		\end{split}  \Ee
		From (\ref{hypAssump}), for $M\ll 1$, $\| \nabla \phi^\ell \|_\infty \ll1$ and hence
		   \\
		\Be\label{lower_nu_l}
		\nu_\iota^\ell (t,x,v)\geq \frac{4 \nu_0}{5} \langle v\rangle  .
		\Ee
		Let 
		\Be\label{g_ell}
		g^\ell : = -   q_1v\cdot \nabla \phi^\ell \sqrt{\mu} +  \Gamma_{\mathrm{gain}} (\frac{h^\ell}{w_{\vartheta}}, \frac{h^\ell}{w_{\vartheta}})  :=\begin{bmatrix} g^\ell_+ \\ g^\ell_- \end{bmatrix}.
		\Ee
		Note that 
		\Be\label{bound_g_ell}
		| w_{\vartheta} g^\ell |  \lesssim  \| h^\ell \|_\infty +  \langle v\rangle  \| h^\ell \|_\infty^ 2,
		\Ee
		where we have used
		\Be\label{est_Gamma}
		|w_{\vartheta}\Gamma(\frac{h}{w_{\vartheta}}, \frac{h}{w_{\vartheta}})| \lesssim \langle v\rangle \| h \|_\infty^2.
		\Ee

		Consider the trajectories of $h_+^{\ell+1}$ and $h_-^{\ell+1}$ separately from (\ref{hell_local}),
		\Be\begin{split}\label{duhamel_local}
			&\frac{d}{ds} \Big\{ e^{- \int^t_s \nu_\iota^\ell (\tau, X_\iota^\ell (\tau), V_\iota^\ell (\tau)) \dd \tau}
			h_\iota^{\ell+1} (s, X_\iota^\ell(s;t,x,v), V_\iota^\ell (s;t,x,v)) 
			\Big\}\\
			=& \ e^{- \int^t_s \nu_\iota^\ell (\tau, X_\iota^\ell (\tau), V_\iota^\ell (\tau)) \dd \tau}
			\Big\{K_{w_{\vartheta},\iota} h^\ell (s, X_\iota^\ell (s), V_\iota^\ell(s))+ w_{\vartheta} g_\iota^\ell (s, X_\iota^\ell (s), V_\iota^\ell(s))\Big\}.
		\end{split}\Ee
		
		From (\ref{duhamel_local}) and (\ref{bdry_local}), we have 
		\Be
		\begin{split}\label{h_ell_local}
			& h_\iota^{\ell+1} (t,x,v)\\
			= & \ \mathbf{1}_{t_{1,\iota}^\ell   \leq 0}e^{-      \int^t_0     \nu_\iota^\ell 
			} h_\iota^{\ell+1} (0,X_\iota^\ell (0),V_\iota^\ell (0) )   \\
			&+ \int_{\max\{t_{1,\iota}^\ell ,0\}}^{t}e^{-  \int^t_s     \nu_\iota^\ell } [K_{w_{\vartheta},\iota }h^\ell  +w_{\vartheta} g_\iota^\ell](s,X_\iota^\ell(s;t,x,v), V_\iota^\ell(s;t,x,v))\dd s \\
			&+ \mathbf{1}_{t_{1,\iota}^\ell  \geq  0}e^{-      \int^t_{t_{1,\iota}^\ell}     \nu_\iota^\ell 
			} h^{\ell+1}(t_{1,\iota}^\ell,X_\iota^\ell (t_{1,\iota}^\ell  ;t,x,v  ), V_\iota^\ell (t_{1,\iota}^\ell  ;t,x,v  )).
		\end{split}
		\Ee
		We define
		\begin{equation}
		\tilde{w}_{\vartheta} (v)\equiv \frac{1}{ w_{\vartheta}  (v)\sqrt{\mu (v)}}
		.\label{tweight}
		\end{equation}
		From (\ref{bdry_local}),
		\Be\begin{split}\notag
			\textit{the last line of} \ (\ref{h_ell_local}) = \mathbf{1}_{t_{1,\iota}^\ell  \geq  0}e^{-      \int^t_{t_{1,\iota}^\ell}     \nu_\iota^\ell 
			}  \frac{1}{\tilde{w_{\vartheta}} (V_\iota^\ell (t^\ell_{1,\iota}))} \int_{n(x_{1,\iota}^\ell)\cdot v_1>0} h_\iota^{\ell}    (t^\ell_{1,\iota}, x^\ell_{1,\iota}, v_1)
			\tilde{w}_{\vartheta} (v_1) 
			c_\mu \mu  \{n(x_{1,\iota}^\ell) \cdot v_1\} \dd v_1.
		\end{split}\Ee
		
		We define $\mathcal{V}(x)=\{v \in\mathbb{R}^3 : n(x)\cdot v >0\}$
		with a probability measure $\dd\sigma=\dd\sigma(x)$ on $\mathcal{V}(x)$ which is given by
		\begin{equation}
		\dd\sigma \equiv
		c_\mu
		\mu
		(v)\{n(x)\cdot v\}\dd v.\label{smeasure}
		\end{equation}
		Let
		\Be\label{mathcal_V}
		\mathcal{V}_{j,\iota}: = \{v_j \in \R^3: n(x^{\ell - (j-1)}_{j,\iota}) \cdot v_j >0\}.
		\Ee
		
		Then inductively we obtain from (\ref{h_ell_local}), (\ref{duhamel_local}) and (\ref{bdry_local}), 
		\Be\begin{split}\label{h_iteration}
			&|h_\iota^{\ell+1} (t,x,v)| \\
			\leq & \ \mathbf{1}_{t_{1,\iota}^\ell   \leq 0}e^{-      \int^t_0     \nu_\iota^\ell 
			}|h^{\ell+1} (0,X_\iota^\ell (0),V_\iota^\ell (0))|    \\
			&  +
			\int_{\max\{t_{1,\iota}^\ell ,0\}}^{t}e^{-  \int^t_s     \nu_\iota^\ell }|[K_{w_{\vartheta} }h^\ell  +w_{\vartheta} g^\ell](s,X_\iota^\ell(s;t,x,v), V_\iota^\ell(s;t,x,v))|\dd s   \\
			&    +\mathbf{1}_{t_{1,\iota}^\ell >0}   
			\frac{
				e^{-  \int^t_{t_{1,\iota}^\ell}     \nu_\iota^\ell }
			}{\tilde{w_{\vartheta}}_{\varrho}(V_\iota^\ell (t_{1,\iota}^\ell))}\int_{\prod_{j=1}^{k-1}\mathcal{V}_{j,\iota}}|H| , 
		\end{split}\Ee
		where $|H|$ is bounded by
		\begin{eqnarray}
		&&\sum_{l=1}^{k-1}\mathbf{1}_{\{t^{\ell-l}_{l+1,\iota}\leq
			0<t_{l,\iota}^{\ell - (l-1)}\}}   |h^{\ell-l} (0,  
		X_\iota^{\ell-l}(0
		; v_l
		)
		,V_\iota^{\ell-l}(0; v_l))|\dd\Sigma _{l,\iota}(0) \label{h1} \\
		&+& \sum_{l=1}^{k-1}\int_{\max\{ t_{l+1,\iota}^{\ell-l}, 0 \}}^{t_{l,\iota}^{\ell - (l-1)}}\mathbf{1}_{\{t_{l+1,\iota}^{\ell-l}\leq
			0<t_{l,\iota}^{\ell - (l-1)}\}} \nonumber
		\\
		&& \ \ \   \times
		|[K_{w_{\vartheta} }h^{\ell-l} +w_{\vartheta}   g^{\ell-l}](s,
		X_\iota^{\ell-l}(s; v_l), V_\iota^{\ell-l}(s; v_l)
		|\dd \Sigma
		_{l,\iota}(s)\dd s  \label{h2} \\
		&+& \mathbf{1}_{\{0<t_{k,\iota}^{\ell - (k-1)}\}}|h^{\ell-(k-1)} (t_{k,\iota}^{\ell - (k-1)},x_{k,\iota}^{\ell - (k-1)},v_{k-1})|\dd\Sigma
		_{k-1,\iota}(t_{k,\iota}^{\ell - (k-1)}),  \label{h5}
		\end{eqnarray}%
		where
		\Be\begin{split}
			\label{measure} 
			d\Sigma _{l,\iota}^{k-1}(s) &= \{\Pi _{j=l+1}^{k-1}\dd\sigma _{j,\iota}\}\times\{
			e^{
				-\int^{t_{l,\iota}}_s
				\nu^\ell_{\iota} 
			}
			\tilde{w_{\vartheta}}(v_{l})\dd\sigma _{l,\iota}\}\times \Pi
			_{j=1}^{l-1}\{{{
					e^{
						-\int^{t_{j,\iota}}_{t_{j+1,\iota}}
						\nu^j_\iota 
					} \frac{ w_{\vartheta} (v_{j,\mathbf b})  \sqrt {\mu (v_{j,\mathbf b})}}{w_{\vartheta} (v_{j}) \sqrt  {\mu (v_j)} }
					\dd\sigma _{j,\iota}}}\},
		\end{split} \Ee 
		and 
		\Be\begin{split}\label{X_ell_l}
			X_\iota^{\ell-l} (s; v_l)&:= X_\iota^{\ell-l}(s;t_{l,\iota}^{\ell- (l-1)},x_{l,\iota}^{\ell- (l-1)},v_l),\\
			V_\iota^{\ell-l} (s;v_l)&:= V_\iota^{\ell-l}(s;t_{l,\iota}^{\ell- (l-1)},x_{l,\iota}^{\ell- (l-1)},v_l), \\
			v_{j,\mathbf{b}} &: = V_\iota^{\ell - j } ( t_{j,\iota}^{\ell - (j - 1 ) }-\tb;t_{j,\iota}^{\ell - (j - 1 ) }, x_{j,\iota}^{\ell - (j - 1 ) }, v_j ).
		\end{split}\Ee

		\vspace{2pt}
		
		\textit{Step 2-2. } We claim that there exist $T>0$ and $k_{0} >0$ such that for all $k\geq  k_{0}$ and for all $(t,x,v) \in [0,T] \times \bar{\O} \times \R^{3}$, we have 
		\Be\label{small_k}
		\int_{\prod_{j=1}^{k-1} \mathcal{V}_{j,\iota}} \mathbf{1}_{\{ t_{k,\iota}^{\ell - (k-1)} (t,x,v,v^{1}, \cdots , v^{k-1}) >0 \}} \dd \Sigma_{k-1,\iota}^{k-1} \lesssim_{\O} \Big\{\frac{1}{2}\Big\}^{ k/5}.
		\Ee
		The proof of the claim is a modification of a proof of Lemma 14 of \cite{GKTT1}.
		
		For $0<\delta\ll 1$ we define
		\Be\label{V_j^delta}
		\mathcal{V}_{j,\iota}^{\delta} := \{ v_j \in \mathcal{V}_{j,\iota} : |v_j \cdot n(x_{j,\iota}^{\ell -(j-1)})| > \delta, \ |v_j| \leq \delta^{-1} \}.
		\Ee 
		
		Choose   
		\Be\label{large_T}
		T= \frac{2}{\delta^{2/3} (1+ \| \nabla \phi \|_\infty)^{2/3}}.
		\Ee
		We claim that 
		\Be\label{t-t_lowerbound}
		|t_{j,\iota}^{\ell - (j-1) } -t_{j+1,\iota}^{\ell - j } |\gtrsim   \delta^3, \ \ \text{for} \ v_j \in \mathcal{V}^\delta_{j,\iota},  \ 0 \leq t\leq T, \ 0 \leq t_{j,\iota}^{\ell - (j-1)}.
		\Ee
		
		For $j \geq 1$,
		\Bes
		&&\Big| \int^{t_{j+1,\iota}^{\ell - j}}_{t_{j,\iota}^{\ell - (j-1) }} V^{\ell - j }_\iota(s;t_{j,\iota}^{\ell - (j-1) }, x_{j,\iota}^{\ell -(j-1)}, v_j) \dd s   \Big|^{2}\\
		&=& |x_{j+1,\iota}^{\ell -j} -x_{j,\iota}^{\ell -(j-1)}|^{2}\\
		&\gtrsim& |(x_{j+1,\iota}^{\ell -j} -x_{j,\iota}^{\ell -(j-1)}) \cdot n(x_{j,\iota}^{\ell -(j-1)})|\\
		&=&\Big| \int^{t_{j+1,\iota}^{\ell - j}}_{t_{j,\iota}^{\ell - (j-1)}}  V^{\ell-j}_\iota(s;t_{j,\iota}^{\ell - (j-1) },x_{j,\iota}^{\ell -(j-1)},v_j) \cdot  n(x_{j,\iota}^{\ell -(j-1)})  \dd s 
		\Big|\\
		&=&\Big| \int_{t_{j,\iota}^{\ell - (j-1) }}^{t_{j+1,\iota}^{\ell - j}}  
		\Big(
		v_j - \int^{s}_{t_{j,\iota}^{\ell - (j-1) }} \nabla \phi^{\ell - j} (\tau, X^{\ell - j}_\iota(\tau;t_{j,\iota}^{\ell - (j-1)},x_{j,\iota}^{\ell -(j-1)},v_j))    \dd \tau
		\Big)\cdot   n(x_{j,\iota}^{\ell -(j-1)}) 
		\dd s 
		\Big|\\
		&\geq& |v_j \cdot n(x_{j,\iota}^{\ell -(j-1)})| |t_{j,\iota}^{\ell - (j-1) }-t_{j+1,\iota}^{\ell - j}|
		- \Big|
		\int^{t_{j+1,\iota}^{\ell - j}}_{t_{j,\iota}^{\ell -(j-1)}} \int^{s}_{t_{j,\iota}^{\ell -(j-1)}}  \nabla \phi^{\ell -j } (\tau, X^{\ell -j}_\iota(\tau; t_{j,\iota}^{\ell - (j-1) },x_{j,\iota}^{\ell -(j-1)},v_j)) \cdot n(x_{j,\iota}^{\ell -(j-1)})\dd \tau \dd s 
		\Big|.
		\Ees
		Here we have used the fact if $x,y \in \p\O$ and $\p\O$ is $C^2$ and $\O$ is bounded then $|x-y|^2\gtrsim_\O |(x-y) \cdot n(x)|$. Hence
		\Be\label{lower_tb}
		\begin{split}
			|v_j \cdot n(x_{j,\iota}^{\ell -(j-1)})| 
			& \lesssim \frac{1}{|t_{j,\iota}^{\ell - (j-1) } - t_{j+1,\iota}^{\ell - j}| } \Big| \int^{t_{j+1,\iota}^{\ell - j}}_{t_{j,\iota}^{\ell - (j-1) }}
			V^{\ell - j}_\iota(s;t_{j,\iota}^{\ell - (j-1) },x_{j,\iota}^{\ell -(j-1)},v_j) \dd s
			\Big|^{2}\\
			& \ \  + 
			\frac{1}{|t_{j,\iota}^{\ell - (j-1) } - t_{j+1,\iota}^{\ell - j}| } \Big| \int^{t_{j+1,\iota}^{\ell - j}}_{t_{j,\iota}^{\ell - (j-1) }} 
			\int^{s}_{t_{j,\iota}^{\ell - (j-1) }} 
			\nabla \phi^{\ell-j} (\tau, X^{\ell-j}_\iota(\tau;t_{j,\iota}^{\ell - (j-1) },x_{j,\iota}^{\ell -(j-1)},v_j) )\cdot n(x_{j,\iota}^{\ell -(j-1)})\dd \tau  \dd s
			\Big|\\
			& \lesssim 
			|t_{j,\iota}^{\ell - (j-1) } - t_{j+1,\iota}^{\ell - j}|   \big\{ |v_j|^2 + |t_{j,\iota}^{\ell - (j-1)} - t_{j+1,\iota}^{\ell - j}|^3 \|\nabla \phi\|^2_\infty
			\\
			& \ \ \ \ \ \ \ \  \ \ \ \ \ \  \ \ \ \     +    \frac{1}{2}\sup_{t_{j+1,\iota}^{\ell - j} \leq \tau \leq t_{j,\iota}^{\ell - (j-1) }} 
			|  \nabla \phi^{\ell-j} (\tau, X^{\ell-j}_\iota(\tau;t_{j,\iota}^{\ell - (j-1) },x_{j,\iota}^{\ell -(j-1)},v_j) )\cdot n(x_{j,\iota}^{\ell -(j-1)})|
			\big\}.
		\end{split}
		\Ee
		For $v_j \in \mathcal{V}^\delta_{j,\iota}$, $0 \leq t\leq T$, and $t_{j,\iota}^{\ell - (j-1)}\geq 0$,
		\Be\notag
		|v_j \cdot n(x_{j,\iota}^{\ell -(j-1)})|\lesssim |t_{j,\iota}^{\ell - (j-1)} -t_{j+1,\iota}^{\ell - j}|
		\{
		\delta^{-2} + T^3 \| \nabla \phi^{\ell-j} \|_\infty^2 + \| \nabla \phi^{\ell-j} \|_\infty
		\}.
		\Ee
		We choose $T$ as (\ref{large_T}) then prove (\ref{t-t_lowerbound}).

		Therefore if $t_{k,\iota}^{\ell-(k-1)}  \geq 0$ then there can be at most $\left\{\left[\frac{C_\Omega}{%
			\delta^3}\right]+1\right\}$ numbers of $v_m \in \mathcal{V}_{m,\iota}^\delta$ for $1\leq m
		\leq k-1$. Equivalently there are at least $k-2- \left[\frac{C_\Omega}{%
			\delta^3}\right]$ numbers of $v_{i} \in \mathcal{V}_{i,\iota} \backslash \mathcal{V}_{i,\iota}^\delta$ for $0 \leq i \leq m$. 
		
		Let us choose $k=N \times \left(\left[\frac{C_\Omega}{\delta^3}\right]%
		+1\right)$ and $N=\left(\left[\frac{C_\Omega}{\delta^3}\right]%
		+1\right)   \gg C>1$. Then we have 
		\begin{eqnarray*}
			&&\int_{\prod_{j=1}^{k-1}\mathcal{V}_{j,\iota}} 1_{\{t_{k,\iota}^{\ell - (k-1)}(t,x,v,v^{1},\cdots, v^{k-1})
				> 0\}} d\Sigma_{k-1}^{k-1} \\
			&\leq& \sum_{m=1}^{\left[\frac{C_\Omega}{\delta^3}\right]+1} \int_{\Big\{\begin{array}{ccc}{\text{there are exactly } m \text{ of } v_{i} \in \mathcal{V}_{i,\iota}^\delta} \\ {\small \text{ and } \ k-1-m \
						of \ v_{i} \in \mathcal{V}_{i,\iota} \backslash \mathcal{V}_{i,\iota}^{\delta}}\end{array}\Big\}}
			\prod_{j=1}^{k-1} C_0 \mu(v_j)^{1/4} \mathrm{d} v_j \\
			&\leq& \sum_{m=1}^{\left[\frac{C_\Omega}{\delta^3}\right]+1} \left(%
			\begin{array}{ccc}
				k-1  \\
				m
			\end{array}%
			\right) \left\{ \int_{\mathcal{V}}C_0 \mu(v)^{1/4}\mathrm{d} v\right\}^{m} \left\{ \int_{%
				\mathcal{V}\backslash \mathcal{V}^\delta}C_0 \mu(v)^{1/4}\mathrm{d} v\right\}^{k-1-m} \\
			&\leq& \left(\left[\frac{C_\Omega}{\delta^3}\right]+1\right) \{k-1\}^{\left[%
				\frac{C_\Omega}{\delta^3}\right]+1} \{ \delta\}^{k-2-\left[\frac{%
					C_\Omega}{\delta^3}\right]} \left\{ \int_{\mathcal{V}} C_0\mu(v)^{1/4}
			dv\right\}^{\left[\frac{C_\Omega}{\delta^3}\right]+1} \\
			&\leq& \{CN\}^{\frac{k}{N}} \left\{\frac{k}{N}\right\}^{\frac{k}{N}} \left\{\frac{k}{N}\right\}^{-\frac{k}{N} \frac{N^2}{20}}
			\leq  \left\{\frac{k}{N}\right\}^{\frac{k}{N} \left(- \frac{N^2}{20} + 3\right) } 
			\leq
			\left\{\frac{1}{2}\right\}^{ k}
			,
		\end{eqnarray*}
		where we have chosen $k=N \times \left(\left[\frac{C_\Omega}{\delta^3}\right]%
		+1\right)$ and $N=\left(\left[\frac{C_\Omega}{\delta^3}\right]%
		+1\right)   \gg C>1$.
		
		\vspace{2pt}
		
		\textit{Step 2-3. }  
				We define a notation
		\begin{equation}\label{kzeta}
		\mathbf{k}_{  \varrho}(v,u) := \frac{1}{|v-u| } \exp\left\{- {\varrho} |v-u|^{2}  
		-  {\varrho} \frac{ ||v|^2-|u|^2 |^2}{|v-u|^2}
		\right\}.
		\end{equation}
		For $0<\frac{\vartheta}{4}<\varrho$, if $0<\tilde{\varrho}<  \varrho- \frac{\vartheta}{4}$ then
		\Be\label{k_vartheta_comparision}
		\mathbf{k}_{  \varrho}(v,u) \frac{e^{\vartheta |v|^2}}{e^{\vartheta |u|^2}} \lesssim  \mathbf{k}_{\tilde{\varrho}}(v,u) .
		\Ee
		See the proof in the appendix. 
		
		Moreover, for $0<\frac{\vartheta}{4}<\varrho$, (see the proof of Lemma 7 in \cite{Guo10})
		\Be\label{grad_estimate}
		\int_{\R^3}\mathbf{k}_{  \varrho}(v,u) \frac{e^{\vartheta |v|^2}}{e^{\vartheta |u|^2}} \dd u  \lesssim \langle v\rangle^{-1}.
		\Ee
		Also, we have
		\Be \label{measurechangesmall}
		\prod_{ j = 1 }^{l - 1 } \sup_{v_j }  \frac{ w_{\vartheta} (v_{j,\mathbf b})  \sqrt {\mu (v_{j,\mathbf b})}}{w_{\vartheta} (v_{j}) \sqrt  {\mu (v_j)} } \lesssim \prod_{ j = 1 }^{l - 1 } e^{\int_{t_{j+1,\iota}}^{t_{j,\iota}} \| \nabla \phi ^{\ell - j } (s) \|_\infty^2} \lesssim \prod_{ j = 1 }^{l - 1 } e^{M^2 ( t_{j,\iota} - t_{j+1,\iota} ) } \lesssim e^{M^2 t },
		\Ee
and
\Be \label{sumtkle0}
 \sum_{l=1}^{k-1}\mathbf{1}_{\{t^{\ell-l}_{l+1,\iota}\leq
			0<t_{l,\iota}^{\ell - (l-1)}\}}  = \mathbf{1}_{ \{ t_k^{\ell - ( k - 1) } \le 0\}  }.
\Ee	
		Then from (\ref{bound_g_ell}), (\ref{lower_nu_l}), and (\ref{h_iteration})-(\ref{measure}), (\ref{small_k}), \eqref{measurechangesmall}, and \eqref{sumtkle0}, if we choose $\ell \geq k_0$ and $0 \leq t \leq T$ where $k_0$ and $T$ in (\ref{small_k}), and let $M^2 \ll \nu_0$, we have
		\Be
		\begin{split}\label{stochastic_h^ell}
			& | h_\iota^{\ell+1} (t,x,v )|\\
			\leq & \  \| e^{- \frac{3}{4}\nu_0  t} h_0 \|_\infty \\
			& + \ \int^t_{\max\{t^\ell_{1,\iota}, 0 \}}
			e^{ - \frac{3}{4}\nu_0 (t-s)}\int_{\R^3}
			\mathbf{k}_\varrho (V_\iota^\ell (s;t,x,v),u) |h^\ell (s, X_\iota^\ell (s;t,x,v), u)|
			\dd u
			\dd s\\
			& +C_k \sup_l \int^{t_{l,\iota}^{\ell- (l-1)}}_{\max\{ t^{\ell-1}_{l+1,\iota}, 0  \}}   e^{ - \frac{3}{4}\nu_0 (t-s)}\int_{\R^3} \int_{\R^3} 
			\mathbf{k}_\varrho (V_\iota^{\ell-l} (s; v_l), u) 
			\\
			& \ \ \ \ \ \ \ \ \ \ \  \times
			|h^{\ell-l} (s, X_\iota^{\ell-l}(s; v_l)  ,u )|
			\{n(x_l)\cdot v_l\}\frac{\sqrt{\mu (v_l)}}{w_{\vartheta}(v_l)}
			\dd v_l
			\dd u
			\dd s \\
			& +  \int^t_{\max\{ t^\ell_{1,\iota}, 0  \}} 
			\langle V_\iota^\ell (s;t,x,v) \rangle     e^{- \int^t_s \frac{ \nu_\iota^\ell (\tau)}{2} \dd \tau}
			\| 
			e^{- \frac{3}{4}\nu_0 (t-s)}  h^\ell (s) \|_\infty^2
			\dd s \\
			&+ C_k \sup_l \int^{t_{l,\iota}^{\ell - (l-1)}}_{\max\{ t^{\ell-l}_{l+1,\iota} ,0 \}} \langle V_\iota^{\ell-l} (s;v_l) \rangle  \\
			& \ \ \ \ \ \ \ \ \ \ \  \times
			e^{
				- \int^{t_{l,\iota}^{\ell - (l-1)}}_s \frac{ \nu_\iota^{\ell-l } (\tau)}{2} \dd \tau 
			}
			\| e^{-\frac{3}{4}\nu_0 (t-s)}   h^{\ell-l} (s) \|_\infty^2 \dd s\\
			& +  \int^t_{\max\{ t^\ell_{1,\iota}, 0  \}} 
			\| 
			e^{- \frac{3}{4}\nu_0 (t-s)}  \nabla  \phi^\ell (s) \|_\infty 
			\dd s \\
			&+ C_k \sup_l \int^{t_{l,\iota}^{\ell - (l-1)}}_{\max\{ t^{\ell-l}_{l+1,\iota} ,0 \}}  
			\| e^{-\frac{3}{4}\nu_0 (t-s)}  \nabla   \phi ^{\ell-l} (s) \|_\infty  \dd s\\
			& +\Big\{\frac{1}{2}\Big\}^{ k/5} \| e^{-\frac{3}{4}\nu_0  (t - t_{k,\iota}^{\ell - (k-1)}) }  h(t_{k,\iota}^{\ell - (k-1)})  \|_\infty,
		\end{split}
		\Ee
		where we used the abbreviation of (\ref{X_ell_l}).
		
		From $\int^t_{0}\langle V_\iota^{\ell-l} (s;v_l) \rangle 
		e^{
			- \int^{t_{l,\iota}^{\ell - (l-1)}}_s \frac{ \nu_\iota^{\ell-l } (\tau)}{2} \dd \tau 
		} \dd s \lesssim 1$ and (\ref{grad_estimate}), we derive that 
		\Be\begin{split}\label{uniform_h^ell}
			\| h^{\ell+1} (t) \|_\infty & \lesssim \| h_+^{\ell+1} (t) \|_\infty +\| h_-^{\ell+1} (t) \|_\infty  
			\\ &\lesssim_k \  \| h(0) \|_\infty + o(1) \| h (t_{k}^{\ell - (k-1)}) \|_\infty
			+ t  \max_{l \geq 0 } \sup_{0 \leq s \leq t}   \| h^{\ell-l} (s) \|_\infty
			+ \max_{l \geq 0 } \sup_{0 \leq s \leq t}  \| h^{\ell-l} (s) \|_\infty^2 .
		\end{split}\Ee
		By taking supremum in $\ell$ and choosing $M \ll1$ and $0\leq t\leq T^* \leq T$ with $T^*\ll 1$, we conclude (\ref{uniform_h_ell}).

\end{proof}

		\section{Weighted ${W}^{1,p}$ estimates }

				\begin{proposition}\label{prop_W1p}
The main goal of this section is to prove the following weighted $W^{1,p}$ estimate for the sequence $f^\ell$ in \eqref{fell_local}
			Let us choose $0< \tilde{\vartheta}< \vartheta \ll 1$ and 
			\Be\label{beta_condition}
			\frac{p-2}{p}<\beta< \frac{2}{3} 
			, \quad\text{for}\quad 3< p < 6. 
			\Ee
			Assume $f^l$ solves \eqref{fell_local}, and for some $T > 0$ 
			\Be\label{small_bound_f_infty}
			\sup_{\ell \ge 0 } \sup_{0 \leq t \leq T} \| w_\vartheta f^\ell(t) \|_\infty \ll 1,
			\Ee
						\Be\label{integrable_nabla_phi_f}
			\sup_{\ell \ge 0 }\sup_{0 \leq t \leq T} e^{\Lambda_1 t} \|  \nabla_x \phi^\ell (t) \|_\infty  < \delta_1,
			\Ee
			with 
			\Be\label{delta_1/lamdab_1}
			0< \frac{  \delta_1}{\Lambda_1}  \ll_\O 1.
			\Ee
		Then there exists $T^{**} \ll 1 $ and $C > 0$ such that the sequence \eqref{fell_local} satisfies

			\Be \label{unif_Em}
		   \max_{ \ell\geq0}\sup_{0 \leq t \leq T^{**}} \mathcal{E}^\ell(t)
			\leq C \{ \| w_\vartheta f_0 \|_\infty +
			\| w_{\tilde{\vartheta}} f _0 \|_p^p
			+
			\| w_{\tilde{\vartheta}} \alpha_{f_{0 },\e}^\beta \nabla_{x,v} f  _0 \|_{p}^p + |\nabla_{\tau, v } f_0 |_{p,+}^p \}< \infty.
		 \Ee 
		 where we define, for $0 < \e \ll 1$,
		\Be\label{mathcal_E}
		\mathcal{E}^{\ell+1}(t):= \| w_\vartheta f^{\ell+1} (t) \|_\infty +
		\| w_{\tilde{\vartheta}} f^{\ell+1} (t) \|_p^p
		+
		\| w_{\tilde{\vartheta}}\alpha_{f^{\ell},\e}^\beta \nabla_{x,v} f^{\ell+1}(t) \|_{p}^p
		+ 
		\int^t_0  | w_{\tilde{\vartheta}}\alpha_{f^{\ell},\e}^\beta \nabla_{x,v} f^{\ell+1}(t) |_{p,+}^p + \int_0^t | w_{\tilde{\vartheta}} f^{\ell+1}(t) |_{p,+}^p.
		\Ee					
		\end{proposition}
To prove this, we need the following results:
		
		\begin{proposition}\label{prop_int_alpha}
	
	Assume $\phi_f(t,x)$ obtained from \eqref{smallfphi} with $\nabla \phi_f $ satisfies \eqref{decay_E}   
	and
	\Be\label{Decay_phi_2}
	\sup_{t\geq 0}e^{\Lambda_2 t} \|\nabla^2 \phi_f(t)\|_{\infty} \leq \delta_2 \ll1 .
	\Ee
	Then for $\iota = +$ or $-$ as in \eqref{iota}, for all $0<\sigma<1$ and $N>1$ and for all $s\geq 0,$ $x \in \bar{\O}$,
	\Be\label{NLL_split2}
	\int_{|u| \leq N } 
	\frac{   \dd u
	}{\alpha_{f,\e,\iota}(s,x,u)^{ \sigma}}
	\lesssim_{\sigma, \O, \Lambda_1, \delta_1, \Lambda_2, \delta_2,N}  1, 
	\Ee
	and, for any $0< \kappa\leq2$,  
	\Be\label{NLL_split3} 
	\int_{  |u|\geq N} \frac{e^{-C|v-u|^2}}{|v-u|^{2-\kappa}} \frac{1}{\alpha_{f,\e,\iota}(s,x,u)^\sigma} \dd u
	\lesssim_{\sigma, \O, \Lambda_1, \delta_1, \Lambda_2, \delta_2,N,\kappa}  1.
	\Ee
\end{proposition}
\begin{proof}
It is important to note that from \eqref{decay_E} we have $  n(x) \cdot \nabla \phi_f = 0 $ for all $x \in \O$. Thus  for both trajectories $(X_{\pm}(s;t,x,v), V_\pm(s;t,x,v) )$, their corresponding fields $\mp \nabla_x \phi_f $ satisfy $\mp \nabla_x \phi_f  \cdot n(x) = 0 $. Therefore we can apply Proposition 3 from \cite{VPBKim} to $\alpha_{f,\e,+}$ and $\alpha_{f, \e , - }$ separately to conclude \eqref{NLL_split2} and \eqref{NLL_split3}.
\end{proof}

		\begin{lemma}
			For any $0< \delta <1$, we claim that if $(f, \phi_f)$ solves (\ref{smallfphi}) then 
			\Be\label{Morrey}
			\|   \phi_f(t) \|_{C^{1, 1- \delta}(\bar{\O})}\lesssim_{\delta,  \O}  \| w_\vartheta f(t) \|_\infty  \ \ \text{for all} \ \ t\geq 0.
			\Ee\end{lemma}

		\begin{proof} 
			We have, for any $p>1$, 
			\Be\notag
			\begin{split}
				\left\|\int_{\R^3} (f_+ - f_- ) \sqrt{\mu(v)} \dd v\right\|_{L^p (\O)} 
				\leq        |\O|^{1/p} \left( \int_{\R^3} w_\vartheta(v)^{-1} \sqrt{\mu (v)} \dd v\right) \| w_\vartheta f (t) \|_\infty.
			\end{split}\Ee
			Then we apply the standard elliptic estimate to (\ref{smallfphi}) and deduce that 
			\Be\label{phi_2p}
			\|\phi_f (t) \|_{W^{2,p} (\O)} \lesssim 
			\| w_\vartheta f(t) \|_\infty
			.\notag
			\Ee
			On the other hand, from the Morrey inequality, we have, for $p>3$ and $\O \subset \R^3$, 
			\[
			\|   \phi_f(t) \|_{C^{1, 1-  {3}/{p}}(\O)}\lesssim_{p, \O} \|   \phi(t) \|_{W^{2,p} (\O)}.
			\]
			Now we choose $p=3/\delta$ for $0< \delta <1$. Then we can obtain (\ref{Morrey}). 
		\end{proof}

		To close the estimate, we use the following lemma crucially.
		\begin{lemma}\label{lemma_apply_Schauder}
			Assume (\ref{beta_condition}). If $\phi_f$ solves (\ref{smallfphi}) then 
			\Be\label{phi_c,gamma}
			\| \phi_f (t) \|_{C^{2, 1- \frac{3}{p}} (\bar{\O})}    \leq (C_1)^{1/p} \big\{ \|  f(t) \|_p  + \| \alpha_{f,\e}^\beta \nabla_x f(t) \|_{p} \big\} \ \ \ \text{for} \ \ p>3.
			\Ee

		\end{lemma}
		
		\begin{proof}

			Applying the Schauder estimate to (\ref{smallfphi}), we deduce 
			\Be
			\begin{split}\label{apply_Schauder}
				\|\phi_{f} (t)\|_{C^{2,1-\frac{3}{p}}(\bar{\O})} &\lesssim_{p,\O} \Big\|\int_{\R^{3}} (f_+(t) - f_-(t))\sqrt{\mu} \dd v \Big\|_{C^{0,1-\frac{3}{p}}(\bar{\O})}  \ \ \ \text{for} \ \ p>3.
			\end{split}\Ee

			By the Morrey inequality, $W^{1,p}   \subset C^{0,1- \frac{3}{p}} $ with $p>3$ for a domain $\O \subset \R^3$ with a smooth boundary $\p\O$, we derive 
			%
			%
			\Be\begin{split}\label{apply_Morrey}
				&\Big\| \int_{\R^{3}} (f_+(t) - f_-(t)) \sqrt{\mu }\dd v \Big\|_{C^{0,1-\frac{n}{p}} (\bar{\O}) }  \\
				\lesssim &  \ \Big\| \int_{\R^{3}} (f_+(t) - f_-(t))  \sqrt{\mu }\dd v \Big\|_{W^{1,p} (\O) } 
				\\
				\lesssim & \ \left(\int_{\R^3}   {\mu}^{q /2}\dd v\right)^{1/q}   \|  f (t) \|_{L^p (\O \times \R^3)} 
				+\Big\| \int_{\R^{3}} \nabla_{x} (f_+(t) - f_-(t))  \sqrt{\mu }\dd v \Big\|_{L^p (\O) }.
			\end{split}
			\Ee
			By the H\"older inequality, for $\iota = +$ or $-$ as in \eqref{iota},
			\Be\begin{split}\label{nabla_p}
				&\Big|\int_{\R^{3}} \nabla_{x} f_\iota (t,x,v) \sqrt{\mu(v)}\dd v\Big|\\
				\leq  & \  \Big\|\frac{\sqrt{\mu(\cdot)}}{ \alpha_{f,\e,\iota}(t,x,\cdot) ^{\beta}}  \Big\|_{L^{\frac{p}{p-1} }(\R^{3})}
				\Big\| \alpha_{f,\e,\iota}(t,x,\cdot )^{\beta} \nabla_{x} f_\iota (t,x,\cdot ) \Big\|_{L^{p} (\R^{3})}\\
				= & \ \underbrace{\left( \int_{\R^{3}}  \frac{ \mu(v)^{\frac{p}{2(p-1)}}}{\alpha_{f,\e,\iota} (t,x,v)^{\frac{\beta p}{p-1}}}   \dd v \right)^{\frac{p-1}{p}}}_{(\ref{nabla_p})_1} \|\alpha_{f,\e,\iota}(t,x, \cdot)^{\b}\nabla_{x}f_\iota(t,x, \cdot)\|_{L^{p} (\R^{3})}.
			\end{split}\Ee
			Note that $\frac{p-2}{p-1}<\frac{\beta p}{p-1}< \frac{2}{3} \frac{p}{p-1}<1$ from (\ref{beta_condition}). We apply Proposition \ref{prop_int_alpha} and conclude that $(\ref{nabla_p})_1\lesssim 1$. Taking $L^p(\O)$-norm on (\ref{nabla_p}) and from (\ref{apply_Morrey}), we conclude (\ref{phi_c,gamma}).\end{proof}

		We need some basic estimates to prove Proposition \ref{prop_W1p}. Recall the decomposition of $L$ in (\ref{L_decomposition}). From (\ref{collision_frequency})
		\Be\label{nabla_nu}
		|\nabla_v \nu(v)| \leq  \int_{\R^3} \int_{\S^2} |\o| \mu(u) \dd \omega \dd u \lesssim 1. 
		\Ee

		\hide
		\begin{lemma} \label{DKG}
			For $0<\varrho< \frac{1}{8}$,
			\Be\label{vK}\begin{split}
				| \nabla_v Kg(v) | \lesssim    \ \| w g \|_\infty
				+ \int_{\R^3} \mathbf{k}_\varrho (v,u) |\nabla_v g(u)| \dd u,
			\end{split}
			\Ee
			and
			\Be\label{vGamma}
			\begin{split}
				|\nabla_v \Gamma (g,g) (v)| 
				\lesssim  & \ \| w g \|_\infty \int_{\R^3} \mathbf{k}_\varrho (v,u) |\nabla_v g (u)| \dd u \\
				&+ \langle v\rangle  \| w g \|_\infty |\nabla_v g (v)|  + \langle v \rangle w(v)^{-1} \| w g \|_\infty^2.
			\end{split}
			\Ee
			
		\end{lemma}
		\unhide
	Recall the definition of $\mathbf{k}_\varrho(v,u)$ from \eqref{kzeta}. From (\ref{k_estimate}) and a direction computation, for $0<\varrho< \frac{1}{8} $,
		\Be \label{nabla_k1}
		\begin{split}
			|\p_{v_{i}} \mathbf{k}_{1}(v,u+v)| &= C_{\mathbf{k}_{1}} \p_{v_{i}}\Big( |u|e^{-\frac{|v|^{2}+|u+v|^{2}}{4}} \Big)  
			\lesssim 
			\mathbf{k}_{\varrho}(v,u+v),
		\end{split} 
		\Ee
		and 
		\Be \label{nabla_k2}
		\begin{split}
			\p_{v_{i}} \mathbf{k}_{2}(v,u+v) &= C_{\mathbf{k}_{2}} \p_{v_{i}}\Big( \frac{1}{|u|}e^{-\frac{|u|^{2}}{8}} e^{-\frac{||v|^{2}-|u+v|^{2}|^{2}}{8|u|^{2}}}  \Big) \\
			&= -\frac{C_{\mathbf{k}_{2}}}{|u|}e^{-\frac{|u|^{2}}{8}} e^{-\frac{||v|^{2}-|u+v|^{2}|^{2}}{8|u|^{2}}} \frac{||v|^{2}-|u+v|^{2}|}{4|u| } \frac{u_{i}}{|u|}  \\
			&\lesssim \frac{e^{-\frac{|u|^{2}}{8}} 
			}{|u|}
			e^{-\frac{
					||v|^{2}-|u+v|^{2}|^{2}
				}{16|u|^{2}}} 
			\\
			&
			\lesssim \mathbf{k}_{\varrho}(v,u+v) .  \\
		\end{split} 
		\Ee

For $g_1, g_2: \mathbb R^3 \to \mathbb R$, $g = \begin{bmatrix} g_1 \\ g_2 \end{bmatrix}$, we define
		\Be \label{K_v}
		 K_v g (v)		 : = \begin{bmatrix} \int_{\mathbb R^3 }\nabla_v \mathbf{k}_2 (v,u) (3g_1(u) + g_2(u) ) du - \int_{\mathbb R^3} \nabla_v \mathbf{k}_1 (v,u) (g_1(u) + g_2(u) ) du \\   \int_{\mathbb R^3 }\nabla_v  \mathbf{k}_2 (v,u) (3g_2(u) + g_1(u) ) du - \int_{\mathbb R^3} \nabla_v \mathbf{k}_1 (v,u) (g_1(u) + g_2(u) ) du \end{bmatrix}.
		\Ee

		From (\ref{k_vartheta_comparision}), (\ref{nabla_k1}), and (\ref{nabla_k2}),
		\Be\label{vKsum}\begin{split}
			| w_{\tilde{\vartheta}}   K \nabla_v g(v) |
			&\lesssim \ \sum_{i}\int_{\R^3} |\mathbf{k}_i  (v,u+v) | \frac{w_{\tilde{\vartheta}}(v)}{w_{\tilde{\vartheta}}(u+v)}  (|w_{\tilde{\vartheta}} \nabla_v g_1(u+v)| + |w_{\tilde{\vartheta}} \nabla_v g_2(u+v)|)\}\dd u\\
			&\lesssim \int_{\R^3} \mathbf{k}_{\tilde{\varrho}} (v,u) |w_{\tilde{\vartheta}}\nabla_v g(u)| \dd u, \\
		| w_{\tilde{\vartheta}}  K_{v} g(v) |
			&\lesssim \ \sum_{i}\int_{\R^3} | \nabla_v\mathbf{k}_i  (v,u+v)| \frac{w_{\tilde{\vartheta}}(v)}{w_{\vartheta}(u+v)} ( |w_{\vartheta} g_1(u+v) | + |w_{\vartheta} g_2(u+v) |) \dd u  \\
			&\lesssim 
			\int_{\R^3} \mathbf{k}_\varrho (v,u) \frac{w_{\tilde{\vartheta}}(v)}{w_{\vartheta}(u)} |w_{\vartheta} g(u)|  \dd u  \\
			&\lesssim \| w_\vartheta g \|_\infty.  \\
		\end{split}
		\Ee

For $g = \begin{bmatrix} g_1 \\ g_2 \end{bmatrix}$ and $h = \begin{bmatrix} h_1 \\ h_2 \end{bmatrix} $, the nonlinear Boltzmann operator  $ \Gamma (g, h ) $ in (\ref{Gamma_def}) equals
		\begin{equation}\label{carleman}
		\begin{split}
		& 
		 \Gamma (g, h)  = \begin{bmatrix} \int_{\mathbb{R}^{3}}   \int_{\S^2}   | u \cdot \omega|
		(h_1+h_2)(v+ u_\perp) g_1(v + u_\parallel)
		\sqrt{\mu(v+u)} \dd \omega  \mathrm{d}u
		 - \int_{\mathbb{R}^{3}}   \int_{\S^2}   | u \cdot \omega|
		(h_1+h_2)(v+u) g_1(v  )
		\sqrt{\mu(v+u)} \dd \omega  \mathrm{d}u
		\\  \int_{\mathbb{R}^{3}}   \int_{\S^2}   | u \cdot \omega|
		(h_1+h_2)(v+ u_\perp) g_2(v + u_\parallel)
		\sqrt{\mu(v+u)} \dd \omega  \mathrm{d}u
		 - \int_{\mathbb{R}^{3}}   \int_{\S^2}   | u \cdot \omega|
		(h_1+h_2)(v+u) g_2(v  )
		\sqrt{\mu(v+u)} \dd \omega  \mathrm{d}u
		 \end{bmatrix},
		\end{split}
		\end{equation}
		where $u_\parallel = (u \cdot \omega)\omega$ and $u_\perp = u - u_\parallel$. Following the derivation of (\ref{k_estimate}) in Chapter 3 of \cite{gl}, by exchanging the role of $\sqrt{\mu}$ and $w^{-1}$, we have 
		\Be\begin{split}\label{bound_Gamma_k}
			& |w_{\vartheta} \Gamma  (g, h)|   \lesssim \| w_{\vartheta} g \|_\infty \int_{\R^3} \mathbf{k}_{\tilde{\varrho }} (v,u) |w_{\vartheta}h(u) | \dd u, 
			\\ & |w_{\vartheta} \Gamma  (g, h)|   \lesssim \| w_{\vartheta} h \|_\infty \left( \int_{\R^3} \mathbf{k}_{\tilde{\varrho }} (v,u) |w_{\vartheta}g(u) | \dd u + \langle v \rangle |w_\vartheta g (v)|\right).
		\end{split}\Ee

		By direct computations
		\Be\begin{split}\label{nabla_Gamma}
			&\nabla_v  \Gamma(g,h) (v) \\
			=& \  \nabla_{v}  \Gamma_{\textrm{gain}}(g, h) - \nabla_{v}   \Gamma_{\textrm{loss}}(g, h) \\
			= & \   \Gamma_{\textrm{gain}} (\nabla_v g,h) +   \Gamma_{\textrm{gain}} ( g,\nabla_v h)
			-   \Gamma_{\textrm{loss}} (\nabla_v g,h) -  \Gamma_{\textrm{loss}} ( g,\nabla_vh)+   \Gamma_v (g,h).
		\end{split}\Ee
		Here we have defined
		\Be\begin{split}\label{Gamma_v}
			   & \Gamma_v (g, h)(v) :=   \Gamma_{v,\textrm{gain}} -   \Gamma_{v,\textrm{loss}} 
			  \\
			&  = \begin{bmatrix} \int_{\mathbb{R}^{3}}   \int_{\S^2}   | u \cdot \omega|
		(h_1+h_2)(v+ u_\perp) g_1(v + u_\parallel)
		\nabla_v \sqrt{\mu(v+u)} \dd \omega  \mathrm{d}u
		 - \int_{\mathbb{R}^{3}}   \int_{\S^2}   | u \cdot \omega|
		(h_1+h_2)(v+u) g_1(v  )
		\nabla_v \sqrt{\mu(v+u)} \dd \omega  \mathrm{d}u
		\\  \int_{\mathbb{R}^{3}}   \int_{\S^2}   | u \cdot \omega|
		(h_1+h_2)(v+ u_\perp) g_2(v + u_\parallel)
		\nabla_v \sqrt{\mu(v+u)} \dd \omega  \mathrm{d}u
		 - \int_{\mathbb{R}^{3}}   \int_{\S^2}   | u \cdot \omega|
		(h_1+h_2)(v+u) g_2(v  )
		\nabla_v \sqrt{\mu(v+u)} \dd \omega  \mathrm{d}u
		 \end{bmatrix}	.	
		 \end{split}\Ee
		
		
		Note that 
		\Be\begin{split}\notag 
			&|w_{\tilde{\vartheta}}  \Gamma_{\textrm{gain}} (\nabla_v g,h) |+ |w_{\tilde{\vartheta}}  \Gamma_{\textrm{gain}} ( g,\nabla_vh)|\\
			\lesssim & \ (\| w_{\vartheta} g \|_\infty + \| w_{\vartheta} h \|_\infty ) \big\{|w_{\tilde{\vartheta}}  \Gamma_{\textrm{gain}} (|\nabla_v g|, w_{\vartheta}^{-1}) | + |w_{\tilde{\vartheta}}  \Gamma_{\textrm{gain}} ( w_{\vartheta}^{-1}, |\nabla_vh|)|\big\} \\
			\lesssim & \ \ (\| w_{\vartheta} g \|_\infty + \| w_{\vartheta} h \|_\infty )
			\int_{\R^3} \int_{\S^2} 
			|(v-u) \cdot \omega| \frac{w_{\tilde{\vartheta}}(v)}{w_{\vartheta}(u )}  \Big\{ \frac{|\nabla_v h (u^\prime)|}{w_{\vartheta}(v^\prime) } + \frac{ |\nabla_v g (v^\prime)|}{w_{\vartheta}(u^\prime)}\Big\}
			\dd \omega \dd u.
		\end{split}\Ee

		Then following the derivation of (\ref{k_estimate}) in Chapter 3 of \cite{gl}, by exchanging the role of $\sqrt{\mu}$ and $w_{\vartheta}^{-1}$, we can obtain a bound of 
		\Be\label{bound_Gamma_nabla_vf1}
		\begin{split}
		|w_{\tilde{\vartheta}}   \Gamma_{\textrm{gain}} (\nabla_v g,h) |+ |w_{\tilde{\vartheta}}  \Gamma_{\textrm{gain}} ( g,\nabla_v h)|
		&\lesssim 
		(\| w_{\vartheta} g \|_\infty + \| w_{\vartheta} h \|_\infty ) \int_{\R^3} \mathbf{k}_{\varrho} (v,u) \frac{w_{\tilde{\vartheta}}(v)}{w_{\tilde{\vartheta}}(u)} (|w_{\tilde{\vartheta}}\nabla_v g(u)|  + |w_{\tilde{\vartheta}}\nabla_v h(u)| )\dd u  \\
		&\lesssim (\| w_{\vartheta} g \|_\infty + \| w_{\vartheta} h \|_\infty ) \int_{\R^3} \mathbf{k}_{\tilde{\varrho}} (v,u) ( |w_{\tilde{\vartheta}}\nabla_v g(u)|  + |w_{\tilde{\vartheta}}\nabla_v h(u)|)\dd u . \\
		\end{split}
		\Ee

		Clearly 
		\Be\label{bound_Gamma_nabla_vf2}
		\begin{split}
			 |w_{\tilde{\vartheta}}  \Gamma_{\textrm{loss}}( g, \nabla_v h)| &\lesssim  \| w_{\vartheta} g \|_\infty   \int_{\R^3} 
			\frac{w_{\tilde{\vartheta}}(v)}{w_{\vartheta}(v)w_{\tilde{\vartheta}}(u)}
			|w_{\tilde{\vartheta}}\nabla_v h (u)| \mu(u)^{\frac{1}{2}} 
			\dd u \\
			&\lesssim \|w_{\vartheta} g\|_{\infty} \int_{\R^3} \mathbf{k}_{\tilde{\varrho}} (v,u) |w_{\tilde{\vartheta}}\nabla_v h(u)| \dd u , \\
			|w_{\tilde{\vartheta}}  \Gamma_{\textrm{loss}}(\nabla_v g, h)| &\lesssim \langle v\rangle  \| w_{\vartheta} h \|_\infty   
			|w_{\tilde{\vartheta}}\nabla_v g (v)|  . 
		\end{split}
		\Ee
		For $  \Gamma_{v,\textrm{loss}}(g,h)$ defined in (\ref{Gamma_v}),
		\Be \label{Gvloss}
		\begin{split}
			&|w_{\tilde{\vartheta}}  \Gamma_{v,\textrm{loss}} (g,h)|\\
			&\lesssim \frac{w_{\tilde{\vartheta}}(v)}{w_{\vartheta} (v)} | w_{\vartheta} g | \iint_{\R^{3}\times {\S}^{2}} |(u-v) \cdot\omega|  \frac{1}{w_{\vartheta}(u)} |w_{ \vartheta}h(u)| \nabla_{v}\sqrt{\mu(u)} \dd u \dd \omega  \\
			&\lesssim  {\langle v \rangle} |w_{\vartheta} g| \|w_{\vartheta} h\|^{}_{\infty}. 
		\end{split}
		\Ee	
		For $  \Gamma_{v,\textrm{gain}}(g,h)$, following the derivation of (\ref{k_estimate}) in Chapter 3 of \cite{gl}, by exchanging the role of $\sqrt{\mu}$ and $w_{\vartheta}^{-1}$
		\Be \label{Gvgain}
		\begin{split}
			|w_{\tilde{\vartheta}}  \Gamma_{v,\textrm{gain}} (g,h)| &\lesssim \|w_{\vartheta}h\|_{\infty} \iint_{\R^{3}\times {\S}^{2}} |(u-v)\cdot\omega| \frac{w_{\tilde{\vartheta}}(v)}{w_{\vartheta}(v^{\prime})} \frac{w_{\vartheta}g(v^\prime)}{w_{\vartheta}(u^{\prime}) } \nabla_{v}\sqrt{\mu(u)} \dd u \dd\omega  \\
			&\lesssim   \|w_{\vartheta}h\|_{\infty} \int_{\R^3} \mathbf{k}_{\tilde{\varrho}} (v,u) | w_{\vartheta} g(u)| \dd u. \\
		\end{split}
		\Ee
		
		\hide
		we know that $w(v) \lesssim w(v^{\prime})w(u^{\prime})$ from consevation $|u|^{2} + |v|^{2} = |u^{\prime}|^{2} + |v^{\prime}|^{2}$ to derive 
		\Be \label{Gvgain}
		\begin{split}
			| \hat \Gamma_{v,\textrm{gain}} (g,g)| &\lesssim \|w_{\vartheta} g\|_{\infty}  \iint_{\R^{3}\times {\S}^{2}} |(u-v)\cdot\omega| \frac{1}{w_{\vartheta}(u^{\prime})w_{\vartheta}(v^{\prime})} \nabla_{v}\sqrt{\mu}(u) du d\omega  \\
			&\lesssim \|w_{\vartheta}g\|_{\infty} \int_{\R^3} \mathbf{k}_{\varrho} (v,u) |\nabla_v g(u)| \dd u. 
		\end{split}
		\Ee

		We combine (\ref{nabla_Gamma}), (\ref{bound_Gamma_nabla_vf1}), (\ref{bound_Gamma_nabla_vf2}), (\ref{Gvloss}), and (\ref{Gvgain}) to obtain (\ref{vGamma}).
		
		\unhide

		The next result is about estimates of derivatives on the boundary. Assume (\ref{F_ell_BC}) and (\ref{fell_local}). We claim that for $(x,v) \in\gamma_-$,
		\Be\label{BC_deriv}\begin{split}
			|\nabla_{x,v} f^{\ell+1}(t,x,v) |
			\lesssim
			\langle v\rangle\sqrt{\mu(v)}
			\Big(1+ \frac{1}
			{|n(x) \cdot v|}  \Big)
			\times (\ref{BC_deriv_R}) .
		\end{split}\Ee   
		with
		\Be\begin{split}\label{BC_deriv_R}
			& 
			\int_{n(x) \cdot u>0}  \Big\{
			(
			\langle u\rangle  + |\nabla_x \phi^\ell | + |\nabla_x \phi^{\ell-1} | 
			)
			| \nabla_{x,v} f^\ell (t,x,u ) |  \\ 
			&
			\ \ \ \ \    \ \ \  \ \ \  \ \ \   + \langle u \rangle (|f^{\ell +1}|  + |f^\ell | ) 
			+(1+ \| w_{\vartheta} f ^{\ell }\|_\infty + \| w_\vartheta f^{\ell -1 } \|_\infty) \int_{\R^3}\mathbf{k}_{\varrho }(u, u^\prime) ( |f^\ell(u^\prime)| + |f^{\ell - 1 } (u' ) | ) \dd u^\prime
			\\
			&
			\ \ \ \ \    \ \ \  \ \ \  \ \ \   + 
			(\langle u \rangle (|f^{\ell +1}| + | f^\ell | )  + \mu(u)^{\frac{1}{4}} ) ( |\nabla_x \phi^\ell | + |\nabla_x \phi^{\ell-1} | )
			\Big\}\sqrt{\mu(u)}\{n(x) \cdot u\} \dd u.
		\end{split}
		\Ee
		
		From (\ref{fell_local}),
		%
		\Be\label{fn}
		\begin{split}
			&\partial _{n}f^{\ell+1}(t,x,v)\\
			=&\frac{-1}{n(x)\cdot v}\bigg\{ \partial
			_{t}f^{\ell +1}+\sum_{i=1}^{2}(v\cdot \tau _{i})\partial _{\tau _{i}}f^{\ell +1} 
			-q \nabla_x \phi^l \cdot \nabla_v f^{\ell +1}  \\
			& \ \ \ \  \ \ \ \  \   \ \ \ + q \frac{v}{2} \cdot \nabla_x \phi^l f^{\ell +1} + \nu f^{\ell +1} - K f^\ell
			-\Gamma_\text{gain} (f^\ell,f^\ell) + \Gamma_\text{loss} (f^{\ell +1}, f^\ell) + q_1 v\cdot \nabla_x \phi^l \sqrt{\mu}\bigg\}. 
		\end{split}\Ee%
		%

		Let $\tau _{1}(x)$ and $\tau _{2}(x)$ be unit tangential vectors to $\partial\Omega$ satisfying $\tau
		_{1}(x)\cdot n(x)=0=\tau _{2}(x)\cdot n(x)$ and $\tau _{1}(x)\times \tau
		_{2}(x)=n(x)$. Define the orthonormal transformation from $\{n,\tau
		_{1},\tau _{2}\}$ to the standard basis $\{\mathbf{e}_{1},\mathbf{e}_{2},%
		\mathbf{e}_{3}\}$, i.e. $\mathcal{T}(x)n(x)=\mathbf{e}_{1},\ \mathcal{T}%
		(x)\tau _{1}(x)=\mathbf{e}_{2},\ \mathcal{T}(x)\tau _{2}(x)=\mathbf{e}_{3},$
		and $\mathcal{T}^{-1}=\mathcal{T}^{T}.$ Upon a change of variable: $%
		u^{  \prime }=\mathcal{T}(x)u,$ we have%
		\begin{equation*}
		n(x)\cdot u=n(x)\cdot \mathcal{T}^{t}(x)u^{\prime }=n(x)^{t}%
		\mathcal{T}^{t}(x)u^{ \prime }=[\mathcal{T}(x)n(x)]^{t}u^{
			\prime }=\mathbf{e}_{1}\cdot u^{  \prime }=u_{1}^{  \prime },
		\end{equation*}%
		then the RHS of the diffuse BC (\ref{F_ell_BC}) equals
		\begin{equation*}
		c_{\mu }%
		\sqrt{\mu (v)}\int_{u_{1}^{  \prime }>0}f^\ell(t,x,\mathcal{T}%
		^{t}(x)u^{  \prime })\sqrt{\mu (u^{  \prime })}\{u_{1}^{
			\prime }\}\mathrm{d}u^{  \prime }.
		\end{equation*}%
		Then we can further take tangential derivatives $\partial _{\tau _{i}}$
		as, for $(x,v)\in \gamma _{-},$
		\begin{equation}\label{boundary_tau}
		\begin{split}
		&\partial _{\tau _{i}}f^{\ell +1}(t,x,v)\\
		& =c_{\mu }\sqrt{\mu (v)}\int_{n(x)\cdot u >0}\partial _{\tau
			_{i}}f^\ell(t,x,u)\sqrt{\mu (u)}\{n(x)\cdot u\}%
		\mathrm{d}u\\
		& \ \ +c_{\mu }\sqrt{\mu (v)}\int_{n(x)\cdot u>0}\nabla
		_{v}f^\ell(t,x,u)\frac{\partial \mathcal{T}^{t}(x)}{\partial \tau _{i}}%
		\mathcal{T}(x)u\sqrt{\mu (u)}\{n(x)\cdot u\}%
		\mathrm{d}u.
		\end{split}
		\end{equation}%


		
		We can take velocity derivatives directly to (\ref{diffusef}) and obtain that for $%
		(x,v)\in \gamma _{-},$
		\begin{eqnarray} \label{vtderivbdry}
		\nabla _{v}f^{\ell +1}(t,x,v) &=&c_{\mu }\nabla _{v}\sqrt{\mu (v)}\int_{n(x)\cdot
			u>0}f^\ell(t,x,u)\sqrt{\mu (u)}\{n({x})\cdot
		u\}\mathrm{d}u,\label{boundary_v}\\
		\partial _{t}f^{\ell +1}(t,x,v) &=&c_{\mu }\sqrt{\mu (v)}\int_{n(x)\cdot u>0}\partial _{t}f^\ell(t,x,u)\sqrt{\mu (u)}\{n({x})\cdot
		u\}\mathrm{d}u. \notag
		\end{eqnarray}%
		For the temporal derivative, we use (\ref{systemf}) again to deduce that 
		\Be\label{boundary_t} 
		\begin{split}
			&\partial _{t}f^{\ell +1}(t,x,v)\\
			=&c_{\mu }\sqrt{\mu (v)}\int_{n(x)\cdot u>0}
			\Big\{
			- u\cdot \nabla_x f^\ell + q \nabla_x  \phi^{\ell-1}  \cdot \nabla_v f^\ell - q \frac{u}{2} \cdot \nabla_x \phi^{l-1} f^\ell + \nu f^\ell - K f^{l-1} \\
			& \ \ \ \ \ \ \ \  \ \ \ \ \ \ \ \  \ \ \ \ \ \ \ \ \ \ 
			+  \Gamma_\text{gain}(f^{\ell-1},f^{\ell-1}) - \Gamma_\text{loss} (f^\ell, f^{\ell-1})  - q_1 u\cdot \nabla_x \phi^{\ell-1} \sqrt{\mu}
			\Big\}
			\sqrt{\mu (u)}\{n({x})\cdot
			u\}\mathrm{d}u.
		\end{split}
		\Ee 
		From (\ref{fn})-(\ref{boundary_t}), (\ref{k_estimate}), and (\ref{bound_Gamma_k}), we conclude (\ref{BC_deriv}).




		\begin{proof}[\textbf{Proof of Proposition \ref{prop_W1p}}]
			\textit{Step 1.} 
			Note that by our choice of $f^1$, we have $\p_t f^1(t,x,v)|_{\gamma_- }  = 0 $. Therefore combing \eqref{boundary_tau}, \eqref{fn}, and \eqref{vtderivbdry} and the assumption that $ |\nabla_{\tau, v } f_0 |_{p,+}^p < \infty$, we get \eqref{unif_Em} is valid for $\ell \le 1 $.
			
			Thus it suffices to prove the following induction statement: there exist $T^{**} \ll 1$ (and $T^{**} < T^*(M)$) and $C>0$ such that  
		\Be \label{induc_hypo}
		\begin{split}
			\text{ if }   \   \max_{0\leq m\leq\ell}\sup_{0 \leq t \leq T^{**}} \mathcal{E}^m(t)
			\leq C \{ \| w_\vartheta f_0 \|_\infty + \| w_{\tilde{\vartheta}} f _0 \|_p^p
			+
			\| w_{\tilde{\vartheta}} \alpha_{f_{0 },\e}^\beta \nabla_{x,v} f  _0 \|_{p}^p
			+ |\nabla_{\tau, v } f_0 |_{p,+}^p  \} < \infty, \\
			\text{ then } \  \sup_{0 \leq t \leq T^{**}}
			\mathcal{E}^{\ell+1} (t)\leq C \{ \| w_\vartheta f_0 \|_\infty +\| w_{\tilde{\vartheta}} f _0 \|_p^p
			+
			\| w_{\tilde{\vartheta}} \alpha_{f_{0 },\e}^\beta \nabla_{x,v} f_0 \|_{p}^p
			+|\nabla_{\tau, v } f_0 |_{p,+}^p \}. 
		\end{split} \Ee

			Define 
			\Be\label{def_nu_phi}
			\nu_{\phi^\ell}(t,x,v):= \begin{bmatrix} \nu(v) +  \frac{v}{2} \cdot \nabla_x \phi^\ell & 0 \\ 0 & \nu(v) -  \frac{v}{2} \cdot \nabla_x \phi^\ell \end{bmatrix}.
			\Ee
			From the assumption (\ref{small_bound_f_infty}), we have that $\nu(v) +  \frac{v}{2} \cdot \nabla_x \phi^\ell \gtrsim \frac{\nu(v)}{2}$, and $\nu(v) -  \frac{v}{2} \cdot \nabla_x \phi^\ell \gtrsim \frac{\nu(v)}{2}$.
			
			From (\ref{smallfphi}), (\ref{k_estimate}), and (\ref{bound_Gamma_k}), we can easily obtain that, for $0<\varrho \ll 1$ 
			\Be
			\begin{split}\label{energy_f_p}
				&\|w_{\tilde{\vartheta}} f^{\ell+1}(t)\|_p^p  + \int^t_0 \| \nu_{\phi^\ell}^{1/p} w_{\tilde{\vartheta}} f^{\ell+1} \|_p^p +  \int^t_0 | w_{\tilde{\vartheta}}f^{\ell+1}|_{p,+}^p\\
				\lesssim & \ \|w_{\tilde{\vartheta}}f( 0)\|_p^p + (1+ \| w_\vartheta f^\ell \|_\infty) \int^t_0 \iint_{\O \times \R^3} |w_{\tilde{\vartheta}}f^{\ell+1}(v)|^{p-1} \int_{\R^3} \mathbf{k}_\varrho (v,u)\frac{w_{\tilde{\vartheta}}(v)}{w_{\tilde{\vartheta}}(u)} | w_{\tilde{\vartheta}}f^\ell(u)|\dd u\\
				&  +  \|  w_{\tilde{\vartheta}} f^\ell \|_\infty \int^t_0   \int_{\Omega \times \mathbb R^3 }  \langle v \rangle| w_{\tilde{\vartheta}} f^{\ell+1} | ^p +    o(1) \int^t_0 \|  w_{\tilde{\vartheta}}f^{\ell+1} \|_p^p  + \int_0^t \| \nabla \phi^\ell \| _p^p +  \int^t_0 |  w_{\tilde{\vartheta}}f^{\ell + 1 }|_{p,-}^p .
			\end{split}
			\Ee
			Note that by the H\"older inequality, (\ref{grad_estimate}), and (\ref{k_vartheta_comparision}),			\Be\begin{split}\label{k_p_bound}
				&\int_{\R^3}|w_{\tilde{\vartheta}}f^{\ell + 1 }(v)|^{p-1} \int_{\R^3} \mathbf{k}_{\tilde{\varrho}} (v,u)  | w_{\tilde{\vartheta}}f^\ell(u)|\dd u \dd v  \\
				\lesssim & \ \| w_{\tilde{\vartheta}} f^{\ell + 1 }  \|_{L^p_v}^{\frac{1}{p-1}}  \left\| \int_{\R^3} \mathbf{k}_{\tilde{\varrho}} (v,u)^{1/q} \mathbf{k}_{\tilde{\varrho}} (v,u)^{1/p}  | w_{\tilde{\vartheta}} f^\ell(u)|\dd u  \right\|_{L^p_v }\\
				\lesssim & \  \| w_{\tilde{\vartheta}} f^{\ell + 1 }  \|_{L^p_v}^{\frac{1}{p-1}} \left( \int_{\R^3} \mathbf{k}_{\tilde{\varrho}} (v,u) \dd u\right)^{1/q}\left\| 
				\left( \int_{\R^3} \mathbf{k}_{\tilde{\varrho}}  (v,u) |w_{\tilde{\vartheta}}f^\ell(u)|^p \dd u \right)^{1/p} \right\|_{L^p_v }\\
				\lesssim &  \ o(1) \| w_{\tilde{\vartheta}} f ^{\ell + 1 } \|_{L^p_v}^ p + \| w_{\tilde{\vartheta}} f ^{\ell} \|_{L^p_v}^ p.
			\end{split}\Ee
			From a standard elliptic theorem and (\ref{smallfphi}), we have 
			\Be\label{phi_p_bound}
			\int_0^t\| \nabla \phi ^\ell \|_p^p\lesssim \int^t_0 \| w_{\tilde{\vartheta}}f^\ell \|_p^p. 
			\Ee
			
			Now we focus on $\int^t_0 |w_{\tilde{\vartheta}}f^{\ell + 1}|^p_{p,-}$ in (\ref{energy_f_p}). We plug in (\ref{F_ell_BC}) and then decompose $\gamma_+^\e  \cup \gamma_+   \backslash \gamma_+^\e $ where $\e$ is small but satisfies (\ref{lower_bound_e}). This leads
			\Be\notag
			\begin{split}
				&\int^t_0 |w_{\tilde{\vartheta}}f^{\ell + 1}|_{p,-}^p\\
				\lesssim & \ 
				\int^t_0  \int_{\p\O} \left(
				\int_{\gamma_+^\e(x)} w_{\tilde{\vartheta}} | f^\ell | \sqrt{\mu} \{n \cdot u\} \dd u 
				\right)^p +  \int^t_0  \int_{\p\O}  \left(
				\int_{\gamma_+(x)\backslash \gamma_+^\e(x)} w_{\tilde{\vartheta}} | f^\ell|\sqrt{\mu} \{n \cdot u\} \dd u 
				\right)^p\\
				\lesssim & \  \Big( \int_{\gamma_+^\e} \sqrt{\mu} \{n \cdot u\} \dd u  \Big)^{p/q}\int^t_0 |w_{\tilde{\vartheta}}f^\ell |_{p,+}^p
				+ \int^t_0 \int_{\gamma_+ \backslash \gamma_+^\e} |w_{\tilde{\vartheta}}f^\ell \sqrt{\mu}|^p\\
				\lesssim & \ o(1)  \int^t_0 |w_{\tilde{\vartheta}}f^\ell |_{p,+}^p
				+  \int^t_0 \int_{\gamma_+ \backslash \gamma_+^\e} |w_{\tilde{\vartheta}}f^\ell \sqrt{\mu}|^p.
			\end{split}
			\Ee
			From (\ref{fell_local}), Lemma \ref{le:ukai}, (\ref{bound_Gamma_k}), and (\ref{k_p_bound})
			\Be\label{bound_p_est}
			\int^t_0 |w_{\tilde{\vartheta}}f^{\ell+1} |_{p,-}^p\lesssim \| w_{\tilde{\vartheta}}f(0)\|_p^p + o(1)  \int^t_0 |w_{\tilde{\vartheta}}f^\ell |_{p,+}^p  +(1+ \| w_\vartheta f^{\ell - 1} \|_\infty) \int^t_0 \|   ( w_{\tilde{\vartheta}} f^\ell \|_p^p +  w_{\tilde{\vartheta}} f^{\ell-1} \|_p^p) .
			\Ee
			
			Collecting terms from (\ref{energy_f_p}), (\ref{k_p_bound}), (\ref{phi_p_bound}), and (\ref{bound_p_est}), we conclude that for $\sup_{\ell \ge 0 } \|  w_{\tilde{\vartheta}} f^\ell \|_\infty \ll 1 $,
			\Be \begin{split}\label{Lp_estimate_f}
			\| w_{\tilde{\vartheta}}f^{\ell + 1 } (t) \|_p^p  & + \int^t_0 \| \nu_{\phi^\ell}^{1/p} w_{\tilde{\vartheta}} f^{\ell + 1 } \|_p^p + \int^t_0 |w_{\tilde{\vartheta}}f^{\ell + 1 }|_{p,+}^p 
			\\ & \lesssim \| w_{\tilde{\vartheta}}f(0) \|_p^p + (1+ \| w_\vartheta f^\ell \|_\infty) \int^t_0  (\| w_{\tilde{\vartheta}} f^\ell \|_p^p + \| w_{\tilde{\vartheta}} f^{\ell-1} \|_p^p +  \int^t_0 |w_{\tilde{\vartheta}}f^\ell |_{p,+}^p )
			\\ & \lesssim w_{\tilde{\vartheta}}f(0) \|_p^p + \left( o(1) +t (1 +  \| w_\vartheta f^\ell \|_\infty) \right) \max_{m = \ell, \ell-1 } \sup_{0 \le t \le T^{**}} \mathcal E^m .
			\end{split}  \Ee

			\vspace{4pt}
			
			\textit{Step 2.} By taking derivatives $\p \in\{ \nabla_{x_{i}}, \nabla_{v_{i}} \}$ to (\ref{fell_local}),
			\Be \label{eqtn_nabla_f}   [ \p_{t} + v\cdot\nabla_{x} - q \nabla_{x}\phi^\ell \cdot\nabla_{v} + \nu_{\phi^\ell, w_{\tilde{\vartheta}}} ]( w_{\tilde{\vartheta}}\p f^{\ell +1}) =  w_{\tilde{\vartheta}}\mathcal{G}^{\ell +1}, \Ee
			where
			\Be  
			\begin{split}	\label{mathcal_G}
				\mathcal{G}^{\ell +1}  =& 
				- \p v \cdot\nabla_{x}f^{\ell +1} 
				+ q \p \nabla \phi^{\ell} \cdot \nabla_v f^{\ell +1}\\&
				+ \p \Gamma_\text{gain}(f^\ell,f^\ell) - \p \Gamma_\text{loss} (f^{\ell +1}, f^\ell ) - \p \big[\nu(v) + q \frac{v}{2} \cdot \nabla \phi^\ell(t,x)\big] f^{\ell +1} - \p K f^\ell - q_1 \p(v\cdot\nabla_{x}\phi^\ell \sqrt{\mu}).
			\end{split}\Ee
			Here we have used
			\Be
			\nu_{\phi^\ell, w_{\tilde{\vartheta}}} =  \nu_{\phi^\ell, w_{\tilde{\vartheta}}} (t,x,v) := 
			\begin{bmatrix} \nu(v) +  \frac{v}{2} \cdot \nabla \phi^\ell (t,x) 
			+ \frac{\nabla_x \phi^\ell \cdot \nabla_v w_{\tilde{\vartheta}}}{w_{\tilde{\vartheta}}} & 0
						\\ 0 &  \nu(v) -  \frac{v}{2} \cdot \nabla \phi^\ell(t,x) + \frac{\nabla_x \phi^\ell \cdot \nabla_v w_{\tilde{\vartheta}}}{w_{\tilde{\vartheta}}} \end{bmatrix}
			. \label{nu_phi}
			\Ee
Denote
\Be \notag
\nu_{\phi^\ell, w_{\tilde{\vartheta}} + } = \nu(v) +  \frac{v}{2} \cdot \nabla \phi^\ell(t,x) 
			+ \frac{\nabla_x \phi^\ell \cdot \nabla_v w_{\tilde{\vartheta}}}{w_{\tilde{\vartheta}}} , \, 
			\nu_{\phi^\ell, w_{\tilde{\vartheta}} - } = \nu(v) -  \frac{v}{2} \cdot \nabla \phi^\ell(t,x) 
			+ \frac{\nabla_x \phi^\ell \cdot \nabla_v w_{\tilde{\vartheta}}}{w_{\tilde{\vartheta}}}.
\Ee
			\hide
			Note that we used $K_v$ to denote 
			\begin{equation} \label{Kv}
			K_v f(v) =\int_{\mathbb{R}^3} \left\{ \nabla_u \mathbf{k}_1 + \nabla_ v\mathbf{k}_1 + \nabla_u \mathbf{k}_2 + \nabla_v \mathbf{k}_2\right\}(v,u) f(u) \dd u \lesssim \|wf\|_{\infty},
			\end{equation}
			where the last inequality comes from (\ref{nabla_k1}) and (\ref{nabla_k2}).   \\
			
			Let us choose 
			\Be\label{beta_p_1}
			\frac{p-2}{p}< \beta <\frac{p-1}{p} , \ \ p>3.
			\Ee\unhide

			From (\ref{alpha_invariant}) and (\ref{eqtn_nabla_f}), for $\iota = +$ or $-$ we have
			\Be \label{Boltzmann_p} \begin{split}
				&\frac{1}{p} |w_{\tilde{\vartheta}}\alpha_{f^{\ell},\e,\iota}^\beta  \p f^{\ell +1}_\iota |^{p-1}  \big[\p_t + v\cdot \nabla_x {-\iota} \nabla_x \phi^\ell \cdot \nabla_v
				+ \nu_{\phi^\ell, w_{\tilde{\vartheta}}\iota} 
				\big] |w_{\tilde{\vartheta}} \alpha_{f^\ell,\e,\iota}^\beta  \p f^{\ell +1}_\iota | \\
				&=   \alpha_{f^\ell,\e,\iota}^{\b p} | w_{\tilde{\vartheta}} \p f^{\ell +1}_\iota|^{p-1} \big[\p_t + v\cdot \nabla_x {-\iota} \nabla_x \phi^\ell \cdot \nabla_v + \nu_{\phi^\ell,w_{\tilde{\vartheta}} \iota} \big]|w_{\tilde{\vartheta}} \p f^{\ell +1}_\iota | \\
				&= w_{\tilde{\vartheta}}^p \alpha_{f^\ell,\e,\iota}^{\b p} |  \p f^{\ell +1}_\iota|^{p-1}    \mathcal{G}^{\ell +1}_\iota .
			\end{split}\Ee
			
			From (\ref{nabla_nu}), (\ref{nu_phi}), (\ref{K_v}), and (\ref{Gamma_v})
			\Be
			\begin{split}\label{G est}
				|\mathcal{G}| &\lesssim |\nabla_{x} f^{\ell +1}| 
				+  | \nabla^2 \phi^\ell | |\nabla_v f^{\ell +1}|
				+ |\Gamma_\text{gain}(\p f^\ell, f^\ell)| + |\Gamma_\text{gain} (f^\ell, \p f^\ell )|  + |\Gamma_\text{loss}(f^{\ell +1}, \p f^\ell)| + | \Gamma_\text{loss} ( \p f^{\ell +1}, f^\ell ) |+   \\
				&\quad  + | K \p f^\ell |  + |f^{\ell+1}|  +  |\Gamma_{v,\text{gain}}(f^\ell,f^\ell)| + |\Gamma_{v, \text{loss} } ( f^{\ell +1}, f^\ell ) | + |K_{v}f^\ell |\\
				& \quad  + 
				w_\vartheta(v)^{-1/2}(  | \nabla \phi^\ell |  + | \nabla^2\phi^\ell | ) (1+\|w_\vartheta f^{\ell +1}\|_\infty )    .\end{split}
			\Ee
			Now we apply Lemma \ref{lem_Green} to (\ref{Boltzmann_p}) to both $f^{\ell +1}_+$ and $f^{\ell + 1}_-$ separately and add them together to obtain
			\Be\begin{split} \label{pf_green}
				& \| w_{\tilde{\vartheta}} \alpha_{f^\ell,\e}^\beta 
				\p f^{\ell +1} (t)\|_p^p
				+ \int^t_0 \| \nu_{\phi^\ell, w_{\tilde{\vartheta}}}^{1/p} w_{\tilde{\vartheta}} \alpha_{f^\ell,\e}^\beta 
				\p f ^{\ell +1}  \|_p^p  
				+ 
				\int^t_0 | w_{\tilde{\vartheta}} \alpha_{f^\ell,\e}^\beta 
				\p f^{\ell +1}   |_{p,+}^p  
				\\
				\leq&  \  \| w_{\tilde{\vartheta}} \alpha_{f^\ell,\e}^\beta 
				\p f (0)\|_p^p  + 
				\underbrace{\int^t_0 |w_{\tilde{\vartheta}} \alpha_{f^\ell,\e}^\beta 
					\p f ^{\ell +1} |_{p,-}^p  } _{(\ref{pf_green})_{\gamma_-}}
				+ \underbrace{\int_{0}^{t} \iint_{\O\times\R^{3}}p | \alpha_{f^\ell,\e}^{\b p} w_{\tilde{\vartheta}}^p ( \p f^{\ell +1})^{p-1} |\mathcal{G}^{\ell +1}| } _{(\ref{pf_green})_{\mathcal{G}}}.
			\end{split}\Ee

			First we consider $(\ref{pf_green})_{\mathcal{G}}$. Directly, the contribution of $|\nabla_{x} f^{\ell +1}| 
			+  | \nabla^2 \phi^\ell | |\nabla_v f^{\ell +1}|$ of (\ref{G est}) in $(\ref{pf_green})_{\mathcal{G}}$ is bounded by
			\Be
			\begin{split}\label{pf_p_nabla}
				(1+  \sup_{0 \leq s \leq t} \| \nabla^2 \phi^\ell \|_\infty ) \int_0^t \| w_{\tilde{\vartheta}}\alpha_{f^\ell,\e}^\beta \p f^{\ell +1} \|_p^p.
			\end{split}
			\Ee

			From (\ref{vKsum}), (\ref{bound_Gamma_nabla_vf1}), and (\ref{bound_Gamma_nabla_vf2}), the contribution of $|\Gamma_\text{gain}( f^\ell, \p f^\ell)| + |\Gamma_{\mathrm{gain}}(\p f^\ell, f^\ell)|+ |\Gamma_\text{loss} (f^{\ell +1}, \p f^l ) | + |K \p f^\ell|$ of (\ref{G est}) in $(\ref{pf_green})_{\mathcal{G}}$ is bounded by 
			\Be\begin{split}\label{pf_p_K}
				&(1 +\sup_{0 \leq s\leq  t} \|w_\vartheta f^\ell(s)\|_{\infty} + \sup_{0 \le s \le t } \| w_\vartheta f^{\ell +1} (s) \|_\infty) \\
				& \times \int^t_0    { \iint_{\O\times\R^{3}}  |\alpha_{f^\ell,\e}^{\b  } w_{\tilde{\vartheta}}  \p f^{\ell +1}(v)|^{p-1} 
					\int_{\R^3} |\alpha_{f^\ell,\e}(v)^\beta \mathbf{k}_\varrho(v,u)w_{\tilde{\vartheta}} (v) \p f^\ell(u)| \dd u  
					\dd v \dd x } \dd s.
			\end{split}\Ee
			The estimate of (\ref{pf_p_K}) is carried out in \textit{Step 3}.
			
			From (\ref{bound_Gamma_nabla_vf2}), the contribution of $|\Gamma_{\mathrm{loss}}(\p f^{\ell +1}, f^\ell)|$ of (\ref{G est}) in $(\ref{pf_green})_{\mathcal{G}}$ is bounded by 
			\Be\label{pf_p_0}
			\sup_{0 \leq s\leq t}\| w_\vartheta f^\ell(s) \|_\infty
			\int^t_0  \| \nu_{\phi_f, w_{\tilde{\vartheta}}}^{1/p} w_{\tilde{\vartheta}}   \alpha_{f^\ell,\e}^\beta \p f^{\ell +1} \|_p^p
			.
			\Ee

			For the $|f^{\ell+1}|$ contribution of (\ref{G est}) in $(\ref{pf_green})_{\mathcal{G}}$, we bound
			\Be
			\begin{split}\label{pf_p_1}
				&\int^t_0 
				\iint_{\O \times \R^3}p  |w_{\tilde{\vartheta}}^p \alpha_{f^\ell,\e}^{\beta p } (\p f^{\ell +1})^{p-1} |
				|f^{\ell+1}|  \dd x \dd v \dd s\\
				\lesssim & \ \int^t_0 \iint_{\O \times \R^3}
				| \nu_{\phi_f,w_{\tilde{\vartheta}} }^{1/p}  w_{\tilde{\vartheta}} \alpha_{f^\ell,\e}^{\beta} \p f^{\ell +1}| ^{p-1}
				| w_{\tilde{\vartheta}}f^{\ell +1}|
				\frac{  
					 |\alpha_{f^\ell,\e}(s,x,v)|^{\beta} }{\langle v \rangle^ {{(p-1)}/{p}} }   \dd x \dd v \dd s\\
				\lesssim & \ o(1) \int^t_0 \iint_{\O \times \R^3}
				| \nu_{\phi_f, w_{\tilde{\vartheta}}}^{1/p} w_{\tilde{\vartheta}} \alpha_{f^{\ell},\e}^{\beta} \p f^{\ell +1}| ^{p}
				+(1+\delta_1/\Lambda_1 )  \int^t_0 \iint_{\O \times \R^3} |w_{\tilde{\vartheta}} f^{\ell +1} |^p.
			\end{split}
			\Ee
			Here we have used the fact that, from (\ref{weight}) and (\ref{integrable_nabla_phi_f})
			\Be\label{upper_vb}
			\begin{split}
				&| \alpha_{f^\ell,\e}(s,x,v) |\\  \leq & \ 
				2 (\mathbf{1}_{s +1 \geq \tb(s,x,v)}
				|\vb(s,x,v)| 
				+ \mathbf{1}_{s \leq \tb(s,x,v) +1})\\
				\lesssim &  \ 1+   |v| + \int^{0}_{-1} |\nabla \phi^\ell (\tau,X(\tau;s,x,v))| \dd \tau +  \int^s_{0} |\nabla \phi^\ell (\tau,X(\tau;s,x,v))| \dd \tau\\
				\lesssim &  \ ( 1+ \| w_{\vartheta} f_0 \|_\infty+  {\delta_1}/{\Lambda_1} ) \langle v\rangle, 
			\end{split}
			\Ee
			and from (\ref{beta_condition}), $
			\frac{  
				|\alpha_{f^\ell,\e}(s,x,v)|^{\beta} }{\langle v\rangle^{{(p-1)}/{p}} }  \lesssim (1+ {\delta_1}/{\Lambda_1} )\times \frac{\langle v\rangle^\beta }{\langle v\rangle^ {{(p-1)}/{p}} }\lesssim (1 + {\delta_1}/{\Lambda_1}).$

			From (\ref{Gvloss}), the contribution of $|\Gamma_{v, \mathrm{loss}} (f^{\ell+1},f^\ell)|$ of (\ref{G est}) in $(\ref{pf_green})_{\mathcal{G}}$ is bounded by
			\Be
			\begin{split}\label{pf_p_Gamma_v_-}
				& \|w_{\vartheta}f^{\ell+1} \|_{\infty} \int^t_0 \iint_{\O \times \R^3} p   | w_{\tilde \vartheta} \alpha_{f^\ell,\e}^\beta \p f^{\ell+1}|^{p-1} 
				|\alpha_{f^\ell,\e}(v) |^\beta {\langle v \rangle} w_{\tilde \vartheta} (v) w_{\vartheta}(v)^{-1} 
				\| w_{ \vartheta} f^\ell(s,x, \cdot) \|_{L^p(\R^3)}\\
				\lesssim & \ 	
				\|w_{\vartheta}f^{\ell+1}\|_{\infty} 
				\left\{
				\int^t_0 \iint_{\O \times \R^3}  | \alpha_{f^\ell,\e}^\beta \p f^{\ell+1}|^p
				+   \int^t_0 \iint_{\O \times \R^3}  | w_{ \vartheta} f^\ell|^p
				\right\},
			\end{split}
			\Ee
			where we have used, from (\ref{upper_vb}), $|\alpha_{f^\ell,\e}(v)|^\beta {\langle v \rangle}w_{\tilde \vartheta} (v)w_{\vartheta}(v)^{-1}  \lesssim 1$.
			
			From (\ref{vKsum}) and (\ref{Gvgain}), the contribution of $|\Gamma_{v, \mathrm{gain}}|$ and $|K_vf|$ in $(\ref{pf_green})_{\mathcal{G}}$ is bounded by 
			\Be\begin{split}\label{pf_p_Kv}
				&(1 +\sup_{0 \leq s\leq  t} \|w_\vartheta f^\ell \|_{\infty}) \int^t_0  
				\iint_{\O\times\R^{3}}  | \alpha_{f^\ell,\e}^{\b p} ( w_{\tilde{\vartheta}} \p f^{\ell+1} (v))^{p-1} |
				\int_{\R^3} \mathbf{k}_\varrho(v,u) \frac{w_{\tilde{\vartheta}}(v)}{w_{\tilde{\vartheta}}(u)} |f^\ell(u)| 
				\\
				\lesssim & \ 
				o(1) \int^t_0 \iint_{\O \times \R^3}
				| \nu_{\phi^\ell}^{1/p} \alpha_{f^\ell,\e}^{\beta} \p f^{\ell+1}| ^{p}
				+(1 +\sup_{0 \leq s\leq  t} \|w_{\vartheta}f^\ell (s)\|_{\infty}) \int^t_0 \iint_{\O \times \R^3} |w_{\tilde \vartheta} f^\ell |^p,
			\end{split}\Ee
			where we have used, for $1/p+1/p^*=1$ and $0< \tilde{\varrho} \ll \varrho$, from (\ref{k_vartheta_comparision}), (\ref{grad_estimate}),
			\Be\begin{split}\notag
				&\int_{\R^3} | \alpha_{f^\ell,\e}(v) ^{\b p}  (w_{\tilde{\vartheta}} \p f^{\ell+1} (v))^{p-1} |\int_{\R^3} \mathbf{k}_{\tilde{\varrho}}(v,u) w_{\tilde{\vartheta}}(u) |f^\ell(u)| \dd u \dd v \\
				\lesssim & \ \int_{\R^3} | \alpha_{f^\ell,\e}(v)^{\b p}  ( w_{\tilde{\vartheta}} \p f^{\ell+1} (v))^{p-1} |\int_{\R^3}\mathbf{k}_{\tilde{\varrho}}(v,u)^{1/{p^*}} \mathbf{k}_{\tilde{\varrho}}(v,u)^{1/p}  | w_{\tilde{\vartheta}}f^\ell(u)| \dd u \dd v\\
				\lesssim & \ \int_{\R^3} \frac{ |\alpha_{f^\ell,\e}(v) |^\beta}{\langle v\rangle^{\frac{p-1}{p}}} | \langle v\rangle^{1/p} w_{\tilde{\vartheta}} \alpha_{f^\ell,\e}^\beta\p f^{\ell+1} (v)|^{p-1}\\
				& \ \ \  \times \left(\int_{\R^3} \mathbf{k}_{\tilde{\varrho}}(v,u) \dd u \right)^{1/{p^*}}
				\left(
				\int_{\R^3}  \mathbf{k}_{\tilde{\varrho}}(v,u)   | w_{\tilde{\vartheta}} f^\ell (u)|  ^p  \dd u 
				\right)^{1/p} \dd v\\
				\lesssim & \ \left( \int_{\R^3}  |\langle v\rangle^{1/p}w_{\tilde{\vartheta}} \alpha_{f^\ell,\e}^\beta\p f^{\ell+1} (v)|^p  \dd v\right)^{\frac{p-1}{p}}
				\left(
				\int_{\R^3} \int_{\R^3} \mathbf{k}_{\tilde{\varrho}}(v,u)  |w_{\tilde{\vartheta}}f^\ell(u)|^p \dd u \dd v
				\right)^{\frac{1}{p}}\\
				\lesssim & \   \left(\int_{\R^3}
				| \nu_{\phi^\ell}^{1/p} w_{\tilde{\vartheta}}\alpha_{f^\ell,\e}^{\beta} \p f^{\ell+1}| ^{p}\right)^{\frac{p-1}{p}}
				\left( \int_{\R^3} |w_{\tilde{\vartheta}}f^\ell|^p\right)^{\frac{1}{p}}.
			\end{split}\Ee

			Note that from the standard elliptic estimate and (\ref{smallfphi}),
			\Be\label{nabla_2_phi_p}
			\|   \phi^\ell (t) \|_{W^{2,p}(\O  ) } \lesssim \left\|\int_{\R^3} (f^\ell_+ - f_-^\ell)   (t,x,v)\sqrt{\mu(v)} \dd v \right\|_{L^p(\O )}
			\lesssim \|   f^\ell(t) \|_{L^p (\O  \times \R^3)}
			.
			\Ee
			Then from (\ref{nabla_2_phi_p}) we bound the contribution of $w_{\tilde{\vartheta}}(v)^{-1/2}(  | \nabla \phi^\ell |  + | \nabla^2\phi^\ell | ) (1+\|w_{\vartheta}f^{\ell+1}\|_\infty )$ of (\ref{G est}) in (\ref{Boltzmann_p}) by 
			\Be \label{pf_p_2}
			\begin{split}
				&(1+\|w_\vartheta f^{\ell+1}\|_\infty ) \int^t_0 
				\iint_{\O \times \R^3}p | w_{\tilde{\vartheta}} \alpha_{f^\ell,\e}^{\beta   }  \p f^{\ell+1}|^{p-1} 
				\frac{ |\alpha_{f^\ell,\e}(v)|^{\beta}}{w_{\tilde{\vartheta}}(v)^{1/2}} (  | \nabla \phi^\ell |  + | \nabla^2\phi^\ell | )  
				\\
				\lesssim  & \ (1+\|w_\vartheta f^{\ell+1}\|_\infty ) \int^t_0 
				\iint_{\O \times \R^3}  | w_{\tilde{\vartheta}} \alpha_{f^\ell,\e}^{\beta   }  \p f^{\ell+1}|^{p-1} 
				w_{\tilde{\vartheta}}^{-1/4}  (  | \nabla \phi^\ell |  + | \nabla^2\phi^\ell | )  
				\\
				\lesssim  & \ (1+\|w_\vartheta f^{\ell+1}\|_\infty ) 
				\left\{ o(1)
				\int^t_0 
				\iint_{\O \times \R^3}  |  w_{\tilde{\vartheta}} \alpha_{f^\ell,\e}^{\beta   }  \p f^{\ell+1}|^{p } 
				+\int_0^t \iint_{\O \times \R^3} | w_{\tilde{\vartheta}} \alpha_{f^{\ell-1},\e}^\beta \p f^\ell |^p 
				+\int^t_0 
				\int_{\O  }  \| \phi^\ell \|_{W^{2,p}}^p \int_{\R^3} w_{\tilde{\vartheta}}^{-p/4}  \right\}\\
				\lesssim& \ o(1) (1+\|w_\vartheta f^{\ell+1}\|_\infty ) \int^t_0 \iint_{\O \times \R^3} |w_{\tilde{\vartheta}}\alpha_{f^\ell,\e}^\beta \p f^{\ell+1}|^p +\int_0^t \iint_{\O \times \R^3} | w_{\tilde{\vartheta}} \alpha_{f^{\ell-1},\e}^\beta \p f^\ell |^p + \int^t_0 \iint_{\O \times \R^3} |f^\ell|^p.
			\end{split}
			\Ee
			where we have used, from (\ref{upper_vb}), $\alpha_{f^\ell,\e}(v)^\beta w_\vartheta(v)^{-1/2} \lesssim w_\vartheta(v)^{-1/4}$.

			\hide
			
			Note that $\Gamma_v (f,f)$ is defined in (\ref{Gamma_v}) and bounded by 
			\[
			\Gamma_{v}(f,f) \lesssim \frac{\langle v \rangle}{w(v)}\|wf\|_{\infty}^{2}
			\]
			from (\ref{Gvloss}) and (\ref{Gvgain}). Also we note that 
			\[
			| \Gamma(\p f, f) + \Gamma(f, \p f) | \lesssim \|wf\|_{\infty} \int_{\R^3} \mathbf{k}_{\varrho}(v,u )  |  \p f (u)| \dd u
			\dd v.
			\]

			We further apply Lemma \ref{lem_Green} to (\ref{Boltzmann_p}) and use (\ref{vKsum}), (\ref{bound_Gamma_nabla_vf1}), (\ref{bound_Gamma_nabla_vf2}), (\ref{Gvloss}), (\ref{Gvgain}), and (\ref{nabla_2_phi_p}) to obtain, together with (\ref{pf_p_1}) and (\ref{pf_p_2}),   
			\Be\begin{split}\label{pf_green}
				& \| \alpha^\beta 
				\p f (t)\|_p^p
				+ \int^t_0 \| \nu_{\phi_f}^{1/p} \alpha^\beta 
				\p f   \|_p^p  
				+ \underbrace{\int^t_0 |\alpha^\beta 
					\p f   |_{p,+}^p  }_{(\ref{pf_green})_{1}}
				\\
				\lesssim&  \  \| \alpha^\beta 
				\p f (0)\|_p^p   + 
				\underbrace{
					\int^t_0 |\alpha^\beta 
					\p f  |_{p,-}^p
				}_{(\ref{pf_green})_2}+  (1+ \sup_{0 \leq s\leq  t} \| \nabla^2 \phi_{f} (t) \|_\infty) \int^t_0 \| \alpha^\beta 
				\p f  \|_p^p + \int^t_0 \iint_{\O \times \R^3} |f|^p
				\\
				&+ (1 +\sup_{0 \leq s\leq  t} \|wf(s)\|_{\infty}) \int^t_0  \underbrace{ \iint_{\O\times\R^{3}}
					\alpha(v)^{p\beta}    |  \p f (v)|^{p-1}  
					\int_{\R^3} \mathbf{k}_{\varrho}(v,u )  |  \p f (u)| \dd u
					\dd v
					\dd x }_{(\ref{pf_green})_3} \dd s \\
				&+ (1 +\sup_{0 \leq s\leq  t} \|wf(s)\|_{\infty}) \int^t_0   \underbrace{ \iint_{\O\times\R^{3}} \alpha(v)^{\b p} |\p f(v)|^{p-1} 
					\int_{\R^3} \mathbf{k}_\varrho(v,u)  |f(u)| \dd u  
					\dd v \dd x }_{(\ref{pf_green})_4} \dd s.
			\end{split}\Ee
			\unhide
			
			\hide
			
			Estimates $(\ref{pf_green})_4$ and $(\ref{pf_green})_5$ are easily obtained by H\"older inequality,
			\Be \label{151_4}
			\begin{split}
				(\ref{pf_green})_4 &= \iint_{\O\times\R^{3}} \alpha^{\b p} |\p f|^{p-1} \langle v \rangle |f| \dd v \dd x = \iint_{\O\times\R^{3}} |\langle v \rangle\alpha^{\b}f| |\alpha^{\b}\p f|^{p-1} \dd v \dd x  \\
				&\lesssim \|\alpha^{\b}\p f\|_{p}^{p-1} \|\langle v \rangle\alpha^{\b}f\|_{p} \lesssim \|wf\|_{\infty} \|\alpha^{\b}\p f\|_{p}^{p-1},
			\end{split}
			\Ee  
			and
			\Be \label{151_5}
			\begin{split}
				(\ref{pf_green})_5  &\lesssim \iint_{\O\times\R^{3}} \alpha^{\b p} |\p f|^{p-1} \Big( \sqrt{\mu}^{1-\delta} + \frac{\langle v \rangle}{w(v)} \Big) \dd v \dd x  \\
				&\lesssim \|\alpha^{\b}\p f\|_{p}^{p-1} \Big\|\alpha^{\b}\Big( \sqrt{\mu}^{1-\delta} + \frac{\langle v \rangle}{w(v)} \Big) \Big\|_{p} \lesssim \|\alpha^{\b}\p f\|_{p}^{p-1}.
			\end{split}
			\Ee\unhide
			
			\vspace{4pt}

			\textit{Step 3.} We focus on (\ref{pf_p_K}). 
			With $N>0$, we split the $u$-integration of (\ref{pf_p_K}) into the integrations over $\{|u| \leq N\}$ and $\{|u| \geq N\}$.
			
			For $\{|u| \geq N\}$ and $0< \tilde{\varrho} \ll \varrho$, by Holder inequality with $\frac{1}{p} + \frac{1}{p^*}=1$
			\Be\begin{split}\label{est_k_p_f}
				&\int_{|u| \geq N} | \alpha_{f^\ell,\e}^\beta(v)  \mathbf{k}_{\tilde{\varrho}} (v,u) \p f^{\ell}(u)|\\
				\leq & \ 
				| \alpha_{f^\ell,\e}^\beta (v) |
				{\left( \sum_{\iota = \pm } \int_{|u| \geq N}  \mathbf{k}_{\tilde{\varrho}}  (v,u) \frac{1}{\alpha_{f^{\ell-1},\e,\iota}(u)^{\beta {p^*}}}   \right)^{1/{p^*}}}
				\left(
				\int_{|u| \geq N} \mathbf{k}_{\tilde{\varrho}}  (v,u)  |  \alpha_{f^{\ell-1},\e}^{\beta  }\p f^\ell (u)|^p  
				\right)^{1/p}\\
				\lesssim & \ \alpha_{f^\ell,\e}^\beta(v) \left(
				\int_{|u| \geq N}  \mathbf{k}_{\tilde{\varrho}}  (v,u)  |\alpha_{f^{\ell-1},\e}^{\beta  }\p f^\ell (u)|^p \dd u 
				\right)^{1/p},
			\end{split}\Ee
			where have used Proposition \ref{prop_int_alpha} with $\beta q< \frac{p-1}{p}\frac{p}{p-1}=1$ from (\ref{beta_condition}).
			
			Then the contribution of $\{|u| \geq N\}$ in (\ref{pf_p_K}) is bounded by 
			\Be\begin{split}\label{small result}
				&\int^t_0 \int_{\O} \int_{v \in \R^3}  | \nu_{\phi_f}^{1/p} w_{\tilde{\vartheta}}\alpha_{f^\ell,\e}^\beta \p f ^{\ell+1}(v)|^{p-1} \frac{ |\alpha_{f^\ell,\e}(v)|^\beta}{ \langle v\rangle^{\frac{p-1}{p}}}\\
				& \ \ \ \   \ \ \ \ \ \ \ \ \times 
				\int_{|u| \geq N} \mathbf{k}_\varrho (v,u)\frac{w_{\tilde{\vartheta}}(v)}{w_{\tilde{\vartheta}}(u)} | w_{\tilde{\vartheta}} \p f^\ell(u)| \dd u \dd v \dd x \dd s  \\
				\leq& \ \int^t_0 \int_{\O} \bigg(\int_{v}  |\nu_{\phi_f}^{1/p}w_{\tilde{\vartheta}} \alpha_{f^\ell,\e}^\beta \p f^{\ell+1} (v)|^{p }\bigg)^{1/q}
				\bigg(
				\int_{|u| \geq N} |w_{\tilde{\vartheta}}\alpha_{f^{\ell-1},\e}^{\beta  }\p f^\ell (u)|^p  \int_{v}\mathbf{k}_{\tilde{\varrho}}(v,u)  
				\bigg)^{1/p}\\
				\lesssim & \ 
				o(1)
				\int_0^t \|\nu_{\phi_f}^{1/p} w_{\tilde{\vartheta}} \alpha_{f^\ell,\e}^\beta \p f^{\ell+1}  (s)\|_p^p \dd s
				+\int_0^t \| w_{\tilde{\vartheta}} \alpha_{f^{\ell-1},\e}^\beta \p f^\ell  (s)\|_p^p \dd s,
			\end{split}\Ee
			where we have used, from (\ref{upper_vb}), $\frac{ |\alpha_{f^\ell,\e}(v)|^\beta}{ \langle v\rangle^{\frac{p-1}{p}}}\lesssim 1$ for $\beta$ in (\ref{beta_condition}), (\ref{k_vartheta_comparision}), and (\ref{grad_estimate}).

			The contribution of $\{|u| \leq N\}$ in (\ref{pf_p_K}) is bounded by, from the H\"older inequality, 
			\begin{eqnarray}
			&& \int^t_0  \int_\O 
			\int_{\R^3}
			| \nu_{\phi_f}^{1/p} w_{\tilde{\vartheta}}\alpha_{f^\ell,\e}^\beta \p f^{\ell+1} (v)|^{p-1} \notag\\
			&& \ \ \  \times 
			\int_{|u| \leq N} \sum_{\iota = \pm } \mathbf{k}_\varrho(v,u) \frac{ w_{\tilde{\vartheta}}(v)}{ w_{\tilde{\vartheta}}(u)}\frac{ | \alpha_{f^\ell,\e}(v) |^{\beta} |  w_{\tilde{\vartheta}}  \alpha_{f^{\ell-1},\e} ^\beta \p f^\ell (u)| }{\langle v\rangle^{(p-1)/p}  \alpha_{f^{\ell-1},\e,\iota}(u)^{\beta}}     \dd u
			\dd v
			\dd x \dd s \notag\\
			&\leq& 
			\int^t_0    \| \nu_{\phi_f}^{1/p} w_{\tilde{\vartheta}}  \alpha_{f^\ell,\e}^\beta \p f^{\ell+1}  (s)\|_p^{p-1}\notag\\
			&& \ \ \   \times 
			\Big[
			\int_\O
			\int_{\R^3}
			\Big(
			\underline{\underline{
					\int_{|u| \leq N} \sum_{\iota = \pm}
					\mathbf{k}_{\tilde{\varrho}} (v,u) 
					\frac{ |
						w_{\tilde{\vartheta}} \alpha_{f^{\ell-1},\e}^\beta \p f^\ell ( u)|}{ 
						\alpha_{f^{\ell-1},\e,\iota}(u)^{\beta}}
					\dd u}}
			\Big)^p
			\dd v
			\dd x  \Big]^{1/p}
			\dd s 
			.\label{double_underline}
			\end{eqnarray}
			where we have used (\ref{k_vartheta_comparision}) and the fact $ |\alpha_{f^\ell,\e} |^\beta/ \langle v \rangle^{\frac{p-1}{p}} \lesssim 1$ from (\ref{upper_vb}) and (\ref{beta_condition}).
			
			By the H\"older inequality, we bound an underlined $u$-integration inside (\ref{double_underline}) as
			\Be
			\|w_{\tilde{\vartheta}} \alpha_{f^{\ell-1},\e}^\beta \p f^\ell(\cdot) \|_{L^p(\R^3)}
			\times 
			\bigg(  \sum_{\iota= \pm}
			\int_{\R^3}
			\frac{e^{-p^*\tilde{\varrho}  |v-u|^2}}{|v-u|^{p^*}} \frac{\mathbf{1}_{|u| \leq N}}{\alpha_{f^{\ell-1},\e,\iota}(u)^{\beta p^*}}
			\dd u
			\bigg)^{1/q}
			\label{double_underline_split},
			\Ee
			where $1/p + 1/p^* =1$.

			\hide
			By the H\"older inequality, we bound it by 
			\begin{equation} \label{whole}
			\int^t_0  \int_\O 
			\int_{\R_{v}^3}
			| \nu_\phi^{1/p}\alpha^\beta \p f (v)|^{p-1} 
			\int_{\R_{u}^3} \mathbf{k}(v,u) \frac{ |  \alpha^\beta \p f (u)| }{ \alpha(u)^{\beta}}     \dd u
			\dd v
			\dd x \dd s.   
			\end{equation}
			We note that $\frac{p-2}{p} < \b < \frac{p-1}{p}$ implies $-1 < (\b-1) p + 1 < 0$ and therefore 
			$
			\frac{\alpha^{\b p}(v)}{\nu_{\phi}(v)^{p-1}} \lesssim \frac{\alpha^{\b p}(v)}{\langle v \rangle^{p-1}} \lesssim 1.   
			$

			Recall a standard estimate of $\mathbf{k}(v,u)\lesssim  \frac{e^{- C|v-u|^2}}{|v-u|}$ in (\ref{estimate_k}). 
			With $M>0$, we split $\{|u| \leq M\} \cup \{|u| \geq M\}$.
			In (\ref{whole}), for $\{|u| \geq M \}$,
			\Be \label{double_underline_split}
			\begin{split}
				&\int^t_0  \int_\O 
				\int_{\R_{v}^3}
				| \nu_\phi^{1/p}\alpha^\beta \p f (v)|^{p-1} 
				\int_{|u| \geq M} \mathbf{k}(v,u) \frac{ |  \alpha^\beta \p f (u)| }{ \alpha(u)^{\beta}}     \dd u
				\dd v
				\dd x \dd s  \\
				&\leq \int^t_0 \int_{\O} \int_{\R^{3}_{v}}  |\nu_{\phi}^{1/p} \alpha^\beta \p f (v)|^{p-1} \\
				&\quad\times 
				\underbrace{ \Big[ \int_{|u| \geq M} \mathbf{k}_\zeta (v,u) \frac{1}{\alpha^{\b q}} \Big] ^{1/q} }
				\Big[ \int_{|u| \geq M} \mathbf{k}_\zeta (v,u) |\alpha^{\b}\p f(u)|^{p} \Big]^{1/p}  dv dx ds \\
				&\leq \ \int^t_0 \int_{\O} \Big[ \int_{\R^{3}_{v}}  |\nu_{\phi}^{1/p}\alpha^\beta \p f (v)|^{p } \Big]^{1/q}
				\Big[
				\int_{|u| \geq M} |\alpha^{\beta  }\p f (u)|^p \int_{\R^{3}_{v}} \mathbf{k}_\zeta(v,u)  
				\Big]^{1/p}  dx  ds  \\
				&\lesssim \ \int_0^t \| \nu_{\phi}^{1/p}\alpha^\beta \p f  (s)\|_{p}^{p-1} \| \alpha^\beta \p f  (s)\|_{p} \dd s  \\
				&\lesssim \ O(\varepsilon)\int_0^t \| \nu_{\phi}^{1/p}\alpha^\beta \p f  (s)\|_{p}^{p} + \int_0^t \| \alpha^\beta \p f  (s)\|_{p}^{p}  \dd s  \\
			\end{split}
			\Ee
			where we used Proposition \ref{prop_int_alpha} for above underbraced term.  \\
			
			For $\{|u| \leq M\}$ part of (\ref{whole}),
			\Be  \label{small}
			\begin{split}
				&\int^t_0  \int_\O 
				\int_{\R_{v}^3}
				| \nu_\phi^{1/p}\alpha_{f^\ell,\e}^\beta \p f (v)|^{p-1} 
				\int_{|u| \leq N} \mathbf{k}(v,u) \frac{ |  \alpha_{f^\ell,\e}^\beta \p f (u)| }{ \alpha_{f^\ell,\e}(u)^{\beta}}     \dd u
				\dd v
				\dd x \dd s  \\
				&\lesssim \int^t_0    \| \nu_\phi^{1/p} \alpha_{f^\ell,\e}^\beta \p f  (s)\|_p^{ p/q}  
				\Big[
				\int_\O
				\|  \alpha_{f^\ell,\e}^\beta \p f(\cdot) \|^{p}_{L^p(\R^3)}
				\int_{\R_{v}^3}
				|(**)|^{p}
				\dd v
				\dd x  \Big]^{1/p}
				\dd s  \\
			\end{split}
			\Ee
			where  \unhide
			
			It is important to note that for $\iota = + $ or $-$,
			\Be\label{convolution}
			\bigg(
			\int_{\R^3}
			\frac{e^{-p^*\tilde{\varrho}  |v-u|^2}}{|v-u|^{p^*}} \frac{\mathbf{1}_{|u| \leq N}}{\alpha_{f^{\ell-1},\e,\iota}(u)^{\beta p^*}}
			\dd u
			\bigg)^{1/{p^*}}  \leq \bigg|
			\frac{1}{| \cdot |^{p^*}} *  \frac{\mathbf{1}_{|\cdot| \leq N}}{\alpha_{f^{\ell-1},\e,\iota}(\cdot)^{p^*\beta}}
			\bigg| ^{1/{p^*}}.
			\Ee
			%
			By the Hardy-Littlewood-Sobolev inequality with $$1+ \frac{1}{p/{p^*}} = \frac{1}{3/{p^*}} + \frac{1}{
				\frac{3}{2} \frac{p-1}{p}
			},$$ we have  
			\Be \label{153_1p}
			\begin{split}
				\left\|
				\bigg|
				\frac{1}{| \cdot |^{p^*}} *  \frac{\mathbf{1}_{|\cdot| \leq N}}{\alpha_{f^{\ell-1},\e,\iota}(\cdot)^{{p^*}\beta}}
				\bigg| ^{1/{p^*}}\right\|_{L^p(\R^3)} 
				& = 
				\bigg\|
				\frac{1}{| \cdot |^{p^*}} *  \frac{\mathbf{1}_{|\cdot| \leq N}}{\alpha_{f^{\ell-1},\e,\iota}(\cdot)^{p^*\beta}}
				\bigg\|_{L^{p/{p^*}} (\R^3)}  ^{1/{p^*}}\\
				&\lesssim 
				\left\| \frac{\mathbf{1}_{|\cdot| \leq N}}{\alpha_{f^{\ell-1},\e,\iota}(\cdot)^{p^*\beta}}\right\|_{L^{
						\frac{3(p-1)}{2p}
					} (\R^3)}^{1/{p^*}}\\
				&\lesssim
				\left( \int_{\R^3} \frac{\mathbf{1}_{|v| \leq N}}{\alpha_{f^{\ell-1},\e,\iota}(v) ^{\frac{p}{p-1} \beta 
						\frac{3(p-1)}{2p}
					}
				} \dd v\right)^{
					\frac{2p}{3(p-1)} \frac{p-1}{p}
				}
				\\
				&=
				\left( \int_{\R^3} \frac{\mathbf{1}_{|v| \leq N}}{\alpha_{f^{\ell-1},\e,\iota}(v) ^{  3\beta/2
					}
				} \dd v\right)^{
					2/3
				}
				.
			\end{split}
			\Ee
			For $3<p<6$, we have $\frac{3}{2}\frac{p-2}{p} < 1$ and $\frac{2}{3} < \frac{p-1}{p}$. 
			Importantly from (\ref{beta_condition}) we have $\frac{3\beta}{2}<1$. Now we apply (\ref{NLL_split2}) in Proposition \ref{prop_int_alpha} to conclude that 
			\Be\notag
			\left( \int_{\R^3} \frac{\mathbf{1}_{|v| \leq M}}{\alpha_{f^{\ell-1},\e,\iota}(v) ^{ 3\beta/2}
			} \dd v\right)^{2/3}\lesssim_{p, \beta, M,\O} 1 .
			\Ee 
			Finally from (\ref{double_underline}), (\ref{double_underline_split}), (\ref{convolution}), (\ref{153_1p}), and (\ref{small result}) we bound 
			\Be\label{151_3}
			(\ref{pf_p_K}) \lesssim
			o(1)\int_{0}^{t}  \| \nu_\phi^{1/p} w_{\tilde{\vartheta}}\alpha_{f^\ell,\e}^\beta \p f^{\ell+1}   \|_{p}^{p} + (1+ \sup_{0 \leq s \leq t } \| w_\vartheta f^\ell(s) \|_\infty + \sup_{0 \leq s \leq t } \| w_\vartheta f^{\ell+1}(s) \|_\infty) \int_{0}^{t} \|w_{\tilde{\vartheta}}\alpha_{f^{\ell-1},\e}^{\b}\p f^\ell \|_{p}^{p}   .
			\Ee
			
			Collecting terms from (\ref{pf_p_nabla}) (\ref{pf_p_K}), (\ref{pf_p_0}), (\ref{pf_p_1}), (\ref{pf_p_Gamma_v_-}), (\ref{pf_p_Kv}), (\ref{pf_p_2}), (\ref{small result}), and (\ref{151_3}) we have
			\Be
			\begin{split}\label{final_est_G}
				(\ref{pf_green})_\mathcal{G}
				\lesssim& o(1) 
				\int^t_0  \| \nu_{\phi^\ell}^{1/p} w_{\tilde{\vartheta}}  \alpha_{f^\ell,\e}^\beta \p f^{\ell+1} \|_p^p\\
				&+  (1  + \sup_{0 \leq s\leq t}\| \nabla^2 \phi^\ell(s) \|_\infty
				)
				\int^t_0   \| w_{\tilde{\vartheta}}  \alpha_{f^\ell,\e}^\beta \p f^{\ell+1} \|_p^p 
				\\
				& 
				+(1 + \sup_{0 \leq s\leq t}\| w_\vartheta f^\ell(s) \|_\infty+ \delta_1/\Lambda_1)
				\int^t_0 \|w_{\tilde{\vartheta}} f^{\ell+1} \|_p^p
				\\
				& 
				+(1 + \sup_{0 \leq s\leq t}\| w_\vartheta f^\ell(s) \|_\infty  +\sup_{0 \leq s\leq t} \| w_\vartheta f^{\ell+1}(s) \|_\infty)
				\int^t_0 (\|w_{\tilde{\vartheta}} f^{\ell} \|_p^p + \|w_{\tilde{\vartheta}}\alpha_{f^{\ell-1},\e}^{\b}\p f^\ell \|_{p}^{p})  .
			\end{split}
			\Ee
			
			\hide
			Therefore
			\Be \label{small result}
			\begin{split}
				(\ref{small}) &\lesssim \int^t_0    \| \nu_\phi^{1/p} \alpha^\beta \p f  (s)\|_p^{ p/q}
				\|  \alpha^\beta \p f(\cdot) \|_{p}	\dd s  \\
				&\lesssim O(\varepsilon)\int_{0}^{t}  \| \nu_\phi^{1/p} \alpha^\beta \p f  (s)\|_{p}^{p} ds + \int_{0}^{t} \|\alpha^{\b}\p f(s)\|_{p}^{p} ds.
			\end{split}
			\Ee

			From (\ref{small result}) and (\ref{double_underline_split}) we conclude
			\Be \label{151_3}
			\begin{split}
				\int_{0}^{t} (\ref{pf_green})_{3} \dd s &\lesssim  o(1)\int_{0}^{t}  \| \nu_\phi^{1/p} \alpha^\beta \p f  (s)\|_{p}^{p} \dd s + \int_{0}^{t} \|\alpha^{\b}\p f(s)\|_{p}^{p} \dd s .
			\end{split}
			\Ee

			\unhide

			\vspace{4pt}

			\textit{Step 4.}
			We focus on $(\ref{pf_green})_{\gamma_-}$. From (\ref{BC_deriv}) and (\ref{BC_deriv_R}),
			\Be \label{BC_deriv_1}
			\begin{split}
				&
				\int_{n(x) \cdot v<0}
				|n(x) \cdot v|^{\beta p}
				|w_{\tilde{\vartheta}}  \nabla_{x,v} f^{\ell+1}(t,x,v) |^p
				|n(x) \cdot v| \dd v
				\\
				&\lesssim  \int_{n(x) \cdot v<0}
				\langle v\rangle^p \mu(v)^{\frac{p}{2}}w_{\tilde{\vartheta}} ^p
				\Big(|n(x) \cdot v|^{\beta p+1}+ |n(x) \cdot v|^{(\beta-1) p+1} 
				\Big)  \times |(\ref{BC_deriv_R})|^p  \dd v.
			\end{split}\Ee
			Note that for $0< \tilde{\vartheta} \ll_p 1$ we have $\mu(v)^{\frac{p}{2}}w_{\tilde{\vartheta}} ^p\lesssim e^{C|v|^{2}}$ for some $C>0$ when $|v|\gg1$. 
			
			On the other hand, from (\ref{beta_condition}), we have   
			\Be\label{L1_loc}
			(\beta-1) p + 1> \frac{p-2}{p} p - p+1 = -1, \ \ \  |n(x) \cdot v|^{(\beta-1) p+1} \in L^1_{loc}(\R^3).
			\Ee
			
			\hide
			\Be\notag
			\begin{split}
				&|\nabla_{x,v} f(t,x,v) |\\
				\lesssim& \ 
				\langle v\rangle\sqrt{\mu(v)}
				\Big(1+ \frac{1}
				{|n(x) \cdot v|}  \Big)\\
				& \times
				\int_{n(x) \cdot u>0}  \Big\{
				(
				\langle u\rangle  + |\nabla_x \phi_f|
				)
				|\nabla_{x,v} f(t,x,u) | \\ &  \ \ \ \ \   \ \ \  \ \ \  \ \ \  \ \ \   + 
				\langle u \rangle (1+ |\nabla_x \phi_f |) |f|+ |\nabla_x \phi_f| \mu(u)^{\frac{1}{4}}
				\\
				&  \ \ \ \ \   \ \ \  \ \ \  \ \ \  \ \ \  
				+(1+ \| w f \|_\infty) |\int_{\R^3}\mathbf{k}_{\varrho }(u, u^\prime)|f(u^\prime)| \dd u^\prime|
				\Big\}\sqrt{\mu(u)}\{n(x) \cdot u\} \dd u.
			\end{split}\Ee   
			\unhide

			Now we bound $|(\ref{BC_deriv_R})|^p$. For the first line of (\ref{BC_deriv_R}), 
			%
			we split the $u$-integration into $\gamma_+^\e(x) \cup \gamma_+ (x) \backslash \gamma_+^\e(x)$ where $\e$ is small but satisfies (\ref{lower_bound_e}). By the H\"older inequality 
			\Be
			\begin{split}
				&   \bigg\{  \sum_{\iota = \pm} \int_{n(x) \cdot u>0} |w_{\tilde{\vartheta}}\alpha_{f^{\ell-1},\e}^\beta\nabla_{x,v}  f^{\ell} (s,x,u)| \{w_{\tilde{\vartheta}}\alpha_{f^\ell,\e,\iota}( u)\}^{- \beta}\langle u\rangle \sqrt{\mu(u)} \{n(x) \cdot u\}  \dd u\bigg\}^p
				\label{W1p_bdry}
				\\
				&  \lesssim
				\bigg\{\int_{\gamma_+^\e (x)}  |w_{\tilde{\vartheta}}\alpha_{f^{\ell-1},\e}^\beta\nabla_{x,v} f^\ell(s,x,u)|^p\{n(x) \cdot u\}  \dd u\bigg\}  \\
				& \ \ \ \  \times 
				\underbrace{ \bigg\{  \sum_{\iota = \pm} \int_{\gamma_+^\e (x)}
					\{w_{\tilde{\vartheta}}\alpha_{f^{\ell-1},\e,\iota}( u)\}^{- \beta p^*}
					|n(x) \cdot u|  \mu ^{\frac{q}{4}}\dd u 
					\bigg\}^{p/{p^{*} }} }  \\
				& + \bigg\{\int_{\gamma_+(x) \backslash\gamma_+^\e (x)}   | w_{\tilde{\vartheta}}\alpha_{f^{\ell-1},\e}^\beta\nabla_{x,v} f^\ell(s,x,u)|^p
				\mu ^{\frac{p}{8}}
				\{n(x) \cdot u\}  \dd u\bigg\}  \\
				& \ \ \ \   \times \underbrace{ \bigg\{  \sum_{\iota = \pm} \int_{\gamma_+(x) \backslash\gamma_+^\e (x)}
					\{w_{\tilde{\vartheta}}\alpha_{f^{\ell-1},\e,\iota}(s,x,u)\}^{- \beta p^* }
					|n(x) \cdot u|\mu ^{\frac{p^{*} }{8}}\dd u 
					\bigg\}^{p/{p^{*} }} },  \quad  p^{*} := \frac{p}{p-1}.
			\end{split}
			\Ee
			Note that $\alpha_{f^\ell,\e,\iota}(s,x,u)\neq |n(x) \cdot u|$ for $(x,u) \in \gamma_+$ in general. From (\ref{beta_condition}), $\beta p^*<1$. From (\ref{NLL_split2}) and (\ref{NLL_split3}) with $v=0$, we have $\alpha_{f^{\ell-1},\e,\iota}^{- \beta p^*} |n(x) \cdot u |   \lesssim \alpha_{f^{\ell-1},\e,\iota}^{- \beta p^*}   \in L^1_{loc} (\{ u \in \R^3\})$. Since $\mathbf{1}_{\gamma^\e_+ (x)}(v) \downarrow 0$ almost everywhere in $\R^3$ as $\e \downarrow 0$, by the dominant convergence theorem, for (\ref{delta_1/lamdab_1}), we choose $\e : = \frac{2 \delta_1}{\Lambda_1}\ll_\O 1$  
			\Be\begin{split}\label{W1p_bdry_1}
				(\ref{W1p_bdry}) \lesssim& \  o(1)  \int_{\gamma_+^\e (x)}  |w_{\tilde{\vartheta}}\alpha_{f^{\ell-1},\e}^\beta\nabla_{x,v} f^\ell(s,x,u)|^p\{n(x) \cdot u\}  \dd u\\
				&+\int_{\gamma_+(x) \backslash\gamma_+^\e (x)}   |w_{\tilde{\vartheta}}\alpha_{f^{\ell-1},\e}^\beta\nabla_{x,v} f^\ell(s,x,u)|^p \mu(u)^{{p}/{8}}\{n(x) \cdot u\}  \dd u.
			\end{split} 
			\Ee

 Now applying Lemma \ref{le:ukai} and (\ref{Boltzmann_p}) to $f^\ell_+$ and $f^\ell_-$ separately and adding them together, the last term of (\ref{W1p_bdry_1}) has a bound as 
			\Be\begin{split}\label{nongrazing_nabla_f}
				&\int^t_0 \int_{\p\O} \int_{\gamma_+(x) \backslash\gamma_+^\e (x)}   | w_{\tilde{\vartheta}}\alpha_{f^{\ell-1},\e}^\beta\nabla_{x,v} f^\ell(s,x,u)|^p \mu(u)^{{p}/{8}}\{n(x) \cdot u\}  \dd u \dd S_x \dd s\\
				\lesssim & \ \| w_{\tilde{\vartheta}} \alpha_{f_0,\e}^\beta\nabla_{x,v} f(0)  \mu^{{1}/{8}}\|^p_p + \int^t_0  \| w_{\tilde{\vartheta}} \alpha_{f^{\ell-1},\e}^\beta \nabla_{x,v} f^\ell \|_p^p + (\ref{8 Last}),
			\end{split} 
			\Ee
			where, from (\ref{eqtn_nabla_f}), (\ref{mathcal_G}),
			%
			%
			\begin{eqnarray} 
			&&
			 \int^t_0\iint_{ \O \times \R^3}
			  \big|[\p_t + v\cdot \nabla_x - q \nabla_x \phi^{\ell-1} \cdot \nabla_v + \nu_{\phi^\ell,w_{\tilde{\vartheta}}}  ] (
			w_{\tilde{\vartheta}}\mu^{1/8}
			\alpha_{f^{\ell-1},\e}^\beta\nabla_{x,v} f^\ell 
			)^p
			\big| \label{8 Last} \\
			&\leq&    \int^t_0
			\iint_{ \O \times \R^3} p |\alpha_{f^{\ell-1},\e}^{\beta p}  (\nabla_{x,v} f^\ell)^{p-1} | |w_{\tilde{\vartheta}}  \mu^{1/8}|^p
			\big|
			\mathcal{G}^\ell \big|\label{8 Last_1} \\
			&& +   \int^t_0\iint_{ \O \times \R^3} 
			|\nabla_x \phi^\ell| \mu^{0+} 
			|\alpha_{f^{\ell-1},\e}^\beta\nabla_{x,v} f^\ell |^p   \label{8 Last_2}.
			\end{eqnarray}
			\hide
			&\lesssim 
			\int^t_0\iint_{ \O \times \R^3} p\alpha^{\beta p}  |\nabla_{x,v} f|^{p-1}  \mu^{p/8} |\mathcal{G}|  
			+ ( 1 + \| \nabla \phi_{f}\|_{\infty} ) \|\alpha^{\b}\p f\|_{p}^{p}  \\
			&\lesssim 
			(o(1) + \sup_{0 \leq s\leq t}\| w_\vartheta f(s) \|_\infty)
			\int^t_0  \| \nu_{\phi^\ell}^{1/p} w_{\tilde{\vartheta}}\alpha^\beta \p f \|_p^p\\
			& \  \ +  (1 + \sup_{0 \leq s\leq t}\| w_{\vartheta} f(s) \|_\infty + \sup_{0 \leq s\leq t}\| \nabla^2 \phi^\ell(s) \|_\infty
			)
			\int^t_0   \| w_{\tilde{\vartheta}} \alpha_{f,\e}^\beta \p f \|_p^p 
			\\
			&  \ \ \  
			+(1 + \sup_{0 \leq s\leq t}\| w f(s) \|_\infty+ \delta_1/\Lambda_1) \int^t_0  \|  w_{\tilde{\vartheta}}  f \|_p^p.
		\end{eqnarray}\unhide
		Clearly $(\ref{8 Last_1})\lesssim  (\ref{final_est_G})|_{\ell \longleftrightarrow \ell-1} $. And, from (\ref{integrable_nabla_phi_f}),  
		\Be\notag
		(\ref{8 Last_2}) \lesssim \delta_1 \int^t_0 \| w_{\tilde{\vartheta}} \alpha_{f^{\ell-1}, \e}^\beta \nabla_{x,v} f^\ell \|_p^p.
		\Ee

		Now we consider the third term of (\ref{BC_deriv_R}). From the trace theorem $W^{1,p} (\O) \rightarrow W^{1- \frac{1}{p}, p} (\p\O)$ and (\ref{nabla_2_phi_p})
		\Be \label{phi_bdry_estimate}
		\| \nabla \phi^m \|_{L^p (\p\O)}\lesssim \| \nabla \phi ^m \|_{W^{1- \frac{1}{p}, p} (\p\O)}\lesssim \| \nabla \phi^m \|_{W^{1,p} (\O)}\lesssim \| w_{\tilde{\vartheta}} f^m \|_{L^p (\O \times \R^3)}.
		\Ee
		Then 
		\Be
		\begin{split}\label{W1p_bdry_2}
			&\int_{\p \O} \Big\{  \int_{n(x) \cdot u > 0 } 
			(\langle u \rangle (|f^{\ell +1}| + | f^\ell | )  + \mu(u)^{\frac{1}{4}} ) ( |\nabla_x \phi^\ell | + |\nabla_x \phi^{\ell-1} | )
			\Big\}\sqrt{\mu(u)}\{n(x) \cdot u\} \dd u \Big\}^p \dd S_{x}
			\\
			\lesssim & \ (1+ \| w_\vartheta f^\ell \|_\infty + \| w_\vartheta f^{\ell+1} \|_\infty) \left( \sum_{m = \ell}^{\ell + 1 } \| \nabla \phi^m \|_{L^p(\p\O)}^p \right)\\
			\lesssim & \ (1+ \| w_\vartheta f^\ell \|_\infty + \| w_\vartheta f^{\ell+1} \|_\infty) (  \| w_{\tilde{\vartheta}} f^\ell \|^{p}_{L^p (\O \times \R^3)} + \| w_{\tilde{\vartheta}} f^{\ell+1}\|^{p}_{L^p (\O \times \R^3)}).
		\end{split}
		\Ee
		
		For the second term of (\ref{BC_deriv_R}), by the H\"older inequality with $\frac{1}{p}+ \frac{1}{q}=1$ for $3<p<6$,
		\Be
		\begin{split}\label{W1p_bdry_3}
			&\Bigg\{
			\int_{n \cdot u>0} 
			\Big( 
			\langle u\rangle ( |f|^\ell + |f|^{\ell + 1 } )  
			\\ &+ (1+ \| w_{\vartheta} f^\ell \|_\infty + \| w_{\vartheta} f^{\ell-1} \|_\infty) \int_{\R^3}\frac{ \mathbf{k}_\varrho (u,u^\prime)^{1/q}}{|n\cdot u^\prime|^{1/p} }
			 \mathbf{k}_\varrho (u,u^\prime)^{1/p}(|f^\ell(u^\prime)| + f^{\ell -1 } (u') ) |n\cdot u^\prime|^{1/p} \dd u^\prime 
			 \Big)
			\sqrt{\mu }  \{n  \cdot u\} \dd u 
			\Bigg\}^p\\
			\lesssim & \ \left(  \int_{n\cdot u>0} (|f^{\ell+1}|^p + | f^\ell | ^p ) \{n\cdot u\} \sqrt \mu \dd u \right)^p \\
			&  + (1+ \| w_{\vartheta} f^\ell \|_\infty +\| w_{\vartheta} f^{\ell+1} \|_\infty) \Big(\int_{\R^3} \mathbf{k}_{\varrho} (u,u^\prime) |n \cdot u^\prime|^{-q/p} \dd u^\prime\Big)^{p/q} \int_{\R^3}\int_{\R^3} \mathbf{k}_\varrho (u,u^\prime) ( |f^\ell(u^\prime)|^p +|f^{\ell+1}(u^\prime)|^p) |n \cdot u^\prime| \dd u^\prime \dd u\\
			\lesssim & \ (1+  \| w_{\vartheta} f^\ell \|_\infty +\| w_{\vartheta} f^{\ell+1} \|_\infty) \int_{n\cdot u>0}  (|f^{\ell+1}|^p + | f^\ell | ^p ) \{n\cdot u\} \dd u
			\\  \lesssim & \ (1+  \| w_{\vartheta} f^\ell \|_\infty +\| w_{\vartheta} f^{\ell+1} \|_\infty) ( | w_{\vartheta} f^\ell \|_\infty +\| w_{\vartheta} f^{\ell+1} \|_\infty ).
		\end{split}
		\Ee
		
		Collecting terms from (\ref{BC_deriv_1}), (\ref{W1p_bdry_1}), (\ref{8 Last}), (\ref{W1p_bdry_2}), and (\ref{W1p_bdry_3}) we derive that 
		\Be \label{mid} \begin{split} 
			&(\ref{pf_green})_{\gamma_-}\\
			\lesssim & \
			\|  w_{\tilde{\vartheta}}\alpha_{f_0,\e}^\beta\nabla_{x,v} f(0)  \mu(u)^{{1}/{8}}\|^p_p
			\\
			& + o(1) \left(  \int^t_0 | w_{\tilde{\vartheta}}\alpha_{f^{\ell-1},\e}^\beta \p f^\ell |_{p,+}^p  \right) + \int_0^t  (1+  \| w_{\vartheta} f^\ell \|_\infty +\| w_{\vartheta} f^{\ell+1} \|_\infty)  
			\\ & + \left(o(1) +  \sup_{0 \leq s\leq t}\| w_{\vartheta} f^{\ell-1} (s) \|_\infty +  \sup_{0 \leq s\leq t}\| w_{\vartheta} f^\ell (s) \|_\infty +  \sup_{0 \leq s\leq t}\| w_{\vartheta} f^{\ell+1} (s) \|_\infty \right)
			\int^t_0  \sum_{m = \ell -1 }^\ell \| \nu_{\phi_f}^{1/p} w_{\tilde{\vartheta}}\alpha_{f^{m-1},\e}^\beta \p f^m \|_p^p\\
			& + \sum_{m = \ell}^{\ell+1} (1 + \sup_{0 \leq s\leq t}\| w_{\vartheta} f^m(s) \|_\infty+\sup_{0 \leq s\leq t}\| \nabla^2 \phi^{m-1}(s) \|_\infty   ) \int^t_0  \|   \alpha_{f^{m-1},\e}^\beta w_{\tilde{\vartheta}} \p f^m \|_p^p 
			\\
			&    
			+(1 +\sum_{m = \ell}^{\ell +1 }  \sup_{0 \leq s\leq t}\| w_{\vartheta} f^m(s) \|_\infty) \int^t_0  
			\sum_{m = \ell}^{\ell +1 }  \big(  \| w_{\tilde{\vartheta}} f^m \|_p^p+   | w_{\tilde{\vartheta}}f^m |_{p,+}^p \big) .
		\end{split} \Ee

		\hide

		Using Lemma \ref{le:ukai} and (\ref{8 Last}) to (\ref{mid}),
		\Be \label{151_2}
		\begin{split}
			& (\ref{pf_green})_2 = \int_{0}^{t} |\alpha^\beta 
			\p f(s)  |_{p,-}^p ds   \\
			&\lesssim O(\e) (1 + \|\phi_{f}\|_{C^{2}})^{p} \int^t_0 |\alpha^\beta \p f (s) |_{p,+}^p ds  + O(\varepsilon) (1+\|wf\|_{\infty}) \int_{0}^{t} \| \nu_\phi^{1/p} \alpha^\beta \p f  (s)\|_{p}^{p}  \\
			&\quad + \|\alpha^{\b}\p f (0) \|_{p}^{p} + \int_{0}^{t} (1 + \|wf\|_{\infty})^{2p}(1 + \|\phi_{f}\|_{C^{2}})^{p} ( 1 + \|\alpha^{\b}\p f\|_{p}) \|\alpha^{\b}\p f\|_{p}^{p-1}  ds \\
		\end{split}
		\Ee
		\unhide
		
		\vspace{4pt}

		\textit{Step 5.} From (\ref{Lp_estimate_f}), (\ref{pf_green}), (\ref{final_est_G}), (\ref{mid}) we have 		
		\Be
		\begin{split}\label{control_bdry_ell}
&\sup_{0 \leq s\leq t}  \mathcal{E}^{\ell+1} (s)
			\\  \le  &     
			C_0  \left( \|  w_{\tilde{\vartheta}}\alpha^\beta_{f_0, \e} \nabla_{x,v} f_0 \|_p^p 
			+ t \left(1  + \|w_{\vartheta}f^{\ell-1 }\|_{\infty} +   \sum_{ m = \ell - 1 }^{\ell + 1}\| w_{\vartheta} f^ml \|_\infty   + \| \nabla^2 \phi^{\ell-1 } \|_\infty + \| \nabla^2 \phi^{\ell } \|_\infty\right)   \right) \max_{0 \leq m \leq \ell +1} \sup_{0 \leq s \leq t} \mathcal{E}^{m}(s) 
			\\ & +o(1)  \max_{0 \leq m \leq \ell }\sup_{0 \leq s \leq t}\mathcal{E}^{m}(s)
			 .
		\end{split}\Ee		
		On the other hand, from Lemma \ref{lemma_apply_Schauder}, 
		\Be\label{phi_ell<E}
		\| \nabla^2 \phi^\ell (t)\|_\infty + \| \nabla^2 \phi^{\ell-1}(t) \|_\infty 
		\lesssim [\mathcal{E}^{\ell}(t) + \mathcal{E}^{\ell-1}(t)]^{1/p}.
		\Ee
Therefore from (\ref{control_bdry_ell}), (\ref{phi_ell<E}), and the induction hypothesis in (\ref{induc_hypo}), we first choose a small $o(1)$, then large $C \gg C_0$, and finally small $ 0< T^{**} \ll 1 $ to conclude 
		\Be \begin{split} \notag
		\sup_{0 \leq s\leq t}\mathcal{E}^{\ell + 1} (s) \le & \frac{C}{10} \| w_{\tilde{\vartheta}} \alpha_{f_0,\e}^{\b} 
		\nabla_{x,v} f_0\|_p^p + \frac{1}{10}\sup_{0 \leq s\leq t}\sup_{m\leq \ell}\mathcal{E}^{m} (s)
		\\ \le & C \{ \| w_\vartheta f_0 \|_\infty + \| w_{\tilde{\vartheta}} f _0 \|_p^p
			+
			\| w_{\tilde{\vartheta}} \alpha_{f_{0 },\e}^\beta \nabla_{x,v} f  _0 \|_{p}^p \}.
		\end{split} \Ee
		This proves (\ref{induc_hypo}).

		\hide

		\vspace{4pt}

		\textit{Step 5.} By simple modification we can prove 
		\Be\label{tilde_w_W1p}
		\| \tilde{w} \alpha^\beta \nabla_{x,v} f(t) \|_{L^p (\O \times \R^3)}^p + \int^t_0 \| \langle v \rangle^{1/p}\tilde{w}  \alpha^\beta \nabla_{x,v} f(t) \|_{L^p (\O \times \R^3)}^p  \lesssim_t 1.
		\Ee
		The equation for $\tilde{w} \p f$ equals 
		\Be\label{eqtn_nabla_f_tilde_w}
		[\p_t + v\cdot \nabla_x - \nabla_x \phi_f \cdot \nabla_v + \tilde{\nu}_{\phi_f,\tilde{w}}] \p f 
		=   \tilde{w} \mathcal{G},
		\Ee
		where ${\nu}_{\phi_g,\tilde{w}}$ in (\ref{nu_w}) and $\mathcal{G}$ in (\ref{mathcal_G}). Note that, from {\color{red}(??)}, for some $0<\tilde{\varrho}\ll \varrho$
		\Be \label{K_tilde_w}
		\tilde{w} (v)| K \p f(v)| 
		\lesssim  \int_{\R^3} \mathbf{k}_\varrho (v,u) \frac{\tilde{w}(v)}{\tilde{w}(u)} | \tilde{w}\p f (u)| \dd u 
		\lesssim  \int_{\R^3} \mathbf{k}_{\tilde{\varrho}} (v,u)   | \tilde{w}\p f (u)| \dd u.
		\Ee
		Similarly we have 
		\Be\begin{split}\label{Gamma_tilde_w}
			&\tilde{w} (v) |\Gamma( \tilde{w}^{-1} \|w f\|_\infty,\p f)|\lesssim  \tilde{w} (v) |\Gamma(w^{-1}\|w f\|_\infty, \tilde{w}^{-1} \tilde{w}\p f)|\\
			\lesssim & \ 
			\|w f\|_\infty \int_{\R^3} \mathbf{k}_{\tilde{\varrho}} (v,u)   | \tilde{w}\p f (u)| \dd u.
		\end{split}\Ee
		The rest of proof is same as \textit{Step 1}-\textit{Step 4}.\hide
		
		We define
		\Be \label{def W}
		\mathcal{W}(t) := 
		\| f(t) \|_p^p
		+\| \alpha^\beta 
		\p f (t)\|_p^p  .
		\Ee
		From (\ref{final_est_W1p}),
		\Be\label{eqtn_W}
		\mathcal{W} (t) \lesssim \mathcal{W} (0)  + (1 
		+ \sup_{0 \leq s \leq t} \| \nabla^2 \phi_f(s) \|_\infty)
		\int^t_0 \mathcal{W} (s) \dd s.
		\Ee

		\vspace{4pt}
		
		\textit{Step 5. }

		and assume $\|\alpha^{\b}\p f(t)\|_{p} \geq 1$ WLOG, since we apply growth estimate. We use (\ref{151_2}), (\ref{151_3}), (\ref{151_4}), and (\ref{151_5}) to (\ref{pf_green}) with sufficiently small $\varepsilon \ll 1$ to obtain
		\Be
		\begin{split}
			\mathcal{W}(t) &\lesssim \| \alpha^\beta \p f (0)\|_{p}^{p} + \int_{0}^{t}\mathcal{P}\big( \|wf\|_{\infty},  \|\phi_{f}\|_{C^{2}} \big) \mathcal{W}(s) ds, 
		\end{split}	
		\Ee
		and Gronwall's inequality yield
		\Be \label{e_growth}
		\mathcal{W}(t) := \| \alpha^\beta 
		\p f (t)\|_p^p
		+ \int^t_0 \| \nu_{\phi_f}^{1/p} \alpha^\beta 
		\p f (s) \|_p^p  ds \lesssim \| \alpha^\beta \p f (0)\|_{p}^{p} \ e^{ t C_{p}(f,\phi_{f}) },
		\Ee
		where 
		\[
		C_{p}(f,\phi_{f}) = \mathcal{P}\big( \|wf\|_{\infty},  \|\phi_{f}\|_{C^{2}} \big) = (1 + \|wf\|_{\infty})^{2p}(1 + \|\phi_{f}\|_{C^{2}})^{p}.
		\]
		
		From Schauder estimate,
		\Be
		\begin{split}\label{apply_schauder}
			\|\phi_{f}\|_{C^{2,1-\frac{3}{p}}} &\lesssim_{p,\O} \Big\|\int_{\R^{3}} f\sqrt{\mu} dv \Big\|_{C^{0,1-\frac{n}{p}}(\O)}  \\
			&\lesssim \|wf\|_{\infty} + \Big\|\int_{\R^{3}} \nabla_{x} f\sqrt{\mu} dv \Big\|_{L^{p}(\O)}  \\
			&\lesssim  \|wf\|_{\infty} + \Big( \int_{\R^{3}} \Big| \frac{\sqrt{\mu}}{\alpha^{\b}(x,v)} \Big|^{\frac{p}{p-1}} dv \Big)^{\frac{p-1}{p}} \|\alpha^{\b}\nabla_{x}f\|_{p}  \\
			&\lesssim e^{t C_{p}(f,\phi_{f})},
		\end{split}
		\Ee
		since we can apply Proposition \ref{prop_int_alpha} with $\b \frac{p}{p-1} < 1$. Now, let us choose 
		\[
		\delta_{2} = 1-\frac{3}{p}, \quad \lambda_{0} > \frac{1}{\delta_{2}} C_{p}(f,\phi_{f}).
		\]
		From (\ref{morrey}), $\|\nabla_{x}\phi_{f}\|_{C^{0,1-\frac{3}{p_{1}}}} \lesssim \|wf\|_{\infty} \lesssim e^{-\lambda t}$ holds for any $p_{1} > 3$. We choose $p_{1} \gg 3$ so that $\frac{3}{p_{1}}\lambda_{0} < \lambda$. Finally we apply (\ref{phi_interpolation}) to conclude
		\Be\begin{split}
			\|\nabla^2_x \phi(t )\|_{L^\infty_x}
			&\leq
			e^{\frac{3}{p_{1}}\lambda_0t}[\nabla_x \phi(t)]_{C^{0,1-\frac{3}{p_{1}}}_x}
			+ e^{-(1-\frac{3}{p}) \lambda_0t}[\nabla^2 \phi(t)]_{C_x^{0, 1-\frac{3}{p}}}  \\
			&\lesssim e^{ -\lambda_{1} t } ,
		\end{split}\Ee
		where
		\[
		\lambda_{1} := \min\{ \lambda-\frac{3}{p}\lambda_{0}, (1-\frac{3}{p})\lambda_{0} - C_{p}(f,\phi_{f}) \}.	
		\]
		\unhide\unhide
	\end{proof}

	\hide
	\vspace{4pt}

	\section{Weighted $W^{1,p}$ estimates of $f$ and $C^2$ estimates of $\phi$}

	For the hard sphere cross section and the global Maxwellian $\mu(v) = e^{-\frac{|v|^{2}}{2}}$,
	\begin{equation}\label{def_Kf}
	\begin{split}
	Kg  &:=K_{2}g -K_{1}g\\
	&:= \frac{1}{\sqrt{\mu}} \big[Q_{\mathrm{gain}}(\mu, \sqrt{\mu}g) + Q_{\mathrm{gain}}(  \sqrt{\mu}g,\mu)\big]
	- \frac{1}{\sqrt{\mu}} Q_{\mathrm{loss}} (\sqrt{\mu}g, \mu)\\
	&:=\int_{\mathbb{R}^{3}}\mathbf{k}_{2}(v,u)g(u)\mathrm{d}u - \int_{\mathbb{R}^{3}}\mathbf{k}_{1}(v,u)g(u)\mathrm{d}u,  
	\end{split}
	\end{equation}%
	where, for some constants $C_{\mathbf{k}_{1}}>0$ and $C_{\mathbf{k}_{2}}>0$,
	\begin{equation}\label{k_estimate}
	\begin{split}
	\mathbf{k}_{1}(u,v)=  & C_{\mathbf{k}_{1}}  |v-u|   e^{-\frac{|v|^{2} +|u|^{2}}{4}} ,
	\\
	\mathbf{k}_{2}(u,v) =& C_{\mathbf{k}_{2}} 
	|v-u|^{-1} e^{- \frac{|v-u|^2}{8}} 
	e^{-   \frac{  | |v|^2- |u|^2   |^2}{8|v-u|^2}}.
	\end{split}
	\end{equation}%
	%
	%

	The nonlinear Boltzmann operator equals
	\begin{equation}\label{carleman}
	\begin{split}
	\Gamma (g_{1},g_{2})(v) 
	=  &
	\int_{\mathbb{R}^{3}}   \int_{\S^2}   | u \cdot \omega|
	g_1(v+ u_\perp) g_2(v + u_\parallel)
	\sqrt{\mu(v+u)} \dd \omega  \mathrm{d}u \\
	& - \int_{\mathbb{R}^{3}}   \int_{\S^2}   | u \cdot \omega|
	g_1(v+u) g_2(v  )
	\sqrt{\mu(v+u)} \dd \omega  \mathrm{d}u
	\end{split}
	\end{equation}
	where $u_\parallel = (u \cdot \omega)\omega$ and $u_\perp = u - u_\parallel$.
	
	We define a notation
	\begin{equation}\label{kzeta}
	\mathbf{k}_{  \varrho}(v,u) := \frac{e^{-\frac{\varrho}{4}|v-u|^{2}    }}{|v-u| } .
	\end{equation}
	
	\begin{lemma} For $0<\varrho< \frac{1}{8}$,
		\Be\label{vK}\begin{split}
			| \nabla_v Kg(v) | \lesssim    \ \| w g \|_\infty
			+ \int_{\R^3} \mathbf{k}_\varrho (v,u) |\nabla_v g(u)| \dd u,
		\end{split}
		\Ee
		and
		\Be\label{vGamma}
		\begin{split}
			&|\nabla_v \Gamma (g,g) (v)|\\
			& \  \lesssim \| w g \|_\infty \int_{\R^3} \mathbf{k}_\varrho (v,u) |\nabla_v g (u)| \dd u + \langle v\rangle  \| w g \|_\infty |\nabla_v g (v)| + w(v)^{-1} \| w g \|_\infty^2.
		\end{split}
		\Ee
		
	\end{lemma}
	\begin{proof}
		
		Clearly $|\nabla_v \mathbf{k}_1 (u+v,v)| \lesssim \mathbf{k}_{  \varrho}(v,u)$. Note that 
		\Be\label{nabla_k2}
		\begin{split}
			|\nabla_v \mathbf{k}_2 (u+v,v)|
			\lesssim & \ \frac{\big||v|^2 - |u+v|^2\big|}{|u|^2}  e^{- \frac{| u|^2}{8}} 
			e^{-   \frac{  | |v|^2- |u+v|^2   |^2}{8| u|^2}}\\
			\lesssim & \ \frac{  |v| + |u+v| }{|u|}   e^{ - \frac{ ( |v|+ |u+v|  )^2}{8|u|^2}}e^{- \frac{| u|^2}{8}}  e^{- (|v|-|u+v|)^2}\\
			\lesssim & \ \mathbf{k}_{  \varrho}(v,u) .
		\end{split}\Ee
		
		Therefore we have
		\Be\label{vK}\begin{split}
			&| \nabla_v Kg(v) |\\
			\leq& \ \sum_{i}\int_{\R^3}\{| \nabla_v\mathbf{k}_i  (v,u+v)||  g(u+v) | + |\mathbf{k}_i  (v,u+v) ||\nabla_v g(u+v)|\}\dd u\\
			\lesssim&
			\ 
			\int_{\R^3} \mathbf{k}_\varrho (v,u) \{ |g(u)| +|\nabla_v g(u)|\} \dd u,
		\end{split}
		\Ee 
		and conclude (\ref{vK}).

		By taking derivatives 
		\Be\begin{split}\label{nabla_Gamma}
			&\nabla_v \Gamma(f,f) (v) \\
			=& \  \Gamma(\nabla_v f, f) + \Gamma(f, \nabla_v f) \\
			= & \ \Gamma_{\textrm{gain}} (\nabla_v f,f) + \Gamma_{\textrm{gain}} ( f,\nabla_vf)
			- \Gamma_{\textrm{loss}} (\nabla_v f,f) -\Gamma_{\textrm{loss}} ( f,\nabla_vf)+ \Gamma_v (f,f),
		\end{split}\Ee
		where we define
		\Be\begin{split}\label{Gamma_v}
			\Gamma_v (f,f)(v)  =&\int_{\mathbb{R}^{3}}   \int_{\S^2}   | u \cdot \omega|
			g_1(v+ u_\perp) g_2(v + u_\parallel)
			\nabla_v\sqrt{\mu(v+u)} \dd \omega  \mathrm{d}u  \\
			& - \int_{\mathbb{R}^{3}}   \int_{\S^2}   | u \cdot \omega|
			g_1(v+u) g_2(v  )
			\nabla_v \sqrt{\mu(v+u)} \dd \omega  \mathrm{d}u.
		\end{split}\Ee
		
		Note that 
		\Be\begin{split}\notag
			&|\Gamma_{\textrm{gain}} (\nabla_v f,f) |+ |\Gamma_{\textrm{gain}} ( f,\nabla_vf)|\\
			\lesssim & \ \| w f \|_\infty \big\{|\Gamma_{\textrm{gain}} (|\nabla_v f|, w^{-1}) | + |\Gamma_{\textrm{gain}} ( w^{-1}, |\nabla_vf|)|\big\} \\
			\lesssim & \ \| w f \|_\infty  
			\int_{\R^3} \int_{\S^2} 
			|(v-u) \cdot \omega| w(u )  \Big\{ \frac{|\nabla_v f (u^\prime)|}{w(v^\prime) } + \frac{ |\nabla_v f (v^\prime)|}{w(u^\prime)}\Big\}
			\dd \omega \dd u.
		\end{split}\Ee

		Then following the derivation of (\ref{k_estimate}) we can obtain a bound of 
		\Be\label{bound_Gamma_nabla_vf1}
		|\Gamma_{\textrm{gain}} (\nabla_v f,f) |+ |\Gamma_{\textrm{gain}} ( f,\nabla_vf)|
		\lesssim 
		\| w f \|_\infty \int_{\R^3} \mathbf{k}_{\varrho} (v,u) |\nabla_v f(u)| \dd u
		\Ee
		
		Clearly 
		\Be\label{bound_Gamma_nabla_vf2}
		\begin{split}
			|\Gamma_{\textrm{loss}}(\nabla_v f, f)| &\lesssim  \| w f \|_\infty   \int_{\R^3} 
			|\nabla_v f (u)| \mu(u)^{\frac{1}{4}} 
			\dd u,\\
			|\Gamma_{\textrm{loss}}( f, \nabla_v f)| &\lesssim \langle v\rangle  \| w f \|_\infty   
			|\nabla_v f (v)|  ,\\
			|\Gamma_v (f,f) | & \lesssim w^{-1} \| w f \|_\infty^2.
		\end{split}
		\Ee
		
	\end{proof}


	\begin{lemma}
		\label{lemma_K} 
		For $\varrho>0$ and $-2\varrho<  \vartheta <2\varrho$ and $\zeta\in\mathbb{R}$, we have \begin{equation}\notag
		\int_{\mathbb{R}^{3}}
		\mathbf{k}_{ \varrho}(v,u)
		\frac{\langle
			v\rangle ^{\zeta }e^{\theta |v|^{2}}}{\langle u\rangle ^{\zeta }e^{\theta
				|u|^{2}}}\mathrm{d}u \ \lesssim \ \langle v\rangle ^{-1}.  \label{int_k}
		\end{equation}%
	\end{lemma}
	
	\begin{proof}
		The proof is based on \cite{Guo10}. Note that
		\[
		\frac{\langle v\rangle^\zeta e^{\theta|v|^2}}{\langle u \rangle^\zeta e^{\theta|u|^2}}
		\ \lesssim \  [1+|v-u|^2]^{\frac{\zeta}{2}} e^{-\theta(|u|^2 -|v|^2)}.
		\]
		Set $v-u = \eta$ and $u =v-\eta$ in the integration of (\ref{int_k}). Now we compute the total exponent of the integrand of (\ref{int_k}) and (\ref{int_k_maxwell}) as
		\begin{equation*}
		\begin{split}
		&-\varrho|\eta|^2 -\varrho \frac{||\eta|^2-2 v\cdot \eta|^2}{|\eta|^2} -\theta\{|v-\eta|^2 -|v|^2\}
		=  - 2\varrho|\eta|^2 + 4\varrho \{v\cdot \eta\} - 4\varrho \frac{|v\cdot \eta|^2}{|\eta|^2} -\theta\{|\eta|^2 -2v\cdot \eta\}\\
		& =\left(-\theta-2\varrho\right) |\eta|^2 + \left( 4\varrho+ 2\theta\right)v\cdot \eta -4\varrho\frac{|v\cdot \eta|^2}{|\eta|^2}.
		\end{split}
		\end{equation*}
		Since $-2\varrho <\theta< 2\varrho,$ the discriminant of the above quadratic form of $|\eta|$ and $\frac{v\cdot \eta}{|\eta|}$ is negative : $\left(4\varrho+ 2\theta\right)^2 +16 \varrho(-\theta-2\varrho)   = 4\theta^2 -16 \varrho^2<0$. We thus have
		\begin{equation*}
		-\varrho|\eta|^2 -\varrho \frac{||\eta|^2-2 v\cdot \eta|^2}{|\eta|^2} -\theta\{|v-\eta|^2 -|v|^2\} \
		\lesssim_{\varrho, \theta} \ \frac{|\eta|^2}{2}+|v\cdot \eta| .
		\end{equation*}
		Therefore, for $0< \kappa \leq 1$
		\begin{equation*}
		\begin{split}
		&\int_{\mathbb{R}^{3}}\Big\{|v-u|^{\kappa }+|v-u|^{-2+\kappa }\Big\}e^{-\varrho
			|v-u|^{2}-\varrho\frac{(|v|^{2}-|u|^{2})^{2}}{|v-u|^{2}}}\frac{\langle
			v\rangle ^{\zeta }e^{\theta |v|^{2}}}{\langle u\rangle ^{\zeta }e^{\theta
				|u|^{2}}}\mathrm{d}u \\
		&\lesssim \ \int_{\mathbb{R}^3} \Big\{|\eta|^\kappa + |\eta|^{-2+\kappa}\Big\}\langle\eta\rangle  ^{\zeta } e^{- {C_{\varrho,\theta}}   |\eta|^2 }  \ \lesssim_{\varrho, \theta, \kappa} \ 1.
		\end{split}
		\end{equation*}%
		Therefore, in order to show (\ref{int_k}) it suffices to consider the case $|v|\geq 1$. We make another change of variables $\eta_\parallel = \left\{\eta\cdot \frac{v}{|v|}\right\}\frac{v}{|v|}$ and $\eta_\perp = \eta-\eta_\parallel$, so that $|v\cdot\eta|=|v||\eta_\parallel|$ and $|v-u|\geq |\eta_\perp|.$ We can absorb $\langle \eta \rangle^\zeta, \ |\eta|\langle \eta \rangle^\zeta$ by $e^{-C_{\varrho,\theta}|\eta|^2}$, and bound the integral of (\ref{int_k}) by \begin{equation*}
		\begin{split}
		& \ \ \int_{\mathbb{R}^3} \big\{ 1+|\eta|^{-2+ \kappa}\big\} e^{-C_{\varrho,\theta}\left\{\frac{|\eta|^2}{2}+|v\cdot \eta|\right\}} \mathrm{d} \eta \ \leq \ \int_{\mathbb{R}^3}  \big\{ 1+|\eta|^{-2+ \kappa }\big\} e^{-\frac{C_{\varrho,\theta}}{2}|\eta|^2}     e^{-C_{\varrho,\theta}|v\cdot \eta|} \mathrm{d} \eta\\
		&\leq \int_{\mathbb{R}^2} \{1+|\eta_\perp|^{-2+\kappa  }\} e^{-\frac{C_{\varrho,\theta}}{2}|\eta_{\perp}|^2} \left\{\int_{\mathbb{R}}  e^{-C_{\varrho,\theta} |v|\times |\eta_\parallel|} \mathrm{d} |\eta_\parallel|\right\} \mathrm{d} \eta_\perp\\
		& \lesssim \langle v\rangle^{-1 } \int_{\mathbb{R}^2} \{1+|\eta_\perp|^{-2+\kappa  }\} e^{-\frac{C_{\varrho,\theta}}{2}|\eta_{\perp}|^2} \left\{ \int_{0}^\infty  e^{-C_{\varrho,\theta}  y } \mathrm{d} y\right\} \mathrm{d} \eta_\perp, \ \ (y=|v||\eta_\parallel|).
		\end{split}
		\end{equation*}
	\end{proof}



	\begin{proof}[Proof of Proposition ]
		\textit{Step 1.} By taking derivatives to (\ref{smallfphi}),
		\Be \label{eqtn_nabla_f} 
		\begin{split}
			&\p_t f_{x_{i}} + v\cdot \nabla_x f_{x_{i}} - \nabla_x  \phi_f  \cdot \nabla_v f_{x_{i}} + \frac{v}{2} \cdot \nabla_x \phi_f f_{x_{i}} 
			+   \nu f _{x_{i}} - K f _{x_{i}} \\
			&
			=  
			\nabla_x    \p_{x_{i}} \phi_f  \cdot \nabla_v f 
			- \frac{v}{2} \cdot \nabla_x \p_{x_{i}} \phi_f f
			+
			\Gamma(f,f)_{x_{i}} - v\cdot \nabla_x \p_{x_{i}} \phi_f \sqrt{\mu}  
			, \\
			&
			\p_{t} f_{v_{i}} + v\cdot \nabla_{x} f_{v_{i}}- \nabla_{x} \phi_f \cdot \nabla_{v} f_{v_{i}} + \frac{v}{2}  \cdot \nabla \phi _f f_{v_{i}} + \nu f_{v_{i}}- (Kf)_{v_{i}}\\
			&= 
			(\Gamma(f,f))_{v_{i}}
			- f_{x_{i}} - \frac{1}{2} \p_{x_{i}} \phi_f f - \nu_{v_{i}} f - \p_{x_{i}} \phi_f \sqrt{\mu} - v\cdot \nabla_{x}  \phi_f (\sqrt{\mu})_{v_{i}}.
		\end{split}\Ee
		Let us adopt the next notation as a matter of convenience
		\Be
		\nu_{\phi_f} = \nu_{\phi_f} (t,x,v) := \nu(v) + \frac{v}{2} \cdot \nabla \phi_f(t,x).
		\Ee
		
		Let us choose 
		\Be\label{beta_p_1}
		\frac{p-2}{p}< \beta <\frac{p-1}{p} , \ \ p>3.
		\Ee
		We will restrict more a range of $\beta$ later in (\ref{parameters_k}). From (\ref{alpha_invariant}) and  the notations $K_v f, \Gamma_v (f,f)$ in (\ref{Kv}), (\ref{Gammav}), we deduce
		\Be \begin{split}
			&  \big[\p_t + v\cdot \nabla_x - \nabla_x \phi_f \cdot \nabla_v + \nu_{\phi_f} 
			\big] |\alpha^\beta  \p f |^p \\
			& = 
			\alpha^{\beta p}  \text{sign}(\p f)|\p f|^{p-1}\{ K\p f +  \Gamma(\p f,f) +  \Gamma(f,\p f)\}  + I + II + III + IV,\label{Boltzmann_p}
		\end{split}\Ee 
		where
		\Bes  
		|I|  &\leq&\alpha^{\beta p}  |\p f|^{p-1}    | K_v f| +\alpha^{\beta p}  |\p f|^{p-1}  |\Gamma_v (f,f)|,\\
		|II| &\leq & \{1+ |\nabla_x^2 \phi_f| \}\alpha^{\beta p}  |\p f|^{p }   ,\\
		|III| &\leq& \{|\nabla_x \phi_f| + |v||\nabla_x^2 \phi_f| + |v| \}\alpha^{\beta p}  |\p f|^{p-1}   |f|
		,\\
		| IV| &\leq &\{ |v| \sqrt{\mu } |\nabla^2 _x\phi_f| + \langle v\rangle^2 \sqrt{\mu } |\nabla_x \phi _f|\} \alpha^{\beta p}  |\p f|^{p-1} . 
		\Ees

		From (\ref{Boltzmann_p}), and the Green's identity, we obtain that 
		\Be\begin{split}\label{pf_green}
			& \| \alpha^\beta 
			\p f (t)\|_p^p
			+ \int^t_0 \| \nu_{\phi_f}^{1/p} \alpha^\beta 
			\p f   \|_p^p  
			+ \underbrace{\int^t_0 |\alpha^\beta 
				\p f   |_{p,+}^p  }_{(\ref{pf_green})_{1}}
			\\
			\lesssim&  \  \| \alpha^\beta 
			\p f (0)\|_p^p + (1+ \| \nabla^2 \phi \|_\infty) \int^t_0 \| \alpha^\beta 
			\p f  \|_p^p  + 
			\underbrace{
				\int^t_0 |\alpha^\beta 
				\p f  |_{p,-}^p
			}_{(\ref{pf_green})_1}
			\\
			& 
			+(1+ \| w f \|_\infty) \\
			& \ \ \  \times \underbrace{\int^t_0  \int_\O 
				\int_{\R^3}
				\alpha(v)^{p\beta}    |  \p f (v)|^{p-1}  
				\int_{\R^3} \mathbf{k}_\zeta(v,u )  |  \p f (u)| \dd u
				\dd v
				\dd x \dd s}_{(\ref{pf_green})_2} \\
			&+ \cdots.
		\end{split}\Ee
		
		\vspace{4pt}
		
		\textit{Step 2.} We focus on $(\ref{pf_green})_1$.

		$\clubsuit$See Page 120-122, 173-175, 182-185 plus Lemma 8 in \cite{GKTT1}.
		
		Note that $\alpha (t,x,v) = |n(x) \cdot v|$ for $(x,v) \in \gamma_-$. It is important to note that, from (\ref{beta_p_1}), we have $(\beta-1)p +1>-1$ and $|n(x) \cdot v|^{\beta p+1}, |n(x) \cdot v|^{(\beta-1) p+1} \in L^1_{loc} (\R^3)$. Then from (\ref{BC_deriv})
		\Be \label{BC_deriv_1}
		\begin{split}
			&
			\int_{n(x) \cdot v<0}
			|n(x) \cdot v|^{\beta p}
			|\nabla_{x,v} f(t,x,v) |^p
			|n(x) \cdot v| \dd v
			\\
			\lesssim&  \int_{n(x) \cdot v<0}
			\langle v\rangle^p \mu(v)^{\frac{p}{2}}
			\Big(|n(x) \cdot v|^{\beta p+1}+ |n(x) \cdot v|^{(\beta-1) p+1} 
			\Big)\\
			& \times \bigg\{
			\int_{n(x) \cdot u>0} 
			\big\{
			\langle u\rangle  + |\nabla_x \phi_f|
			\big\}
			|\nabla_{x,v} f(t,x,u) |
			\sqrt{\mu(u)}\{n(x) \cdot u\} \dd u\\
			&  \ \ \ \ \  
			+ 
			\| w f \|_\infty (1+ \| w f \|_\infty) (1+ |\nabla_x \phi_f|)
			+  |\nabla_x \phi_f| \bigg\}^p \dd v \\
			\lesssim  & 
			(1+ |\nabla_x \phi_f|)^p \bigg\{
			\int_{n(x) \cdot u>0}  |\nabla_{x,v} f(t,x,u)|\langle u\rangle \sqrt{\mu(u)} \{n(x) \cdot u\} \dd u\bigg\}^p
			\\
			&
			+ \big\{ \| w f \|_\infty (1+ \| w f \|_\infty) (1+ |\nabla_x \phi_f|)
			+  |\nabla_x \phi_f|\big\}^p.
		\end{split}\Ee
		Now we split the above $u$-integration into $\gamma_+^\e(x) \cup \gamma_+ (x) \backslash \gamma_+^\e(x)$ and apply the H\"older inequality to deduce that 
		\begin{eqnarray}
		&& \bigg\{\int_{n(x) \cdot u>0}  |\alpha^\beta\nabla_{x,v} f(s,x,u)| |n(x) \cdot u|^{- \beta}\langle u\rangle \sqrt{\mu(u)} \{n(x) \cdot u\}  \dd u\bigg\}^p
		\label{W1p_bdry}
		\\
		&\lesssim& 
		\bigg\{\int_{\gamma_+^\e (x)}  |\alpha^\beta\nabla_{x,v} f(s,x,u)|^p\{n(x) \cdot u\}  \dd u\bigg\}\notag\\
		&& \ \ \ \  \times 
		\bigg\{ \int_{\gamma_+^\e (x)}
		\alpha(s,x,u)^{- \beta q}
		|n(x) \cdot u|  \mu ^{\frac{q}{4}}\dd u 
		\bigg\}^{p/q}\notag\\
		& &+\bigg\{\int_{\gamma_+(x) \backslash\gamma_+^\e (x)}  |\alpha^\beta\nabla_{x,v} f(s,x,u)|^p
		\mu ^{\frac{p}{8}}
		\{n(x) \cdot u\}  \dd u\bigg\}\notag\\
		&& \ \ \ \   \times \bigg\{ \int_{\R^3}
		\alpha(s,x,u)^{- \beta q}
		|n(x) \cdot u| \mu ^{\frac{q}{8}}\dd u 
		\bigg\}^{p/q}.\notag
		\end{eqnarray}
		From (\ref{beta_p_1}), $1- \beta q>0$ and $|n(x) \cdot u|^{1- \beta q} \in L^1_{loc} (\R^3)$. Since $\mathbf{1}_{\gamma^\e_+ (x)} \downarrow 0$ almost everywhere in $\R^3$ as $\e \downarrow 0$
		\Be\begin{split}\notag
			(\ref{W1p_bdry}) \lesssim& \  O(\e)  \int_{\gamma_+^\e (x)}  |\alpha^\beta\nabla_{x,v} f(s,x,u)|^p\{n(x) \cdot u\}  \dd u\\
			&+\int_{\gamma_+(x) \backslash\gamma_+^\e (x)}  |\alpha^\beta\nabla_{x,v} f(s,x,u)|^p \mu(u)^{{p}/{8}}\{n(x) \cdot u\}  \dd u.
		\end{split} 
		\Ee
		Now we collect estimates of (\ref{BC_deriv_1}), (\ref{W1p_bdry}), and the above estimate to derive that 
		\Be\begin{split}\notag
			(\ref{pf_green})_1 \lesssim & \ 
			(1+ \| \nabla \phi_f \|_\infty^p) \times  \Big\{O(\e) \int^t_0 |\alpha^\beta \p f |_{p,+}^p + \underline{\underline{ \int^t_0
					\int_{\gamma_+ \backslash \gamma_+^\e} 
					|
					\alpha^\beta \p f | ^p \mu^{p/8}}}\Big\}\\
			&+  \big\{ \| w f \|_\infty (1+ \| w f \|_\infty) (1+ \|\nabla_x \phi_f\|_\infty)
			+  \|\nabla_x \phi_f\|_\infty\big\}^p.
		\end{split} \Ee
		
		For the underlined term above we use Lemma \ref{le:ukai}. From (\ref{Boltzmann_p})
		
		Note that 
		\Be\begin{split}\notag
			&
			\iint_{ \O \times \R^3}
			\big|[\p_t + v\cdot \nabla_x - \nabla_x \phi_f \cdot \nabla_v + \nu_\phi  ] |\alpha^\beta\nabla_{x,v} f 
			\mu^{1/8}|^p
			\big|\\
			\leq& \ 
			p\alpha^{\beta p}  |\nabla_{x,v} f|^{p-1}  \mu^{p/8}
			\big|
			[\p_t + v\cdot \nabla_x - \nabla_x \phi_f \cdot \nabla_v + \nu_\phi  ]  \nabla_{x,v} f \big|
			+ ( |\nabla_x \phi_f| |v| 
			+\nu_\phi)\mu^{p/8} |\alpha^\beta\nabla_{x,v} f |^p\\
			\lesssim&  \  \int_{\O} \int_{\R^3}\alpha(v)^{\beta p } |\p f(v)|^{p-1} \int_{\R^3} \mathbf{k}_\varrho (v,u) |\p f(u)| \dd u \dd v \dd x\\
			= & \ 
			\int_{\O} \int_{\R^3}\alpha(v)^{\beta (p-1) } |\p f(v)|^{p-1} \int_{\R^3}  \mathbf{k}_\varrho (v,u)  \frac{1
			}{ \alpha(u)^{\beta}
			}
			\alpha(u)^{\beta} | \p f(u)| \dd u \dd v \dd x
			\\
			\lesssim & \ \| \alpha^\beta \p f \|_{L^p(\O \times \R^3)}^{p-1}
			\left\|  \int_{\R^3}  \mathbf{k}_\varrho (v,u)  \frac{1
			}{\alpha(u)^{\beta}}
			\alpha(u)^{\beta} | \p f(x,u)| \dd u \right\|_{L_{x,v}^p(\O \times \R^3)}\\
			\lesssim & \ \| \alpha^\beta \p f \|_{L^p(\O \times \R^3)}^{p-1}
			\| \alpha^\beta \p f 
			\|_{L^p_{x,v}} \times \sup_{x} 
			\left\| 
			\left|
			\int_{\R^3}
			|\mathbf{k}_{\varrho} (v,u)|^q \frac{
				1
			}{\alpha(u)^{q \beta}}
			\dd u
			\right|^{1/q}
			\right\|_{L^p_v (\R^3)},
		\end{split}\Ee
		where we have used $\alpha(v)^\beta \mu(v)^{p/8}\lesssim 1$.
		
		Note that 
		\Be
		\begin{split}
			& \left\| 
			\left|
			\int_{\R^3}
			| \mathbf{k}_{\varrho} (v,u)|^q \frac{
				1
			}{\alpha(u)^{q \beta}}
			\dd u
			\right|^{1/q}
			\right\|_{L^p_v (\R^3)}\\
			\lesssim  &  \    \left\| 
			\left|
			\int_{\R^3}
			\frac{e^{- \frac{q\varrho}{4}|v-u|^2 }}{|v-u|^q} \frac{
				1
			}{\alpha(u)^{q \beta}}
			\dd u
			\right|^{1/q}
			\right\|_{L^p_v (\R^3)}
		\end{split}
		\Ee

		\vspace{4pt}
		
		\textit{Step 3.} We focus on $(\ref{pf_green})_2$. 
		Recall a standard estimate of $\mathbf{k}(v,u)\lesssim  \frac{e^{- C|v-u|^2}}{|v-u|}$ in (\ref{estimate_k}). 
		With $M>0$, we split $\{|u| \leq M\} \cup \{|u| \geq M\}$.
		
		For $\{|u| \geq M\}$, by Holder inequality
		\Be\begin{split}\notag
			&\alpha(v)^\beta \int_{|u| \geq M} \mathbf{k}_\zeta (v,u) |\p f(u)|\\
			\leq & \ 
			\alpha(v)^\beta
			{\left(\int_{|u| \geq M} \mathbf{k}_\zeta(v,u) \frac{1}{\alpha(u)^{\beta q}}  \dd u \right)^{1/q}}
			\left(
			\int_{|u| \geq M} \mathbf{k}_\zeta(v,u)  |  \alpha^{\beta  }\p f (u)|^p \dd u 
			\right)^{1/p}\\
			\lesssim & \ \alpha(v)^\beta \left(
			\int_{|u| \geq M} \mathbf{k}_\zeta(v,u)  |\alpha^{\beta  }\p f (u)|^p \dd u 
			\right)^{1/p},
		\end{split}\Ee
		where have used Proposition ?? at the last line. 
		
		Then the contribution of $\{|u| \geq M\}$ in $(\ref{pf_green})_2$ is bounded by 
		\Be\begin{split}
			&\int^t_0 \int_{\O} \int_{v \in \R^3}  | \nu_\phi^{1/p}\alpha^\beta \p f (v)|^{p-1} \frac{\alpha(v)^\beta}{ \nu_\phi (v)^{\frac{p}{p-1}}}
			\int_{|u| \geq M} \mathbf{k}_\zeta (v,u) |\p f(u)| \\
			\leq& \ \int^t_0 \int_{\O} \bigg(\int_{v}  |\nu_\phi^{1/p}\alpha^\beta \p f (v)|^{p }\bigg)^{1/q}
			\bigg(
			\int_{|u| \geq M} |\alpha^{\beta  }\p f (u)|^p  \int_{v}\mathbf{k}_\zeta(v,u)  
			\bigg)^{1/p}\\
			\lesssim & \ 
			O(\e) 
			\int_0^t \| \nu_\phi^{1/p}\alpha^\beta \p f  (s)\|_p^p \dd s
			+\int_0^t \| \alpha^\beta \p f  (s)\|_p^p \dd s,
		\end{split}\Ee
		where we have used $\frac{\alpha(v)^\beta}{ \nu_\phi (v)^{\frac{p}{p-1}}}\lesssim 1$ for $\beta$ in (\ref{beta_condition}).

		For $\{|u| \leq M\}$ we use the H\"older inequality with $\frac{1}{p} +\frac{1}{q}=1$ 
		\begin{eqnarray}
		&& \int^t_0  \int_\O 
		\int_{\R^3}
		| \nu_\phi^{1/p}\alpha^\beta \p f (v)|^{p-1} 
		\int_{\R^3} \mathbf{k}(v,u) \frac{ \alpha(v)^{\beta} |  \alpha^\beta \p f (u)| }{\nu_\phi(v)^{(p-1)/p}  \alpha(u)^{\beta}}     \dd u
		\dd v
		\dd x \dd s \notag\\
		&\leq& 
		\int^t_0    \| \nu_\phi^{1/p} \alpha^\beta \p f  (s)\|_p^{ {p}/{q}}\notag\\
		&& \ \ \   \times 
		\Big[
		\int_\O
		\int_{\R^3}
		\Big(
		\underline{\underline{
				\int_{|u| \leq M} 
				\mathbf{k}(v,u) 
				\frac{\alpha(v)^{\beta} |
					\alpha^\beta \p f ( u)|}{ \nu_\phi(v)^{(p-1)/p}
					\alpha(u)^{\beta}}
				\dd u}}
		\Big)^p
		\dd v
		\dd x  \Big]^{1/p}
		\dd s 
		.\label{double_underline}
		\end{eqnarray}
		By the H\"older inequality, we bound an underlined $u$-integration inside (\ref{double_underline}) as
		\begin{eqnarray}
		&
		&
		\frac{\alpha(v)^{\beta}}{\nu_\phi(v)^{(p-1)/p}}
		\|  \alpha^\beta \p f(\cdot) \|_{L^p(\R^3)}
		\notag
		\\
		&\times& 
		\bigg(
		\int_{u}
		\frac{e^{-qC |v-u|^2}}{|v-u|^q} \frac{\mathbf{1}_{|u| \leq M}}{\alpha(u)^{\beta q}}
		\bigg)^{1/q}
		\label{double_underline_split}.
		\end{eqnarray}
		
		We bound
		\[
		(\ref{double_underline_split})
		\leq \bigg|
		\frac{1}{| \cdot |^q} *  \frac{\mathbf{1}_{|\cdot| \leq M}}{\alpha(\cdot)^{q\beta}}
		\bigg| ^{1/q}
		\]
		%
		By the Hardy-Littlewood-Sobolev inequality with $1+ \frac{1}{p/q} = \frac{1}{3/q} + \frac{1}{
			\frac{3}{2} \frac{p-1}{p}
		}$
		\Bes
		\big\|(\ref{double_underline_split})_1 \big\|_{L^p_{ v}}
		&\lesssim& 
		\left\|
		\bigg|
		\frac{1}{| \cdot |^q} *  \frac{\mathbf{1}_{|\cdot| \leq M}}{\alpha(\cdot)^{q\beta}}
		\bigg| ^{1/q}\right\|_{L^p(\R^3)} = 
		\bigg\|
		\frac{1}{| \cdot |^q} *  \frac{\mathbf{1}_{|\cdot| \leq M}}{\alpha(\cdot)^{q\beta}}
		\bigg\|_{L^{p/q} (\R^3)}  ^{1/q}\\
		&\lesssim& 
		\left\| \frac{\mathbf{1}_{|\cdot| \leq M}}{\alpha(\cdot)^{q\beta}}\right\|_{L^{
				\frac{3(p-1)}{2p}
			} (\R^3)}^{1/q}\\
		&\lesssim&
		\left( \int_{\R^3} \frac{\mathbf{1}_{|v| \leq M}}{\alpha(v) ^{\frac{p}{p-1} \beta 
				\frac{3(p-1)}{2p}
			}
		} \dd v\right)^{
			\frac{2p}{3(p-1)} \frac{p-1}{p}
		}
		\\
		&=&
		\left( \int_{\R^3} \frac{\mathbf{1}_{|v| \leq M}}{\alpha(v) ^{  3\beta/2
			}
		} \dd v\right)^{
			2/3
		}
		.
		\Ees
		For $3<p<6$, we have $\frac{3}{2}\frac{p-2}{p} < 1$.
		In this case let us choose 
		\Be\label{beta_condition}
		\frac{p-2}{p}<\beta< \frac{2}{3}, \ \ \ 3< p < 6. 
		\Ee
		Now we apply Proposition \ref{prop_int_alpha} and conclude that 
		\Be\label{est_DUS1}
		\big\|(\ref{double_underline_split})
		\big\|_{L^p_{ v}}
		\lesssim 
		\left( \int_{\R^3} \frac{\mathbf{1}_{|v| \leq M}}{\alpha(v) ^{ 3\beta/2}
		} \dd v\right)^{2/3}\lesssim_{p, \beta, M,\O} 1 < \infty.
		\Ee

	\end{proof}

	\unhide

		\section{$L^3_xL^{1+}_v$ bound of $\nabla_v f^\ell$ }
	\hide In this section we prove the uniqueness statement assuming the extra condition (\ref{extra_con_uniqueness}). We need the following lemma.
	\begin{lemma}
		Now we claim that for $0<\kappa\leq2$ and $0< \sigma<1$
		\Be\label{nonlocal_finite}
		\sup_{t,x}\int_{\R^3}
		\frac{e^{-C|v-u|}}{|v-u|^{2-\kappa}} \frac{\mathbf{1}_{\tb(t,x,v) \leq 2}}{|n(\xb(t,x,v)) \cdot \vb(t,x,v)|^\sigma} \dd u \lesssim 1.
		\Ee
	\end{lemma}\unhide
	
		\begin{proposition}\label{prop_better_f_v}
Assume the inital condition satisfies (\ref{wf0small}), (\ref{unif_Em}), and 
		\Be \label{Extra_uniq}
		\| w_{\tilde{\vartheta}} \nabla_v f_{0} \|_{L^3_{x,v}} < \infty.
		\Ee 
		Then for $T^{**} \ll1 $, the sequence \eqref{fell_local} satisfies
		\Be\label{bound_nabla_v_g_global}
		\sup_\ell \sup_{0 \le t \le T^{**} } \| \nabla_v f^\ell(t) \|_{L^3_x (\O)L^{1+\delta}_v(\R^3)} \lesssim 1 \ \ \text{for all } \ t\geq 0.
		\Ee
		%
		%
		%
	\end{proposition}

	\hide In this section we prove the uniqueness statement assuming the extra condition (\ref{extra_con_uniqueness}). We need the following lemma.
	\begin{lemma}
		Now we claim that for $0<\kappa\leq2$ and $0< \sigma<1$
		\Be\label{nonlocal_finite}
		\sup_{t,x}\int_{\R^3}
		\frac{e^{-C|v-u|}}{|v-u|^{2-\kappa}} \frac{\mathbf{1}_{\tb(t,x,v) \leq 2}}{|n(\xb(t,x,v)) \cdot \vb(t,x,v)|^\sigma} \dd u \lesssim 1.
		\Ee
	\end{lemma}\unhide

	\begin{proof} 
		\textit{Step 1. } Note that from (\ref{fell_local}) and (\ref{boundary_v}), we have
		\Be\begin{split}\label{eqtn_g_v}
			&[\p_t + v\cdot \nabla_x - q \nabla_x \phi^\ell \cdot \nabla_v + \nu(v) + q \frac{v}{2} \cdot \nabla_x \phi^\ell  ] \p_v f^{\ell+1} \\
			=&  -\p_x f^{\ell+1}-  q \frac{1}{2} \p_x \phi^\ell f^{\ell+1} - \p_v \nu f^{\ell+1} + \p_v(Kf^\ell) + \p_v (\Gamma_\text{gain} (f^\ell,f^\ell)) - \p_v (\Gamma_\text{loss} ( f^{\ell+1}, f^\ell ) )- q_1 (\p_x \phi^\ell \sqrt \mu - \frac{v_i^2}{2} \p_x \phi^\ell \sqrt \mu)
		\end{split}\Ee
		with the boundary bound for $(x,v) \in\gamma_-$
		\Be\label{bdry_g_v}
		\big|\p_v f^{\ell+1}  \big| \lesssim   |v| \sqrt{\mu} \int_{n \cdot u>0} |f^\ell| \sqrt{\mu} \{n \cdot u \} \dd u \ \ \text{on } \ \gamma_-.
		\Ee
		
		From (\ref{nabla_nu}), (\ref{vKsum}), (\ref{bound_Gamma_nabla_vf1}), (\ref{bound_Gamma_nabla_vf2}), (\ref{Gvloss}), and (\ref{Gvgain}), we obtain the following bound along the characteristics for $f_+$ and $f_-$ seperately. For $\iota = +$ or $-$ as in \eqref{iota}, 
		\begin{eqnarray}
		&&|\p_v f^{\ell+1}_\iota(t,x,v)|\nonumber
		\\
		&\leq &   \mathbf{1}_{t^{\ell}_{1,\iota}(t,x,v)> t}  
		|\p_v f^{\ell+1}(0,X_\iota^\ell(0;t,x,v), V_\iota^\ell(0;t,x,v))|\label{g_initial}\\
		& +&   \ \mathbf{1}_{t^{\ell}_{1,\iota}(t,x,v)<t}
		\mu(\vb)^{\frac{1}{4}}  \int_{n(\xb) \cdot u>0} 
		| f^{\ell+1}(t-t^{\ell}_{1,\iota}, \xb, u) |\sqrt{\mu} \{n(\xb) \cdot u\} \dd u\label{g_bdry}\\
		&  +&   \int^t_{\max\{t-t^{\ell}_{1,\iota}, 0\}} 
		|\p_x f^\ell(s, X_\iota^\ell(s;t,x,v),V_\iota^\ell(s;t,x,v))|
		\dd s\label{g_x}\\
		&   + &\int^t_{\max\{t-t^{\ell}_{1,\iota}, 0\}} 
		(1+ \| w_{\vartheta} f^\ell \|_\infty + \| w_\vartheta f^{\ell+1} \|_\infty)
		\int_{\R^3} \mathbf{k}_\varrho (V_\iota^\ell(s),u) |\p_v f^\ell(s,X_\iota^\ell(s),u)| \dd u 
		\dd s\label{g_K}\\
		& +  &
		\int^t_{\max\{t-t^{\ell}_{1,\iota}, 0\}} \| f^{\ell+1}(s) \|_\infty
		|\nabla_x \phi^\ell (s, X_\iota^\ell(s;t,x,v))| \mu ^{1/4}
		\dd s \label{g_phi}
		,
		\end{eqnarray}
		where $\delta_1$ is in (\ref{integrable_nabla_phi_f}). Here we used that from \eqref{bound_Gamma_nabla_vf2}, on the RHS of \eqref{eqtn_g_v}, $| \Gamma_\text{loss} ( \p_v f^{\ell+1}, f^\ell ) | \lesssim \langle v \rangle \| w_\theta f^\ell \|_\infty |\p_v f^{\ell + 1 } | \le \frac{\nu(v)}{8} |\p_v f^{\ell+1} | $, and thus can be absorbed to the LHS.
		
		Note that if $|v| > 2 \frac{\delta_1}{\Lambda_1}$, then from (\ref{integrable_nabla_phi_f}) and (\ref{delta_1/lamdab_1}), for $0 \leq s \leq t$,
		\Be\label{V_lower_bound_v}
		\begin{split}
			|V_\iota^\ell(s;t,x,v)| &\geq |v| - \int^t_0 |\nabla_x \phi^\ell (\tau;t,x,v)| \dd \tau \\
			&\geq |v| 
			- \delta_1/\Lambda_1
			\\
			&    \geq \frac{|v|}{2}.
		\end{split}
		\Ee
		Therefore 
		\Be\label{tilde_w_integrable}
		\sup_{s,t,x}\left\|    \frac{1}{{w}_{\tilde{\vartheta}} (V_\iota^\ell(s;t,x,v)) }    \right\|_{ L^{r}_v} \lesssim 1 \  \text{ for any }    1 \leq r \leq \infty .
		\Ee

		We derive 
		\Be\begin{split}\label{est_g_initial}
			&\| (\ref{g_initial})\|_{L^3_x L^{1+ \delta}_v}\\
			\lesssim & \ \left(
			\int_{\O}
			\left(\int_{\R^3} |w_{\tilde{\vartheta}} \p_v f^{\ell+1}(0,X_\iota^\ell(0 ), V_\iota^\ell(0 ))|^3 
			\right)
			\left(
			\int_{\R^3} \frac{\dd v}{|w_{\tilde{\vartheta}} (V_\iota^\ell(0 ))|^{(1+ \delta) \frac{3}{2-\delta}}}
			\right)^{\frac{2-\delta}{1+ \delta}}
			\right)^{1/3} \\
			\lesssim & \
			\left(\iint_{\O \times \R^3} |w_{\tilde{\vartheta}}(V_\iota^\ell(0;t,x,v)) \p_v f^{\ell+1}(0,X_\iota^\ell(0;t,x,v), V_\iota^\ell(0;t,x,v))|^3 \dd v \dd x\right)^{1/3}\\
			\lesssim & \ \| w_{\tilde{\vartheta}} \p_v f (0) \|_{L^3_{x,v}},
		\end{split}
		\Ee
		where we have used a change of variables $(x,v) \mapsto (X_\iota^\ell(0;t,x,v), V_\iota^\ell(0;t,x,v))$ and (\ref{tilde_w_integrable}).

		\hide
		Also we use $|V_\iota^\ell(0;t,x,v)| \gtrsim |v|$ for $|v|\gg 1$, from (\ref{decay_phi}), and hence $\tilde{w}(V_\iota^\ell(0;t,x,v))^{- (1+ \delta) \frac{3}{2-\delta}} \in L_v^1 (\R^3)$.\unhide
		
		Clearly 
		\Be\label{est_g_bdry}
		\|(\ref{g_bdry})\|_{L^3_x L^{1+ \delta}_v}
		 \lesssim \sup_{0 \leq s \leq t} \| w_{\vartheta} f^{\ell+1}(s) \|_\infty.
		\Ee
		
		From $W^{1,2}(\O)\subset L^6(\O)\subset L^2(\O)$ for a bounded $\O \subset \R^3$, and the change of variables $(x,v) \mapsto (X_\iota^\ell(s;t,x,v), V_\iota^\ell(s;t,x,v))$ for fixed $s\in(\max\{t-\tb,0\},t)$,\Be
		\begin{split}\label{est_g_phi}
			\|(\ref{g_phi})\|_{L^3_x L^{1+ \delta}_v} 
			\lesssim & \ \|w_\vartheta f^{\ell+1} \|_\infty \int^t_0 \| \mu^{1/8}  \nabla_x \phi^\ell (s,X_\iota^\ell(s;t,x,v))  \|_{L^3_{x,v} }\| 
			\mu^{1/8}
			\|_{L^{
					\frac{3(1+ \delta)}{2- \delta}
				}_v}\\
			\lesssim & \ \|w_\vartheta f^{\ell+1} \|_\infty  \int^t_0 \| \nabla_x \phi^\ell (s)   \|_{L^3_{x} }
			\lesssim   \|w_\vartheta f^{\ell+1} \|_\infty  \int^t_{\max\{t-\tb,0\}} \| \phi^\ell (s) \|_{W^{2,2}_{x} } 
			\\
			\lesssim & \ \|w_\vartheta f^{\ell+1} \|_\infty \int^t_0 \|w_{\tilde{\vartheta}} f^\ell(s) \|_{2}.
		\end{split}
		\Ee

		\vspace{4pt}
		
		\textit{Step 2. }  We claim 
		\Be\label{est_g_x}
		\|(\ref{g_x})\|_{L^3_x L^{1+ \delta}_v} \lesssim \int^t_0 \| w_{\tilde{\vartheta}} \alpha_{f^{\ell-1},\e}^\beta \p_x f^\ell (s) \|_{L^p_{x,v}}. 
		\Ee
		Now we have for $3<p<6$, by the H\"older inequality $\frac{1}{1+ \delta}= \frac{1}{ \frac{p+p \delta}{p-1 - \delta}}+ \frac{1}{p}$,
		\Be\label{init_p_xf}
		\begin{split}
			&\left\|\left\| \int^t_{\max\{t-t_{1,\iota}^\ell, 0\}}
			| \p_x f^\ell(s,X_\iota^\ell(s;t,x,v),V_\iota^\ell(s;t,x,v)) | \dd s
			\right\|_{L_{v}^{1+ \delta}(  \R^3)}\right\|_{L^{3}_x}\\
			\lesssim & \ \left\|\left\|  \int^t_{\max\{t-t_{1,\iota}^\ell, 0\}} \frac{ |  w_{\tilde{\vartheta}} \alpha_{f^{\ell-1},\e}^{\beta} \p_x f^\ell(s,X_\iota^\ell(s;t,x,v),V_\iota^\ell(s;t,x,v))|}{w_{\tilde{\vartheta}}\alpha_{f^{\ell-1},\e,\iota}(s,X_\iota^\ell(s;t,x,v),V_\iota^\ell(s;t,x,v))^{\beta} }
			\dd s
			\right\|_{L_{v}^{1+ \delta}(  \R^3)}\right\|_{L^{3}_x}
			\\
			\lesssim & \ \left\| \frac{w_{\tilde{\vartheta}}(v)^{-1}}{\alpha_{f^{\ell-1},\e,\iota}(t,x,v)^\beta}\right\|_{L_v^{\frac{p+p \delta}{p-1 - \delta}} (\R^3) }\\
			& \times \left\|
			\left\| \int^t_0 |w_{\tilde{\vartheta}}\alpha_{f^{\ell-1},\e}^\beta \p_x f^\ell(s,X_\iota^\ell(s;t,x,v), V_\iota^\ell(s;t,x,v))| \dd s\right\|_{L_{v}^p( \R^3)}\right\|_{L^{3}_x}\\
			\lesssim & \ \left\| \frac{ w_{\tilde{\vartheta}}(v)^{-1}}{\alpha_{f^{\ell-1},\e,\iota}(t,x,v)^\beta}\right\|_{L_v^{\frac{p+p \delta}{p-1 - \delta}} (\R^3) }
			\times \int^t_0 \| w_{\tilde{\vartheta}} \alpha_{f^{\ell-1},\e}^\beta \p_x f^\ell (s) \|_{L^p_{x,v}} \dd s 
			,
		\end{split}\Ee
		where we have used $\alpha_{f^{\ell-1},\e,\iota}(t,x,v)=\alpha_{f^{\ell-1},\e,\iota}(s,X_\iota^\ell(s;t,x,v),V_\iota^\ell(s;t,x,v))$ for $t-\tb(t,x,v)\leq s \leq t$ and the change of variables $(x,v) \mapsto (X_\iota^\ell(s;t,x,v), V_\iota^\ell(s;t,x,v))$ and the Minkowski inequality.
		
		For $\beta$ in (\ref{W1p_initial}), we have $
		\beta \frac{p}{p-1} < 1$ since $\frac{2}{3} < \frac{p-1}{p}$ for $3 < p$. Therefore, we can choose $0 < \delta \ll 1$ so that $\beta$ in (\ref{W1p_initial}) satisfies 
		\Be\label{less_than_1}
		\beta \times \frac{p+p \delta}{p-1 - \delta} < 1.
		\Ee
		\hide
		For $\beta$ in (\ref{W1p_initial}) we have $
		\beta \frac{p+p \delta}{p-1 - \delta}> \frac{p+p\delta}{p} \frac{p-2}{p-1-\delta}.$ Therefore for $0<\delta\ll1$ we can choose $\beta$ in (\ref{W1p_initial}) and satisfies 
		\Be\label{less_than_1}
		\beta \times \frac{p+p \delta}{p-1 - \delta} < 1.
		\Ee\unhide
		We apply Proposition \ref{prop_int_alpha} to conclude that 
		\Be\label{bound_g_x}
		\sup_{t,x} \left\| \frac{w_{\tilde{\vartheta}} (v)^{-1}}{\alpha_{f^{\ell-1},\e,\iota}(t,x,v)^\beta}\right\|_{L_v^{\frac{p+p \delta}{p-1 - \delta}} (\R^3) }^{\frac{p+p \delta}{p-1 - \delta}} 
		= \sup_{t,x} \int_{\R^3}  \frac{
			e^{- \tilde{\vartheta} \frac{p+p \delta}{p-1 - \delta} |v|^2 }
		}{\alpha_{f^{\ell-1},\e,\iota}(t,x,v)^{\beta \frac{p+p \delta}{p-1 - \delta}   }}   \dd v \lesssim 1.
		\Ee
		\hide
		We have 
		\Be
		\begin{split}
			| \tilde{w} \mathcal{G}|\lesssim&  \tilde{w} (\ref{G est})\\
			\lesssim & \tilde{w} |\nabla_{x,v} g| +
		\end{split}
		\Ee

		{\color{red}DETAILS}\unhide
		Finally, from (\ref{init_p_xf}), (\ref{bound_g_x}), and (\ref{W1p_main}), we conclude the claim (\ref{est_g_x}).
		
		\hide
		Now for $6>p>3$ and $\beta$ in (\ref{W1p_initial}) and a bounded domain $\O$, by the Minkowski inequality and a change of variables $(x,v) \mapsto (X_\pm^\ell(s;t,x,v), V_\pm^\ell(s;t,x,v))$ for fixed $s$, whose Jacobian equals 1,
		\Be\notag
		\begin{split}
			&\left\| (\ref{init_p_xf}) \right\|_{L^{3}_x}\\
			\lesssim& \  \left\| \int^t_0 \tilde{w} \alpha^\beta \p_x f(s,X_\pm^\ell(s;t,x,v), V_\pm^\ell(s;t,x,v)) \dd s\right\|_{L ^p(\O \times \R^3)}\\
			\lesssim & \ \int^t_0 \| \tilde{w} \alpha^\beta \p_x f(s) \|_p^p \dd s.
		\end{split}
		\Ee
		From (\ref{tilde_w_W1p}) we conclude the claim (\ref{est_g_x}).
		\unhide
		
		\vspace{4pt}
		
		\textit{Step 3. } We consider (\ref{g_K}). We split the $u$-integration of (\ref{g_K}) into two parts with $N\gg 1$ as 
		\begin{eqnarray}
		&&\int_{|u| \leq N} \mathbf{k}_\varrho (V_\iota^\ell(s), u) |\nabla_v f^\ell(s,X_\iota^\ell(s ), u) | \dd u \label{g_K_split1}\\
		&+& \int_{|u| \geq N} \mathbf{k}_\varrho (V_\iota^\ell(s), u) |\nabla_v f^\ell(s,X_\iota^\ell(s ), u) | \dd u .\label{g_K_split2}
		\end{eqnarray}

		First we bound (\ref{g_K_split1}). From the change of variables $(x,v) \mapsto (X_\iota^\ell(s;t,x,v), V_\iota^\ell(s;t,x,v))$ for $t-t^\ell_{1,\iota} \leq s \leq t$
		\Be\label{g_K_COV}
		\begin{split}
			&\left\|\int_{|u| \leq N} \mathbf{k}_\varrho (V_\iota^\ell(s;t,x,v), u) |\nabla_v f^\ell(s,X_\iota^\ell(s;t,x,v ), u) | \dd u  \right\|_{L^3_x L^3_v}\\
			= & \ \left\|\int_{|u| \leq N} \mathbf{k}_\varrho (v, u) |\nabla_v f^\ell(s,x, u) | \dd u  \right\|_{L^3_x L^3_v}.
		\end{split}
		\Ee
		If $|v|\geq 2N$ then $|v-u|^2\gtrsim |v|^2$ and $\mathbf{k}_\varrho (v,u) \lesssim \frac{e^{-C|v|^2}}{|v-u|^2}$ for $|v|\geq 2N$ and $|u| \leq N$. For $0 < \delta \ll 1$ with $\frac{3(1+ \delta)}{1-2\delta}>3$,  
		\Be
		\begin{split}\label{g_K_COV1}
			&(\ref{g_K_COV})\\
			\lesssim & \ C_N \left\| \left\| \int_{|u| \leq N} \mathbf{k}_\varrho (v, u) |\nabla_v f^\ell(s,x, u) | \dd u \right\|_{L^{\frac{3(1+ \delta)}{1-2\delta}}_v
				(\{|v| \leq 2N\})
			} \right\|_{L^3_x  }\\
			&+ \left\| \left\|  e^{-C|v|^2}\right\|_{L^{3/2}_v} \left\| \int_{|u| \leq N} \frac{1}{|v-u|} |\nabla_v f^\ell(s,x, u) | \dd u \right\|_{L^{\frac{3(1+ \delta)}{1-2\delta}}_v
				(\{|v| \geq 2N\})
			} \right\|_{L^3_x  }\\
			\lesssim & \ \left\|   \left\|   \frac{1}{|v- \cdot |} * |\nabla_v f^\ell(s,x, \cdot ) |   \right\|_{L^{\frac{3(1+ \delta)}{1-2\delta}}_v 
			} \right\|_{L^3_x  }.
		\end{split}
		\Ee
		Then by the Hardy-Littlewood-Sobolev inequality with $1+ \frac{1}{\frac{3(1+ \delta)}{1- 2\delta}} = \frac{1}{3} + \frac{1}{1+ \delta}$, we derive that   
		\Be\begin{split}\notag
			(\ref{g_K_COV1}) 
			\lesssim & \ \left\| \| \nabla_v f^\ell(s, x,v)  \|_{L^{1+ \delta}_v  }\right\|_{L^3_x}
			= \| \nabla_v f^\ell(s) \|_{L^3_x L^{1+ \delta}_v}
			.
		\end{split}\Ee 
		Combining the last estimate with (\ref{g_K_COV}), (\ref{g_K_COV1}), we prove that 
		\Be\label{est_g_K_split1}
		\|(\ref{g_K_split1})\|_{L^3_x L^{1+ \delta}_v} \lesssim    \| \nabla_v f^\ell(s) \|_{L^3_x L^{1+ \delta}_v}.
		\Ee
		
		Now we consider (\ref{g_K_split2}). Choose $0< {\delta'} \ll  1$. We have
		\Be\notag
		\begin{split}
			&(\ref{g_K_split2}) \\
			\le \  &\int_{|u| \geq N} \frac{1}{w_{\tilde{\vartheta}} (V_\iota^\ell(s;t,x,v))^{1- {\delta'}}} \frac{w_{\tilde{\vartheta}}(V_\iota^\ell(s;t,x,v))  }{w_{\tilde{\vartheta}}(u)} \frac{\mathbf{k}_\varrho (V_\iota^\ell(s;t,x,v), u)}{\alpha_{f^{\ell-1},\e,\iota}(s,X_\iota^\ell(s;t,x,v), u)^\beta} \\
			& \ \  \times \frac{1}{w_{\tilde{\vartheta}}(V_\iota^\ell(s;t,x,v))^{ {\delta'}}} w_{\tilde{\vartheta}}(u)|\alpha_{f^{\ell-1},\e}(s,X_\iota^\ell(s;t,x,v), u)^\beta  \nabla_v f^\ell(s,X_\iota^\ell(s;t,x,v), u)| \dd u .
		\end{split} 
		\Ee

		By the H\"older inequality with $\frac{1}{p} + \frac{1}{p^*}=1$ with $3<p<6$
		\Be
		\begin{split}\label{g_K_split2_Holder}
			&|(\ref{g_K_split2})| \\
			&\lesssim  \ \frac{1}{w_{\tilde{\vartheta}} (V_\iota^\ell(s;t,x,v))^{1- {\delta'}}} \left\|  \frac{w_{\tilde{\vartheta}}(V_\iota^\ell(s;t,x,v))  }{w_{\tilde{\vartheta}}(u)} \frac{\mathbf{k}_\varrho (V_\iota^\ell(s;t,x,v), u)}{\alpha_{f^{\ell-1},\e,\iota}(s,X_\iota^\ell(s;t,x,v), u)^\beta} \right\|_{L^{p^*} (\{ |u|\geq N \})} \\
			&  \times  \left\|  \frac{w_{\tilde{\vartheta}}(u) }{w_{\tilde{\vartheta}}(V_\iota^\ell(s;t,x,v))^{ {\delta'}}} \alpha_{f^{\ell-1},\e}(s,X_\iota^\ell(s;t,x,v), u)^\beta  \nabla_v f^\ell(s,X_\iota^\ell(s;t,x,v), u) \right\|_{L^p_u (\R^3)}.
		\end{split}
		\Ee
		
		Then by the H\"older inequality with $\frac{1}{1+ \delta} = \frac{1}{p} + \frac{1}{\frac{(1+ \delta)p}{p - (1+ \delta)}}$,
		\Be
		\begin{split}\notag
			&\|(\ref{g_K_split2})\|_{L^{1+ \delta}_v}\\
			\lesssim & \ \left\| \frac{1}{w_{\tilde{\vartheta}} (V_\iota^\ell(s;t,x,v))^{1- {\delta'}}} \right\|_{L_v^{\frac{(1+ \delta) p}{ p- (1+ \delta)}}} \\
			& \times \sup_v \left\|  \frac{w_{\tilde{\vartheta}}(V_\iota^\ell(s;t,x,v))  }{w_{\tilde{\vartheta}}(u)} \frac{\mathbf{k}_\varrho (V_\iota^\ell(s;t,x,v), u)}{\alpha_{f^{\ell-1},\e,\iota}(s,X_\iota^\ell(s;t,x,v), u)^\beta} \right\|_{L^{p^*} (\{ |u|\geq N \})} \\
			& \times \left\| \left\|  \frac{w_{\tilde{\vartheta}}(u) }{w_{\tilde{\vartheta}}(V_\iota^\ell(s;t,x,v))^{ {\delta'}}} \alpha_{f^{\ell-1},\e}(s,X_\iota^\ell(s;t,x,v), u)^\beta \nabla_v f^\ell(s,X_\iota^\ell(s;t,x,v), u) \right\|_{L^p_u  }\right\|_{L^{p}_v}.
		\end{split}
		\Ee
		
		Note that, from (\ref{k_vartheta_comparision}), $\mathbf{k}_\varrho (v,u) \frac{e^{\tilde{\vartheta} |v|^2}}{e^{\tilde{\vartheta} |u|^2}} \lesssim \mathbf{k}_{\tilde{\varrho}} (v,u)$ for some $0<\tilde{\varrho}< \varrho$. Hence we derive, using (\ref{tilde_w_integrable}) 
		\Be
		\begin{split}\notag
			&\left\|\|(\ref{g_K_split2})\|_{L^{1+ \delta}_v}\right\|_{L^3_x}\\
			\lesssim_\O &  \ \sup_{X_\iota^\ell,V_\iota^\ell} \left\| \frac{e^{- \frac{\tilde{\vartheta}}{10}|V_\iota^\ell-u|^2  }}{|V_\iota^\ell-u|}  \frac{1}{\alpha_{f^{\ell-1},\e,\iota}(s,X_\iota^\ell , u)^\beta} \right\|_{L^{p^*} (\{ |u|\geq N \})} \\
			&\times  \left\|  \frac{\tilde{w}(u) }{\tilde{w}(V_\iota^\ell(s;t,x,v))^{ {\delta'}}} \alpha_{f^{\ell-1},\e}(s,X_\iota^\ell(s;t,x,v), u)^\beta  \nabla_v f^\ell(s,X_\iota^\ell(s;t,x,v), u) \right\|_{L^p_{u,v,x}  } .
		\end{split}
		\Ee
		Finally using (\ref{NLL_split3}) in Proposition \ref{prop_int_alpha} with $ \frac{p-2}{p-1}<\beta p^*< 1$ from (\ref{W1p_initial}) and applying the change of variables $(x,v) \mapsto (X_\iota^\ell(s;t,x,v), V_\iota^\ell(s;t,x,v))$, we derive that 
		\Be
		\begin{split}\label{g_K_split2_Holder2}
			&\left\|\|(\ref{g_K_split2})\|_{L^{1+ \delta}_v}\right\|_{L^3_x}\\
			\lesssim_\O &    
			\left\| \left\|  \frac{1 }{w_{\tilde{\vartheta}}(v)^{ {\delta'}}}w_{\tilde{\vartheta}}(u) \alpha_{f^{\ell-1},\e}(s,x, u)^\beta  \nabla_v f^\ell(s,x, u) \right\|_{L^p_v} \right\|_{L^p_{u ,x}  } \\
			\lesssim \ & \left\|  \frac{1 }{w_{\tilde{\vartheta}}(v)^{ {\delta'}}} \right\|_{L^p_v}    \left\| w_{\tilde{\vartheta}}(u) \alpha_{f^{\ell-1},\e}(s,x, u)^\beta  \nabla_v f^\ell(s,x, u)   \right\|_{L^p_{u ,x}  }\\
			\lesssim \ &   \left\| w_{\tilde{\vartheta}} \alpha_{f^{\ell-1},\e}^\beta \nabla_v f^\ell(s )  \right\|_{L^p  }.
		\end{split}
		\Ee
		Combining (\ref{g_K_split2_Holder}) and (\ref{g_K_split2_Holder2}) we conclude that 
		\Be
		\|(\ref{g_K_split2})\|_{L^3_x L^{1+ \delta}_v} \lesssim   \|w_{\tilde{\vartheta}} \alpha_{f^{\ell-1},\e}^\beta   \nabla_v f^\ell(s )  \|_{L^p_{x,v}}.\label{est_g_K_split2}
		\Ee
		
		Finally from (\ref{est_g_K_split1}) and (\ref{est_g_K_split2}), and using the Minkowski inequality, we conclude that 
		\Be\label{est_g_K}\begin{split}
			& \| (\ref{g_K}) \|_{L^3_v L^{1+ \delta}_x} \\
			\lesssim & \  (1+ \| w_{\vartheta} f^\ell \|_\infty + \|w_\vartheta f^{\ell+1} \|_\infty) \int^t_0  \big[
			\| \nabla_v f^\ell (s) \|_{L^3_x L^{1+ \delta}_v}
			+
			\| w_{\tilde{\vartheta}} \alpha_{f^{\ell-1},\e}^\beta  \nabla_v f^\ell(s ) \|_{L^p_{x,v}}\big] \dd s. 
		\end{split}\Ee
		
		Collecting terms from (\ref{g_initial})-(\ref{g_phi}), and (\ref{est_g_initial}), (\ref{g_bdry}), (\ref{est_g_phi}), (\ref{est_g_x}), (\ref{est_g_K}), we derive
		\Be\begin{split}\label{bound_nabla_v_g}
			\sup_{0 \leq s \leq t}& \| \nabla_vf^{\ell+1}(s) \|_{L^3_xL^{1+ \delta}_v}
			\\ \lesssim &  \sup_{0 \leq s \leq t}\| \nabla_vf^{\ell+1}_+(s) \|_{L^3_xL^{1+ \delta}_v} + \sup_{0 \leq s \leq t}\| \nabla_vf^{\ell+1}_-(s) \|_{L^3_xL^{1+ \delta}_v}
			\\ \lesssim &  \
			\| w_{\tilde{\vartheta}} \nabla_v f(0) \|_{L^3_{x,v}} + 
			\sup_{0 \leq s \leq t}
			\| w_{\vartheta} f^{\ell+1}(s) \|_\infty + \sup_{0 \le s \le t }\| w_{\vartheta} f^\ell(s)\|_\infty\\
			&  
			+
			t (1+ \sup_{0 \leq s \leq t}\| w_{\vartheta} f^{\ell+1}(s) \|_\infty + \sup_{0 \le s \le t }\| w_{\vartheta} f^\ell(s)\|_\infty )
			   ( \sup_{0 \leq s \leq t} \|w_{\tilde{\vartheta}} \alpha_{f,\e}^\beta \nabla_{x,v} f^\ell(s) \|_{p} + \| \nabla_v f^\ell(s) \|_{L^3_x L^{1+ \delta}_v} ).
		\end{split} \Ee
Therefore from (\ref{uniform_h_ell}) and (\ref{unif_Em}), we can choose $T^{**} \ll 1 $ and conclude \eqref{bound_nabla_v_g_global}.

		\hide

		Then by the Holder inequality $\frac{1}{3} = \frac{1}{\frac{3p}{p-3}} + \frac{1}{p}$
		\Be\label{g_K_split2_Holder_Holder}
		\begin{split}
			&\|(\ref{g_K_split2_Holder})\|_{L^3_x L^{1+ \delta}_v }\\
			\lesssim  & \  \left\| \left\|   \frac{\mathbf{k}_\varrho (V_\iota^\ell(s;t,x,v), u)}{\tilde{w} (u)\alpha(s, X_\iota^\ell(s;t,x,v), u )^\beta} \right\|_{L^{p^*} (\{ |u|\geq N \})  } \right\|_{L_x^{\frac{3p}{p-3} } L^{1+ \delta}_v } \\
			& \times \| \tilde{w} \alpha^\beta \p_v g(s, X_\iota^\ell(s;t,x,v), u) \|_{L^p_{x,u }}  .
		\end{split}
		\Ee
		
		\Be\begin{split}\label{g_K_split2_Holder_Holder1}
			&\left\|  \left\|   \frac{\mathbf{k}_\varrho (V_\iota^\ell(s;t,x,v), u)}{\tilde{w} (u)\alpha(s, X(s;t,x,v), u )^\beta} \right\|_{L^{p^*}(\{ |u|\geq N \})   } \right\|_{L^{1+ \delta}_v }
			\\
			\lesssim&  \ 
			\left\| \frac{1}{\tilde{w} (V_\iota^\ell(s;t,x,v))} \left\|   \frac{\mathbf{k}_{\tilde{\varrho}} (V_\iota^\ell(s;t,x,v), u)}{ \alpha(s, X(s;t,x,v), u )^\beta} \right\|_{L^{p^*}(\{ |u|\geq N \})  }
			\right\|_{ L^{1+ \delta}_v  }\\
			\lesssim & \  
			\left\|    \tilde{w} (V_\iota^\ell(s;t,x,v))^{-1}   \right\|_{ L^{1+ \delta}_v  }
			\times \sup_v  \left\|   \frac{\mathbf{k}_{\tilde{\varrho}} (V_\iota^\ell(s;t,x,v), u)}{ \alpha(s, X(s;t,x,v), u )^\beta} \right\|_{L^{p^*}(\{ |u|\geq N \})  }.
		\end{split}\Ee

		On the other hand from (\ref{NLL_split3}) in Proposition \ref{prop_int_alpha} the last term of (\ref{g_K_split2_Holder_Holder1}) is bounded.

		From (\ref{g_K_split2_Holder}), (\ref{g_K_split2_Holder_Holder}), (\ref{g_K_split2_Holder_Holder1})
		\Be
		\|(\ref{g_K_split2})\|_{L^3_x L^{1+ \delta}_x}\lesssim  \| \tilde{w} \alpha^\beta \nabla_v g (s) \|_{L^p_{x,v}}
		\Ee

		Now we check that a change of variables $v \mapsto V_\iota^\ell(s;t,x,v)$ for fixed $s,t,x$ with $|t-s|\ll \lambda_\infty$ in (\ref{main_Linfty}) is well-defined and has a Jacobian 
		\Be\label{V_v_short}
		\begin{split}
			&\det \left(\frac{\p V(s;t,x,v)}{\p v}\right)\\
			= & \  \det \left(\text{Id}_{3 \times 3}- \int^s_t \frac{\p X(\tau;t,x,v)}{\p v} \cdot \nabla_x \nabla_x \phi (\tau, X(\tau;t,x,v)) \dd \tau \right)\\
			\gtrsim & \ 1  - \int^t_s |t- \tau| e^{- \frac{\lambda_\infty}{2} \tau } \dd \tau\\
			\gtrsim & \ 1- \frac{|t-s|}{\lambda_\infty}\\
			\gtrsim & \ 1 \ \ \  \text{for} \ \   |t-s| \ll \lambda_\infty,
		\end{split}
		\Ee
		where we have used (\ref{result_X_v}).

		We need the following estimate
		\Be\label{X_x_V_v_local}
		|\nabla_x X(s;t,x,v)| + |\nabla_v V(s;t,x,v)| \lesssim e^{C|t-s|},
		\Ee
		where $C>0$ depends on $\sup_{s \leq \tau \leq t }\|\nabla_x^2 \phi_g(\tau) \|_\infty$. This is a direct consequence of Gronwall inequality.

		Note that for fixed $t,x,v$ with $|t-s| \ll 1$ in (\ref{main_Linfty}), the change of variables $x \mapsto X(s;t,x,v)$ is well-defined and has a Jacobian
		\Be\label{X_x_jacobian}
		\begin{split}
			&\det \left(\frac{\p X(s;t,x,v)}{\p x}\right)\\
			= & \  \det \left( \text{Id}_{3 \times 3} - \int^s_t \int^{\tau }_t \frac{\p X(\tau^\prime;t,x,v)}{\p x} \cdot \nabla_x \nabla_x \phi (\tau^\prime, X(\tau^\prime;t,x,v)) \dd \tau^\prime \dd \tau \right)\\
			\gtrsim &  \ 1 - O\left(|t-s|^2\right)\\
			\gtrsim& \ 1 .
		\end{split}
		\Ee
		Here we have used (\ref{main_Linfty}) and (\ref{X_x_V_v_local}).
		
		Then from the above change of variables with (\ref{X_x_jacobian}) and the Minkowski inequality we derive
		\Be
		\begin{split}
			&\left\|(\ref{g_K_split1})
			\right\|_{L^3_x (\O)}\\
			\lesssim & \ \int_{|u| \leq N} \mathbf{k}_\varrho (V(s), u) \|\nabla_v g(s,x, u) \|_{L^3_x  } \dd u
		\end{split}
		\Ee

		If $|v|\geq 10N\gg  \frac{\| w f_0 \|_\infty}{\lambda_\infty} $ then from (\ref{Morrey})
		\Be\notag
		\begin{split}
			|V(s;t,x,v)| &\geq |v| - \int^t_s |\nabla_x \phi (\tau;t,x,v)| \dd \tau \\
			&\geq |v| - C \| w f_0 \|_\infty   \int^t_s e^{- \lambda_\infty \tau } \dd \tau \\
			&\geq |v| - \frac{C}{\lambda_\infty}\| w f_0 \|_\infty \\
			& \geq \frac{|v|}{2}.
		\end{split}
		\Ee
		Therefore if $|u| \leq N$ and $|v|\geq 10N$ then 
		\[
		|V(s;t,x,v)- u|^2 \gtrsim  |v|^2,
		\]
		which implies that 
		\Be \notag
		\mathbf{k}_\varrho (V(s;t,x,v), u) 
		\lesssim    \frac{e^{-C|v  |^2}}{|V(s;t,x,v) - u|} .
		\Ee
		Therefore we have 
		\Be
		\begin{split}\label{g_K_small}
			&\left\|\int_{|u | \leq N} \mathbf{k}_\varrho (V(s;t,x,v), u) |\p_v  g(s,X(s;t,x,v), u)| \dd u\right\|_{L^{1+ \delta}_v (\R^3)}\\
			\lesssim & \ \left\{1+\| e^{-C|v|^2} \|_{L^{3/2}_v (\R^3)}\right\}\\
			& \times \left\| 
			\frac{1}{|V(s;t,x,v)-\cdot |} *|\nabla_v g (s,X(s;t,x,v), \cdot)|
			\right\|_{L^{\frac{3(1+ \delta)}{1-2\delta}}_v (\R^3)}.
		\end{split}
		\Ee
		Using (\ref{V_v_short}) and the Hardy-Littlewood-Sobolev inequality with $1+ \frac{1}{\frac{3(1+ \delta)}{1- 2\delta}} = \frac{1}{3} + \frac{1}{1+ \delta}$, we derive that   
		\Be\begin{split}
			(\ref{g_K_small}) \lesssim& \  \left\| 
			\frac{1}{|v-\cdot |} *|\nabla_v g (s,X(s;t,x,v), \cdot)|
			\right\|_{L^{\frac{3(1+ \delta)}{1-2\delta}}_v (\R^3)}\\
			\lesssim & \  \| \nabla_v g(s,X(s;t,x,V(s;t,x,v)), \cdot)\|_{L^{1+ \delta}_v (\R^3)}
			.
		\end{split}\Ee

		Therefore we have 
		\Be
		\begin{split}
			&\left\|\int_{\R^3} \mathbf{k}_\varrho (V(s;t,x,v), u) |\p_v  g(s,X(s;t,x,v), u)| \dd u\right\|_{L^{1+ \delta}_v (\R^3)}\\
			\lesssim & \
			\int_{\R^3} \mathbf{k}_{\tilde{\varrho}} (V(s;t,x,v), u) |\nabla_v  g(s,X(s;t,x,v), u)|  \dd u
			\\
			\lesssim & \ \left\| 
			\frac{1}{|V(s;t,x,v)-\cdot |} *|\nabla_v g (s,X(s;t,x,v), \cdot)|
			\right\|_{L^{\frac{3(1+ \delta)}{1-2\delta}}_v (\R^3)}\\
			\lesssim & \  \| \nabla_v g(s,X(s;t,x,v), \cdot)\|_{L^{1+ \delta}_v (\R^3)}
		\end{split}
		\Ee

		\vspace{4pt}
		
		\textit{Step .}
		
		\unhide

		\hide
		we derive that 
		\Be\label{improve_v}
		\sup_{0 \leq s \leq   t} \| \nabla_v g (s)\|_{L^3_x L^{1+ \delta}_v} \lesssim 1.
		\Ee
		From (\ref{improve_v}), we conclude the uniqueness by combining (\ref{f-g_energy}), (\ref{gronwall_f-g}), (\ref{gronwall_f-g_last}), (\ref{gamma_+,1+delta}), (\ref{improve_v}) and applying the Gronwall inequality. 
		
		\unhide\end{proof}
	
	\section{Local existence}
		\begin{theorem}\label{local_existence}
		Let $0< \tilde{\vartheta}< \vartheta \ll1$. Assume that for sufficiently small $M>0$, $F_0= \mu+ \sqrt{\mu}f_0\geq 0$ satisfying \eqref{wf0small}, \eqref{unif_Em}, \eqref{Extra_uniq} and the compatibility condition (\ref{compatibility_condition}).

	Then there exists $T^*(M)>0$ and a unique solution $F(t,x,v) =\mu+ \sqrt{\mu} f(t,x,v)\geq 0$ to (\ref{2fVPB}), (\ref{smallfphi}), and (\ref{diffusef}) in $[0, T^*(M)) \times \O \times \R^3$ such that 
		\Be\begin{split}\label{infty_local_bound} \sup_{0 \leq t \leq T*} 
			\| w_\vartheta f  (t) \|_{\infty}
			\leq M.
		\end{split}\Ee 
		
		Moreover
		\Be\label{31_local_bound}
		\sup_{0 \leq t \leq T^*}\| \nabla_v f (t) \|_{L^3_xL^{1+ \delta}_v}< \infty \ \ \text{for} \ 0< \delta \ll1,
		\Ee
		and 
		\Be\begin{split}\label{W1p_local_bound}
			\sup_{0 \leq t \leq T*} 
			\Big\{
			\| w_{\tilde{\vartheta}}\alpha_{f,\e }^\beta \nabla_{x,v} f (t) \|_{p} ^p
			+ 
			\int^t_0  |w_{\tilde{\vartheta}} \alpha_{f,\e}^\beta \nabla_{x,v} f (t) |_{p,+}^p
			\Big\}
			< \infty.
		\end{split}\Ee
		Furthermore, $\|  w_\vartheta f  (t) \|_{\infty}$, $\| \nabla_v f (t) \|_{L^3_xL^{1+ \delta}_v}$ and $\| w_{\tilde{\vartheta}}\alpha_{f,\e }^\beta \nabla_{x,v} f (t) \|_{p} ^p
		+ 
		\int^t_0  |w_{\tilde{\vartheta}} \alpha_{f,\e }^\beta \nabla_{x,v} f (t) |_{p,+}^p$ are continuous in $t$.	
\end{theorem}		

\begin{proof}

\textit{Step 1. }  We claim that for $T^{**} \ll 1$, the whole sequence \eqref{fell_local} satisfies
\Be \label{1+delta_stability}
f^\ell \to f \text{ strongly in }  L^\infty ((0,T); L^{1 + } (\Omega \times \mathbb R^3 )).
\Ee
		Note that $f^{\ell + 1 } - f^\ell $ satisfies $(f^{\ell +1 } - f^\ell ) |_{t = 0 } = 0$, so 
		\begin{equation}\label{eqtn_f-g}
		\begin{split}
		& \p_t [f^{\ell+1}-f^{\ell}] + v\cdot \nabla_x [f^{\ell+1}-f^{\ell}] - q \nabla_x  \phi^\ell \cdot \nabla_v [f^{\ell+1}-f^{\ell}]+ q \frac{v}{2} \cdot \nabla_x \phi^\ell [f^{\ell+1}-f^{\ell}]+  \nu [f^{\ell+1}-f^{\ell}] \\
		= & q \nabla_x \phi_{f^{\ell}-f^{\ell-1}} \cdot \nabla_v f^{\ell-1}
		\\ & +\Gamma_\text{gain}(f^{\ell},f^{\ell}) - \Gamma_\text{loss}(f^{\ell +1} ,f^\ell ) -\Gamma_\text{gain}(f^{\ell-1},f^{\ell-1}) + \Gamma_\text{loss}(f^\ell, f^{\ell - 1 } )
		\\ &  + K[f^{\ell}-f^{\ell-1}]  - q \frac{v}{2} \cdot \nabla_x \phi_{f^{\ell}-f^{\ell-1}} f^{\ell-1}  - q_1 v\cdot \nabla_x \phi_{f^{\ell}-f^{\ell-1}} \sqrt{\mu} 
		.
		\end{split}\end{equation} 
		By Lemma \ref{lem_Green} for $L^{1+ \delta}$-space with $0 < \delta \ll 1$, we obtain 
		\Be\begin{split}\label{f-g_energy}
			&\| [f^{\ell+1}- f^{\ell}](t)\|_{1+ \delta}^{1+ \delta} + \int^t_0 \| \nu_{\phi^\ell}^{ {1}/{1+ \delta}} [f^{\ell+1}-f^{\ell}] \|_{1+ \delta}^{1+ \delta}  + \int_0^t |[f^{\ell+1}-f^{\ell}]|_{1+ \delta, + }^{1 + \delta} \\
			\leq & \ \| [f^{\ell+1}- f^{\ell}](0)\|_{1+ \delta}^{1+ \delta}+
			\int^t_0 \iint_{\O \times \R^3}
			|\text{RHS of } (\ref{eqtn_f-g}) | |f^{\ell+1}-f^{\ell}|^{\delta}
			+\int_0^t |[f^{\ell+1}-f^{\ell}]|_{1+ \delta, - }^{1 + \delta},
		\end{split}
		\Ee
where $\nu_{\phi^\ell } $ is defined as (\ref{def_nu_phi}).		
		
		
Now for $0 < \delta \ll 1$, by the H\"older inequality with $1=\frac{1}{\frac{3(1+ \delta)}{2- \delta}} + \frac{1}{3} +  \frac{1}{\frac{1+ \delta}{\delta}}$ and the Sobolev embedding $W^{1, 1+ \delta} (\O)\subset L^{\frac{3(1+ \delta)}{2- \delta}}(\O)$ when $\O \subset \R^3$,
		\Be\label{gronwall_f-g}
		\begin{split}
			&\int^t_0 \iint_{\O \times \R^3} |\nabla_x \phi_{f^{\ell}-f^{\ell-1}} \cdot \nabla_v f^{\ell-1}|  |f^{\ell+1}-f^{\ell}|^{\delta}  \\
			& \lesssim 
			\int^t_0 \| \nabla_x \phi_{f^{\ell}-f^{\ell-1}}\|_{L^{ \frac{3(1+ \delta)}{2- \delta}}_x} \| \nabla_v f^{\ell-1} \|_{L^{3}_xL^{1+ \delta}_v} \left\| |f^{\ell+1}-f^{\ell}|^\delta\right\|_{L_{x,v}^{\frac{1+ \delta}{\delta}}}\\
			& \lesssim \sup_{0 \leq s \leq t} \| \nabla_v f^{\ell-1} (s) \|_{L^{3}_xL^{1+ \delta}_v}\times \int^t_0 \| [f^{\ell+1}-f^\ell](s) \|_{1+ \delta}^{1+ \delta} \dd s.
		\end{split} \Ee

		A simple modification of (\ref{k_p_bound}) and (\ref{phi_p_bound}) as 
		\Be
		\begin{split}\notag
			&\int^t_0 \int_x \int_v \int_u \mathbf{k}_\varrho(v,u) |f^{\ell}(u)-f^{\ell-1}(u)| |f^{\ell+1}(v)-f^{\ell}(v)|^\delta\\
			\lesssim & \ \int^t_0 \int_x \int_v \int_u\mathbf{k}_\varrho(v,u)^{\frac{1}{1+\delta}} |f^{\ell}(u)-f^{\ell-1}(u)|  \mathbf{k}_\varrho(v,u)^{\frac{\delta}{1+\delta}}|f^{\ell+1}(v)-f^{\ell}(v)|^\delta\\
			\lesssim & \ \int^t_0 \int_x \int_{v } ( |f^{\ell}(v)- f^{\ell-1}(v)|^{1+ \delta} + |f^{\ell+1}(v)-f^{\ell}(v)|^{1 +\delta} )\int_u\mathbf{k}_\varrho (v,u)\\
			\lesssim & \ \int^t_0   \|f^{\ell} - f^{\ell-1} \|^{1+ \delta}_{1+\delta} + \ \int^t_0   \|f^{\ell+1} - f^{\ell} \|^{1+ \delta}_{1+\delta}  ,
		\end{split}
		\Ee
		leads to 
		\Be\label{gronwall_f-g_last}\begin{split}
			&\int^t_0 \iint_{\O \times \R^3} | \text{the }  2^{\text{nd} } \text{and }  3^{\text{rd} } \text{ line of RHS of } (\ref{eqtn_f-g})|  |f^{\ell+1}-f^\ell|^{\delta}  \\
			\lesssim& \sup_{0 \leq s \leq t}\big\{ 1+ \| w_{\vartheta}f^{\ell+1}(s) \|_\infty + \|w_{\vartheta} f^\ell(s) \|_\infty
			\big\}  \left(\int^t_0   \|f^{\ell} - f^{\ell-1} \|^{1+ \delta}_{1+\delta} + \ \int^t_0   \|f^{\ell+1} - f^{\ell} \|^{1+ \delta}_{1+\delta} \right).
		\end{split}\Ee

		Then following the proof of (\ref{bound_p_est}) and applying (\ref{gronwall_f-g}) to (\ref{gronwall_f-g_last}), we can obtain
		\Be\label{gamma_+,1+delta}
		\begin{split}
			&\int_0^t |[f^{\ell+1}-f^\ell]|_{1+ \delta, - }^{1 + \delta} \\
			&\lesssim   o(1) \int_0^t |[f^{\ell}-f^{\ell-1}]|_{1+ \delta, + }^{1 + \delta}+\| [f^{\ell}- f^{\ell-1}](0)\|_{1+ \delta}^{1+ \delta} \\
			&  +
			\sup_{0 \leq s \leq t}\big\{ 1+ \| \nabla_v f^{\ell-2} (s) \|_{L^{3}_xL^{1+ \delta}_v} + \| w_{\vartheta} f^{\ell}(s) \|_\infty + \|w_{\vartheta} f^{\ell-1}(s) \|_\infty
			\big\}  \left( \int^t_0 \| f^{\ell}-f^{\ell-1} \|_{1+ \delta}^{1+ \delta} + \int^t_0 \| f^{\ell-1}-f^{\ell-2} \|_{1+ \delta}^{1+ \delta}\right).
		\end{split}
		\Ee
		Using (\ref{uniform_h_ell}), (\ref{bound_nabla_v_g_global}), (\ref{f-g_energy}), (\ref{gronwall_f-g}), (\ref{gronwall_f-g_last}), (\ref{gamma_+,1+delta}) and $[f^{\ell+1} -f^\ell ]|_{t = 0 } = 0 $ we get
		\Be \begin{split} \label{indll1stability}
		\sup_{0 \le s \le t } & \| f^{\ell+1} (s) - f^\ell (s) \|_{1 +\delta}^{1 + \delta} + \int_0^t  | f^{\ell+1} -f^\ell |_{1+\delta, +}^{1+\delta} 
		\\ \le & \left[O(t) + o(1) \right]\left( \sup_{0 \le s \le t } \| f^\ell - f^{\ell - 1 } \|_{1 + \delta}^{1 + \delta} + \int_0^t  | f^{\ell} -f^{\ell-1} |_{1+\delta, +}^{1+\delta}  + \sup_{0 \le s \le t } \| f^{\ell-1} - f^{\ell - 2} \|_{1 + \delta}^{1 + \delta}
		 \right).
		\end{split} \Ee
Thus adding (\ref{indll1stability}) with the same estimate  $\eqref{indll1stability}|_{f^{\ell +2 } -f^{\ell +1}}$ we get
\Be \notag \begin{split}
& \sup_{0 \le s \le t } \| f^{\ell+1} (s) - f^\ell (s) \|_{1 +\delta}^{1 + \delta} + \int_0^t  | f^{\ell+1} -f^\ell |_{1+\delta, +}^{1+\delta}  + \sup_{0 \le s \le t }  \| f^{\ell+2} (s) - f^{\ell+1} (s) \|_{1 +\delta}^{1 + \delta} + \int_0^t  | f^{\ell+2} -f^{\ell+1} |_{1+\delta, +}^{1+\delta} 
		\\ \le & \left[O(t) + o(1) \right]\left( \sup_{0 \le s \le t } \| f^\ell - f^{\ell - 1 } \|_{1 + \delta}^{1 + \delta} + \int_0^t  | f^{\ell} -f^{\ell-1} |_{1+\delta, +}^{1+\delta}  + \sup_{0 \le s \le t } \| f^{\ell-1} - f^{\ell - 2} \|_{1 + \delta}^{1 + \delta}
		 \right).
\end{split} \Ee
Therefore, inductively we have
\Be \notag
\sup_{0 \le s \le t } \| f^{\ell+1} (s) - f^\ell (s) \|_{1 +\delta}^{1 + \delta} + \int_0^t  | f^{\ell+1} -f^\ell |_{1+\delta, +}^{1+\delta}  \le \left[O(t) + o(1) \right]^m.
\Ee
Hence we derive stability
\Be \label{l1+stabfinal}
\sup_{0 \le s \le t } \| f^{\ell} (s) - f^m (s) \|_{1 +\delta}^{1 + \delta} \le \left[O(t) + o(1) \right]^{\min\{m,\ell\}},
\Ee
and this concludes \eqref{1+delta_stability}.		
		
		 
		
\textit{Step 2.}		We combine (\ref{uniform_h_ell}) and (\ref{1+delta_stability}) to get unique weak-$\ast$ convergence (up to subsequence if necessary),
		 $(w_{\vartheta} f^\ell, w_{\vartheta} f^{\ell+1}) \overset{\ast}{\rightharpoonup} (w_{\vartheta} f, w_{\vartheta} f)$ weakly$-*$ in $L^\infty (\R \times \O \times \R^3; \R^2) \cap L^\infty (\R \times \gamma; \R^2)$. For $\varphi  = \begin{bmatrix} \varphi_+ \\ \varphi_- \end{bmatrix} \in C^\infty_c (\R \times \bar{\O} \times \R^3;\R^2)$, 
		\Be\begin{split}\label{weak_form_ell}
			&\int_0^T \langle f^{\ell+1} ,  [-\p_t  - v\cdot \nabla_x  + \nu   ] \varphi \rangle
			+\underbrace{ \langle q f^{\ell+1} , \nabla_x \phi^\ell \cdot \nabla_v\varphi
				+ \frac{v}{2} \cdot \nabla_x \phi^\ell \varphi \rangle }_{(\ref{weak_form_ell})_\phi}
			\\
			=& \int_0^T\langle K f^\ell , \varphi \rangle - \langle q_1 v\cdot \nabla_x \phi^\ell \sqrt{\mu} ,  \varphi  \rangle
			+   \underbrace{ \langle \Gamma_{\text{gain}} ( {f^\ell} , {f^\ell} ),  \varphi \rangle}_{(\ref{weak_form_ell})_\text{gain}} -   \underbrace{ \langle \Gamma_{\text{loss}} ( {f^{\ell+1 }} , {f^{\ell}}  ), \varphi  \rangle}_{(\ref{weak_form_ell})_{\text{loss}}}\\
			&+ \int^T_0 \langle f^{\ell+1} ,\varphi \rangle_{\gamma_+} - \int^T_0 \langle  
			c_\mu \sqrt{\mu} \int_{n \cdot u>0} f^\ell \sqrt{\mu}\{n \cdot u\} \dd u , 
			\varphi \rangle_{\gamma_- }.
		\end{split}\Ee
Except the underbraced terms in (\ref{weak_form_ell}) all terms converges to limits with $f$ instead of $f^{\ell+1}$ or $f^\ell$.

		We define, for $(t,x,v) \in \mathbb{R} \times  \bar{\Omega} \times \mathbb{R}^{3}$ and for $0 < \delta \ll 1$, 
		\begin{equation}\label{Z_dyn}
		\begin{split}
		f_{\delta}^\ell(t,x,v) 
		:=  &  \ \kappa_\delta (x,v) f^\ell(t,x,v)
		\\
		: =  &  \  \chi\Big(\frac{|n(x) \cdot v|}{\delta}-1\Big) 
		\Big[ 1- \chi(\delta|v|) \Big]  \chi\Big(\frac{|v|}{\delta}-1\Big)  f^\ell(t,x,v )
		.
		\end{split}
		\end{equation}
		Note that $f_\delta (t,x,v)=0$ if either $|n(x) \cdot v| \leq \delta$, $|v| \geq \frac{1}{\delta}$, or $|v| \leq \delta$.

		%
		
		From (\ref{uniform_h_ell})
		\Be\notag
		\begin{split}
			&\left|\int^T_0  (\ref{weak_form_ell})_\text{loss}
			- \int^T_0 \langle \Gamma_{\text{loss}} (f,f) , \varphi \rangle \right|
			\\
			\leq &\left| \int^T_0 \langle
			\int_{\R^3}   |v-u| \{f_+^\ell ( u) - f_+(u) + f_-^\ell (u ) - f_-(u)\}\sqrt{\mu(u)}  \dd u f^{\ell+1} (v),  \varphi(t,x,v
			) \rangle \dd t\right|\\
			&+\left|\int^T_0  \langle
			\int_{\R^3}   |v-u|  (f_+(u) + f_-(u) )\sqrt{\mu(u)}  \dd u \{ f^{\ell+1} (v)- f(v) \} , \varphi(t,x,v
			)\rangle \dd t\right|.
		\end{split}\Ee
		The second term converges to zero from the weak$-*$ convergence in $L^\infty$ and (\ref{uniform_h_ell}). The first term is bounded by, from (\ref{uniform_h_ell}),
		\Be \label{difference_f^ell-f}
		\begin{split}
			&\left[\int^T_0 \left\|  \int_{\R^3} \kappa_\delta (x,u) (f^\ell(t,x,u) - f(t,x,u)) \langle u\rangle \sqrt{\mu(u)} \dd u \right\|^2_{L^2(\O \times \R^3)}\right]^{1/2}\\
			&\times \sup_{0 \leq t \leq T} \|w_{\vartheta} f^{\ell+1}(t)\|_\infty 
			+O(\delta).
		\end{split} \Ee
		
		On the other hand, from Lemma \ref{extension_dyn}, we have an extension $\bar{f}^\ell(t,x,v)$ of $\kappa_\delta (x,u)  f^\ell(t,x,u)$. We apply the average lemma (see Theorem 7.2.1 in page 187 of \cite{gl}, for example) to $\bar{f}^\ell(t,x,v)$. From (\ref{fell_local}) and (\ref{uniform_h_ell})
		
		\Be\label{H_1/4}
		\sup_\ell\left\| \int_{\R^3}   \bar{f}^\ell(t,x,u)  
		\langle u\rangle \sqrt{\mu(u)} \dd u
		\right\|_{H^{1/4}_{t,x} (\R \times \R^3)}< \infty.
		\Ee\hide
		\Be
		\sup_\ell\left\| \int_{\R^3}   f^\ell(t,x,u)  
		\langle u\rangle \sqrt{\mu(u)} \dd u
		\right\|_{H^{1/4}_{t,x} (\R \times \R^3)}< \infty.
		\Ee\unhide
		Then by $H^{1/4} \subset\subset L^2$, up to subsequence, we conclude that 
		\[
		\int_{\R^3} \kappa_\delta (x,u)  f^\ell(t,x,u)  \langle u\rangle \sqrt{\mu(u)} \dd u 
		\rightarrow \int_{\R^3} \kappa_\delta (x,u)  f(t,x,u) \langle u\rangle \sqrt{\mu(u)} \dd u \  \ \text{strongly in } L^2_{t,x}.
		\]
		So we conclude that $(\ref{difference_f^ell-f})\rightarrow 0$ as $\ell \rightarrow \infty$.
		
		For $(\ref{weak_form_ell})_{\text{gain}}$ let us use a test function $\varphi_1(v) \varphi_2 (t,x)$. From the density argument, it suffices to prove a limit by testing with $\varphi(t,x,v)$.
		
		We use a standard change of variables $(v,u) \mapsto(v^\prime,u^\prime)$ and $(v,u) \mapsto(u^\prime,v^\prime)$ (for example see page 10 of \cite{gl}) to get    
		\begin{eqnarray} 
		&&\int^T_0   (\ref{weak_form_ell})_{\text{gain}} - \int^T_0 \langle \Gamma_{\text{gain}}(f,f), \varphi \rangle \nonumber\\
		& =& \  \int^T_0 \langle \Gamma_{\text{gain}}(f^\ell ,f^\ell - f), \varphi \rangle + \int^T_0 \langle \Gamma_{\text{gain}}( f^\ell-f ,f) ,\varphi \rangle \nonumber \\
		&= &  \sum_{\iota = \pm}  \int^T_0 \iint_{\O \times \R^3} \left( \int_{\R^3}  \int_{\S^2} 
		(   f_+^\ell (t,x,u) -  f_+(t,x,u) + f_-^\ell (t,x,u) -  f_-(t,x,u)  )   \sqrt{\mu(u^\prime)}     |(v-u) \cdot \o|  \varphi_{1,\iota}(v^\prime)            \dd \o\dd u\right)\nonumber\\
		&& \ \ \ \ \ \ \ \ \ \ \ \ \ \   \times 
		f_\iota^\ell (t,x,v)   
		\varphi_{2,\iota} (t,x)
		\dd v \dd x\dd t  \label{weak_fell-f1}\\
		&  +& \sum_{\iota = \pm}  \int^T_0 \iint_{\O \times \R^3} \left( \int_{\R^3}  \int_{\S^2} 
		(   f_\iota^\ell (t,x,u) -  f_\iota(t,x,u)  )   \sqrt{\mu(v^\prime)}     |(v-u) \cdot \o|  \varphi_{1,\iota}(u^\prime)            \dd \o\dd u\right)\nonumber\\
		&& \ \ \ \ \ \ \ \ \ \ \ \ \ \   \times 
		(f_+ (t,x,v) +f_-(t,x,v))   
		\varphi_{2,\iota} (t,x)
		\dd v \dd x\dd t  . \label{weak_fell-f2}
		\end{eqnarray} 
		For $N\gg 1$ we decompose the integration of (\ref{weak_fell-f1}) and (\ref{weak_fell-f2}) using 
		\Be
		\begin{split}\label{decomposition_N}
			1=& \{1- \chi (|u|-N)\}\{1- \chi (|v|-N)\}\\
			& +  \chi (|u|-N) + \chi (|v|-N) - \chi (|u|-N)   \chi (|v|-N).
		\end{split} \Ee
		Note that $\{1- \chi (|u|-N)\}\{1- \chi (|v|-N)\} \neq 0$ if $|v| \leq N+1$ and $|u| \leq N+1$, and if $ \chi (|u|-N) + \chi (|v|-N) - \chi (|u|-N)   \chi (|v|-N) \neq 0$ then either $|v|\geq N$ or $|u| \geq N$. From (\ref{uniform_h_ell}), the second part of (\ref{weak_fell-f1}) and (\ref{weak_fell-f2}) from (\ref{decomposition_N}) are bounded by 
		\Be
		\begin{split}
			&\int^T_0 \iint_{\O \times \R^3} \int_{\R^3} \int_{\S^2} [\cdots]  \times \{\chi (|u|-N) + \chi (|v|-N) - \chi (|u|-N)   \chi (|v|-N)\}\\
			\leq &  \ \sup_\ell \| w_{\vartheta}f^\ell \|_\infty \| w_{\vartheta} f\|_\infty \times 
			\{
			e^{-\frac{\vartheta}{2} |v|^2} e^{-\frac{\vartheta}{2} |u|^2} 
			\}\{ \mathbf{1}_{|v|\geq N} +  \mathbf{1}_{|u|\geq N}  \}\\
			\leq& \  O(\frac{1}{N}). \notag
		\end{split}
		\Ee
		
		Now we only need to consider the parts with $ \{1- \chi (|u|-N)\}\{1- \chi (|v|-N)\}$. Then 
		\Be\begin{split}\label{bound_weak_fell-f1}
			&(\ref{weak_fell-f1})\\
			= &  \sum_{\iota = \pm}  \int^T_0 \iint_{\O \times \R^3}    \int_{\R^3}  
			(   f_+^\ell (t,x,u) -  f_+(t,x,u) + f_-^\ell (t,x,u) -  f_-(t,x,u)  ) \\
			& \ \ \ \ \ \ \ \ \ \ \ \ \ \   \ \  \times  \{1- \chi (|u|-N)\}\left(\int_{\S^2}  \sqrt{\mu(u^\prime)}     |(v-u) \cdot \o|  \varphi_{1,\iota}(v^\prime)            \dd \o \right)\dd u \\
			& \ \ \ \ \ \ \ \ \ \ \ \ \ \   \ \  \times 
			\{1- \chi (|v|-N)\}  f_\iota^\ell (t,x,v)   
			\varphi_{2,\iota} (t,x)
			\dd v \dd x\dd t.
		\end{split}
		\Ee
		Now, let us define 
		\Be\label{Phi_v}
		\Phi _{v,\iota}(u ) :=  \{1- \chi (|u|-N)\}\int_{\S^2}  \sqrt{\mu(u^\prime)}     |(v-u) \cdot \o|  \varphi_{1,\iota}(v^\prime)            \dd \o \ \ \text{for} \ |v| \leq N+1.
		\Ee          
		
		For $0<\delta\ll1$ we have $O(\frac{N^3}{\delta^3})$ number of $v_i \in \R^3$ such that $\{v \in \R^3: |v| \leq N+1\}\subset\bigcup_{i=1}^{O(\frac{N^3}{\delta^3})} B(v_i, \delta)$. Since (\ref{Phi_v}) is smooth in $u$ and $v$ and compactly supported, for $0<\e\ll1$ we can always choose $\delta>0$ such that
		\Be\label{Phi_v_continuous}
		|\Phi_{v,\iota}(u) - \Phi_{v_i,\iota} (u)|< \e  \ \ \text{if} \ v \in B(v_i, \delta).
		\Ee
		
		Now we replace $\Phi_{v,\iota}(u)$ in the second line of (\ref{bound_weak_fell-f1}) by $\Phi_{v_i,\iota}(u)$ whenever $v \in B(v_i, \delta)$. Moreover we use $\kappa_\delta$-cut off in (\ref{Z_dyn}). If $v$ is included in several balls then we choose the smallest $i$.  From (\ref{Phi_v_continuous}) and (\ref{uniform_h_ell}) the difference of (\ref{bound_weak_fell-f1}) and the one with $\Phi_{v_i}(u)$ can be controlled and we conclude that
		\Be
		\begin{split}\label{bound_weak_fell-f1_ell}
			(\ref{bound_weak_fell-f1}) & =\{O( \e)+ O(\delta)\} \sup_{\ell} \| w_{\vartheta} f^\ell \|_\infty  ^2\\
			&+ \sum_{\iota= \pm} \int^T_0 \int_{\O  } \sum_{i} \int_{\R^3} \mathbf{1}_{v \in B(v_i,\delta) } \int_{\R^3}  
			\kappa_\delta(x,u)  (   f_+^\ell (t,x,u) -  f_+(t,x,u) +f_-^\ell (t,x,u) -  f_-(t,x,u)  ) \Phi_{v_i,\iota} (u)  \dd u \\
			& \ \ \ \ \ \ \ \ \ \ \ \ \ \ \ \ \ \ \ \ \  \ \ \ \ \ \ \ \ \ \  \times 
			\{1- \chi (|v|-N)\}  f_\iota^\ell (t,x,v)   
			\varphi_{2,\iota} (t,x)
			\dd v \dd x\dd t.
		\end{split}\Ee
		From Lemma \ref{extension_dyn} and the average lemma
		\Be\label{H_1/4_gain}
		\max_{ 1\leq i \leq O(\frac{N^3}{\delta^3})}\sup_\ell\left\| \int_{\R^3} \kappa_\delta (x,u) f^\ell(t,x,u)  
		\Phi_{v_i,\iota} (u) \dd u
		\right\|_{H^{1/4}_{t,x} (\R \times \R^3)}< \infty.
		\Ee
		For $i=1$ we extract a subsequence $\ell_1 \subset\mathcal{I}_1$ such that 
		\Be\label{strong_converge_extract}
		\int_{\R^3} \kappa_\delta (x,u) f^{\ell_1}(t,x,u)  
		\Phi_{v_i,\iota} (u) \dd u \rightarrow \int_{\R^3} \kappa_\delta (x,u) f (t,x,u)  
		\Phi_{v_i,\iota} (u) \dd u \  \text{ strongly in }  \ L^2_{t,x}.
		\Ee
		Successively we extract subsequences $\mathcal{I}_{O(\frac{N^3}{\delta^3})} \subset \cdots\subset \mathcal{I}_2 \subset \mathcal{I}_1$. Now we use the last subsequence $\ell \in\mathcal{I}_{O(\frac{N^3}{\delta^3})}$ and redefine $f^\ell$ with it. Clearly we have (\ref{strong_converge_extract}) for all $i$. Finally we bound the last term of (\ref{bound_weak_fell-f1_ell}) by
		\Be\notag
		\begin{split}
			&C_{\varphi_2,N} \max_i \int^T_0 \sum_{\iota = \pm } \left\|
			\int_{\R^3} \kappa_\delta (x,u) (f^\ell(t,x,u)- f(t,x,u)  )
			\Phi_{v_i,\iota} (u) \dd u
			\right\|_{L^2_{t,x}} \sup_\ell \| w_{\vartheta} f^\ell \|_\infty\\
			&\rightarrow 0 \ \ \text{as} \ \ \ell \rightarrow \infty.
		\end{split}
		\Ee
		Together with (\ref{bound_weak_fell-f1_ell}) we prove $(\ref{weak_fell-f1})\rightarrow 0$. Similarly we can prove $(\ref{weak_fell-f2})      \rightarrow 0$.  
		
		Now we consider $(\ref{weak_form_ell})_{\phi}$. From 
		\Be\begin{split}\notag
			- ( \Delta \phi^\ell-  \Delta \phi) 
			= \int  \kappa_\delta (f_+^\ell- f_+ + f_-^\ell -f_-)\sqrt{\mu} +   \int (1-\kappa_\delta) (f_+^\ell- f_+ + f_-^\ell -f_-)\sqrt{\mu} ,
		\end{split} \Ee
		we have
		\Be\begin{split}
			\| \nabla_x \phi^\ell - \nabla_x \phi \|_{L^2_{t,x}} 
			\leq  \left\| \int  \kappa_\delta (f^\ell- f)\sqrt{\mu}\right\|_{L^2_{t,x}}
			+ O(\delta) \sup_\ell \| w_{\vartheta} f ^\ell \|_\infty.
		\end{split}
		\Ee
		Then following the previous argument, we prove $  \nabla_x \phi^\ell \rightarrow  \nabla_x \phi$ strongly in $L^2_{t,x}$ as $\ell \rightarrow \infty$. Combining with $w_{\vartheta} f^\ell \overset{\ast}{\rightharpoonup} w_{\vartheta} f$ in $L^\infty$, we prove $\int^T_0 (\ref{weak_form_ell})_{\phi}$ converges to $\int^T_0  \langle q f , \{\nabla_x \phi \cdot \nabla_v\varphi
		+ \frac{v}{2} \cdot \nabla_x \phi \varphi\} \rangle$. This proves the existence of a (weak) solution $f \in L^\infty$. 
		
		\vspace{4pt}

		\textit{Step 7. }  We claim (\ref{W1p_local_bound}).
	\hide	\Be\begin{split}  \notag
			&\sup_{0 \leq t \leq T^{**}}  \|w_{\tilde{\vartheta}} f  (t) \|_p^p
			+
			\sup_{0 \leq t \leq T^{**}}  \| w_{\tilde{\vartheta}}\alpha_{f ,\e}^\beta \nabla_{x,v} f (t) \|_{p}^p
			+ 
			\int^{T^{**}}_0  | w_{\tilde{\vartheta}}\alpha_{f ,\e}^\beta \nabla_{x,v} f (t) |_{p,+}^p \\
			&\lesssim  \   \| w_{\tilde{\vartheta}}f _0 \|_p^p
			+
			\| w_{\tilde{\vartheta}}\alpha_{f_{0 },\e}^\beta \nabla_{x,v} f  _0 \|_{p}^p
			.
		\end{split}\Ee
		This implies .\unhide
		By the weak lower-semicontinuity of $L^p$ we know that (if necessary we further extract a subsequence out of  the subsequence of \textit{Step 6})
		\[
			w_{\tilde{\vartheta}}\alpha^\beta_{f^\ell, \e} \nabla_{x,v} f^{\ell+1} \rightharpoonup \mathcal{F},\quad 
			\sup_{0 \leq t \leq T^{**}}\|\mathcal{F}(t)\|_p^p \leq \liminf   \sup_{0 \leq t \leq T^{**}}  \| w_{\tilde{\vartheta}}\alpha_{f^{\ell} ,\e}^\beta \nabla_{x,v} f^{\ell+1} (t) \|_{p}^p,
		\]
		and 
		\[ 
		\int^{T^{**}}_0  | \mathcal{F}  |_{p,+}^p \leq \liminf   
		\int^{T^{**}}_0  | w_{\tilde{\vartheta}}\alpha_{f^{\ell} ,\e}^\beta \nabla_{x,v} f^{\ell+1} (t) |_{p,+}^p.
		\]
		We need to prove that 
		\Be\label{a_nabla_f_seq}
			\mathcal{F}
			=w_{\tilde{\vartheta}}\alpha^{\beta}_{f,\e} \nabla_{x,v}f  \ \ \text{almost everywhere except } \gamma_0.
		\Ee

		We claim that, up to some subsequence, for any given smooth test function $\psi \in C_c^\infty(\bar{\O} \times \R^3\backslash \gamma_0; \mathbb R^2)$ 
\Be\label{lim_Fell=F}
\lim_{\ell  \rightarrow \infty}\int^t_0\iint_{\O \times \R^3} w_{\tilde{\vartheta}} \alpha^\beta_{f^{\ell },\e} \nabla_{x,v} f^{\ell +1} \psi \dd x\dd v
= \int^t_0\iint_{\O \times \R^3} w_{\tilde{\vartheta}} \alpha^\beta_{f ,\e} \nabla_{x,v} f  \psi \dd x\dd v.
\Ee 
We note that we need to extract a single subsequence, let say $\{\ell_*\} \subset \{\ell\}$, satisfying (\ref{lim_Fell=F}) for all test functions in $C_c^\infty(\bar{\O} \times \R^3\backslash \gamma_0;\mathbb R^2)$. Of course the convergent rate needs not to be uniform and it could vary with test functions. 

For each $N\in\mathbb{N}$ we define a set
\Be\label{set_N}
\mathcal{S}_N:= \Big\{ (x,v) \in \bar{\O} \times \R^3: \text{dist}(x, \p\O) \leq \frac{1}{N}  \ \text{and} \ |n(x) \cdot v| \leq \frac{1}{N}  \Big\}
		 \cup \{|v|>N\}.
\Ee
For a given test function we can always find $N\gg1$ such that 
		\Be\label{test_function_condition}
		supp(\psi) \subset (\mathcal{S}_N)^c:=\bar{\O} \times \R^3 \backslash \mathcal{S}_N.
		 \Ee 
		 
We will exam (\ref{lim_Fell=F}) by the identity obtained from the integration by parts
		\begin{eqnarray}
			&&\int^t_0 \langle w_{\tilde{\vartheta}} \alpha^\beta_{f^{\ell },\e}, \nabla_{x,v} f^{{\ell }+1} \psi \rangle
			\notag
			\\
			&= &  
			-  \int^t_0 \langle \alpha^\beta_{f^{\ell },\e}  f^{{\ell }+1}, \nabla_{x,v}(w_{\tilde{\vartheta}}\psi)\rangle
			\label{af_l_1}
			\\ &&  + \sum_{\iota = \pm }\int^t_0 \iint_{\gamma} n  \alpha^\beta_{f^{\ell },\e, \iota}   f_\iota^{{\ell }+1} (w_{\tilde{\vartheta}}\psi)
			\label{af_l_2}\\
			&&-  {\int^t_0 \langle \nabla_{x,v}\alpha^\beta_{f^{\ell },\e}  ,f^{{\ell }+1} (w_{\tilde{\vartheta}}\psi) \rangle}. \label{af_l_3}
		\end{eqnarray}
	We finish this step by proving the convergence of (\ref{af_l_1}) and (\ref{af_l_2}). From (\ref{alphaweight}) and (\ref{uniform_h_ell}), if $(x,v) \in (\mathcal{S}_N)^c$ then
		\Be
		\begin{split}\notag
		\sup_{\ell \geq 0}	| \alpha^\beta_{f^\ell,\e,\iota}(t,x,v)|  \lesssim |v|^\beta + (t+\e)^\beta\sup_{\ell \geq 0} \| \nabla \phi_{f^\ell} \|_\infty^\beta \lesssim N^\beta  +(T^{**}+\e)^\beta \sup_{\ell \geq 0} \| w_{\vartheta} f^\ell \|_\infty^\beta\leq C_N <+\infty  .
		\end{split}
		\Ee
	Hence we extract a subsequence (let say $\{\ell_N\}$) out of subsequence in \textit{Step 6} such that $
	 \alpha_{f^{\ell_N},\e, \iota}^\beta  \overset{\ast}{\rightharpoonup} A_\iota \in L^\infty  \textit{ weakly}-* \textit{ in }  L^\infty ((0,T^{**}) \times (\mathcal{S}_N)^c) \cap L^\infty ((0,T^{**})  \times (\gamma \cap  (\mathcal{S}_N)^c)).$ Note that $\alpha_{f^{\ell_N},\e,\iota}^\beta$ satisfies $[\p_t + v\cdot \nabla_x {-\iota} \nabla_x \phi^{\ell_N}\cdot \nabla_v]  \alpha_{f^{\ell_N},\e,\iota}^\beta=0$ and $\alpha_{f^{\ell_N},\e,\iota}^\beta|_{\gamma_-} =|n \cdot v |^\beta$. By passing a limit in the weak formulation we conclude that $[\p_t + v\cdot \nabla_x {-\iota} \nabla_x \phi_f \cdot \nabla_v]  A_\iota =0$ and $A_\iota |_{\gamma_-} =|n \cdot v |^\beta$. By the uniqueness of the Vlasov equation ($\nabla \phi_f \in W^{1,p}$ for any $p<\infty$) we derive $A_\iota=  \alpha_{f ,\e,\iota}^\beta$ almost everywhere and hence conclude that 
	\Be\label{converge_alpha_ell}
	 \alpha_{f^{\ell_N},\e,\iota}^\beta  \overset{\ast}{\rightharpoonup}   \alpha_{f ,\e,\iota}^\beta  \textit{ weakly}-* \textit{ in }  L^\infty ((0,T^{**})  \times (\mathcal{S}_N)^c) \cap L^\infty ((0,T^{**})  \times (\gamma \cap  (\mathcal{S}_N)^c)).
	\Ee	
Now the convergence of (\ref{af_l_1}) and (\ref{af_l_2}) is a direct consequence of strong convergence of (\ref{1+delta_stability}) and the weak$-*$ convergence of (\ref{converge_alpha_ell}).

		\vspace{4pt}

		\textit{Step 8. } We devote the entire \textit{Step 8} to prove the convergence of (\ref{af_l_3}).

		\textit{Step 8-a. } Let us choose $(x,v) \in (\mathcal{S}_N)^c$. From (\ref{weight}) 
	\Be\label{alpha=1}
If \  \ \tbpm^{f^\ell} \geq t+ \e \  \ then \  \ 	\alpha_{f^\ell, \e,\iota}(t,x,v)=1.
	\Ee
	 From now we only consider that case 
	\Be\label{tb_upperT^**}
	\tbpm^{f^\ell} (t,x,v) \leq \e +  t.
	\Ee
	
	If $|v|\geq 2 (\e + T^{**})\sup_{\ell} \| \nabla \phi^\ell \|_\infty$ then 
	\Be \label{lowerbddVfell}
	\begin{split}
	|V_\iota^{f^\ell}(s;t,x,v)| &\geq |v| - \int^t_s \| \nabla \phi^\ell(\tau) \|_\infty \dd \tau \\
	&\geq  (\e + T^{**})\sup_{\ell} \| \nabla \phi^\ell \|_\infty \ \ \ for \ all \  \ell  \ and \ s \in [-\e ,T^{**}].
	\end{split}
	\Ee
		For this case we need a version of velocity lemma of $\tilde \alpha$ in \eqref{alphatilde}, which shows up in the author's previous paper \cite{CK}, but this time with neutral boundary condition $\pm \nabla \phi^\ell \cdot n = 0$ on $\p \Omega$. So $\tilde \alpha$ now takes the form 	
	\Be\label{beta}
			\tilde{\alpha}(t,x,v) : = \sqrt{\xi(x)^2  + |\nabla \xi(x) \cdot u|^2 - 2 (u\cdot \nabla_x^2 \xi(x) \cdot u) \xi(x)
			}.
			\Ee
						From a direct computation, 
			\Be\begin{split}\label{velocity_lemma_N}
				& [\p_t + u\cdot \nabla_x -\iota \nabla \phi^\ell (t,x) \cdot \nabla_u ]\{ \xi(x)^2  + |\nabla \xi(x) \cdot u|^2 - 2 (u\cdot \nabla_x^2 \xi(x) \cdot u) \xi(x) \}\\
				&=2 \{u \cdot \nabla  \xi \} \xi  \cancel{+ 2 \{u\cdot \nabla ^2 \xi  \cdot u\} \{ u \cdot \nabla_x \xi  \}}-2 u \cdot (u \cdot \nabla  \nabla ^2 \xi  \cdot u) \xi \\
				& \ \ \cancel{ - 2 \{u\cdot \nabla ^2 \xi  \cdot u\}\{ u\cdot \nabla  \xi\}} +2 \{-\iota  \nabla \phi^\ell \cdot \nabla \xi  \} \{ \nabla  \xi  \cdot u\}-4 \{ -\iota \nabla \phi^\ell \cdot \nabla^2 \xi \cdot u \}\xi\\
				&\lesssim |u \cdot \nabla  \xi|^2 + |\xi|^2 + \{|u|+ \frac{1}{|u|}\} (- 2 (u\cdot \nabla_x^2 \xi(x) \cdot u) \xi(x))
				+ | \nabla \phi^\ell  \cdot \nabla \xi  
				|| \nabla  \xi  \cdot u|.
			\end{split}\Ee
			From the Neumann BC ($n(x) \cdot E(t,x) =0$ on $x \in \p\O$), we have 
			\Be\begin{split}\label{phi_xi}
				& |\nabla \phi^\ell (t,x) \cdot \nabla \xi(x)| \\
				&\leq |\nabla \phi^\ell (t,x_*) \cdot \nabla \xi(x_*)| + \|  \nabla \phi^\ell(t) \|_{C^1(\bar{\O})} \| \xi \|_{C^2 (\bar{\O})}
				|x-x_*|\\
				&\lesssim_\O \| \nabla \phi^\ell(t) \|_{C^1(\bar{\O})} 
				|\xi(x )|,\end{split}
			\Ee
			where $x_* \in\p\O$ such that $|x-x_*|= \inf_{y \in \p\O}|x-y|$.
			
			By controlling the last term of (\ref{velocity_lemma_N}) by (\ref{phi_xi}) and using (\ref{lowerbddVfell}), we conclude that 
			\Be\begin{split}\notag
				& \frac{d}{ds} \tilde{\alpha}(s,X_\iota^{f^\ell}(s;t,x,u),V_\iota^{f^\ell}(s;t,x,u))^2 \\
				&\lesssim_\O \Big(1+ |V_\iota^{f^\ell}(s;t,x,u)| + \frac{1}{|V_\iota^{f^\ell}(s;t,x,u)|}\Big)\tilde{\alpha}(s,X_\iota^{f^\ell}(s;t,x,u),V_\iota^{f^\ell}(s;t,x,u))^2\\
				&\lesssim_{\O,R, N} \Big(1+ |V_\iota^{f^\ell}(s;t,x,u)|\Big)\tilde{\alpha}(s,X_\iota^{f^\ell}(s;t,x,u),V_\iota^{f^\ell}(s;t,x,u))^2,
			\end{split}\Ee
so	
	\Be\begin{split}\notag
	 |\tilde{\alpha}(s,X_\iota^{f^\ell}(s;t,x,v),V_\iota^{f^\ell}(s;t,x,v))| 
	\geq & \   \frac{1}{C_\O}\tilde{\alpha} (t,x,v) e^{-C_\O \frac{|t-s|}{(\e + T^{**})\sup_{\ell} \| \nabla \phi^\ell  \|_\infty}}\\
	\geq & \ 
	\frac{e^{- \frac{C_\O }{\sup_{\ell} \| \nabla \phi^\ell  \|_\infty}}}{C_\O } \times \frac{1}{N}
	   \ \ \ for \ all \  \ell  \ and \ s \in [-\e ,T^{**}].
	\end{split}\Ee
Especially at $s= t-\tbpm^{f^\ell}(t,x,v)$, from (\ref{beta}),
\Be\label{alpha_|v|>}
|n(\xbpm^{f^\ell}) \cdot \vbpm^{f^\ell}|
 \geq \frac{e^{- \frac{C_\O }{\sup_{\ell} \| \nabla \phi^\ell  \|_\infty}}}{C_\O } \times \frac{1}{N} \ \ \ for \ all \ \ell. 
\Ee

	 	\vspace{4pt}
	 
	 \textit{Step 8-b. }  From now on we assume (\ref{tb_upperT^**}) and  
	\Be\label{upper_|V|}
	\begin{split}
	&|v|\leq 2 (\e + T^{**})\sup_{\ell} \| \nabla \phi^\ell \|_\infty,\\
	or& ,\ from \ (\ref{hamilton_ODE}), \ |V_\iota^{f^\ell}(s;t,x,v)|\leq 3 (\e + T^{**})\sup_{\ell} \| \nabla \phi^\ell  \|_\infty \ \    for \  \ s \in [-\e, T^{**}].
	\end{split}\Ee
	 Let $(X_{n ,\iota}^{f^\ell}, X_{\parallel ,\iota}^{f^\ell},V_{n ,\iota}^{f^\ell},V_{\parallel ,\iota}^{f^\ell})$ satisfy (\ref{dot_Xn_Vn}), (\ref{def_V_parallel}), and (\ref{hamilton_ODE_perp}) with $E=-\iota \nabla \phi^\ell$. 
	 
	 	 Let us define 
		\Be \label{tau_1}
		\tau_1  
		: = \sup  \big\{ \tau \geq  0: V_{n ,\iota}^{f^\ell} (s;t,x,v)\geq 0 \ for \ all \ s \in [t-\tbpm^{f^\ell}(t,x,v),  \tau ]
		\big\}
		. 
		\Ee
		Since $(X_\iota^{f^\ell} (s;t,x,v),V_\iota^{f^\ell} (s;t,x,v))$ is $C^1$ (note that $\nabla\phi^\ell \in C^1_{t,x}$) in $s$ we have $V_{n ,\iota}^{f^\ell} (\tau_1 ;t,x,v)=0$.

	 We claim that, there exists some constant $\delta_{**} = O_{\e, T^{**}, \sup_{\ell} \| \nabla \phi^\ell  \|_{C^1}}(\frac{1}{N})$ in (\ref{choice_delta**}) which does not depend on $\ell$ such that 
	 \Be\label{claim_Vperp_growth}
	 \begin{split}
	 If \ \ &0\leq V_{n ,\iota}^{f^\ell} (t-\tbpm^{f^\ell}(t,x,v);t,x,v) < \delta_{**} and \  (\ref{upper_|V|}),
	 \\
	&then \ \ V_{n ,\iota}^{f^\ell} (s;t,x,v) \leq e^{C|s-(t-\tbpm^{f^\ell}(t,x,v))|^2}V_{n ,\iota}^{f^\ell} (t-\tbpm^{f^\ell}(t,x,v);t,x,v) \ \ for  \ s \in [ t-\tbpm^{f^\ell},\tau_1].
	 \end{split}\Ee
For the proof we regard the equations (\ref{dot_Xn_Vn}), (\ref{def_V_parallel}), and (\ref{hamilton_ODE_perp}) as the forward-in-time problem with an initial datum at $s=t-\tbpm^{f^\ell} (t,x,v)$. 	Clearly we have $X_{n ,\iota}^{f^\ell}(t-\tbpm^{f^\ell}(t,x,v);t,x,v)=0$ and $V_{n ,\iota}^{f^\ell}(t-\tbpm^{f^\ell}(t,x,v);t,x,v)\geq0$ from Lemma \ref{cannot_graze}. Again from Lemma \ref{cannot_graze}, if $V_{n ,\iota}^{f^\ell}(t-\tbpm^{f^\ell}(t,x,v);t,x,v)=0$ then $X_{n ,\iota}^{f^\ell}(s;t,x,v)=0$ for all $s\geq t-\tbpm^{f^\ell} (t,x,v)$. From now on we assume $V_{n ,\iota}^{f^\ell} (t-\tbpm^{f^\ell}(t,x,v);t,x,v)] >0$. From (\ref{hamilton_ODE_perp}), as long as $t - \tbpm^{f^\ell} (t,x,v) \leq s \leq T^{**}$ and 
	\Be\label{small_s_V_n} 
	V_{n ,\iota}^{f^\ell}(s;t,x,v) \geq 0 \ \ and \ \ 
	X_{\perp,\iota}^{f^\ell} (s;t,x,v) \leq \frac{1}{N} \ll 1,
	\Ee  
	then we have
	\Be\label{hamilton_ODE_perp_bound}\begin{split}
	\dot{V}_n^{f^\ell} (s) 
		= & \  \underbrace{[V^{f^\ell}_{\parallel,\iota} (s)\cdot \nabla^2 \eta (X^{f^\ell}_{\parallel,\iota}(s)) \cdot V^{f^\ell}_{\parallel,\iota}(s) ] \cdot n(X^{f^\ell}_{\parallel,\iota}(s)) }_{\leq 0  \    from \ (\ref{convexity_eta})} \\
	- 	& \underbrace{\nabla \phi^\ell  (s , X^{f^\ell} (s ) ) \cdot [-n(X^{f^\ell}_{\parallel,\iota}(s)) ]}_{
		=O(1) \sup_\ell  \|  \nabla \phi^\ell \|_{C^1} \times X_{n ,\iota}^{f^\ell} (s)
		 \   from \  (\ref{expansion_E})} \\
		  - &\underbrace{X_{n ,\iota}^{f^\ell} (s) [V^{f^\ell}_{\parallel,\iota}(s) \cdot \nabla^2 n (X^{f^\ell}_{\parallel,\iota}(s)) \cdot V^{f^\ell}_{\parallel,\iota}(s)]  \cdot n(X^{f^\ell}_{\parallel,\iota}(s)) }_{
		=O(1)  \{3 (\e + T^{**})\sup_{\ell} \| \nabla \phi^\ell  \|_\infty \}^2 \times X_{n ,\iota}^{f^\ell} (s) \ from \ 
		(\ref{upper_|V|})
		}\\
		 \leq & \ 
		C (1+ \e  + T^{**})^2 ( \sup_{\ell} \| \nabla \phi^\ell  \|_{C^1}  \sup_{\ell} \| \nabla \phi^\ell  \|_{\infty} )
		  \times X_{n ,\iota}^{f^\ell} (s).
	\end{split}\Ee	
	
Let us consider (\ref{hamilton_ODE_perp_bound}) together with $\dot{X}^{f^\ell}_{n}(s;t,x,v) = V^{f^\ell}_{n}(s;t,x,v)$. Then, as long as $s$ satisfies (\ref{small_s_V_n}), 
\Be
\begin{split}\notag
V_{n ,\iota}^{f^\ell} (s)& =    V_{n ,\iota}^{f^\ell} (t-\tbpm^{f^\ell} ) + \int^s_{t-\tbpm^{f^\ell}} \dot{V}_n^{f^\ell}(\tau) \dd \tau \\
&\leq V_{n ,\iota}^{f^\ell} (t-\tbpm^{f^\ell} )+\int^s_{t-\tbpm^{f^\ell}} 
C (1+ \e  + T^{**})^2  ( \sup_{\ell} \| \nabla \phi^\ell  \|_{C^1}  \sup_{\ell} \| \nabla \phi^\ell  \|_{\infty} )
		  \times X_{n ,\iota}^{f^\ell} (\tau)
\dd \tau\\
&= V_{n ,\iota}^{f^\ell} (t-\tbpm^{f^\ell} )+\int^s_{t-\tbpm^{f^\ell}} 
C (1+ \e  + T^{**})^2  ( \sup_{\ell} \| \nabla \phi^\ell  \|_{C^1}  \sup_{\ell} \| \nabla \phi^\ell  \|_{\infty} )
	\int^\tau_{t-\tbpm^{f^\ell}} V_{n ,\iota}^{f^\ell}(\tau^\prime) \dd \tau^\prime
\dd \tau.
\end{split}
\Ee
Following the same argument of the proof of Lemma \ref{est_X_v}, we derive that 
	\Be
\begin{split}\notag
V_{n ,\iota}^{f^\ell} (s) \leq&  \  V_{n ,\iota}^{f^\ell} (t-\tbpm^{f^\ell} )\\
& +
C (1+ \e  + T^{**})^2  ( \sup_{\ell} \| \nabla \phi^\ell  \|_{C^1}  \sup_{\ell} \| \nabla \phi^\ell  \|_{\infty} )
\int^s_{t-\tbpm^{f^\ell}} 
|s -  (t-\tbpm^{f^\ell}) |  V_{n ,\iota}^{f^\ell}(\tau^\prime)  \dd \tau^\prime.
\end{split}
\Ee
From the Gronwall's inequality, we derive that, as long as (\ref{small_s_V_n}) holds, 	
\Be\label{upper_bound_Vn}
V_{n ,\iota}^{f^\ell} (s;t,x,v) \leq  V_{n ,\iota}^{f^\ell} (t-\tbpm^{f^\ell}(t,x,v)) e^{C (1+ \e  + T^{**})^2 ( \sup_{\ell} \| \nabla \phi^\ell  \|_{C^1}  \sup_{\ell} \| \nabla \phi^\ell  \|_{\infty} ) \times |s -  (t-\tbpm^{f^\ell}(t,x,v)) | ^2}.
\Ee	

Now we verify the conditions of (\ref{small_s_V_n}) for all $- \e \leq t - \tbpm^{f^\ell} (t,x,v) \leq s \leq T^{**}$. Note that we are only interested in the case of $V_{n ,\iota}^{f^\ell} (t-\tbpm^{f^\ell}(t,x,v);t,x,v)< \delta_{**}$. From the argument of (\ref{hamilton_ODE_perp_bound}), ignoring negative curvature term,
		 \Be
		 \begin{split}\notag
		| X_{n ,\iota}^{f^\ell} (s;t,x,v)| \leq& 
	\ 	(\e+ T^{**}) | V_{n ,\iota}^{f^\ell} (\tbpm^{f^\ell};t,x,v)  | \\
	&
		+   C[1 + (\e + T^{**})^2 \sup_\ell \| \nabla \phi^\ell \|_\infty]  \sup_\ell\| \nabla \phi^\ell \|_{C^1}
		 \int^s_{t-\tbpm^{f^\ell}} \int^\tau_{t-\tbpm^{f^\ell}} |X_{n ,\iota}^{f^\ell}(\tau;t,x,v)| \dd \tau  \dd s\\
		  \leq & \  (\e+ T^{**}) | V_{n ,\iota}^{f^\ell} (\tbpm^{f^\ell};t,x,v)  | +  C  \int_{t-\tbpm^{f^\ell}} ^s |\tau- (t-\tbpm^{f^\ell})| |X_{n ,\iota}^{f^\ell}(\tau;t,x,v)|  \dd \tau.
		 \end{split}
		 \Ee
		 Then by the Gronwall's inequality we derive that, in case of (\ref{tb_upperT^**}),
		 \Be\label{upper_X_n}
		 | X_{n ,\iota}^{f^\ell} (s;t,x,v)|\leq C_{\e+ T^{**}} | V_{n ,\iota}^{f^\ell} (t-\tbpm^{f^\ell};t,x,v)  | \ \ for  \ all \  -\e \leq t-\tbpm^{f^\ell} \leq s \leq t  \leq T^{**}.
		 \Ee 
If we choose 
\Be\label{choice_delta**}
\delta_{**} = \frac{o(1)}{ |T^{**} + \e |}\times \frac{1}{N}, 
\Ee
then (\ref{upper_bound_Vn}) holds for $- \e \leq t - \tbpm^{f^\ell} (t,x,v) \leq s \leq T^{**}$. Hence we complete the proof of (\ref{claim_Vperp_growth}).

	\hide

	\Be\notag
	\dot{X}_{n}(s;t,x,v) = V_{n}(s;t,x,v) , \ \  \dot{X}_{\parallel}(s;t,x,v) = V_{\parallel}(s;t,x,v).
	\Ee

	Then the proof of (\ref{no_graze}) asserts that  
		\Be\label{alpha_lower_bound}
		|V_{n ,\iota}^f(s;t,x,v)| \gtrsim 1 \ \ \text{for all } \ (t,x,v) \in [0,\frac{\e}{2}] \times  \text{supp} (\psi) .
		\Ee
		Now we consider $\alpha_{f^\ell, \e}(t,x,v)$.

		 On the other hand if $\tbpm^{f^\ell}(t,x,v)< 2\e$, since (\ref{uniform_h_ell}), from (\ref{hamilton_ODE_perp}),
		\Be
		\begin{split}\notag
			&\sup_{\ell \geq 0}|V^f_n(s;t,x,v) - V^{f^\ell }_n(s;t,x,v)|\\
			\lesssim & \ \sup_{\ell \geq 0}\max\{ \tbpm^{f}, \tbpm^{f^\ell}\} \times \left\{N^2  + \| \nabla \phi_f \|_\infty + \| \nabla \phi_{f^\ell} \|_\infty\right\}\\
			\lesssim & \ \e (1+ N^2).
		\end{split}\Ee
	Therefore, from (\ref{alpha_lower_bound})	 for small $\e>0$, we prove the lower bound 
		\Be
	\inf_{\ell }	|n(\xb^{f^\ell})  \cdot \vb^{f^\ell}| \gtrsim 1 \ \ \text{for all } \ (t,x,v) \in [0,\frac{\e}{2}] \times  \text{supp} (\psi) .
		\Ee
		
		\unhide

		\vspace{4pt}

		\textit{Step 8-c. } Suppose that (\ref{upper_|V|}) holds and $0 \leq V_{n ,\iota}^{f^\ell} (t-\tbpm^{f^\ell}(t,x,v);t,x,v) < \delta_{**}$ with $\delta_{**}$ of  (\ref{choice_delta**}). Recall the definition of $\tau_1$ in (\ref{tau_1}). Inductively we define $\tau_2  
		: = \sup  \big\{ \tau \geq  0: V_{n ,\iota}^{f^\ell} (s;t,x,v)\leq 0 \ for \ all \ s \in [\tau_1,  \tau ]
		\big\}$ and $\tau_3, \tau_4, \cdots$. Clearly such points can be countably many at most in an interval of $[t-\tbpm^{f^\ell},t]$. Suppose $ \lim_{k\rightarrow \infty}\tau_k=t   $. Then choose $k_0\gg1$ such that $|\tau_{k_0} -t| \ll  |V_{n ,\iota}^{f^\ell} (t-\tbpm^{f^\ell} ;t,x,v)|$. Then, for $s \in [\tau_{k_0}, t] $, from (\ref{hamilton_ODE_perp_bound}) and (\ref{upper_|V|}),
		\Be\begin{split}\label{upper_V_n_0}
		|V_{n ,\iota}^{f^\ell} (t;t,x,v)| \lesssim   {|V_{n ,\iota}^{f^\ell} (t-\tbpm^{f^\ell} ;t,x,v)|} .
		\end{split}\Ee 
		
		Now we assume that $\tau_{k_0} < t \leq \tau_{k_0+1}$. From the definition of $\tau_i$ in  (\ref{tau_1}) we split the case in two.
		
		\textit{\underline{Case 1:} Suppose $V_{n ,\iota}^{f^\ell} (s;t,x,v)>0$ for $s \in (\tau_{k_0}, t)$. }

		From (\ref{hamilton_ODE_perp_bound}) and (\ref{upper_X_n})
		\Be
		\begin{split}\label{upper_V_n_1}
		 V_{n ,\iota}^{f^\ell} (t;t,x,v)  \lesssim \int^{T^{**}}_{\tau_{k_0}} X_{n ,\iota}^{f^\ell}(s) \lesssim |V_{n ,\iota}^{f^\ell} (t-\tbpm^{f^\ell};t,x,v)|.
		\end{split}
		\Ee

		\textit{\underline{Case 2:} Suppose $V_{n ,\iota}^{f^\ell} (s;t,x,v)<0$ for $s \in (\tau_{k_0}, t)$. }
		
		Suppose 
		\Be\label{assumption_lowerbound_V_n}
		-V_{n ,\iota}^{f^\ell} (t;t,x,v)= |V_{n ,\iota}^{f^\ell} (t;t,x,v)| \geq |V_{n ,\iota}^{f^\ell} (t-\tbpm^{f^\ell};t,x,v)|^A \ \ for \ any \  0<A<\frac{1}{2}. 
		\Ee \hide	
In this step we claim that 
\Be\label{max_Xn}
X_{n ,\iota}^{f^\ell} (
t-\tbpm^{f^\ell}(t,x,v)+
\tau_*^\ell (t,x,v);t,x,v)\geq \frac{1}{2} \frac{| V_{n ,\iota}^{f^\ell} (t- \tbpm^{f^\ell} (t,x,v);t,x,v)|^2}{1+\{3 (\e + T^{**})\sup_{\ell} \| \nabla \phi^\ell  \|_\infty\}^2 }.
\Ee		
\unhide
From (\ref{hamilton_ODE_perp_bound}), now taking account of the curvature term this time, we derive that 	\Be
	\begin{split}\notag
	-V_{n ,\iota}^{f^\ell} (t;t,x,v)  \leq & \   \int_{\tau_{k_0}}^{t} (-1) [V^{f^\ell}_{\parallel,\iota} (s)\cdot \nabla^2 \eta (X^{f^\ell}_{\parallel,\iota}(s)) \cdot V^{f^\ell}_{\parallel,\iota}(s) ] \cdot n(X^{f^\ell}_{\parallel,\iota}(s))  \dd s\\
	& +C | V_{n ,\iota}^{f^\ell} (t- \tbpm^{f^\ell} (t,x,v);t,x,v)|,
	\end{split}
	\Ee		
	where we have used (\ref{upper_|V|}) and (\ref{upper_X_n}). From (\ref{assumption_lowerbound_V_n}) the above inequality implies that, for $| V_{n ,\iota}^{f^\ell} (t- \tbpm^{f^\ell} (t,x,v);t,x,v)|\ll 1$, 
	\Be\notag
	|V_{n ,\iota}^{f^\ell} (t-\tbpm^{f^\ell};t,x,v)|^A \lesssim \int_{\tau_{k_0}}^{t} (-1) [V^{f^\ell}_{\parallel,\iota} (s)\cdot \nabla^2 \eta (X^{f^\ell}_{\parallel,\iota}(s)) \cdot V^{f^\ell}_{\parallel,\iota}(s) ] \cdot n(X^{f^\ell}_{\parallel,\iota}(s))  \dd s. 
	\Ee
Note that $|\frac{d}{ds} V^{f^\ell}_{\parallel,\iota} (s)|$ and $|\frac{d}{ds}X^{f^\ell}_{\parallel,\iota} (s)|$ are all bound from $\nabla \phi^\ell \in C^1$, (\ref{upper_|V|}), and (\ref{upper_X_n}). Hence we have 
\Be\label{lower_bound_V_||^2}
	\frac{1}{2}|V_{n ,\iota}^{f^\ell} (t-\tbpm^{f^\ell};t,x,v)|^A \lesssim \int_{\tau_{k_0}}^{t- |V_{n ,\iota}^{f^\ell} (t-\tbpm^{f^\ell};t,x,v)|^A} (-1) [V^{f^\ell}_{\parallel,\iota} (s)\cdot \nabla^2 \eta (X^{f^\ell}_{\parallel,\iota}(s)) \cdot V^{f^\ell}_{\parallel,\iota}(s) ] \cdot n(X^{f^\ell}_{\parallel,\iota}(s))  \dd s. 
	\Ee
	On the other hand, if $t- |V_{n ,\iota}^{f^\ell} (t-\tbpm^{f^\ell};t,x,v)|^A\leq \tau_{k_0}$ then $|t-\tau_{k_0}| \leq |V_{n ,\iota}^{f^\ell} (t-\tbpm^{f^\ell};t,x,v)|^A$, which implies that, from (\ref{hamilton_ODE_perp_bound}), (\ref{upper_|V|}), and (\ref{upper_X_n}),
	\Be\label{upper_V_n_2}
	|V_{n ,\iota}^{f^\ell}(t;t,x,v)| \lesssim  |V_{n ,\iota}^{f^\ell} (t-\tbpm^{f^\ell};t,x,v)|^A  \ \ for \ any \  0<A<\frac{1}{2}. 
	\Ee
	
	Now we consider $X_{n ,\iota}^{f^\ell} (t;t,x,v)$. From (\ref{hamilton_ODE_perp_bound}) and $\dot{X}^{f^\ell}_{n}(s;t,x,v) = V^{f^\ell}_{n}(s;t,x,v)$ together with (\ref{upper_X_n}) and (\ref{upper_|V|})
	\Be\label{X_n<0}
	\begin{split}
	&X_{n ,\iota}^{f^\ell} (t;t,x,v)\\
	 \lesssim & \   | V_{n ,\iota}^{f^\ell} (t-\tbpm^{f^\ell};t,x,v)  | 
	 + \int_{\tau_{k_0}}^t \int_{\tau_{k_0}}^\tau\underbrace{[V^{f^\ell}_{\parallel,\iota} (s)\cdot \nabla^2 \eta (X^{f^\ell}_{\parallel,\iota}(s)) \cdot V^{f^\ell}_{\parallel,\iota}(s) ] \cdot n(X^{f^\ell}_{\parallel,\iota}(s)) }_{\leq 0}  \dd s  \dd \tau \\
	\lesssim  & \   | V_{n ,\iota}^{f^\ell} (t-\tbpm^{f^\ell};t,x,v)  | \\
	&
	 +  |V_{n ,\iota}^{f^\ell} (t-\tbpm^{f^\ell};t,x,v)|^A \int_{\tau_{k_0}}^{t- |V_{n ,\iota}^{f^\ell} (t-\tbpm^{f^\ell};t,x,v)|^A}[V^{f^\ell}_{\parallel,\iota} (s)\cdot \nabla^2 \eta (X^{f^\ell}_{\parallel,\iota}(s)) \cdot V^{f^\ell}_{\parallel,\iota}(s) ] \cdot n(X^{f^\ell}_{\parallel,\iota}(s))   \dd s   \\
	 \lesssim & \  | V_{n ,\iota}^{f^\ell} (t-\tbpm^{f^\ell};t,x,v)  | -  |V_{n ,\iota}^{f^\ell} (t-\tbpm^{f^\ell};t,x,v)|^{2A}
	 \ \ \ from \ (\ref{lower_bound_V_||^2})
	 \\
	 \lesssim & \ | V_{n ,\iota}^{f^\ell} (t-\tbpm^{f^\ell};t,x,v)  | -  |V_{n ,\iota}^{f^\ell} (t-\tbpm^{f^\ell};t,x,v)|^{1-}\\
	 < & \ 0,
	\end{split}
	\Ee
	for $ | V_{n ,\iota}^{f^\ell} (t-\tbpm^{f^\ell};t,x,v)  | \ll 1$. Clearly this cannot happen since $x \in \bar{\O}$ and $x_n\geq 0$. Therefore our assumption (\ref{assumption_lowerbound_V_n}) was wrong and we conclude (\ref{upper_V_n_2}).

		\vspace{4pt}

		\textit{Step 8-d. } From (\ref{claim_Vperp_growth}), (\ref{upper_V_n_0}), (\ref{upper_V_n_1}), and (\ref{upper_V_n_2}) in \textit{Step 8-a} and \textit{Step 8-b}, we conclude that the same estimate (\ref{upper_V_n_2}) for $|V_{n ,\iota}^{f^\ell} (t-\tbpm^{f^\ell};t,x,v)|\ll 1$ in the case of (\ref{tb_upperT^**}) and (\ref{upper_|V|}). Finally from (\ref{alpha=1}), (\ref{alpha_|v|>}), (\ref{claim_Vperp_growth}), and (\ref{upper_V_n_2}) we conclude that 
		\Be\label{lower_bound_V_n_final}
		|V_{n ,\iota}^{f^\ell} (t-\tbpm^{f^\ell} (t,x,v);t,x,v )|\geq \left(\frac{1}{N}\right)^{1/A} \ \ (t,x,v) \in [0,T^{**}] \times (\mathcal{S}_N)^c.
		\Ee

	\hide

	\Be\notag
	\dot{X}_{n}(s;t,x,v) = V_{n}(s;t,x,v) , \ \  \dot{X}_{\parallel}(s;t,x,v) = V_{\parallel}(s;t,x,v).
	\Ee

	Then the proof of (\ref{no_graze}) asserts that  
		\Be\label{alpha_lower_bound}
		|V_{n ,\iota}^f(s;t,x,v)| \gtrsim 1 \ \ \text{for all } \ (t,x,v) \in [0,\frac{\e}{2}] \times  \text{supp} (\psi) .
		\Ee
		Now we consider $\alpha_{f^\ell, \e}(t,x,v)$.

		 On the other hand if $\tb^{f^\ell}(t,x,v)< 2\e$, since (\ref{uniform_h_ell}), from (\ref{hamilton_ODE_perp}),
		\Be
		\begin{split}\notag
			&\sup_{\ell \geq 0}|V^f_n(s;t,x,v) - V^{f^\ell }_n(s;t,x,v)|\\
			\lesssim & \ \sup_{\ell \geq 0}\max\{ \tb^{f}, \tb^{f^\ell}\} \times \left\{N^2  + \| \nabla \phi_f \|_\infty + \| \nabla \phi_{f^\ell} \|_\infty\right\}\\
			\lesssim & \ \e (1+ N^2).
		\end{split}\Ee
	Therefore, from (\ref{alpha_lower_bound})	 for small $\e>0$, we prove the lower bound 
		\Be
	\inf_{\ell }	|n(\xb^{f^\ell})  \cdot \vb^{f^\ell}| \gtrsim 1 \ \ \text{for all } \ (t,x,v) \in [0,\frac{\e}{2}] \times  \text{supp} (\psi) .
		\Ee
		
		\unhide
 From (2.36), (2.37), (2.40), and (2.41) in Lemma 2.4 in \cite{KL1}, and combing with \eqref{upper_|V|}
we have
\Be \begin{split}
 \nabla_{x,v} \tbpm^{f^\ell} & \lesssim_{\Omega} \frac{1}{|v_n (\xbpm^{f^\ell} ) | } |\vbpm^{f^\ell}|  |\tbpm^{f^\ell } |^2 e^{ \| \nabla^2 \phi^\ell \|_\infty (\tbpm^{f^\ell})^2 /2 } \lesssim_{\Omega} |V_{n ,\iota}^{f^\ell} (t-\tbpm^{f^\ell} (t,x,v);t,x,v )|,
\\  \nabla_{x,v} v_n (\xbpm^{f^\ell} ) & \lesssim_\O \frac{1}{|v_n (\xbpm^{f^\ell} ) | } \left( |v|  |\tbpm^{f^\ell } |^2 e^{ \| \nabla^2 \phi^\ell \|_\infty (\tbpm^{f^\ell})^2 /2 } + |\vbpm^{f^\ell}| |\tbpm^{f^\ell } |^3 (1 + \tbpm^{f^\ell } ) e^{ \| \nabla^2 \phi^\ell \|_\infty (\tbpm^{f^\ell})^2 /2 } \right) 
\\ & \lesssim_\O  |V_{n ,\iota}^{f^\ell} (t-\tbpm^{f^\ell} (t,x,v);t,x,v )|.
\end{split}
\Ee

Therefore from above we have
\Be \label{deltaalphal}
|\nabla_{x,v} \alpha^\beta_{f^\ell,\e,\iota}(t,x,v)| \lesssim_\chi \beta |\alpha^\beta_{f^\ell,\e,\iota}(t,x,v)|^{\beta -1 } | \nabla_{x,v} \tbpm^{f^\ell} + \nabla_{x,v} v_n (\xbpm^{f^\ell } ) | \lesssim_\chi  |V_{n ,\iota}^{f^\ell} (t-\tbpm^{f^\ell};t,x,v)|^{\beta -1 }   |V_{n ,\iota}^{f^\ell} (t-\tbpm^{f^\ell};t,x,v)|^{-1}.
\Ee
Combing \eqref{lower_bound_V_n_final} and \eqref{deltaalphal} we achieve

		\Be\notag
	\sup_{ \substack{ \ell \in \mathbb{N}
, \ 
 (x,v) \in (\mathcal{S}_N)^c , \\  - \e \leq t - \tbpm^{f^\ell} (t,x,v)   \leq t \leq  T^{**}}}		|\nabla_{x,v} \alpha^\beta_{f^\ell,\e,\iota}(t,x,v)| \lesssim \frac{1}{|V_{n ,\iota}^{f^\ell} (t-\tbpm^{f^\ell};t,x,v)|^{2- \beta}}  \lesssim_{\e,N,T^{**}  } 1.
		\Ee
	Hence we extract another subsequence out of all previous steps for $\iota = +$ first, and then from that subsequence extract further another subsequence for $\iota = -$ (and redefine this as $\{\ell_N\}$) such that 
	\Be\label{converge_D_alpha_ell}
	\nabla_{x,v} \alpha_{f^{\ell_N},\e,\iota}^\beta  \overset{\ast}{\rightharpoonup}  	\nabla_{x,v} \alpha_{f ,\e,\iota}^\beta  \textit{ weakly}-* \textit{ in } L^\infty
			( (-\e, T^{**}) \times (\mathcal{S}_N)^c
			).
	\Ee	
Note that the limiting function is identified from (\ref{converge_alpha_ell}).  Clearly the convergence of (\ref{af_l_3}) is an easy consequence of strong convergence of (\ref{1+delta_stability}) and the weak$-*$ convergence of (\ref{converge_D_alpha_ell}). 
	
	\vspace{4pt}

		\textit{Step 8-c. } Inductively we extract $\{\ell_N\}\supseteq\{\ell_{N+1}\} \supseteq\{\ell_{N+2}\} \supseteq \cdots$ by following all the process in \textit{Step 7} to \textit{Step 8-b}. Then finally we define the subsequence, by the Cantor's diagonal argument,  
		\Be\label{ell_*}
		\ell_*= \ell_{ \ell} .
		\Ee
		Then clearly (\ref{lim_Fell=F}) holds with this subsequence for any test function $\psi$. For any $\psi \in C^\infty_c (\bar{\O} \times \R^3 \backslash \gamma_0; \mathbb R^2)$ there exists $N_{\psi} \in \mathbb{N}$ such that $supp (\psi) \subset (\mathcal{S}_{N_\psi})^c$. Then all the proofs work. 
		
		This implies (\ref{a_nabla_f_seq}) from (\ref{af_l_1}), (\ref{af_l_2}), (\ref{af_l_3}). Positivity $F= \mu+ \sqrt{\mu}f\geq 0$ comes from \textit{Step 1} and \textit{Step 6}.

		\vspace{4pt}
		
		\textit{Step 5. }  Choose $t>t^\prime\geq0$. Instead of expanding $h_\iota(t,x,v)$ at $t=0$ as (\ref{h_ell_local}), we expand at $t^\prime$. Then by the iteration we have (\ref{h_iteration}) replacing $0$ by $t^\prime$. Collecting all terms at time $t^\prime$, we have
		\Be\begin{split}\label{expansion_t_prime}
			\| h_\iota(t)\|_\infty 
			\leq &  \| h_\iota(t^\prime)\|_\infty \bigg\{
			\mathbf{1}_{t_{1,\iota}\leq t^\prime} e^{- \int^t_{t^\prime} \nu_\iota}  \\
			& \ \ \ \ \ \ \     
			+  \mathbf{1}_{t_{1,\iota} \geq  t^\prime} \frac{e^{- \int^t_{t_{1,\iota} }\nu_\iota}}{\widetilde{w}_\vartheta (V_\iota(t_{1,\iota}))}
			\int_{\prod_{j=1}^{k-1} \mathcal{V}_{j,\iota}} \sum_{l=1}^{k-1}\mathbf{1}_{\{t^{\ell-l}_{l+1,\iota}\leq
				t^\prime<t_{l,\iota}^{\ell - (l-1)}\}} \dd\Sigma _{l_\iota}(t^\prime) \bigg\}. 
		\end{split} \Ee
		Since (\ref{smeasure}) is a probability measure and $|e^{-\int^t_{t^\prime} \nu_\iota} - 1| \ll |t-t^\prime|$ for $|t-t^\prime| \ll 1$, 
		\Be 
		|(\ref{expansion_t_prime})-  \| h_\iota(t^\prime)\|_\infty| \leq  O(|t-t^\prime|)+ \int_{\prod_{j=1}^{k-1} \mathcal{V}_{j,\iota}} \mathbf{1}_{\{ t_\iota^{k} (t,x,v,u^{1}, \cdots , u^{k-1}) >0 \}} \dd \Sigma_{k-1,\iota}^{k-1} 
		.\notag
		\Ee
		Then by (\ref{h_iteration}) we have $
		\| h(t)\|_\infty - \| h (t^\prime) \|_\infty < \frac{1}{2^k}
		+O_k(|t-t^\prime|).$ For large $k$, choosing $|t-t^\prime|\ll1$, we can prove $
		\| h(t)\|_\infty - \| h (t^\prime) \|_\infty\ll 1$ as $|t-t^\prime| \ll 1$. 
		
		Now we can expand $h(t^\prime,x,v)$ at $t$ by (\ref{h_ell_local}). Following the same argument we have $\| h(t^\prime) \|_\infty - \| h(t ) \|_\infty \ll 1$ as $|t-t^\prime| \ll 1$. Hence $\| w_{\vartheta}f(t) \|_\infty$ is continuous in $t$.

		The continuity of $\| \nabla_v f (t) \|_{L^3_xL^{1+ \delta}_v}$ and $\| w_{\tilde{\vartheta}}\alpha_{f,\e }^\beta \nabla_{x,v} f (t) \|_{p} ^p
		+ 
		\int^t_0  |w_{\tilde{\vartheta}} \alpha_{f,\e }^\beta \nabla_{x,v} f (t) |_{p,+}^p$ is an easy consequence of (\ref{g_initial})-(\ref{g_phi}), and \eqref{pf_green}, \eqref{mid}, (\ref{final_est_G}) as well.
		
		\end{proof}

	\section{$L^2$ coercivity}

\begin{proposition} \label{l2coercivity}
Suppose $(f, \phi)$ solves (\ref{2fVPB}), (\ref{smallfphi}), and (\ref{diffusef}). Then 
		%
		%
		there is $0<\lambda_{2} \ll 1$ such that for $0 \leq s \leq t$, 
		\Be\begin{split}\label{completes_dyn}
			& \| e^{\lambda_{2} t}f(t)\|_2^2
			+ \| e^{\lambda_{2} t} \nabla \phi (t) \|_2^2\\
			&
			+  \int_s^t \| e^{\lambda_{2} \tau}  f (\tau)\|_\nu^2 
			+ \| e^{\lambda_{2} \tau} \nabla \phi_f(\tau)\|_2^2 
			\mathrm{d} \tau 
			+  \int_s^t | e^{\lambda_{2} \tau} f |^2_{2,+}    \\
			\lesssim & \ \|e^{\lambda_{2} s} f(s)\|_2^2 +   \|e^{\lambda_{2} s} \nabla \phi_f(s)\|_2^2  \\&   + \sup_{s \leq \tau \leq t} \| w_{\vartheta} f (\tau) \|_\infty \int_s^t \| e^{\lambda_{2} \tau}  f (\tau)\|_\nu^2  .  
		\end{split}\Ee
\end{proposition}
The null space of linear operator $L$ is a six-dimensional subspace of $L^2_v(\mathbb R^3; \mathbb R^2 )$ spanned by orthonormal vectors
\Be 
\left\{ \begin{bmatrix} \sqrt \mu \\ 0  \end{bmatrix}, \begin{bmatrix} 0  \\ \sqrt \mu \end{bmatrix}, \begin{bmatrix} \frac{v_i}{\sqrt 2 } \sqrt \mu \\ \frac{v_i}{\sqrt 2 } \sqrt \mu   \end{bmatrix}, \begin{bmatrix} \frac{|v|^2 - 3}{2\sqrt 2} \sqrt \mu \\ \frac{|v|^2 - 3}{2\sqrt 2} \sqrt \mu   \end{bmatrix}
 \right\}, \, i = 1,2,3,
\Ee
and the projection of $f$ onto the null space $N(L)$ is denoted by
\Be
\mathbf Pf(t,x,v) := \left\{ a_+(t,x) \begin{bmatrix} \sqrt \mu \\ 0  \end{bmatrix} + a_-(t,x) \begin{bmatrix} 0  \\ \sqrt \mu \end{bmatrix} + b(t,x)  \cdot \frac{v}{\sqrt 2 } \begin{bmatrix} \sqrt \mu \\ \sqrt \mu  \end{bmatrix} + c(t,x)  \frac{|v|^2 - 3}{2\sqrt 2}\begin{bmatrix} \sqrt \mu \\ \sqrt \mu  \end{bmatrix}
\right\}.
\Ee
In order to prove the proposition we need the following:
	
	\begin{lemma}
		\label{dabc}
		There exists
		a function $G(t)$ such that, for all $0\le s\leq t$, $G(s)\lesssim
		\|f(s)\|_{2}^{2}$ and 
		\begin{equation}\begin{split}\label{estimate_dabc}
		&\int_{s}^{t}\|\mathbf{P}f(\tau)\|_{\nu }^{2}  + \int^t_s \| \nabla \phi_f \|_2^2\\
		\lesssim & \ 
		G(t)-G(s)
		+  \int_{s}^{t}\|(\mathbf{I}-\mathbf{P}%
		)f(\tau)\|_{\nu }^{2} +\int_{s}^{t}|(1-P_{\gamma })f(\tau)|_{2,+}^{2} \\
		& +\int_{s}^{t}\| \nu^{-1/2} { {\Gamma(f,f)} } \|_{2}^{2} + \int_{s}^{t} \|w_{\vartheta}f(\tau)\|^{}_{\infty} \|\mathbf{P}f(\tau)\|_{2}^{2} .
		\end{split}\end{equation}
	\end{lemma}

\begin{proof}[Proof of Proposition \ref{l2coercivity}] 
		\textit{Step 1. }  Without loss of generality we prove the result with $s=0$. We have an $L^2$-estimate from $\int_0^t \langle 2e^{\lambda_{2} t} f, (\ref{systemf}) \rangle$
		\Be\notag
		\begin{split}
			&\|e^{\lambda_{2} t}  f(t) \|_2^2 - \| f(0) \|_2^2 +  \int^{t}_{0} | e^{\lambda_{2} \tau }
			(1-P_{\gamma} ) f^{j}|_{2,+}^{2}
			\\
			& + \int_0^t  \iint_{\O \times \R^3} v \cdot \nabla \phi_f e^{2\lambda_{2} \tau} (f_+^2 - f_-^2) +
			2\int^t_0 e^{\lambda_{2} \tau} \langle f, Lf \rangle \\
			= &\ 2\int^t_0  e^{2\lambda_{2} \tau} \langle f, \Gamma(f,f) \rangle - 2\int^t_0 e^{2\lambda_{2} \tau}\int_{\O  } \nabla \phi_f \cdot \int_{\R^3} v \sqrt{\mu} (f_+ - f_-)\\
			& 
			+2 \lambda_{2}  \int_0^t \|e^{\lambda_{2} \tau}  f(\tau) \|_2^2,
		\end{split}
		\Ee 
		where 
		\Be \label{Pgamma}
	P_\g f :=  \begin{bmatrix} P_\g f_+ \\ P_\g f_- \end{bmatrix} : = \begin{bmatrix} c_{\mu}\sqrt{\mu(v)} \int_{n(x)\cdot u>0} f_+(u) \sqrt{\mu(u)} \{n(x) \cdot u\} \mathrm{d}u \\
	c_{\mu}\sqrt{\mu(v)} \int_{n(x)\cdot u>0} f_-(u) \sqrt{\mu(u)} \{n(x) \cdot u\} \mathrm{d}u \end{bmatrix}.
		\Ee 
		On the other hand multiplying $\sqrt{\mu(v)} \phi_f(t,x)$ with a test function $\psi(t,x)$ to (\ref{systemf}) and applying the Green's identity, (from the charge conservation) we obtain
		\Be\begin{split}\notag
			&
			\int_{\O} \nabla \phi_f(t,x)  \cdot \int_{\R^3}  v \sqrt{\mu}( f_+ - f_-) \dd v\dd x 
			\\ =&  - \ \int_{\O} \phi_f(t,x)   \left(\int_{\R^3}   v \cdot  \nabla_x \sqrt{\mu} (f_+ - f_- )  \dd v\right)  \dd x
			+ \iint_{\p\O \times \R^3}  \phi_f(t,x) (f_+ - f_-) \sqrt{\mu} \{n \cdot v\} \dd v\dd S_x
			\\
			= & \ \int_{\O} \phi_f(t,x) \p_\tau\left(\int_{\R^3} (f_+  -  f_-)\sqrt{\mu} \dd v\right)  \dd x
			+ \iint_{\p\O \times \R^3}  \phi_f(t,x) (f_+ - f_- ) \sqrt{\mu} \{n \cdot v\} \dd v\dd S_x .
		\end{split}\Ee
		From (\ref{null_flux}), the last boundary contribution equals zero. Now we use (\ref{smallfphi}) and deduce that 
		\Be
		\begin{split}\notag
			& \int^t_0e^{2 \lambda_{2} \tau}\int_{\O} \phi_f(t,x) \p_\tau\left(\int_{\R^3} (f_+ - f_- )(\tau)  \sqrt{\mu} \dd v\right)  \dd x \dd \tau\\
			= & \ -  \int^t_0e^{2 \lambda_{2} \tau}\int_{\O} \phi_f(t,x) \p_\tau  \Delta_x  \phi_f(\tau,x) \dd x \dd \tau\\
			= &  \  \frac{1}{2}\int^t_0e^{2 \lambda_{2} \tau}\int_{\O}  \p_\tau  |\nabla_x  \phi_f(\tau,x)|^2 \dd x \dd \tau
			\\
			= & \ \frac{1}{2}   \left(\int_{\O} e^{2 \lambda_{2} t} |\nabla_x \phi_f (t,x)|^2  \dd x\right)- \frac{1}{2}   \left(\int_{\O} |\nabla_x \phi_f (0,x)|^2  \dd x\right)\\
			& \ 
			- \lambda_{2} \int^t_0 e^{2 \lambda_{2} \tau} \int_\O |\nabla_x  \phi_f(\tau,x)|^2 \dd x \dd \tau
			.
		\end{split}
		\Ee
		
		Hence we derive 
		\Be\notag
		\begin{split}
			&\| e^{\lambda_{2} t} f(t) \|_2^2 + \| e^{\lambda_{2} t}\nabla\phi_f (t) \|_2^2  + \int_0^t  \iint_{\O \times \R^3}
			e^{2\lambda_{2} \tau}
			v \cdot \nabla \phi_f (f_+^2- f_- ^2)\\
			& +
			2C\int^t_0 \iint_{\O \times \R^3}\| e^{\lambda_{2} \tau} (\mathbf{I} - \mathbf{P}) f \|_\nu^2
			+ \int^{t}_{0} | e^{\lambda_{2} \tau }
			(1-P_{\gamma} ) f^{j}|_{2,+}^{2}
			\\
			\lesssim &\ 
			\| f(0) \|_2^2+ \| \nabla \phi_f (0) \|_2^2 +
			\int^t_0 \| e^{\lambda_{2} \tau}  \nu^{-1/2} \Gamma(f,f) \|_2^2\\
			& 
			+\{ \lambda_{2} + o(1)\} \int^t_0 \| e^{\lambda_{2} \tau} f \|_\nu^2 + \lambda_{2} \int^t_0 \| e^{\lambda_{2} \tau} \nabla_x \phi_f \|_2^2
			.
		\end{split}
		\Ee
		Now we apply Lemma \ref{dabc} and add $o(1) \times (\ref{estimate_dabc})$ to the above inequality and choose $0< \lambda_{2} \ll 1$ to conclude (\ref{completes_dyn}) except the full boundary control. 
		
		\vspace{4pt}
		
		\textit{Step 2. } Note that from (\ref{Pgamma}), $P_\gamma f_\pm = z_\pm(t,x)\sqrt{\mu(v)}$ for a suitable functions $z_\pm(t,x)$ on the boundary. Then for $\iota = +$ or $-$, for $0 < \e \ll 1$
		\Be
		\begin{split}\notag
			| P_\gamma f_\iota |_{\gamma,2}^2 
			= & \ \int_{\p\O } |z_\iota(t,x)|^2 \dd x \times \int_{\R^3} \mu (v) |n(x) \cdot v| \dd v\\
			\lesssim & \  \int_{\p\O } |z_\iota(t,x)|^2 \dd x \times \int_{ \gamma_+(x) \backslash \gamma_+^\e(x)} \mu (v)^{3/2} |n(x) \cdot v| \dd v\\
			= & \  | \mathbf{1}_{ \gamma_+  \backslash \gamma_+^\e }\mu^{1/4} P_\gamma f_\iota   |_{2,+}^2.
		\end{split}
		\Ee
		Since $P_\gamma f = f - (1- P_\gamma) f$ on $\gamma_+$ we have $
		| \mathbf{1}_{ \gamma_+  \backslash \gamma_+^\e } \mu^{1/4} P_\gamma f|_{2, +}^2
		\lesssim |\mathbf{1}_{ \gamma_+  \backslash \gamma_+^\e }\mu^{1/4}  f |_{ 2,+}^2 + |(1- P_\gamma) f|_{ 2,+}^2.$ Therefore
		\Be\label{Pgamma_bound}
		\int^t_0 | P_\gamma f|_{\gamma,2}^2 \lesssim \int^t_0 |\mathbf{1}_{ \gamma_+  \backslash \gamma_+^\e }\mu^{1/4}  f |_{ 2,+}^2 +\int^t_0 |(1- P_\gamma) f|_{ 2,+}^2.
		\Ee
		Note that 
		\Be\begin{split}\notag
			&\big|[\p_t + v\cdot \nabla_x - q \nabla \phi \cdot \nabla_v] (\mu^{1/4}  f )\big|\\
			\lesssim&\ \mu^{1/4} \{ |v|| \nabla_x \phi|  f  +  |v|| \nabla_x \phi|  +  |Lf| + |\Gamma (f,f) | \}.
		\end{split}\Ee
		By the trace theorem Lemma \ref{le:ukai},  
		\Be\label{est_trans_f}
		\begin{split}
			& \int^t_0 |\mathbf{1}_{ \gamma_+  \backslash \gamma_+^\e }\mu^{1/4}  f |_{ 2,+}^2\\
			\lesssim & \ \| f_0 \|_2 + (1 + \| wf \|_\infty) \int^t_0 \| f \|_2^2 + \int^t_0\| \nabla \phi \|_2^2.
		\end{split}
		\Ee
		Adding $o(1) \times (\ref{Pgamma_bound})$ to the result of \text{Step 1} and using (\ref{est_trans_f}) we conclude (\ref{completes_dyn}).\end{proof}

\begin{proof} [Proof of Lemma \ref{dabc}]
From the Green's identity, a solution $f$ of \eqref{systemf} satisfies
		\begin{equation} \label{weakformulation}
		\begin{split}
		& \langle f(t) , \psi(t) \rangle -\langle f(s), \psi(s) \rangle
		- \underbrace{ \int_s^t \langle f  , \p_t\psi \rangle }_{(\ref{weakformulation})_{T}} +
		\underbrace{ \int^t_s\int_{\gamma } (\psi\cdot f) (v\cdot n(x)) }_{(\ref{weakformulation})_{B}}  \\
		& - \underbrace{ \int^t_s \langle \mathbf{P}f  ,{v}\cdot \nabla _{x}\psi \rangle }_{(\ref{weakformulation})_{C}} - \int^t_s \langle  (\mathbf{I-P})f , {v}\cdot \nabla _{x}\psi \rangle
		 +  \underbrace{  \int_s^t   \langle q \sqrt{\mu} f , \nabla_x \phi_f \cdot \nabla_v ( \frac{1}{\sqrt{\mu}} \psi )  \rangle}_{(\ref{weakformulation})_{P}}
		\\
		&=\int^t_s \langle \psi,  \{ -L  (\mathbf{I}-\mathbf{P%
		})f + \Gamma(f,f)\} \rangle-   \underbrace{ \int^t_s \langle \psi , q_1 v\cdot \nabla_x \phi_f \sqrt{\mu} \rangle} _{(\ref{weakformulation})_{\phi_f}}  .  
		\end{split}
		\end{equation}%
		
		We use a set of test functions:
		\Be \label{tests}
		\begin{split}
			\psi_{a}  &\equiv  \begin{bmatrix} - (|v|^{2}-\beta_{a} )\sqrt{\mu }v\cdot\nabla_x\varphi _{a_+} \\  - (|v|^{2}-\beta_{a} )\sqrt{\mu }v\cdot\nabla_x\varphi _{a_-} \end{bmatrix} ,  \\
			\psi^{i,j}_{b,1} &\equiv  \begin{bmatrix} (v_{i}^{2}-\beta_ b)\sqrt{\mu }\partial _{j}\varphi _{b}^{j} \\  (v_{i}^{2}-\beta_ b)\sqrt{\mu }\partial _{j}\varphi _{b}^{j} \end{bmatrix}, \quad i,j=1,2,3,   \\
			\psi^{i,j}_{b,2} &\equiv \begin{bmatrix} |v|^{2}v_{i}v_{j}\sqrt{\mu }\partial _{j}\varphi _{b}^{i}(x) \\ |v|^{2}v_{i}v_{j}\sqrt{\mu }\partial _{j}\varphi _{b}^{i}(x) \end{bmatrix},\quad i\neq j,  \\
			\psi_c &\equiv \begin{bmatrix} (|v|^{2}-\beta_c )\sqrt{\mu }v \cdot \nabla_x \varphi_{c}\\ (|v|^{2}-\beta_c )\sqrt{\mu }v \cdot \nabla_x \varphi_{c} \end{bmatrix},  \\
		\end{split}
		\Ee
		where $\varphi_{a_{\pm}}(t,x)$, $\varphi_{b}(t,x)$, and $\varphi_{c}(t,x)$ solve
		\Be\begin{split}\label{phi_abc}
			- \Delta \varphi_{a_\pm} &= a_\pm (t,x),  \quad \p_{n}\varphi_{a_\pm} \vert_{\p\O} = 0,  \\
			- \Delta \varphi_b^j &= b_j(t,x),   \ \ \varphi^j_b|_{\p\O} =0,\  \text{and}  \ - \Delta \varphi_c  = c(t,x),   \ \ \varphi_c|_{\p\O} =0,
		\end{split}
		\Ee
		and $\beta_a=10$, $\beta_b=1$, and $\beta_c= 5$ such that for all $i=1,\notag2,3,$
		\Be \label{defbeta}
		\begin{split}
			&\int_{{\R}^{3}}(|v|^{2}-\beta_a )(\frac{|v|^{2}-3}{2\sqrt 2}%
			) v_{i} ^{2}\mu(v)\dd v =0 ,\\
			&\int_{\R^3} (v_i^2 - \beta_b) \mu(v) \dd v=0 ,\\
			&\int_{\R^3} (|v|^{2}-\beta_c )v_{i}^{2}\mu(v) \dd v=0. 
		\end{split}
		\Ee

		\textit{Step 1. } Estimate of $(\ref{weakformulation})_{\phi_f}$: From (\ref{tests}) and (\ref{defbeta}), we have $(\ref{weakformulation})_{\phi_f}\equiv0$ for $\psi_{b,1}^{i,j}$, $\psi_{b,2}^{i,j}$, and $\psi_c$. For $\psi = \psi_{a}$, because from definition $\phi = \varphi_{a_+} - \varphi_{a_-}$, $(\ref{weakformulation})_{\phi_f}$ equals
		\Be \label{extra}
		\begin{split}
			(\ref{weakformulation})_{\phi_f} \big\vert_{\psi=\psi_{a}} &= \int_{\R^3} - (|v|^2- \beta_a) (v_1)^2\mu \dd v \int^t_s\int_{\O  } (\nabla \varphi_{a_+} -  \nabla \varphi_{a_-} ) \cdot \nabla \phi_f  = C_1 \int_s^t \| \nabla \phi_f \|_2^2,
		\end{split}\Ee
where 
\Be \label{intcst1}
C_1 = \int_{\R^3} - (|v|^2- \beta_a) (v_1)^2\mu \dd v = 5 . 
\Ee

Now we look at $(\ref{weakformulation})_{C}$. For $\psi = \psi_c$, from oddness in velocity integration and (\ref{defbeta}), $(\ref{weakformulation})_{C}$ becomes
		\Be \label{cC}
		\begin{split}
			\int^t_s \langle \mathbf{P}f, {v}\cdot \nabla
			_{x}\psi_{c} \rangle &= - C_{2} \int^t_s \|c(\tau)\|_{2}^{2},
		\end{split}	
		\Ee
		where $C_2 = 2 \int_{\mathbf{R}^3}(|v|^{2}-\beta_c )v_{i}^{2} (\frac{|v|^{2} -3}{2\sqrt2})\mu(v)dv= 20\pi^{3/2}$. 

For $\psi = \psi_a$, from oddness in velocity integration and (\ref{defbeta}), $(\ref{weakformulation})_{C}$ becomes
\Be \label{aC?}
		\begin{split}
			\int^t_s \langle \mathbf{P}f, {v}\cdot \nabla
			_{x}\psi_{a} \rangle &= - C_{1} \int^t_s \|a_+(\tau)\|_{2}^{2}+ \|a_-(\tau)\|_{2}^{2},
		\end{split}	
		\Ee
where $C_1 = 5$ as in \eqref{intcst1}.

For fixed $i,j$, we choose test function $\psi = \psi_{b,1}^{i,j}$ in (\ref{tests}) where $\b_{b}$ and $\varphi_{b}$ are defined in (\ref{defbeta}) and (\ref{phi_abc}). From oddness in velocity integration and definition of $\b_{b}$, $(\ref{weakformulation})_{C}$ in (\ref{weakformulation}) yields
		\Be \label{bC}
		\begin{split}
			(\ref{weakformulation})_{C}\vert_{\psi_{b,1}^{i,j}} &:= \int_s^t  \langle \mathbf{P}f,  v\cdot\nabla \psi_{b,1}^{i,j}  \rangle = - C_{3} \int_s^t \int_{\O} b_{i} (\p_{ij} \Delta^{-1} b_{j}), 
		\end{split}
		\Ee
		where $C_{3} :=  2 \int_{\R^{3}} (v_{i}^{2} - \b_{b})\frac{ v_{i}^{2}}{\sqrt 2 } \mu dv = 4\sqrt{\pi}$. 		
		
		Next, we try test function $\psi_{b,2}^{i,j}$ with $i \neq j$ to obtain 
		\Be \label{bC2}
		\begin{split}
			(\ref{weakformulation})_{C}\vert_{\psi_{b,2}^{i,j}} &:= \int_s^t \langle  \mathbf{P}f  , \psi_{b,2}^{i,j} \rangle = 2  \int_s^t \iint_{\O\times \R^{3}} (b\cdot \frac{v}{\sqrt 2} )\sqrt{\mu} v\cdot\nabla \psi_{b,2}^{i,j} = - C_{4} \int_s^t \int_{\O} \big(  b_{j} (\p_{ij} \Delta^{-1} b_{i}) + b_{i} (\p_{jj} \Delta^{-1} b_{i}) \big).
		\end{split}
		\Ee
		by oddness in velocity integral where $C_{4} := 14\sqrt{\pi}$. Note that the RHS of \eqref{bC} cancel out with the first term in the RHS of \eqref{bC2}, therefore combining them we get
\Be \notag
\left( \sum_{i,j}  - \frac{C_4}{C_3} \times (\ref{weakformulation})_{C}\vert_{\psi_{b,1}^{i,j}} \right) + \left( \sum_{i \neq j }(\ref{weakformulation})_{C}\vert_{\psi_{b,2}^{i,j}}  \right)    = - C_4 \sum_{i,j} \int_s^t \int_{\O} b_{i} (\p_{jj} \Delta^{-1} b_{i}) = - C_4 \int^t_s \|b(\tau)\|_{2}^{2}.
\Ee

		Estimate of $(\ref{weakformulation})_{P}$: From (\ref{tests}),
		\Be \label{extraP}
		\begin{split}
			(\ref{weakformulation})_{P} &= \int_s^t  \langle q \sqrt{\mu} f, \nabla_x \phi_f \cdot \nabla_v (\frac{1}{\sqrt{\mu}} \psi ) \rangle ,\quad  \psi = \psi_{a_\pm,b,c},  \\
			&\lesssim \int_{s}^{t} \|w_{\vartheta}f\|_{\infty} \int_{\O}  \nabla_{x}\phi_{f}\cdot\nabla_{x}\varphi_{a_\pm,b,c}  \lesssim \int_s^t \|w_{\vartheta}f(\tau)\|_{\infty} \|\mathbf{P}f(\tau)\|_{2}^{2}, 
		\end{split}
		\Ee
		by elliptic estimate $\|\nabla_{x}\varphi_{a_\pm,b,c}\|_{2} \lesssim \|\varphi_{a_\pm,b,c}\|_{H^{2}} \lesssim \|\mathbf{P}f\|_{2} $.  \\
		
		{\it Step 2}. {Estimate of} \ ${c}$ : 
For boundary integral $(\ref{weakformulation})_{B}$, we decompose $f_{\gamma } = P_{\gamma }f +\mathbf{1}_{\gamma_{+}}(1-P_{\gamma })f$. Then from (\ref{defbeta}) and trace theorem $|\nabla\varphi_{c}|_{2} \lesssim \|\varphi_{c}\|_{H^{2}} \lesssim \|c\|_{2}$,
		\Be \label{cB}
		\begin{split}
			&\int^t_s\int_{\gamma }\psi_{c} \cdot f (v\cdot n(x)) = \int^t_s\int_{\gamma }\psi_{c} \cdot \mathbf{1}_{\gamma_{+}} (1-P_{\gamma})f d\gamma  \\
			&\lesssim \varepsilon\int^t_s \|c(\tau)\|_{2}^{2} + C_{\varepsilon}\int^t_s |(1-P_{\gamma})f(\tau)|^{2}_{2,+} ,\quad \varepsilon \ll 1. \\
		\end{split}
		\Ee
		If we define
		\Be \label{Re}
		\begin{split}
			Re := 
			 \int^t_s \langle \psi, \{ L  (\mathbf{I}-\mathbf{P%
			})f - \Gamma(f,f)\} \rangle   
			 + \int^t_s \langle (\mathbf{I-P})f , {v}\cdot \nabla _{x}\psi \rangle,
		\end{split}
		\Ee
		then from \eqref{extra}, \eqref{extraP}, elliptic estimate and Young's inequality we have
		\Be \label{c Re}
		Re\vert_{\psi_{c}} \lesssim \varepsilon \int_{s}^{t} \|c\|_{2}^{2} + \int_{s}^{t} \|(\mathbf{I-P})f(\tau)\|_{\nu}^{2} + \int_{s}^{t} \|\nu^{-1/2}\Gamma(f,f)(\tau)\|_{2}^{2},
		\Ee
		We also use even/oddness in velocity integration, (\ref{defbeta}), and Young's inequality to estimate,
		\Be \label{cT}
		\begin{split}
			(\ref{weakformulation})_{T}\vert_{\psi_{c}} &= \int_s^t \langle f , \p_t\psi_{c} \rangle = \int_s^t\langle (\mathbf{I-P})f ,\p_t\psi_{c} \rangle \\
			&\lesssim \varepsilon \int_s^t \|\nabla\Delta^{-1}\p_{t}c(\tau)\|_{2}^{2} + \int_{s}^{t} \|(\mathbf{I-P})f(\tau)\|_{\nu}^{2}.
		\end{split}
		\Ee
		
		Now, we choose a new test function 
		$
		\psi_{c}^{t} :=  \begin{bmatrix} (\frac{|v|^2 -3}{2\sqrt 2})\sqrt{\mu } \p_{t}\varphi_{c} (t,x) \\ (\frac{|v|^2 -3}{2\sqrt 2})\sqrt{\mu } \p_{t}\varphi_{c} (t,x) \end{bmatrix}.
		$
		Note that $\p_{t}\varphi_{c}$ solves $-\Delta \p_{t}\varphi_{c} = \p_{t}c(t,x)$ with $\p_{t}\varphi_{c}(t,x)|_{\p\O} = 0$. We taking difference quotient for $\p_{t}f$ and it replace first three terms in the LHS of (\ref{weakformulation}). With help of Poincar\'e inequality $\|\p_{t}\varphi_{c}\|_{2} \lesssim \|\nabla\p_{t}\varphi_{c}\|_{2}$, we can also compute $(\ref{weakformulation})_{\phi_f} \big|_{\psi=\psi_{c}^{t}} = 0$, and
		\Be \label{ct_extra}
		\begin{split}
		(\ref{weakformulation})_{P} \big|_{\psi=\psi_{c}^{t}} &= \int_s^t \langle q \sqrt{\mu} f , \begin{bmatrix} \nabla_x \phi_f \cdot \frac{v}{\sqrt 2 } \p_{t}\varphi_{c} \\ \nabla_x \phi_f \cdot \frac{v}{\sqrt 2 } \p_{t}\varphi_{c} \end{bmatrix} \rangle    \\
			&\lesssim  \int_{s}^{t} \|w_{\vartheta}f\|_{\infty} \left( \varepsilon \|\nabla\Delta^{-1}\p_{t}c(\tau)\|_{2}^{2} +  \big(\|a_+(\tau)\|_{2}^{2} + \|a_-(\tau)\|_{2}^{2}\big) \right),	\\
		\end{split}
		\Ee
		\Be \label{ct_extra2}
		\int_s^t \langle \mathbf{P}f , v \cdot \nabla \psi_c^t \rangle + \int_s^t \langle (\mathbf{I-P})f , v \cdot \nabla \psi_c^t \rangle \lesssim  \varepsilon \int_s^t  \| \nabla \Delta^{-1} \p_t c (\tau) \|_2 ^2 + \int_s^t \| b(\tau) \|_2 ^2 + \int_s^t \| (\mathbf{I-P}) f \|_{\nu}^2.
		\Ee
		Since $\psi_{c}^{t}$ vanishes when it acts with $Lf$ and $\Gamma(f,f)$, and boundary integral $(\ref{weakformulation})_{B}$ vanishes by Dirichlet boundary condition of $\varphi_{c}$ , from (\ref{ct_extra}), \eqref{ct_extra2}, and (\ref{weakformulation}), we obtain
		\begin{equation} \label{c time}
		\begin{split}
		& \int_s^t \langle \p_t f , \psi_c^t \rangle =\int_s^t \int_{\Omega} \p_{t}\varphi_{c}(\tau,x)\partial_t c(\tau,x) dx = \int_s^t \|\nabla\Delta^{-1} \p_{t} c(\tau)\|_{2}^{2}    \\
		&\lesssim \varepsilon \int^t_s \|\nabla\Delta^{-1}\p_{t}c(\tau)\|_{2}^{2} +\int^t_s \big( \|a_+(\tau)\|_{2}^{2} + \|a_-(\tau)\|_{2}^{2} + \int_s^t \| b(\tau) \|_2 ^2 + \int_s^t \| (\mathbf{I-P}) f \|_{\nu}^2 \big).
		\end{split}
		\end{equation}
		
		We combine (\ref{weakformulation}), (\ref{extra}), (\ref{extraP}), (\ref{cB}), (\ref{c Re}), (\ref{cT}), and (\ref{c time}) with $\varepsilon \ll 1$ to obtain
		\Be \label{c est}
		\begin{split}
			\int^t_s \|c(\tau)\|_{2}^{2} &\lesssim G_{c}(t) - G_{c}(s) + \int_{s}^{t} \|(\mathbf{I-P})f(\tau)\|_{\nu}^{2} + \int^t_s |(1-P_{\gamma})f(\tau)|^{2}_{2,+} \\
			&\quad + \int_{s}^{t} \|\nu^{-1/2}\Gamma(f,f)(\tau)\|_{2}^{2} + \int_{s}^{t} \|w_{\vartheta}f(\tau)\|_{\infty} \|\mathbf{P}f(\tau)\|_{2}^{2}  \\
			&\quad + \varepsilon\int^t_s \big(\|a_+(\tau)\|_{2}^{2} + \|a_-(\tau)\|_{2}^{2} + \| b(\tau ) \|_2^2 \big),
		\end{split}
		\Ee
		for $\varepsilon \ll 1$ where $G_{c}(t) := \langle f(t), \psi_{c}(t) \rangle \lesssim \|f(t)\|_{2}^{2}$.  \\
		
		{\it Step 3.} {Estimate of} \ ${a}$ : From mass conservation $\int_{\O} a_\pm(t,x) dv = 0$, $\varphi_{a_\pm}$ in (\ref{phi_abc}) is well-defined. Moreover, we choose $\varphi_{a_\pm}$ so that has mean zero, $\int_{\O} \varphi_{a_\pm}(t,x) dx = 0$. Therefore, Poincar\'e inequality $\|\varphi_{a_\pm}\|_{2} \lesssim \|\nabla \varphi_{a_\pm}\|_{2}$ holds and these are also true for $\p_{t}\varphi_{a_\pm}$ which solves same elliptic equation with Neumann boundary condition.  
		
		For boundary integral $(\ref{weakformulation})_{B}$, we decompose $f_{\gamma } = P_{\gamma }f +\mathbf{1}_{\gamma_{+}}(1-P_{\gamma })f$. From Neumann boundary condition $\p_{n}\varphi_{a} = 0$ and oddness in velocity integral, $\int_{\gamma} \psi_{a} \cdot P_{\gamma}f (v\cdot n(s)) = 0$ and we obtain similar esimate as (\ref{cB}),
		\Be \label{aB}
		\begin{split}
			&\int^t_s\int_{\gamma }\psi_{a} \cdot f (v\cdot n(x)) = \int^t_s\int_{\gamma }\psi_{a} \cdot \mathbf{1}_{\gamma_{+}} (1-P_{\gamma})f d\gamma  \\
			&\lesssim \varepsilon\int^t_s \|a(\tau)\|_{2}^{2} + C_{\varepsilon}\int^t_s |(1-P_{\gamma})f(\tau)|^{2}_{2,+} ,\quad \varepsilon \ll 1. \\
		\end{split}
		\Ee
		For $(\ref{weakformulation})_{T}$, from oddness,
		\Be \label{aT}
		\begin{split}
			(\ref{weakformulation})_{T}\vert_{\psi_{a}} &= \int_s^t \langle f, \p_t\psi_{a} \rangle = \int_s^t \iint_{\O \times \R^3} (  (b\cdot v) \begin{bmatrix} \sqrt \mu \\ \sqrt \mu \end{bmatrix} + (\mathbf{I-P})f ) \cdot \p_t\psi_{a}  \\
			&\lesssim \varepsilon \int_s^t (\|\nabla\Delta^{-1}\p_{t}a_+(\tau)\|_{2}^{2} +  \|\nabla\Delta^{-1}\p_{t}a_-(\tau)\|_{2}^{2}) + \int_{s}^{t} \|b(\tau)\|_{2}^{2}+ \int_{s}^{t} \|(\mathbf{I-P})f(\tau)\|_{\nu}^{2}.
		\end{split}
		\Ee
		Now let us estimate $\int_{s}^{t} \|\nabla\Delta^{-1}\p_{t}a_+(\tau)\|_{2}^{2} +  \|\nabla\Delta^{-1}\p_{t}a_-(\tau)\|_{2}^{2}$ which appear in (\ref{cT}) type estimate. We use new test function $\psi_{a}^{t} = \begin{bmatrix} \p_t{\varphi_{a_+}}(x)\sqrt{\mu} \\ \p_t{\varphi_{a_-}}(x)\sqrt{\mu} \end{bmatrix} $. It's easy to check 
		\Be \label{at_extra}
		\begin{split}
			(\ref{weakformulation})_{\phi_f} \big|_{\psi=\psi_a^t} &= \int^t_s\iint_{\Omega \times  \R^{3}}  q_1\sqrt{\mu } v\cdot \nabla_x \phi_f \cdot  \psi_a^t =0, 	\\
			(\ref{weakformulation})_{P} \big|_{\psi=\psi_a^t} &= \int_s^t \iint_{\O \times \R^3}  q \sqrt{\mu} f  \cdot \nabla_x \phi_f \cdot \nabla_v \begin{bmatrix} \p_t \varphi_{a_+} \\ \p_t \varphi_{a_-}  \end{bmatrix} = 0,  \\
		\end{split}
		\Ee
		and from the null condition on boundary (\ref{null_flux}), we have $ (\ref{weakformulation})_{B} \vert_{\psi =\psi_a^t } = 0 $. Moreover,
		\Be \label{extra_at2}
		\int_s^t \langle \mathbf{P} f , v \cdot \nabla_x \psi_a^t \rangle + \int_s^t \langle \mathbf{I -P} f , v \cdot \nabla_x \psi_a^t \rangle \lesssim \varepsilon \int_s^t ( \| \nabla \Delta^{-1} \p_t a_+(\tau) \|_2^2 +\| \nabla \Delta^{-1} \p_t a_-(\tau) \|_2^2 )+ \int_s^t \| b(\tau) \|_2^2 + \int_s^t \| \mathbf{I- P} f(\tau) \|_\nu^2,
		\Ee
and from \eqref{collison_invariance}
\Be \label{extra_at3} \begin{split}
\int_s^t \langle \psi_a^t , \Gamma(f,f) \rangle &= 0.
\end{split} \Ee		
		Now taking difference quotient, we obtain from \eqref{at_extra}, \eqref{extra_at2}, and \eqref{extra_at2}, for almost $t$,
		\begin{equation} \label{a time} \begin{split}
		\int_s^t \langle \p_t f , \psi_a^t \rangle & =\int_s^t \int_{\Omega} \p_{t}\varphi_{a_+} \partial_t a_+ + \p_{t}\varphi_{a_-}\partial_t a_- dx \\ & = \int_s^t (\|\nabla\Delta^{-1} \p_{t} a_+(\tau)\|_{2}^{2} + \|\nabla\Delta^{-1} \p_{t} a_-(\tau)\|_{2}^{2} )
		  \lesssim \int_s^t \| b(\tau) \|_2^2 + \int_s^t \| \mathbf{I-P} f(\tau) \|_\nu^2 + \int_s^t \|\nu^{-1/2} \Gamma(f,f) \|_2^2.
		\end{split} \end{equation}
%
		Finally we change $c$ into $a$ in (\ref{c Re}) and combine with  (\ref{weakformulation}), (\ref{extra}), \eqref{aC?}, (\ref{extraP}), (\ref{aB}), (\ref{aT}), and (\ref{a time}) with $\varepsilon \ll 1$ to obtain
		\Be \label{a est}
		\begin{split}
			&\int^t_s \|a(\tau)\|_{2}^{2} + \int^t_s \|\nabla \phi_{f}(\tau)\|_{2}^{2} \\
			&\lesssim G_{a}(t) - G_{a}(s) + \int_{s}^{t} \|(\mathbf{I-P})f(\tau)\|_{\nu}^{2} + \int^t_s |(1-P_{\gamma})f(\tau)|^{2}_{2,+} \\
			&\quad + \int_{s}^{t} \|\nu^{-1/2}\Gamma(f,f)(\tau)\|_{2}^{2} + \int_{s}^{t} \|w_\vartheta f(\tau)\|^{}_{\infty} \|\mathbf{P}f(\tau)\|_{2}^{2}  + \int^t_s \|b(\tau)\|_{2}^{2} ,  \\
		\end{split}
		\Ee
		for $\varepsilon \ll 1$ where $G_{a}(t) := \iint_{\O\times \R^{3}} f(t)\psi_{a}(t) \lesssim \|f(t)\|_{2}^{2}$.  \\
		
		{\it Step 4.} {Estimate of} \ $ {b}$ : For fixed $i,j$, we choose test function $\psi = \psi_{b,1}^{i,j}$ in (\ref{tests}) where $\b_{b}$ and $\varphi_{b}$ are defined in (\ref{defbeta}) and (\ref{phi_abc}). For boundary integration, contribution of $P_{\gamma}f$ vanishes by oddness. 
		\Be \label{bB}
		\begin{split}
			(\ref{weakformulation})_{B}\vert_{\psi_{b,1}^{i,j}} &:= \int_s^t \int_{\gamma} \psi_{b,1}^{i,j} \cdot \mathbf{1}_{\gamma_{+}} (1-P_{\gamma})f (v\cdot n(x)) \lesssim \varepsilon \int_s^t \|b(\tau)\|^{2}_{2} + \int_s^t |(1-P_{\gamma})f|_{2,+}^{2},
		\end{split}
		\Ee
		and similar as (\ref{cT}) and (\ref{c Re}), we use oddness and definition of $\b_{b}$ to vanish contribution of $a$ and $b$. We obtain
		\Be \label{bT}
		\begin{split}
			(\ref{weakformulation})_{T}\vert_{\psi_{b,1}^{i,j}} &\lesssim  \varepsilon \int_s^t \|\nabla\Delta^{-1}\p_{t} b_{j}(\tau)\|_{2}^{2} + \int_{s}^{t} \|c(\tau)\|_{2}^{2} + \int_{s}^{t} \|(\mathbf{I-P})f(\tau)\|_{\nu}^{2},
		\end{split}
		\Ee
		\Be \label{b Re}
		 Re\vert_{\psi_{b,1}^{i,j}} \lesssim \varepsilon \int_{s}^{t} \|b(\tau)\|_{2}^{2} + \int_{s}^{t} \|(\mathbf{I-P})f(\tau)\|_{\nu}^{2} + \int_{s}^{t} \|\nu^{-1/2}\Gamma(f,f)(\tau)\|_{2}^{2}.
		\Ee
		
		Next, we try test function $\psi_{b,2}^{i,j}$ with $i \neq j$.
		We also have the following three estimates using oddness of velocity integral,
		\Be \label{bB2T2Re2}
		\begin{split}
			(\ref{weakformulation})_{B}\vert_{\psi_{b,2}^{i,j}} &:= \int_s^t \int_{\gamma} \psi_{b,1}^{i,j} \cdot \mathbf{1}_{\gamma_{+}} (1-P_{\gamma})f (v\cdot n(x)) \lesssim \varepsilon \int_s^t \|b(\tau)\|^{2}_{2} + \int_s^t |(1-P_{\gamma})f|_{2,+}^{2},  \\
			(\ref{weakformulation})_{T}\vert_{\psi_{b,2}^{i,j}} &\lesssim  \varepsilon \int_s^t \|\nabla\Delta^{-1}\p_{t} b_{i}(\tau)\|_{2}^{2} + \int_{s}^{t} \|(\mathbf{I-P})f(\tau)\|_{\nu}^{2},  \\
			 Re\vert_{\psi_{b,2}^{i,j}} &\lesssim \varepsilon \int_{s}^{t} \|b(\tau)\|_{2}^{2} + \int_{s}^{t} \|(\mathbf{I-P})f(\tau)\|_{\nu}^{2} + \int_{s}^{t} \|\nu^{-1/2}\Gamma(f,f)(\tau)\|_{2}^{2}.
		\end{split}
		\Ee
		
		To obtain estimate for $\|\nabla\Delta^{-1} \p_{t} b_{j}\|_{2}$, we use a test function $\psi_{b,j}^{t} := \begin{bmatrix} \frac{v_{j}}{\sqrt 2}\sqrt{\mu } \p_{t}\varphi^{j}_{b} (t,x) \\ \frac{v_{j}}{\sqrt 2} \sqrt{\mu } \p_{t}\varphi^{j}_{b} (t,x)  \end{bmatrix}$. Note that $\p_{t}\varphi_{b}^{j}$ solves $-\Delta \p_{t}\varphi_{b}^{j} = \p_{t}b_{j}(t,x)$ with $\p_{t}\varphi_{b}^{j}(t,x)|_{\p\O} = 0$. We taking difference quotient for $\p_{t}f$ in (\ref{smallfphi}) and with help of Poincar\'e inequality, we get
		\Be \label{bt_extra}
		\begin{split}
			(\ref{weakformulation})_{\phi_f} \big|_{\psi=\psi_{b,j}^{t}} &= 0,
\\
			(\ref{weakformulation})_{P} \big|_{\psi=\psi_{b,j}^{t}} &= \int_s^t \iint_{\O \times \R^3} \sqrt{\mu} f \cdot \begin{bmatrix} \p_{j} \phi_f \cdot \p_{t}\varphi_{b}^{j} \\  \p_{j} \phi_f \cdot \p_{t}\varphi_{b}^{j} \end{bmatrix}    \\
			&\lesssim  \int^t_s  \| w_\vartheta f \|_\infty (\varepsilon \|\nabla\Delta^{-1}\p_{t}b_{j}(\tau)\|_{2}^{2} +  \|a_+(\tau)\|_{2}^{2}+ a_-(\tau)\|_{2}^{2} ).	\\
		\end{split}
		\Ee
		Moreover,
		\Be \label{bt_extra2}
		\int_s^t \langle \mathbf{P}f , v \cdot \nabla \psi_{b,j}^t \rangle + \int_s^t \langle (\mathbf{I-P})f , v \cdot \nabla \psi_{b,j}^t \rangle \lesssim  \varepsilon \int_s^t  \| \nabla \Delta^{-1} \p_t b_j (\tau) \|_2 ^2 + \int_s^t \| (a_+(\tau) \|_2 ^2+\| a_-(\tau) \|_2 ^2 +  \| c(\tau) \|_2 ^2 ) +\int_s^t \| (\mathbf{I-P}) f \|_{\nu}^2.
		\Ee
		Since $\psi_{b,j}^{t}$ vanishes when it acts with $Lf$ and $\Gamma(f,f)$, and boundary integral $(\ref{weakformulation})_{B}$ vanishes by Dirichlet boundary condition of $ \p_t \varphi_{b}^j$ , from (\ref{bt_extra}), \eqref{bt_extra2}, and (\ref{weakformulation}), we obtain
		\begin{equation} \label{b time}
		\begin{split}
		&\int_s^t \int_{\Omega} \p_{t}\varphi_{b}^{j}(\tau,x)\partial_t b_{j}(\tau,x) dx = \int_s^t \|\nabla\Delta^{-1} \p_{t} b_{j}(\tau)\|_{2}^{2}    \\
		&\lesssim \varepsilon \int^t_s \|\nabla\Delta^{-1}\p_{t}b_{j}(\tau)\|_{2}^{2} +\int^t_s  (\|a_+(\tau)\|_{2}^{2}+\|a_-(\tau)\|_{2}^{2} +\|c(\tau)\|_{2}^{2} ) + \int_{s}^{t} \|(\mathbf{I-P})f(\tau)\|_{\nu}^{2} . \\	
		\end{split}
		\end{equation}
		
		Now we combine (\ref{weakformulation}), (\ref{extra}), (\ref{extraP}), (\ref{bC}), (\ref{bB}), (\ref{bT}), (\ref{b Re}), (\ref{bC2}), and (\ref{bB2T2Re2}) for all $i,j$ with proper constant weights. In particular, we note that RHS of (\ref{bC}) is cancelled by the first term on the RHS of (\ref{bC2}). Therefore,
		\Be \label{b est}
		\begin{split}
			&\int^t_s \|b(\tau)\|_{2}^{2}  = - \sum_{i,j} \int_s^t \int_{\O} b_{i} (\p_{jj} \Delta^{-1} b_{i})   \\
			&\lesssim G_{b}(t) - G_{b}(s) + \int_{s}^{t} \|(\mathbf{I-P})f(\tau)\|_{\nu}^{2} + \int^t_s |(1-P_{\gamma})f(\tau)|^{2}_{2,+} \\
			&\quad + \int_{s}^{t} \|\nu^{-1/2}\Gamma(f,f)(\tau)\|_{2}^{2} + \int_{s}^{t} \|w_{\vartheta}f(\tau)\|^{2}_{\infty} \|\mathbf{P}f(\tau)\|_{2}^{2}  \\
			&\quad + \int^t_s \|c(\tau)\|_{2}^{2} + \varepsilon\int^t_s (\|a_+(\tau)\|_{2}^{2} + \|a_-(\tau)\|_{2}^{2}) , \quad G_{b}(t) \lesssim \|f(t)\|_{2}^{2},\quad \varepsilon \ll 1. \\
		\end{split}
		\Ee
		
		Finally we combine (\ref{c est}), (\ref{a est}), and (\ref{b est}) with $\varepsilon \ll 1$ to conclude (\ref{estimate_dabc}).

\end{proof}

	\section{Global Existence and Exponential decay}
The following time-dependent interpolation estimate is crucial in the proof of Theorem \ref{main_existence}. 
	\begin{lemma}\label{lemma_interpolation}Assume $\O \subset \R^3$ with a smooth boundary $\p\O$. For $0< D_1<1$, $0< D_2<1$, and $\Lambda_0>0$,
		\Be\begin{split}\label{phi_interpolation}
			\|\nabla^2_x \phi(t )\|_{L^\infty (\O)}
			\lesssim_{\O, D_1, D_2}
			e^{D_1 \Lambda_0t}\|  \phi(t)\|_{C^{1,1-D_1}(\O)}
			+ e^{- D_2 \Lambda_0t}\|  \phi(t)\|_{C^{2, D_2}(\O)} \ \ for \ all \  t \geq 0
		\end{split}\Ee
	\end{lemma}

	\begin{proof}
	Let ${\O}_1$ be an open bounded subset of $\R^3$ containing the closure $\bar{\O}$. Suppose ${\phi (t)} \in C^{2,D_2}(\O)$. From a standard extension theorem (e.g. see Lemma 6.37 of \cite{GT} in page 136) there exists a function $\bar{\phi}(t) \in C^{2,D_2} (\O_1)$ and $\bar{\phi}(t)\equiv 0$ in $\R^3 \backslash \O_1$ such that $\phi(t) \equiv \bar{\phi}(t)$ in $\O$ and 
	\Be\label{Holder_extension}
	\| \bar{\phi}(t)\|_{C^{1,1-D_1} (\O_1)} \leq C_{\O,\O_1,D_1, D_2} \| \phi(t)\|_{C^{1,1-D_1} (\O)} \ \ and \ \ 
	\| \bar{\phi}(t)\|_{C^{2,D_2} (\O_1)} \leq C_{\O,\O_1,D_1, D_2} \| \phi(t)\|_{C^{2,D_2} (\O)},
	\Ee	
	where $C_{\O,\O_1,D_1, D_2}$ does not depend on $\phi(t)$ and $t$.

		Choose arbitrary points $x,y$ in $\R^3$. For $0\leq s\leq 1$, $(1-s)x + sy \in \overline{xy}$. Note that 
		\Be\begin{split}\notag
			& [(y-x) \cdot \nabla]\nabla \bar{\phi} (t,(1-s)x + sy)  \\
			=& \ \frac{ 
				[(y-x) \cdot \nabla]\nabla \bar{\phi} (t,(1-s)x + sy) -  [(y-x) \cdot \nabla]\nabla \bar{\phi} (t,x)
			}{|(1-s) x + s y - x|^{D_2}}|(1-s) x + s y - x|^{D_2}\\
			&+   \Big(\frac{y-x}{|y-x|} \cdot \nabla\Big)\nabla \bar{\phi} (t,x) |y-x|\\
			=& \ O( |x-y|^{ 1+ D_2} ) s^{D_2} [\nabla^2 \bar{\phi} (t)]_{C^{0, D_2}}
			+  \Big(\frac{y-x}{|y-x|} \cdot \nabla\Big)  \nabla \bar{\phi} (t,x) |y-x|.
		\end{split}\Ee
		Taking an integration on $s \in [0,1]$, we obtain that 
		\Be\begin{split}\label{C_2_alpha}
			&\left| \Big(\frac{y-x}{|y-x|} \cdot \nabla\Big)  \nabla \bar{\phi} (t,x)\right|\\
			\leq & \  \frac{1}{|y-x|}
			\left| \int_0^1
			[(y-x) \cdot \nabla]\nabla \bar{\phi} (t,(1-s)x + sy) \dd s \right|
			+\frac{1}{1+ D_2} |x-y|^{D_2} [\nabla^2 \bar{\phi}(t)]_{C^{0, D_2}}.
		\end{split}\Ee
		
		On the other hand, from an expansion along $s$,
		\[
		\nabla \bar{\phi}(t,y)- \nabla \bar{\phi}(t,x)  =  \int^1_0 [(y-x) \cdot \nabla]\nabla \bar{\phi} (t,(1-s)x + sy) \dd s.
		\]
		We plug this identity into (\ref{C_2_alpha}) and deduce that for $0<D_1<1$
		\Be\begin{split}\label{C2_interpolation}
			&\left|\left(\frac{x-y}{|x-y|} \cdot \nabla \right) \nabla \bar{\phi} (t,x)\right|\\ 
			\leq&  \ \frac{|\nabla \bar{\phi} (t,x) - \nabla \bar{\phi}(t,y)|}{|x-y|}
			+\frac{1}{1+ D_2}  |x-y|^{  {D_2}}    [\nabla^2 \bar{\phi}(t)]_{C_x^{0, {D_2}}}\\
			\leq& \  \frac{1}{|x-y|^{D_1}} [\nabla \bar{\phi} (t)]_{C^{0, 1-D_1}}+ \frac{1}{1 + D_2}  |x-y|^{  {D_2}}    [\nabla^2 \bar{\phi}(t)]_{C_x^{0, {D_2}}}.
		\end{split}\Ee
		Now let us choose 
		\[
		|x-y| = e^{- \Lambda_0 t}, \ \ \ \hat{\omega}:=\frac{x-y}{|x-y|} \in \mathbb{S}^2.
		\]  
		From (\ref{C2_interpolation})
		\Be\notag
		|\big(\hat{\omega} \cdot \nabla\big) \nabla \bar{\phi} (t,x)|
		\leq e^{D_1 \Lambda_0 t} [\nabla \bar{\phi} (t)]_{C^{0, 1- D_1}} + \frac{1}{1+ D_2}
		e^{- D_2 \Lambda_0 t}   [\nabla^2 \bar{\phi}(t)]_{C_x^{0, {D_2}}}.
		\Ee
		Taking supremum in $x$ and $\hat{\o}$ to the above inequality and using $
		\|\nabla^2_x \bar{\phi}(t )\|_{L^\infty_x } 
		=  \sup_x\sup_{\hat{\omega} \in \mathbb{S}^2}|\big(\hat{\omega} \cdot \nabla\big) \nabla \bar{\phi} (t,x)|$, we get
		\Be\begin{split}\notag
			\|\nabla^2_x \bar{\phi}(t )\|_{L^\infty (\O_1)}
			\leq
			e^{D_1 \Lambda_0t}[\nabla_x \bar{\phi}(t)]_{C^{0,1-D_1}(\O_1)}
			+ e^{- D_2 \Lambda_0t}[\nabla^2 \bar{\phi}(t)]_{C^{0, D_2}(\O_1)}.
		\end{split}\Ee
		Finally from (\ref{Holder_extension}) and the above estimate we conclude (\ref{phi_interpolation}).
	\end{proof}

	Now we are ready to prove the global-in-time result. 
	
	
	\begin{proof}[\textbf{Proof of Theorem \ref{main_existence}}]

	\textit{Step 1. }	For $0< M \ll 1$ and $0<\delta_*\ll 1$, we first assume that an initial datum satisfies
		\Be
		\begin{split}
		\label{initial_M}
		\| w_\vartheta f_0 \|_\infty  + \| w_{\tilde{\vartheta}}  f_0 \|_p + \|w_{\tilde{\vartheta}} \alpha_{f_0,\e}^\beta \nabla_{x,v} f_0 \|_p   \leq \delta_*  M,    \\
		    \| w_{\tilde{\vartheta}} \nabla_v f_0\|_{L^3(\O \times \R^3)}+ \| \nabla_x^2 \phi_f(0) \|_{\infty}< \infty.
		\end{split}\Ee
We will choose $M, \delta_*$ later. For the sake of convenience we choose a large constant $L\gg \max\left(M, \| \nabla_x^2 \phi_f (0) \|_{\infty}\right)$. In order to use the continuation argument along the lines of the local existence theorem, Theorem \ref{local_existence}, we set 
		\Be\label{global_T}
		\begin{split}
			T 
			=& \sup_{t } \Big\{ t \geq 0:     \| e^{\lambda_\infty s} w_\vartheta f(t) \|_\infty  + \| w_{\tilde{\vartheta}} f (t) \|_p    \leq  M ,
			\\
			& \ \ \ \ \ \ \ \   \ \  \ \ \ \ \ \ \    \text{and}  \  \|w_{\tilde{\vartheta}} \alpha_{f,\e}^\beta \nabla_{x,v} f (t) \|_p^p
			+ \int^t_0  | w_{\tilde{\vartheta}}\alpha_{f,\e}^\beta \nabla_{x,v} f(t) |_{p,+}^p  
			<\infty, \\
			& \ \ \ \ \ \ \ \   \ \  \ \ \ \ \ \ \   \text{and}  \ \| \nabla_v f(t) \|_{L^3_x(\O) L^{1+ \delta}_v (\R^3)}< \infty, \\
			& \ \ \ \ \ \ \ \   \ \  \ \ \ \ \ \ \   \text{and}  \  \| \nabla_x^2 \phi_f (t) \|_{\infty}\leq L
			\Big\}.
		\end{split}
		\Ee
		Here for fixed $\delta \ll 1$, we choose $\lambda_{\infty}$ such that 
		\Be\label{condition_lambda_infty}
		\begin{split}
			 &  
			 20\sqrt{   C C_2 M }
			  \leq \lambda_\infty \leq  \min \left( \frac{\lambda_2}{2}, \frac{\nu_0}{4}\right), 
 \ \  for \   M \ll 1,
		\end{split}
		\Ee
		where $\lambda_{2}$ is obtained in Proposition \ref{l2coercivity}. Note that from (\ref{Morrey}) the condition (\ref{delta_1/lamdab_1}) holds for $M \ll 1$.
		
		\hide

		\vspace{4pt}

		\textit{Step 2. } From Proposition \ref{prop_W1p}, (\ref{initial_M}), and (\ref{phi_c,gamma_M}), we have for $t \leq T$ 
		%
		\hide\[
		e^{ {C (1 
				+ \sup_{0 \leq s \leq t} \| \nabla^2 \phi_f(s) \|_\infty) t} } \leq e^{C(1+ (C_1 M)^{1/p})t}\delta_* M.
		\]

		\Be\label{est_W}
		\begin{split}
			&\| f(t) \|_p^p
			+\| \alpha_f^\beta 
			\p f (t)\|_p^p \\
			& 
			\lesssim  \   e^{ {C (1 
					+ \sup_{0 \leq s \leq t} \| \nabla^2 \phi_f(s) \|_\infty) t} }
			\{
			\| f(0) \|_p^p
			+\| \alpha_f^\beta 
			\p f (0)\|_p^p \}
			.\end{split} \Ee

		\unhide
		\Be\label{growth_W1p}
		\| \alpha_f^\beta \nabla_{x,v} f(t) \|_p ^p
		+ \| f(t) \|_p^p
		\leq 
		C_p e^{ C_p(1+ (C_1 M)^{1/p}) t} 
		\times \delta_* M  .
		\Ee

		\vspace{4pt}
		
		\unhide
		
		\textit{Step 2. } We claim that 
		\Be\label{decay_C2}
	\sup_{0 \leq t \leq T}	e^{  \frac{\lambda_\infty}{2}t} \| \nabla_x^2 \phi_f(t) \|_\infty \leq
		C_2 M , \ \   with   \ \ C_2 := C_{  \O }+ (C_1 C_p )^{1/p}\delta_*. 
		\Ee
		Here $C_{\O }$ appears in (\ref{Morrey}), and $C_1$ in (\ref{phi_c,gamma}), and $C_p$ in Proposition \ref{prop_W1p}.

		\hide
		We claim that, for $\delta_1>0$, $\lambda_0>0$, and $3<p<6$ 
		\Be\label{decay_phi_C2}
		\begin{split}
			&\| \nabla_x^2 \phi_f(t) \|_\infty\\
			\leq & \  e^{\delta_1 \lambda_0 t}  \| w f (t)\|_\infty   \\
			& \  +   e^{- (1- \frac{3}{p}) \lambda_0 t} 
			e^{ C\int^t_0 1+ \| w f(s) \|_\infty  + \| \phi_f (s) \|_{C^2}  \dd s} \times  \| \alpha^\beta \nabla_{x,v} f_0 \|_p
		\end{split}
		\Ee

		From (\ref{global_T}), (\ref{phi_c,gamma}) and (\ref{growth_W1p}), for all $p>3$
		\Be\label{growth_C2gamma}
		\| \phi_f (t) \|_{C^{2, 1- \frac{3}{p}} (\bar{\O})} \lesssim e^{CMt} M.
		\Ee\unhide
		
		From (\ref{Morrey}) and (\ref{global_T}), for $0 \leq t \leq T$, for all $D_1>0$
		\Be\begin{split}\label{decay_C0,alpha}
			\|  \phi_f (t)\|_{C ^{1,1- D_1}(\bar{\O})}  
			\leq  C_{ \O} \| w_{\vartheta} f (t) \|_\infty 
			\leq  C_{  \O}M e^{-  {\lambda_\infty}  t}.
		\end{split}\Ee
		
		On the other hand, from Proposition \ref{prop_W1p}, replacing $f^{\ell}$ and $f^{\ell +1 }$ by $f$ in (\ref{final_est_G}), (\ref{mid}), and by Gronwall's inequality and  (\ref{initial_M}), 
		 we derive that for $0\leq t \leq T$ 
		%
		\hide\[
		e^{ {C (1 
				+ \sup_{0 \leq s \leq t} \| \nabla^2 \phi_f(s) \|_\infty) t} } \leq e^{C(1+ (C_1 M)^{1/p})t}\delta_* M.
		\]

		\Be\label{est_W}
		\begin{split}
			&\| w_{\tilde{\vartheta}} f(t) \|_p^p
			+\| \alpha_f^\beta 
			\p f (t)\|_p^p \\
			& 
			\lesssim  \   e^{ {C (1 
					+ \sup_{0 \leq s \leq t} \| \nabla^2 \phi_f(s) \|_\infty) t} }
			\{
			\| f(0) \|_p^p
			+\| \alpha_f^\beta 
			\p f (0)\|_p^p \}
			.\end{split} \Ee

		\unhide
		\Be\label{growth_W1p}
		\begin{split}
			& 
			\| f(t) \|_p^p+\| w_{\tilde{\vartheta}}  \alpha_{f,\e}^\beta \nabla_{x,v} f(t) \|_p ^p 
			+ \int^t_0 | w_{\tilde{\vartheta}}  \alpha_{f,\e}^\beta \nabla_{x,v} f(s)  |_{p,+} ^p
			\\
			\leq  &    \ 
			C_p e^{ C_p(1+ L ) t} 
			\times( \delta_* M)^p  .
		\end{split}
		\Ee
		Now we use Lemma \ref{lemma_apply_Schauder}, from (\ref{phi_c,gamma}), for $p>3$ and $0\leq t \leq T$,
		\Be\label{phi_c,gamma_M}
		 \| \phi_f (t) \|_{C^{2, 1-\frac{3}{p}} (\bar{\O})} \leq 
		(C_1   C_p )^{1/p}  e^{ \frac{1}{p}C_p(1+ L ) t} 
			\times \delta_*  M  .
		\Ee

	\hide	
		On the other hand, from (\ref{phi_c,gamma}) and (\ref{growth_W1p}), 
		\Be\label{phi_c,gamma_M}
		\| \phi_f (t) \| _{C^{2, 1- \frac{3}{p}} (\bar{\O})} \leq  (C_1 C_p \delta_*  M )^{1/p}   e^{\frac{C_p(1+ (C_1 M)^{1/p})}{p}t}.
		\Ee 
	\unhide

	Finally we use an interpolation between $C^{1,1-D_1}(\bar{\O})$ and $C^{2, 1- \frac{3}{p}}(\bar{\O})$ and derive an estimate of $C^2(\bar{\O})$: Applying Lemma \ref{lemma_interpolation} and (\ref{phi_interpolation}) with $D_2 = 1-\frac{3}{p}$, from (\ref{growth_W1p}) and (\ref{decay_C0,alpha}), we derive that for all $0< D_1<1$, $3<p<6$, $\Lambda_0>0$, and $0 \leq t \leq T$,
		\Be 
		\begin{split}\label{apply_Lemma_interpolation}
			  \| \nabla_x^2 \phi_f(t) \|_\infty 
			\leq   e^{- [\lambda_\infty - D_1 \Lambda_0] t} C_{ \O} M  + e^{- [(1- \frac{3}{p}) \Lambda_0
			-\frac{1}{p}C_p(1+ L ) 
			] t} 
			(C_1   C_p )^{1/p}  
		 \delta_*  M  .
			%
			%
		\end{split}
		\Ee 
		Then we choose 
		\Be\label{choice_lambda_0_delta_1}
		\Lambda_0 =  
			\frac{\frac{\lambda_\infty}{2} + \frac{C_p}{p}
			(1+ L)}{1- \frac{3}{p}}  \ \  and  \ then    \ D_1= \frac{\lambda_\infty}{2\Lambda_0}.
		\Ee
		In conclusion we have, for all $0 \leq t \leq T$, 
		\[  \| \nabla_x^2 \phi_f(t) \|_\infty 
			\leq
			e^{- \frac{\lambda_\infty}{2} t} [C_\O+(C_1   C_p )^{1/p}   
		 \delta_*  ]M .
		\]
		As long as $M \ll L$ then $ \| \nabla_x^2 \phi_f(t) \|_\infty\leq L$ for all $0 \leq t \leq T$ and hence the claim (\ref{decay_C2}) holds.


		\hide

		Let us choose 
		\Be
		\lambda_0:= \frac{\frac{\lambda}{2} + C(1+ 2M)}{1- 3/p} \ \ \ \text{and} \ \ \ \delta_1:= \frac{\lambda}{2 \lambda_0}.
		\Ee
		Then 
		\Be
		\begin{split}
			\| \nabla_x^2 \phi_f(t) \|_\infty&\leq e^{\delta_1 \lambda_0 t} e^{-\lambda t}M + 
			e^{- (1- \frac{3}{p}) \lambda_0 t}  e^{C(1+ 2M) }\frac{M}{2}\\
			&\leq e^{- \frac{\lambda}{2}}\frac{3M}{2}.
		\end{split}\Ee
		
		\unhide
		\vspace{4pt}
		
		\hide
		
		\textit{Step 4. } We claim that there exist $T>0$ and $k_{0} >0$ such that for all $k\geq  k_{0}$ and for all $(t,x,v) \in [0,T] \times \bar{\O} \times \R^{3}$, we have 
		\Be
		\int_{\prod_{j=1}^{k-1} \mathcal{V}_{j}} \mathbf{1}_{\{ t^{k} (t,x,v,u^{1}, \cdots , u^{k-1}) >0 \}} \dd \Sigma_{k-1}^{k-1} \lesssim_{\O} \Big\{\frac{1}{2}\Big\}^{-k/5}.
		\Ee
		Here we define 
		\Be\label{mathcal_V}
		\mathcal{V}_j: = \{v^j \in \R^3: n(x^j) \cdot v^j >0\},
		\Ee
		and
		\Be\begin{split}
			\label{measure} 
			d\Sigma _{l}^{k-1}(s) &= \{\Pi _{j=l+1}^{k-1}\dd\sigma _{j}\}\times\{
			e^{
				-\int^{t_l}_s
				\nu_\phi 
			}
			\tilde{w}(v_{l})\dd\sigma _{l}\}\times \Pi
			_{j=1}^{l-1}\{{{
					e^{
						-\int^{t_j}_{t_{j+1}}
						\nu_\phi 
					}
					\dd\sigma _{j}}}\}.
		\end{split} \Ee 
		The proof of the claim is a modification of a proof of Lemma 14 of \cite{GKTT1}.
		
		For $0<\delta\ll 1$ we define
		\Be\label{V_j^delta}
		\mathcal{V}_{j}^{\delta} := \{ v^{j } \in \mathcal{V}^{j} : |v^{j} \cdot n(x^{j})| > \delta, \ |v^{j}| \leq \delta^{-1} \}.
		\Ee 
		
		Choose   
		\Be\label{large_T}
		T= \frac{2}{\delta^{2/3} (1+ \| \nabla \phi \|_\infty)^{2/3}}.
		\Ee
		We claim that 
		\Be\label{t-t_lowerbound}
		|t^j -t^{j+1}|\gtrsim   \delta^3, \ \ \text{for} \ v^j \in \mathcal{V}^\delta_j,  \ 0 \leq t\leq T, \ 0 \leq t^j.
		\Ee
		
		For $j \geq 1$
		\Bes
		&&\Big| \int^{t^{j+1}}_{t^{j}} V(s;t^{j}, x^{j}, v^{j}) \dd s   \Big|^{2}\\
		&=& |x^{j+1} -x^{j}|^{2}\\
		&\gtrsim& |(x^{j+1} -x^{j}) \cdot n(x^{j})|\\
		&=&\Big| \int^{t^{j+1}}_{t^{j }}  V(s;t^{j},x^{j},v^{j}) \cdot  n(x^{j})  \dd s 
		\Big|\\
		&=&\Big| \int_{t^{j}}^{t^{j+1}}  
		\Big(
		v^{j} - \int^{s}_{t^{j}} \nabla \phi (\tau, X(\tau;t^j,x^j,v^j))    \dd \tau
		\Big)\cdot   n(x^{j}) 
		\dd s 
		\Big|\\
		&\geq& |v^{j} \cdot n(x^{j})| |t^{j}-t^{j+1}|
		- \Big|
		\int^{t^{j+}}_{t^{j }} \int^{s}_{t^{j }}  \nabla \phi (\tau, X(\tau; t^{j},x^{j},v^{j})) \cdot n(x^{j})\dd \tau \dd s 
		\Big|.
		\Ees
		Here we have used the fact if $x,y \in \p\O$ and $\p\O$ is $C^2$ and $\O$ is bounded then $|x-y|^2\gtrsim_\O |(x-y) \cdot n(x)|$.

		Hence
		\Be\label{lower_tb}
		\begin{split}
			&|v^{j} \cdot n(x^{j})| \\
			& \lesssim \frac{1}{|t^{j} - t^{j+1}| } \Big| \int^{t^{j+1}}_{t^{j}}
			V(s;t^{j},x^{j},v^{j}) \dd s
			\Big|^{2}\\
			& \ \  + 
			\frac{1}{|t^{j} - t^{j+1}| } \Big| \int^{t^{j+1}}_{t^{j}} 
			\int^{s}_{t^{j}} 
			\nabla \phi (\tau, X(\tau;t^{j},x^{j},v^{j}) )\cdot n(x^{j})\dd \tau  \dd s
			\Big|\\
			& \lesssim 
			|t^{j} - t^{j+1}|   \big\{ |v^j|^2 + |t^j - t^{j+1}|^3 \|\nabla \phi\|^2_\infty
			\\
			& \ \ \ \ \ \ \ \  \ \ \ \ \ \  \ \ \ \     +    \frac{1}{2}\sup_{t^{j+1} \leq \tau \leq t^{j}} 
			|  \nabla \phi (\tau, X(\tau;t^{j},x^{j},v^{j}) )\cdot n(x^{j})|
			\big\}.
		\end{split}
		\Ee
		For $v^j \in \mathcal{V}^\delta_j$, $0 \leq t\leq T$, and $t^j\geq 0$,
		\Be\notag
		|v^j \cdot n(x^j)|\lesssim |t^j -t^{j+1}|
		\{
		\delta^{-2} + T^3 \| \nabla \phi \|_\infty^2 + \| \nabla \phi \|_\infty
		\}.
		\Ee
		If we choose $T$ as (\ref{large_T}) then we can prove (\ref{t-t_lowerbound}).

		{\color{red}Complete the proof by modifying the proof of Lemma 14 of \cite{GKTT1}}. 
		
		\vspace{4pt}
		
		\unhide
		

		\vspace{4pt}
		
		\textit{Step 3. } We claim that there exists $T_\infty\gg1$ such that, for $N\in \mathbb{N}$, $t \in [NT_\infty, (N+1)T_\infty]$, and $(N+1)T_\infty \leq T$, 
		\Be\begin{split}
			&\| w_\vartheta f (t) \|_\infty \\
			\leq&  \   (t - NT_\infty ) e^{-\frac{3}{4}\nu_0 (t-NT_\infty)} \| w_\vartheta f (NT_\infty) \|_\infty 
			+o(1) \sup_{NT_\infty \leq s \leq t}e^{- \frac{3}{4}\nu_0  (t-s)} \| w_\vartheta f(s) \|_\infty
			\\
			&+ C_{T_\infty} \int^t_{NT_\infty} e^{- \frac{3}{4}\nu_0  (t-s)} \| f(s) \|_{L^2_{x,v}} \dd s \\
			&+C_{T_\infty}  \int^t_{NT_\infty} e^{- \frac{3}{4}\nu_0  (t-s)} \| \nabla \phi_f (s) \|_\infty \dd s.
		\end{split}\label{decay_time_interval}
		\Ee
		For the sake of simplicity we present a proof of (\ref{decay_time_interval}) for $N=0$. The proof for $N >0$ can be easily obtained by considering $f(NT_\infty)$ as an initial datum.  
		
		As (\ref{h}) we define $h(t,x,v): = w_\vartheta f(t,x,v)$. Then $h$ solves (\ref{fell_local}) and (\ref{bdry_local}) with exchanging all $(h^{\ell}, h^{\ell+1}, \phi^\ell)$ to $(h,h, \phi_f)$. We define 
		\Be\label{nu_w}
	\begin{split}
		\nu_{\phi_f, w_\vartheta } (t,x,v) &:   = \begin{bmatrix} \nu_{\phi_f, w_\vartheta , +}  & 0 \\ 0 & \nu_{\phi_f, w_\vartheta,- }  \end{bmatrix} =  \begin{bmatrix} \nu(v) + \frac{v}{2} \cdot \nabla \phi_f - \frac{\nabla_x \phi_f \cdot \nabla_v w_{\vartheta}}{w_{\vartheta}}  & 0 
		\\ 0 & \nu(v) - \frac{v}{2} \cdot \nabla \phi_f + \frac{\nabla_x \phi_f \cdot \nabla_v w_{\vartheta}}{w_{\vartheta}} 
		\end{bmatrix}.		\end{split} 		\Ee
		From (\ref{global_T}) and (\ref{Morrey}), for $0 \leq t \leq T$ 
		\Be\begin{split}\label{lower_nu_phi_f,w}
			\nu_{\phi_f,w_\vartheta, \pm}\geq&  \ \big\{\nu_0- \frac{\| \nabla \phi_f \|_\infty}{2}  - 2 \vartheta \| \nabla \phi_f \|_\infty  \big\}\langle v\rangle \\
			\geq & \ \big\{\nu_0 - (\frac{1}{2} -2 \vartheta)M \big\}   \langle v\rangle\\
			\geq & \ \frac{4 \nu_0}{5} \langle v\rangle. 
		\end{split}\Ee
		Then $h$ solves (\ref{duhamel_local}) along the trajectory with deleting all superscriptions of $\ell$ and $\ell+1$ and exchanging $\nu^\ell$ to $\nu_{\phi_f, w_\vartheta}$ and with new $g$
		\Be\label{g_global}
		g: = - q_1 v\cdot \nabla \phi_f \sqrt{\mu} + \Gamma (\frac{h}{w_\vartheta}, \frac{h}{w_\vartheta}).
		\Ee
We define a stochastic cycles for $\iota =  +$ or $-$ as in \eqref{iota},
		\Be\notag
		(t_{l,\iota}(t,x,v,v_1, \cdots, v_{l-1}),x_{l,\iota}(t,x,v,v_1, \cdots, v_{l-1})),
		\Ee
		by deleting all superscriptions in (\ref{cycle}) and (\ref{cycle_ell}). Then by deleting all superscriptions of $\ell$ and $\ell+1$ from (\ref{stochastic_h^ell}), we obtain the bound for $h_\iota$:
		\Be \label{globalfirstexp}
		\begin{split}
			& |h_\iota  (t,x,v)|\\
			\leq & \   \| e^{- \frac{3}{4}\nu_0 t} h_0 \|_\infty
			+ O(k) \sup_{0 \leq s \leq t} \| e^{- \frac{3}{4}\nu_0 (t-s)} h  (s) \|_\infty^2
			+ O(k)  \int^t_0  \| e^{-\frac{3}{4}\nu_0 (t-s)}   \nabla   \phi_{\frac{h}{w_{\vartheta}}} (s) \|_\infty \dd s \\
			&  + \Big\{ \frac{1}{2} \Big\}^{k/5} \sup_{0 \leq s \leq t} \| e^{- \frac{3}{4}\nu_0 (t-s)} h  (s) \|_\infty \\
			& + \underbrace{ \int^{t}_{\max\{t _{1,\iota}, 0 \}}
			e^{ - \frac{3}{4}\nu_0 (t-s)}\int_{
				\R^3} \mathbf{k}_{\varrho} (V_\iota(s;t,x,v),u)
			|h  (s, X_\iota (s;t,x,v), u)|
			\dd u
			\dd s }_{\eqref{globalfirstexp}_1}\\
			& +\underbrace{ O(k) \sup_l  \int^{t_{l,\iota} }_{\max\{ t_{l+1,\iota} , 0  \}}   e^{ - \frac{3}{4}\nu_0 (t-s)} 
			\int_{\mathbb R^3} \int_{\mathbb R^3}
			\mathbf{k}_{\varrho} (V_\iota(s;t_l,x_l,v_l),u) |h (s, X_\iota (s;t_l,x_l, v_l)  ,u )| |n(x_l)\cdot v_l| \sqrt{\mu(v_l ) } \dd u
			\dd v_l
			\dd s  }_{\eqref{globalfirstexp}_2}.
		\end{split}
		\Ee

		For any large $m\gg 1$ we define
		\Be\label{k_m}
		\mathbf{k}_{\varrho,m} (v,u) = \mathbf{1}_{|v-u| \geq \frac{1}{m}, |v| \leq m} \mathbf{k}_\varrho (v,u),
		\Ee
		such that $\sup_v \int_{\R^3} | \mathbf{k}_{\varrho,m} (v,u) - \mathbf{k}_\varrho (v,u) | \dd u \lesssim \frac{1}{m}$, 
		and $|\mathbf{k}_{\varrho,m} (v,u)| \lesssim_m 1$.
		
		Furthermore we split the time interval as, for each $\ell, l$
		\Be\label{time_splitting_l_ell}
		\begin{split}
			\{\max\{  t _{l+1,\iota}, 0  \} \leq s\leq  t  _{l,\iota}\}  
			=
			\{\max\{  t _{l+1,\iota}, 0  \} \leq s\leq  t _{l,\iota}- \delta\} 
			\cup \{  t _{l,\iota} - \delta \leq s\leq  t _{l,\iota}\} ,
		\end{split}\Ee
		where we choose a small constant $0<\delta \ll_k 1$ later in (\ref{choice_m_delta}). 
		
		For $\eqref{globalfirstexp}_1$, we have
		\Be \label{gfirstexp1} \begin{split} 
		\eqref{globalfirstexp}_1 \le & \int_{\max{ \{ t_{1,\iota},0 \} }}^{t-\delta} +  \int_{t-\delta }^t 
		\\ \lesssim &  \int_{\max{ \{ t_{1,\iota},0 \} }}^{t-\delta} e^{ - \frac{3}{4}\nu_0 (t-s)}\int_{
				\R^3} \mathbf{k}_{\varrho} (V_\iota(s;t,x,v),u)
			|h  (s, X_\iota (s;t,x,v), u)|
			\dd u
			\dd s +  \delta  \sup_{0 \leq s \leq t} \| e^{- \frac{3}{4}\nu_0 (t-s)} h  (s) \|_\infty
		\\ \lesssim &   \int_{\max{ \{ t_{1,\iota},0 \} }}^{t-\delta} e^{ - \frac{3}{4}\nu_0 (t-s)}\int_{
				\R^3} \mathbf{k}_{\varrho,m} (V_\iota(s;t,x,v),u)
			|h  (s, X_\iota (s;t,x,v), u)|
			\dd u
			\dd s 
			\\ & + ( \frac{1}{m} +  \delta )  \sup_{0 \leq s \leq t} \| e^{- \frac{3}{4}\nu_0 (t-s)} h  (s) \|_\infty .
		\end{split} \Ee
		
		Now for $\eqref{globalfirstexp}_2$ we separate into several cases:
		
		\textbf{Case 1:} For $|v_l | > \frac{m}{2}$, we have $ | V_\iota(s;t_l,x_l,v_l) |  > \frac{m}{4}$, so 
		\[
		\int_{|v_l| > \frac{m}{2} } \int_{\mathbb R^3 } \mathbf{k}_{\varrho} (V_\iota(s;t_l,x_l,v_l),u) | (n(x_l)\cdot v_l | \sqrt{\mu(v_l ) } \dd u
			\dd v_l \lesssim \frac{1}{m}.
		\]
		Thus
		\[ \begin{split}
		\int^{t_{l,\iota} }_{\max\{ t_{l+1,\iota} , 0  \}}   e^{ - \frac{3}{4}\nu_0 (t-s)} 
			& \int_{ |v_l| > \frac{m}{2}} \int_{\mathbb R^3}
			\mathbf{k}_{\varrho} (V_\iota(s;t_l,x_l,v_l),u) |h (s, X_\iota (s;t_l,x_l, v_l)  ,u )| |n(x_l)\cdot v_l| \sqrt{\mu(v_l ) } \dd u
			\dd v_l
			\dd s 
			 \\ \lesssim & \frac{1}{m}  \sup_{0 \leq s \leq t} \| e^{- \frac{3}{4}\nu_0 (t-s)} h  (s) \|_\infty. 
		\end{split} \]
		
		\textbf{Case 2:} For $|v_l | \le \frac{m}{2}$, $|u| > m $, we have $  | V_\iota(s;t_l,x_l,v_l) - u  | > \frac{m}{4}$ so \[
 \sup_{|v_l | \le \frac{m}{2} } \int_{ |u| > m }  \mathbf{k}_{\varrho} (V_\iota(s;t_l,x_l,v_l),u) | (n(x_l)\cdot v_l | \sqrt{\mu(v_l ) } \dd u \lesssim \frac{1}{m}.
		\]
		Thus
		\[ \begin{split}
		 \int^{t_{l,\iota} }_{\max\{ t_{l+1,\iota} , 0  \}}   e^{ - \frac{3}{4}\nu_0 (t-s)} 
			& \int_{ |v_l| \le \frac{m}{2}} \int_{ |u | > m }
			\mathbf{k}_{\varrho} (V_\iota(s;t_l,x_l,v_l),u) |h (s, X_\iota (s;t_l,x_l, v_l)  ,u )| |n(x_l)\cdot v_l| \sqrt{\mu(v_l ) } \dd u
			\dd v_l
			\dd s 
		\\ 	\lesssim &  \frac{1}{m}  \sup_{0 \leq s \leq t} \| e^{- \frac{3}{4}\nu_0 (t-s)} h  (s) \|_\infty .
		\end{split} \]
		
		\textbf{Case 3:} For $|v_l | \le \frac{m}{2} $ and $|u| \le m $, we split the time integration as $ \int^{t_{l,\iota} }_{\max\{ t_{l+1,\iota} , 0  \}} =  \int^{t_{l,\iota} -\delta }_{\max\{ t_{l+1,\iota} , 0  \}}  +  \int^{t_{l,\iota} }_{t_{l,\iota} -\delta }$ and use \eqref{k_m} to conclude that
		\Be \label{gfirstexp2} \begin{split}
		\eqref{globalfirstexp}_2
		  \lesssim  & O(k) \sup_l \int^{t_{l,\iota} - \delta }_{\max\{ t_{l+1,\iota} , 0  \}}   e^{ - \frac{3}{4}\nu_0 (t-s)} 
			\\ & \times  \int_{ |v_l| \le m} \int_{ |u | \le m }
			\mathbf{k}_{\varrho,m} (V_\iota(s;t_l,x_l,v_l),u) |h (s, X_\iota (s;t_l,x_l, v_l)  ,u )| |n(x_l)\cdot v_l| \sqrt{\mu(v_l ) } \dd u
			\dd v_l
			\dd s  \\ &  +O(k)  ( \frac{1}{m} +  \delta )  \sup_{0 \leq s \leq t} \| e^{- \frac{3}{4}\nu_0 (t-s)} h  (s) \|_\infty
		\\ 	\lesssim & O_m(k) \sup_l  \int^{t_{l,\iota} - \delta }_{\max\{ t_{l+1,\iota} , 0  \}}   e^{ - \frac{3}{4}\nu_0 (t-s)} 
			 \int_{ |v_l| \le m} \int_{ |u | \le m }
			 |h (s, X_\iota (s;t_l,x_l, v_l)  ,u )| |n(x_l)\cdot v_l| \sqrt{\mu(v_l ) } \dd u
			\dd v_l
			\dd s  \\ &  +O(k) ( \frac{1}{m} +  \delta )  \sup_{0 \leq s \leq t} \| e^{- \frac{3}{4}\nu_0 (t-s)} h  (s) \|_\infty.
	\end{split}	\Ee
	
	Combining \eqref{globalfirstexp}, \eqref{gfirstexp1}, and \eqref{gfirstexp2} we get
	
	\Be\label{double_teration}
		\begin{split}
			& |h_\iota  (t,x,v)|\\
			\leq & \  \| e^{- \frac{3}{4}\nu_0 t} h_0 \|_\infty
			+ O(k) \sup_{0 \leq s \leq t} \| e^{- \frac{3}{4}\nu_0 (t-s)} h  (s) \|_\infty^2
			+ O(k)  \int^t_0  \| e^{-\frac{3}{4}\nu_0 (t-s)}   \nabla   \phi_{\frac{h}{w_{\vartheta}}} (s) \|_\infty \dd s \\
			&+ \left\{O(k)  (\delta + \frac{1}{m})   + \Big\{ \frac{1}{2} \Big\}^{k/5}\right\}  \sup_{0 \leq s \leq t} \| e^{- \frac{3}{4}\nu_0 (t-s)} h  (s) \|_\infty \\
			& + \ \int^{t-\delta}_{\max\{t _{1,\iota}, 0 \}}
			e^{ - \frac{3}{4}\nu_0 (t-s)}\int_{
				\R^3} \mathbf{k}_{\varrho,m} (V_\iota(s;t,x,v),u)
			|h  (s, X_\iota (s;t,x,v), u)|
			\dd u
			\dd s\\
			& +O_m(k) \sup_l  \int^{t_{l,\iota} -\delta}_{\max\{ t_{l+1,\iota} , 0  \}}   e^{ - \frac{3}{4}\nu_0 (t-s)}\\
			& \ \ \ \ \  \ \ \ \   \times
			\int_{|u| \leq m} \int_{|v_l| \leq m} 
			|h (s, X_\iota (s;t_l,x_l, v_l)  ,u )| \dd v_l
			\dd u
			\dd s .
		\end{split}
		\Ee

		
		Now for $|h  (s, X_\iota (s;t,x,v), u)|$ we use similar bounds for $|h_+(s,X_\iota (s;t,x,v) ,u ) | $ and $|h_-(s,X_\iota (s;t,x,v) ,u ) | $ separately and add them together to get		
		\hide
		\Be
		\begin{split}
			&|h_\iota  (s, X_\iota (s;t,x,v), u)| 
			\\ =& |h_+(s,X_\iota (s;t,x,v) ,u ) | +|h_-(s,X_\iota (s;t,x,v) ,u ) |   \\
			\leq & \  \| e^{- \frac{\nu_0}{2} t} h_0 \|_\infty\\
			& + O(k) \sup_{0 \leq s \leq t} \| e^{- \frac{\nu_0}{2} (t-s)} h  (s) \|_\infty^2\\
			& + O(k) \int^s_0  \| e^{-\frac{\nu_0}{2} (t-s)}  \nabla   \phi_{\frac{h}{w}} (s) \|_\infty   \\
			&+ \left\{O(k) \left(\delta + \frac{1}{m}\right) + \Big\{ \frac{1}{2} \Big\}^{k/5}\right\}   \sup_{0 \leq s \leq t} \| e^{- \frac{\nu_0}{2} (t-s)} h  (s) \|_\infty \\
			& + \ \int^{t-\delta}_{\max\{t _1, 0 \}}
			e^{ - \frac{\nu_0}{2} (t-s)}\int_{
				|u| \leq m}
			|h  (s, X (s;t,x,v), u)|
			\dd u
			\dd s\\
			& +O(k) \int^{t_l -\delta}_{\max\{ t^{\ell-1} , 0  \}}   e^{ - \frac{\nu_0}{2} (t-s)}\\
			& \ \ \ \ \  \ \ \ \ \  \  \times
			\int_{|u| \leq m} \int_{|v_l| \leq m} 
			|h (s, X (s;t_l,x_l, v_l)  ,u )| \dd v_l
			\dd u
			\dd s .
		\end{split}
		\Ee\unhide
		\Be\label{double_teration_s}
		\begin{split}
			   |h & (s, X_\iota(s;t,x,v), u)| 
			 \\ =  & |h_+(s,X_\iota (s;t,x,v) ,u ) | +|h_-(s,X_\iota (s;t,x,v) ,u ) |\\
			\leq &  \ \| e^{- \frac{3}{4}\nu_0 s} h_0 \|_\infty
			+ O(k) \sup_{0 \leq s^\prime \leq s} \| e^{- \frac{3}{4}\nu_0 (s-s^\prime)} h (s^\prime) \|_\infty^2
			+ O(k)  \int^s_0 \| e^{-\frac{3}{4}\nu_0 (s-s^\prime)}  \nabla   \phi_{\frac{h}{w_{\vartheta}}}  (s^\prime) \|_\infty \dd s^\prime \\
			&+ \left\{O(k) \left(\delta + \frac{1}{m}\right) + \Big\{ \frac{1}{2} \Big\}^{k/5}\right\}   \sup_{0 \leq s^\prime \leq s} \| e^{- \frac{3}{4}\nu_0 (s-s^\prime)} h (s^\prime) \|_\infty \\
			& + O(m) \int^{s-\delta}_{\max\{t^\prime_{1,+}, 0 \}}
			e^{ - \frac{3}{4}\nu_0 (s-s^\prime)}\int_{
				|u^\prime| \leq m}
			|h_+  (s^\prime, X_+ (s^\prime;s,X_\iota(s;t,x,v),u), u^\prime)|
			\dd u^\prime
			\dd s^\prime
			\\
			& +  O(m) \int^{s-\delta}_{\max\{t^\prime_{1,-}, 0 \}}
			e^{ - \frac{3}{4}\nu_0 (s-s^\prime)}\int_{
				|u^\prime| \leq m}
			|h_-  (s^\prime, X_- (s^\prime;s,X_\iota(s;t,x,v),u), u^\prime)|
			\dd u^\prime
			\dd s^\prime
			\\
			& +O_m(k) \sup_{l,l^\prime}   \int^{t^\prime_{l^\prime,+} -\delta}_{\max\{ t^{\prime}_{l^\prime+1,+}, 0  \}}   e^{ - \frac{3}{4}\nu_0 (s-s^\prime)}
			\int_{|u^\prime| \leq m} \int_{|v^\prime_{l^\prime}| \leq m} 
			|h_+ (s^\prime, X_\iota(s^\prime;t^\prime_{l^\prime}, x^\prime_{l^\prime}, v^\prime_{l^\prime})  ,u^\prime )| \dd v^\prime_{l^\prime}
			\dd u^\prime
			\dd s^\prime
			\\
			& +O_m(k) \sup_{l,l^\prime}   \int^{t^\prime_{l^\prime,-} -\delta}_{\max\{ t^{\prime}_{l^\prime+1,-}, 0  \}}   e^{ - \frac{3}{4}\nu_0 (s-s^\prime)}
			 \int_{|u^\prime| \leq m} \int_{|v^\prime_{l^\prime}| \leq m} 
			|h_- (s^\prime, X_\iota(s^\prime;t^\prime_{l^\prime}, x^\prime_{l^\prime}, v^\prime_{l^\prime})  ,u^\prime )| \dd v^\prime_{l^\prime}
			\dd u^\prime
			\dd s^\prime ,
		\end{split}
		\Ee  
		where 
		\Be\notag
		\begin{split} t^\prime_{l^\prime,+} &= t_{l^\prime,+} (s, X_\iota(s;t,x,v), u,v^\prime_1, \cdots, v^\prime_{l^\prime-1}),\\
		t^\prime_{l^\prime,-} &= t_{l^\prime,-} (s, X_\iota(s;t,x,v), u,v^\prime_1, \cdots, v^\prime_{l^\prime-1}),\\
			x^\prime_{l^\prime,+} &= x_{l^\prime,+} (s, X_\iota(s;t,x,v), u,v^\prime_1, \cdots, v^\prime_{l^\prime-1})
			\\
			x^\prime_{l^\prime,-} &= x_{l^\prime,-} (s, X_\iota(s;t,x,v), u,v^\prime_1, \cdots, v^\prime_{l^\prime-1}).
		\end{split}
		\Ee
		
		\hide

		From (\ref{small_k}) for $k\gg 1$
		\Be
		\mathbf{1}_{t_{1}^\ell >0}   
		\frac{
			e^{-  \int^t_{t_1^\ell}     \nu^\ell }
		}{\tilde{w}_{\varrho}(V^\ell (t_1^\ell))}\int_{\prod_{j=1}^{k-1}\mathcal{V}_{j}}  (\ref{h5}) \leq o(1)\max_{ l \geq 0 } \| h^{\ell-l} \|_\infty
		\Ee

		From Lemma 3 in \cite{Guo10} for $\varrho>0$ and $-2\varrho<  \vartheta <2\varrho$ and $\zeta\in\mathbb{R}$, we have \begin{equation} 
		\int_{\mathbb{R}^{3}}
		\mathbf{k} (v,u)
		\frac{\langle
			v\rangle ^{\zeta }e^{\vartheta |v|^{2}}}{\langle u\rangle ^{\zeta }e^{\vartheta
				|u|^{2}}}\mathrm{d}u \ \lesssim \ \langle v\rangle ^{-1}.   \label{int_k}  
		\end{equation}
		Then $|K_w h^{\ell-l}| \leq \int_{\R^3} \mathbf{k} (v,u) \frac{w(v)}{w(u)} \dd u \times \| h^{\ell-l} \|_\infty \lesssim \| h^{\ell-l} \|_\infty$. Together with (\ref{bound_g_ell}) 
		%

		Then $|h(t,x,v)|$ have the same bound of (\ref{h_iteration}), (\ref{h1}), (\ref{h2}), (\ref{h5}).

		. 
		
		We split $\mathbf{k}_w = \mathbf{k}_m + ( \mathbf{k}_w - \mathbf{k}_m)$
		\Be\label{double_teration_ell}
		\begin{split}
			& |h^{\ell+1} (t,x,v)|\\
			\leq & O(k) \| e^{- \frac{\nu_0}{4} t} h_0 \|_\infty\\
			& + O(k)\max_l\sup_{0 \leq s \leq t} \| e^{- \frac{\nu_0}{2} (t-s)} h^{\ell-l} (s) \|_\infty^2\\
			& + O(k) \max_l  \| e^{-\frac{\nu_0}{2} (t-s)}  \nabla   \phi ^{\ell-l} (s) \|_\infty \\
			&+ o(1) \max_l\sup_{0 \leq s \leq t} \| e^{- \frac{\nu_0}{2} (t-s)} h^{\ell-l} (s) \|_\infty \\
			& + \ \int^t_{\max\{t^\ell_1, 0 \}}
			e^{ - \frac{\nu_0}{2} (t-s)}\int_{
				|u| \leq m}
			|h^\ell (s, X^\ell (s;t,x,v), u)|
			\dd u
			\dd s\\
			& +O(k) \sup_l \int^{t_l^{\ell- (l-1)}}_{\max\{ t^{\ell-1}_{l+1}, 0  \}}   e^{ - \frac{\nu_0}{2} (t-s)}
			\int_{|u| \leq m} \int_{|v_l| \leq m} 
			|h^{\ell-l} (s, X^{\ell-l}(s;t_l,x_l, v_l)  ,u )| \dd v_l
			\dd u
			\dd s .
		\end{split}
		\Ee

		We plug this again to achieve 
		\Be\label{double_teration_twice}
		\begin{split}
			& |h (t,x,v)|\\
			\leq & O(k ) \| e^{- \frac{\nu_0}{4} t} h_0 \|_\infty\\
			& + O(k )\max_l\sup_{0 \leq s \leq t} \| e^{- \frac{\nu_0}{2} (t-s)} h  (s) \|_\infty^2\\
			& + O(k) \max_l  \| e^{-\frac{\nu_0}{2} (t-s)}  \nabla   \phi  _f (s) \|_\infty \\
			&+ \{ \frac{1}{N} + \{\frac{1}{2}\}^{-k/5}  \}  \max_l\sup_{0 \leq s \leq t} \| e^{- \frac{\nu_0}{2} (t-s)} h  (s) \|_\infty \\
			& + \ \int^t_{\max\{t_1, 0 \}}
			e^{ - \frac{\nu_0}{2} (t-s)}\int_{
				|u| \leq m}
			\int^s_{\max\{ t_1^\prime, 0 \}} 
			e^{- \frac{\nu_0}{2} (s-s^\prime)} \\
			& \ \ \ \ \ \  \times 
			\int_{|u^\prime| \leq m} 
			|h(s^\prime, X(s^\prime; s, X(s;t,x,v), u), u^\prime)| \dd u^\prime
			\dd s^\prime
			\dd u
			\dd s\\
			& +O(k) \sup_l \int^{t_l^{\ell- (l-1)}}_{\max\{ t^{\ell-1}_{l+1}, 0  \}}   e^{ - \frac{\nu_0}{2} (t-s)}
			\int_{|u| \leq m} \int_{|v_l| \leq m} 
			|h^{\ell-l} (s, X^{\ell-l}(s;t_l,x_l, v_l)  ,u )| \dd v_l
			\dd u
			\dd s 
			%
			%
		\end{split}
		\Ee


		\hide
		
		We define
		\begin{equation}
		\tilde{w} (v)\equiv \frac{1}{ w  (v)\sqrt{\mu (v)}}
		,\label{tweight}
		\end{equation}
		and $\mathcal{V}(x)=\{v \in\mathbb{R}^3 : n(x)\cdot v >0\}$
		with a probability measure $\dd\sigma=\dd\sigma(x)$ on $\mathcal{V}(x)$ which is given by
		\begin{equation}
		\dd\sigma \equiv\mu
		(v)\{n(x)\cdot v\}\dd v.\label{smeasure}
		\end{equation}

		Denote $h=w  f$ and $K_{w }( \ \cdot \ )=w  K( \frac {1}{w } \ \cdot)$.   
		\Be\begin{split}
			&|h (t,x,v)| \\
			\leq & \ \mathbf{1}_{t_{1}\leq 0}e^{-      \int^t_0     \nu_\phi 
			}|h (0,x-t{v},v)|  \notag \\
			&  +
			\int_{\max\{t_1,0\}}^{t}e^{-  \int^t_s     \nu_\phi }|[K_{w }h  +w g](s,X(s;t,x,v), V(s;t,x,v))|\dd s  \notag \\
			&    +\mathbf{1}_{t_{1}>0}   
			\frac{
				e^{-  \int^t_{t_1}     \nu_\phi }
			}{\tilde{w_{\varrho}}(v)}\int_{\prod_{j=1}^{k-1}\mathcal{V}_{j}}|H|  \notag
		\end{split}\Ee
		where $|H|$ is bounded by
		\begin{eqnarray}
		&&\sum_{l=1}^{k-1}\mathbf{1}_{\{t_{l+1}\leq
			0<t_{l}\}}|h (0,
		X(0;t_l,x_l,v_l)
		,V(0;t_l,x_l,v_l))|\dd\Sigma _{l}(0)  \label{h1} \\
		&&+\sum_{l=1}^{k-1}\int_{\max\{ t_{l+1}, 0 \}}^{t_l}\mathbf{1}_{\{t_{l+1}\leq
			0<t_{l}\}} 
		\\
		&& \ \ \ \ \  \ \ \ \ \   \times|[K_{w }h +w   g](s,
		X(s;t_l,x_l,v_l), V(s;t_l,x_l,v_l)
		|\dd \Sigma
		_{l}(s)\dd s  \ \ \ \  \ \ \ \ \ \ \label{h2} \\
		&&+\mathbf{1}_{\{0<t_{k}\}}|h (t_{k},x_{k},v_{k-1})|\dd\Sigma
		_{k-1}(t_{k}),  \label{h5}
		\end{eqnarray}%
		where $\dd\Sigma _{l}^{k-1}(s)$ is defined in (\ref{measure}).
		
		\unhide
		
		\unhide

		Plugging (\ref{double_teration_s}) into (\ref{double_teration}) we conclude that 
		\Be\label{double_teration_double}
		\begin{split}
			& |h_\iota(t,x,v)|\\
			\leq & \   \| ( 1 + t ) e^{- \frac{3}{4}\nu_0 t} h_0 \|_\infty  +  O_m(k)  \sup_{0 \leq s \leq t} \| e^{- \frac{3}{4}\nu_0 (t-s)} h  (s) \|_\infty^2
			\\
			& + O_m(k)  \int^t_0  \| e^{-\frac{3}{4}\nu_0 (t-s)}   \nabla   \phi_{\frac{h}{w_{\vartheta}}} (s) \|_\infty \dd s \\
			&+ \left\{O_m(k) \delta +  O(k)\frac{1}{m}  + O_m(1)\Big\{ \frac{1}{2} \Big\}^{k/5}\right\}  \sup_{0 \leq s \leq t} \| e^{- \frac{3}{4}\nu_0 (t-s)} h  (s) \|_\infty \\
			& + O(m^2) \int^{t }_{0}
			\int_{
				|u^\prime| \leq m}
						\int^{s-\delta}_{0} e^{ - \frac{3}{4}\nu_0 (t-s^\prime)} \int_{
				|u| \leq m}
			|h_+  (s^\prime, X_+ (s^\prime;s,X_\iota(s;t,x,v),u), u^\prime)| \dd u
			\dd s^\prime \dd u^\prime
			\dd s
			 \\
			& + O(m^2) \int^{t }_{0}
			\int_{
				|u^\prime| \leq m}
						\int^{s-\delta}_{0} e^{ - \frac{3}{4}\nu_0 (t-s^\prime)} \int_{
				|u| \leq m}
			|h_-  (s^\prime, X_- (s^\prime;s,X_\iota(s;t,x,v),u), u^\prime)| \dd u
			\dd s^\prime \dd u^\prime
			\dd s\\
			& +O(k) \sup_l  \int^{t_{l,\iota} -\delta}_{0}   e^{ - \frac{3}{4}\nu_0 (t-s)}
			\int_{|u| \leq m} \int_{|v_l| \leq m} 
			|h_\iota (s, X_\iota (s;t_l,x_l, v_l)  ,u )| \dd v_l
			\dd u
			\dd s
			\\& +O(k) \sup_{l,l^\prime} \int^{t }_{0}  \int^{t^\prime_{l^\prime,+} -\delta}_{0}   e^{ - \frac{3}{4}\nu_0 (t-s^\prime)}
			\int_{|u^\prime| \leq m} \int_{|v^\prime_{l^\prime}| \leq m} 
			|h_+ (s^\prime, X_\iota(s^\prime;t^\prime_{l^\prime}, x^\prime_{l^\prime}, v^\prime_{l^\prime})  ,u^\prime )| \dd v^\prime_{l^\prime}
			\dd u^\prime
			\dd s^\prime
			\dd s
			\\
			& +O(k) \sup_{l,l^\prime}  \int^{t }_{0} \int^{t^\prime_{l^\prime,-} -\delta}_{ 0  }   e^{ - \frac{3}{4}\nu_0 (t-s^\prime)}
			 \int_{|u^\prime| \leq m} \int_{|v^\prime_{l^\prime}| \leq m} 
			|h_- (s^\prime, X_\iota(s^\prime;t^\prime_{l^\prime}, x^\prime_{l^\prime}, v^\prime_{l^\prime})  ,u^\prime )| \dd v^\prime_{l^\prime}
			\dd u^\prime
			\dd s^\prime
			\dd s.
		\end{split}\Ee
		
		Choose $T_\infty \gg 1$ and $k\gg 1$ in (\ref{small_k}) and (\ref{large_T}). Then we choose 
		\Be
		m =  {k^2} \ \ \text{and} \ \ \delta = \frac{1}{m^3 k},\label{choice_m_delta}
		\Ee
		so that $O_m(k) \delta +  O(k)\frac{1}{m}  + O_m(1) \{ \frac{1}{2} \}^{k/5}  \ll1$.
		
		Note that 
		\Be\label{X_v}
		\begin{split}
			&\frac{\p  X_\pm(s ;t _{l },x_{l },v _{l }) }{\p v _{l }}\\
			= &   - ( t _{l  }-s ) \mathrm{Id}_{3 \times 3}\\
			& \mp
			\int^{s }_{t _{l }} \int^\tau_{t _{l }} \left( \frac{\p X_\pm(\tau^\prime;t _{l }, x _{l }, v _{l })}{\p v _{l }} \cdot \nabla_x\right)
			\left(          
			\nabla_x \phi_{\frac{h}{w_{\vartheta}}} (\tau^\prime, X_\pm(\tau^\prime;t _{l },x _{l },v _{l }))
			\right) \dd \tau^\prime\dd \tau,
		\end{split}
		\Ee

		\hide
		\Be\notag
		\| \nabla_x^2 \phi_f(t) \|_\infty \leq
		C_2 M^{1/p} 
		e^{- \frac{\lambda_\infty}{2}t}   \ \ \text{with} \ C_2 := C_{\frac{\lambda_\infty}{2 \lambda_0}, \O }+ (C_1 C_p \delta_*)^{1/p},
		\Ee\unhide

		Now we use Lemma \ref{est_X_v}. Note that from (\ref{decay_C2}), the condition (\ref{Decay_phi_2}) of Lemma \ref{est_X_v} is satisfied with $\Lambda_{2} = \frac{\lambda_\infty}{2}$ and $\delta_{2} = C_2 M$. From Lemma \ref{est_X_v} and (\ref{result_X_v}) we have for $\iota = +$ or $ - $,
		\Be\label{X_v_l}
		\left|\frac{\p X_\iota(\tau^\prime;t _{l },x   _{l  },v _{l })}{\p v _{l }}\right| \leq Ce^{ \frac{4C C_2 M}{(\lambda_\infty)^2}   }|t _{l } -\tau^\prime|.
		\Ee
		From (\ref{X_v_l}) and (\ref{decay_C2}), the second term of RHS in (\ref{X_v}) is bounded by 
		\Be\begin{split}\label{double_int_phi}
			&C C_2 M e^{ \frac{4CC_2 M}{(\lambda_\infty)^2}    }  \int^{t _{l }}_{s }   \int^{t _{l }}_\tau (t _{l } - \tau^\prime) e^{-\frac{\lambda_\infty}{2} \tau^\prime}
			\dd \tau^\prime \dd \tau \\ 
			\leq & \ 
	 \frac{4	C C_2 M }{(\lambda_\infty)^2}	 e^{ \frac{4	C C_2 M }{(\lambda_\infty)^2}}
			|t _{l } -s |.
		\end{split}\Ee  
		From our choice of $\lambda_\infty$ in (\ref{condition_lambda_infty}), we have 
		\Be\notag
		 \frac{4	C C_2 M }{(\lambda_\infty)^2}	 e^{ \frac{4	C C_2 M }{(\lambda_\infty)^2}} < \frac{1}{10}. 
		\Ee
		Therefore from (\ref{X_v}), for $0 \leq s  \leq t_l - \delta$
\hide		\Be\label{X_v_estimate}
		\left|\frac{\p  X(s ;t _{l },x _{l },v _{l }) }{\p v _{l }} 
		+ (t _{l }-s^\prime) \mathrm{Id}_{3 \times 3}\right|
		\leq  4Ce^{C \frac{\lambda_\infty}{2} \sqrt{C_2 M^{1/p}}} C_2 \frac{M^{1/p}}{(\lambda_\infty)^2} |t _{l }-s^\prime|.
		\Ee       
		From (\ref{condition_lambda_infty}) and (\ref{choice_m_delta}) we have  
		%
		%
\unhide		\Be\label{lower_jacob_l}
		\begin{split}
		&\det\left(\frac{\p X_\iota(s ;t _{l }, x _{l},v _{l })}{\p v _{l }}\right)
		\\
	= 	&  \  \det\left( - (t_l - s ) \text{Id}_{3 \times 3}
	+  o(1)
	\right) \\ 
		\gtrsim & \   |t_l - s |^3\\
		\gtrsim & \   \delta.
	\end{split}	\Ee
		We can obtain the exactly same lower bound of $\det \left( \frac{\p X_\iota(s;t^\prime_{l^\prime}, x^\prime_{l^\prime},v^\prime_{l^\prime})}{\p v^\prime_{l^\prime}}\right)$, $\det \left( \frac{\p X_+(s^\prime;s,X_\iota(s;t,x,v), u)}{\p u} \right)$, and $\det \left( \frac{\p X_-(s^\prime;s,X_\iota(s;t,x,v), u)}{\p u} \right)$ for $0 \leq s^\prime \leq s-\delta$ and $0 \leq s^\prime \leq t^\prime_{l^\prime} - \delta$.
		
		Now we apply the change of variables 
		\Bes
		v_l &\mapsto& X_\iota(s;t_l, x_l,v_l),\\
		v^\prime_{l^\prime} &\mapsto& X_\iota(s;t^\prime_{l^\prime}, x^\prime_{l^\prime},v^\prime_{l^\prime}),\\
		u &\mapsto& X_+(s^\prime;s,X_\iota(s;t,x,v), u), 
		\\ u &\mapsto& X_-(s^\prime;s,X_\iota(s;t,x,v), u), 
		\Ees 
		and conclude (\ref{decay_time_interval}) from (\ref{double_teration_double}) and (\ref{choice_m_delta}).

		By choosing $T_\infty \gg 1 $ so that $(1 + T_\infty) e^{-\frac{3}{4} \nu_0 T_\infty }  \le e^{-\frac{1}{2} \nu_0 T_\infty } $, and applying (\ref{decay_time_interval}) successively, we achieve that 
		\Be\begin{split}\label{global_decay_N}
			&\| w_{\vartheta} f(t) \|_\infty\\
			\leq & \  C_{T_\infty} e^{- \frac{\nu_0}{2} t} \| w_{\vartheta} f(0) \|_\infty 
			+o(1) \frac{e^{\nu_0 T_\infty}}{1- e^{- \nu_0 T_\infty}}
			\sup_{0 \leq s \leq t} e^{- \nu_0 (t-s)} \| w_{\vartheta} f(s) \|_\infty
			\\
			&
			+\underbrace{ C_{T_\infty} e^{\frac{\nu_0}{2}} \int^t_0 e^{- \frac{\nu_0}{2} (t-s)} \| f(s) \|_2 \dd s}_{(\ref{global_decay_N})_{L^2}}
    			+  \underbrace{C_{T_\infty} e^{\frac{\nu_0}{2}}\int^t_0 e^{- \frac{\nu_0}{2} (t-s)}  \| \nabla \phi_f (s) \|_\infty \dd s
			}_{(\ref{global_decay_N})_{\phi_f}}
			,
		\end{split}\Ee
		where we have used
		\[
	 e^{\nu_0 T_\infty} \{ 1+ e^{- \nu_0 T_\infty} +  \cdots +  e^{- \nu_0 NT_\infty}\}= \frac{e^{\nu_0 T_\infty}}{1- e^{- \nu_0 T_\infty}}.
		\]
		
		\vspace{4pt}
		
		\textit{Step 4. } 	From Proposition \ref{l2coercivity} and (\ref{completes_dyn}) we have 
		\Be\begin{split}\label{completes_dyn_final}
			& \| e^{\lambda_2 t}f(t)\|_2^2
			+ \| e^{\lambda_2 t} \nabla \phi (t) \|_2^2\\
			&
			+  \int_0^t \| e^{\lambda_2 \tau}  f (\tau)\|_\nu^2 
			+ \| e^{\lambda_2 \tau} \nabla \phi_f(\tau)\|_2^2 
			\mathrm{d} \tau 
			+  \int_0^t | e^{\lambda_2 \tau} f |^2_{2,+ }    \\
			\lesssim & \ \| f_0\|_2^2 +   \| \nabla \phi_{f_0}\|_2^2   .  
		\end{split}\Ee
		Hence 
		\Be\label{global_decay_N_L2}
		\begin{split}
		(\ref{global_decay_N})_{L^2}\lesssim& \   t e^{- \min ( \frac{\nu_0}{2}, \lambda_2) \times  t}
		\{\| f_0\|_2  +   \| \nabla \phi_{f_0}\|_2   \}\\
		\lesssim& \  e^{- \min ( \frac{\nu_0}{4}, \frac{\lambda_2}{2}) \times  t}
		\{\| f_0\|_2  +   \| \nabla \phi_{f_0}\|_2   \}.
		\end{split}\Ee

	Now we consider $(\ref{global_decay_N})_{\phi_f}$. In order to close the estimate in (\ref{global_decay_N}) we need to improve the decay rate of $ \| \nabla \phi_f (s) \|_\infty$. We claim that, for $\theta_{2,r,p}>0$ (which is specified in (\ref{theta_2rp})),
		\Be\label{better_decay_phi_C1}
		\| \nabla_x \phi_f (s) \|_\infty \lesssim e^{- (1+ \theta_{2,r,p}) \lambda_\infty s}
		\{\sup_{t\geq 0} \|e^{\lambda_2 t} f(s) \|_2  + 
		\sup_{t\geq 0} \|e^{\lambda_\infty t} f(s) \|_\infty\}.
		\Ee

		By Morrey's inequality for $\O \subset \R^3$ and $r>3$
		\Be
		\begin{split}\label{Morry_phi_infty}
			\| \nabla_x \phi_f \|_\infty 
			\lesssim   \| \nabla_x \phi_f \|_{C^{0, 1-3/r} (\O)} 
			\lesssim  \| \nabla_x \phi_f \|_{W^{1,r} (\O)}.
		\end{split}\Ee
		Then applying the standard elliptic estimate to (\ref{smallfphi}), we get  
		\begin{eqnarray}
		\| \nabla_x \phi_f(t) \|_{W^{1,2} (\O)}
		\lesssim 
		\left\| 
			\int_{\R^3} ( f_+(t) - f_-(t)) \sqrt{\mu}  \dd v
			\right\|_{L^2 (\O)}
		 \lesssim e^{- \lambda_2 t}  \sup_{t\geq 0} \|e^{\lambda_2 t} f(t) \|_2,\label{decay_phi_1r} \\
			\| \nabla_x \phi_f(t) \|_{W^{1,p} (\O)} \lesssim
			\left\| 
			\int_{\R^3} (f_+(t) - f_-(t) )  \sqrt{\mu}  \dd v
			\right\|_{L^p (\O)}
			\lesssim
			e^{- \lambda_\infty t} \sup_{t\geq 0} \|e^{\lambda_\infty t} f(t) \|_\infty
			.\label{decay_phi_1p}
		 \end{eqnarray}
		 
		 Now we use the standard interpolation: For $p>r>3$,
		 \Be\begin{split}\notag
			\| \nabla_x \phi_f \|_{W^{1,r} (\O)}
			\lesssim 
			\| \nabla_x \phi_f(t) \|_{W^{1,2} (\O)}^{\theta_{2,r,p} } \| \nabla_x \phi_f(t)\|_{W^{1,p} (\O)}^{1-\theta_{2,r,p} },
		\end{split}
		\Ee
		for 
		\Be\label{theta_2rp}
		\theta_{2,r,p} : = \frac{\frac{1}{r}- \frac{1}{p}}{\frac{1}{2} - \frac{1}{p}}> \frac{2}{3}\cdot \frac{p-3}{p-2}.
		\Ee
		Then we derive 
		\Be\begin{split}\label{decay_interpolation}
			& \sup_{t\geq 0 }\|e^{  [\theta_{2,r,p} \lambda_2  + (1- \theta_{2,r,p}) \lambda_\infty  ] t } \nabla_x \phi_f (t) \|_\infty \\
			\lesssim & \ \left(\sup_{t\geq 0} \|e^{\lambda_2 t} f(t) \|_2\right)^{\theta_{2,r,p}} \left(\sup_{t\geq 0} \|e^{\lambda_\infty t} f(t) \|_\infty\right)^{1-\theta_{2,r,p}}\\
			\lesssim & \  \sup_{t\geq 0} \|e^{\lambda_2 t} f(t) \|_2  + o(1)
			\sup_{t\geq 0} \|e^{\lambda_\infty t} f(t) \|_\infty .
		\end{split}\Ee
		From our choice (\ref{condition_lambda_infty}) and $0<p-3 \ll 1$,
		\Be\label{decay_phiC1}
		\theta_{2,r,p} \lambda_2  + (1- \theta_{2,r,p}) \lambda_\infty  \geq (1+ \theta_{2,r,p}) \lambda_\infty.
		\Ee
		
		From (\ref{decay_interpolation})
		\Be
		\begin{split}\label{est_gdN_phif}
		(\ref{global_decay_N})_{\phi_f}
			\lesssim \ &\int^t_0 e^{- \frac{\nu_0}{2} (t-s)}  e^{-(1+ \theta_{2,r,p}) \lambda_\infty s}
			 \|e^{\lambda_2 s} f(s) \|_2 \dd s  \\
			 &+o(1) \int^t_0e^{- \frac{\nu_0}{2} (t-s)}    e^{-(1+ \theta_{2,r,p}) \lambda_\infty s}
		  \| e^{\lambda_ \infty s}w_{\vartheta} f(s) \|_\infty 
			\dd s\\
			\lesssim \ & e^{- \min (\frac{\nu_0}{4}, \lambda_\infty) \times t}
		\{	 \| f_0\|_2  +   \| \nabla \phi_{f_0}\|_2 \}
			 \ \  from \ (\ref{completes_dyn_final})\\
			 & + o(1)e^{- \min (\frac{\nu_0}{4}, \lambda_\infty) \times t}\sup_{0 \leq s\leq t} \| e^{\lambda_ \infty s}w_{\vartheta} f(s) \|_\infty .
		\end{split}
		\Ee
		
		Multiplying $e^{\lambda_\infty t}$ and taking $\sup_{t \geq 0}$ to (\ref{global_decay_N}) with $\lambda_\infty \leq \min \left(\frac{\nu_0}{4}, \frac{\lambda_2}{2}\right)$, and from (\ref{global_decay_N_L2}) and (\ref{est_gdN_phif}), we obtain that 
		\Be\label{f_infty_1}
			 \sup_{t \geq 0}e^{\lambda_\infty t}\| w_{\vartheta} f(t) \|_\infty 
			\lesssim   \|w_{\vartheta} f(0) \|_\infty + \| f_0\|_2  +   \| \nabla \phi_{f_0}\|_2   + o(1) \sup_{0 \leq s\leq t} e^{\lambda_\infty s} \| w_{\vartheta} f(s) \|_\infty.
		\Ee
		By absorbing the last (small) term, we conclude that \hide
		
		From (\ref{choice_L_infty})
		\Be\notag
		\begin{split}
			\lambda_\infty- \min \{ {\nu_0} , \lambda_2- \delta \}
			\geq &   -\frac{\lambda_2 - \delta}{2},\\
			\lambda_\infty- 
			\min \big\{
			{\nu_0} 
			,
			\lambda_\infty + \delta \frac{\theta}{1-\theta}
			\big\}\geq &  - \delta \frac{\theta}{1-\theta}.
		\end{split}
		\Ee
		Note that 
		\Be\begin{split}\notag
			\int^t_0 e^{  -\frac{\lambda_2 - \delta}{2} s} \dd s &\lesssim  \frac{1}{\lambda_2- \delta},\\
			\int^t_0 e^{  -    \delta \frac{\theta}{1-\theta} s} \dd s &\lesssim  \frac{1-\theta}{\delta \theta}.
		\end{split}
		\Ee
		
		Using (\ref{f_infty_1}) and choosing $o(1) \ll_{\delta, \theta, \lambda_2} 1$ we conclude that 
		\Be\label{Linfty_L2}
		\sup_{t \geq 0}e^{\lambda_\infty t}\| w_\vartheta f(t) \|_\infty \lesssim \| w_\vartheta f(0) \|_\infty + 
		\sup_{t \geq 0} e^{\lambda_2 t} \| f(t) \|_2.
		\Ee

		From (\ref{Linfty_L2}) and (\ref{completes_dyn_final}) we conclude that, for some $\mathfrak{C}\gg1$,
		 \unhide
		\Be\label{Linfty_final}
		\sup_{0 \leq t \leq T}e^{\lambda_\infty t}\| w_\vartheta f(t) \|_\infty  
		\leq \mathfrak{C}\delta_* M.
		\Ee
		If we choose $\delta_* \ll  1/ \mathfrak{C}$ then by the local existence theorem (Theorem \ref{local_existence}) and continuity of $\| w_\vartheta f(t) \|_\infty$, $\| w_{\tilde{\vartheta}} f (t) \|_p^p+ \|w_{\tilde{\vartheta}} \alpha_{f,\e}^\beta \nabla_{x,v} f (s) \|_p^p
			+ \int^t_0  | w_{\tilde{\vartheta}}\alpha_{f,\e}^\beta \nabla_{x,v} f(s) |_{p,+}^p$, and $\| \nabla_v f(t) \|_{L^3_x(\O) L^{1+ \delta}_v (\R^3)}$, we conclude that $T= \infty$.

		Then the estimates of (\ref{W1p_main}) and (\ref{nabla_v f_31}) are direct consequence of Proposition \ref{prop_W1p}, Lemma \ref{lemma_apply_Schauder}, and Proposition \ref{prop_better_f_v}.
			
		And (\ref{stability_1+}) can be derived from (\ref{eqtn_f-g})-(\ref{l1+stabfinal}) by replacing $f^\ell$, $f^{\ell +1}$ with $f,g$.
		\hide
		
		\vspace{4pt}
		
		\textit{Step 8. } We choose, for $0<p-3\ll1$,
		\Be\label{lambda_infty}
		\begin{split}
			r&= \frac{p+3}{2}, \ \ \theta_{2,r,p} = \frac{2(p-3)}{(p+3) (p-2)}>0, \\
			\lambda_\infty &= \frac{\theta_{2,r,p}}{
				1- \frac{9}{10} (1- \theta_{2,r,p})
			} \times \frac{9}{10} \lambda_2.\end{split}
		\Ee
		
		It is important to note that, from (\ref{lambda_infty}),
		\Be\begin{split}\notag
			& \left(\frac{9}{10}\right)^2 \times  \big\{\theta_{2,r,p} \lambda_2 + (1- \theta_{2,r,p}) \lambda_\infty\big\}\\
			=& \ \left(\frac{9}{10}\right)^2\times 
			\Big\{
			\frac{10}{9}  \Big(1- \frac{9}{10} (1- \theta_{2,r,p})\Big)+ (1- \theta_{2,r,p})
			\Big\} \lambda_\infty\\
			= & \  \frac{9 }{10}\lambda_\infty.
		\end{split}\Ee

		We choose 
		\Be\label{lambda_T}
		\lambda_{\infty} = \frac{9}{10} \times  \min \Big\{  {\nu_0} ,  \lambda_2  ,  {\theta_{2,r,p} \lambda_2 + (1- \theta_{2,r,p}) \lambda_\infty} \Big\}.
		\Ee

		Now we apply

		\vspace{4pt}
		
		\textit{Step 6. } 
		\unhide
	\end{proof}

	\appendix

	\section{Auxiliary Results and Proofs}
	\begin{proof}[\textbf{Proof of (\ref{alpha_invariant})}]
	Let $\iota = +$ or $-$, from (\ref{hamilton_ODE}), for $t-\tbpm(t,x,v)<s\leq t$,
		\Be\begin{split}\notag
			\xbpm(s,X_\pm(s;t,x,v),V_\pm(s;t,x,v))  =& \  \xbpm(t,x,v),\\
			\vbpm(s,X_\pm(s;t,x,v),V_\pm(s;t,x,v))  =& \  \vbpm(t,x,v).
		\end{split}\Ee
		Therefore 
		\Be\notag
		\begin{split}
			& [\p_t + v\cdot \nabla_x \mp \nabla_x \phi_f \cdot \nabla_v] \alpha_{f,\e,\pm}(t,x,v)  \\
			=   & \    \frac{d}{ds} \alpha_{f,\e,\pm}(s,X_\pm(s;t,x,v),V_\pm(s;t,x,v)) \big|_{s=t}\\
			= & \ \frac{d}{ds} \alpha_{f,\e,\pm}(t,x,v) 
			=    0.
		\end{split}
		\Ee
		
		From (\ref{tb}) and (\ref{hamilton_ODE}),
		\Be\begin{split}\notag
			\tbpm(s,X_\pm(s;t,x,v), V_\pm(s;t,x,v)) 
			= \tbpm(t,x,v)- (t-s).
		\end{split}\Ee
		Therefore 
		\Be\notag
		\begin{split}
			& [\p_t + v\cdot \nabla_x \mp \nabla_x \phi_f  \cdot \nabla_v] (t-\tbpm(t,x,v)) \\
			=   & \    \frac{d}{ds} [s- \tbpm(s,X_\pm(s;t,x,v),V_\pm(s;t,x,v))] \big|_{s=t}\\
			= & \ \frac{d}{ds} [t-\tbpm(t,x,v)]
			=    0.
		\end{split}
		\Ee
		These prove (\ref{alpha_invariant}).
		\hide
		\Be \notag
		\begin{split}
			& [\p_{t} + v\cdot\nabla_{x} - \nabla_{x}\phi\cdot\nabla_{v}] \alpha(t,x,v)  \\
			&= \frac{d}{ds} \alpha(s,X(s;t,x,v), V(s;t,x,v)) \Big\vert_{s=t}  \\
			&= \mathbf{1}_{ t+1\geq \tbpm(t,x,v)  } (t,x,v) \frac{d}{ds} |n(\xbpm(s,X(s), V(s)))\cdot \vbpm(s,X(s), V(s)) | \Big\vert_{s=t}  \\
			&\quad + |n(\xbpm(t,x,v))\cdot \vbpm(t,x,v)| \frac{d}{ds} \mathbf{1}_{ s+1 \geq \tbpm(s,X(s),V(s))  } \Big\vert_{s=t} \\
			&= \mathbf{1}_{ t+1\geq \tb(t,x,v)  } (t,x,v) \frac{d}{ds} |n(t,x,v)\cdot \vbpm(t,x,v) | \Big\vert_{s=t}  \\
			&\quad + |n(\xbpm(t,x,v))\cdot \vbpm(t,x,v)| \frac{d}{ds} \mathbf{1}_{ t+1 \geq \tb(t,x,v) } \Big\vert_{s=t}  \\
			&= 0,
		\end{split}
		\Ee
		where $X(s)=X(s;t,x,v)$ and $V(s) = V(s;t,x,v)$. We have used deterministic properties:
		\Be
		\begin{split}
			n(\xbpm(s,X(s), V(s)))\cdot \vbpm(s,X(s), V(s)) &= n(\xbpm(t,x,v))\cdot \vbpm(t,x,v) ,  \\
			t - \tb(t,x,v) &= s - \tb(s,X(s), V(s)),
		\end{split}
		\Ee
		in the last step. \unhide
	\end{proof}

	\begin{proof}[\textbf{Proof of (\ref{k_vartheta_comparision})}] The proof follows the argument of Lemma 7 in \cite{Guo10}. Note
		\Be\notag
		\begin{split}
			\mathbf{k}_{  \varrho}(v,u) \frac{e^{\vartheta |v|^2}}{e^{\vartheta |u|^2}} 
			=  \frac{1}{|v-u| } \exp\left\{- {\varrho} |v-u|^{2}  
			-  {\varrho} \frac{ ||v|^2-|u|^2 |^2}{|v-u|^2} + \vartheta |v|^2 - \vartheta |u|^2
			\right\}.
		\end{split}\Ee
		%
		Let $v-u=\eta $ and $u=v-\eta $. Then the exponent equals
		\begin{eqnarray*}
			&&- \varrho|\eta |^{2}-\varrho\frac{||\eta |^{2}-2v\cdot \eta |^{2}}{%
				|\eta |^{2}}-\vartheta \{|v-\eta |^{2}-|v|^{2}\} \\
			&=&-2 \varrho |\eta |^{2}+ 4 \varrho v\cdot \eta - 4 \varrho\frac{|v\cdot
				\eta |^{2}}{|\eta |^{2}}-\vartheta \{|\eta |^{2}-2v\cdot \eta \} \\
			&=&(-2 \varrho-\vartheta  )|\eta |^{2}+(4 \varrho+2\vartheta )v\cdot \eta -%
			4 \varrho\frac{\{v\cdot \eta \}^{2}}{|\eta |^{2}}.
		\end{eqnarray*}%
		If $0<\vartheta <4 \varrho$ then the discriminant of the above quadratic form of 
		$|\eta |$ and $\frac{v\cdot \eta }{|\eta |}$ is 
		\begin{equation*}
		(4 \varrho+2\vartheta )^{2}-4
		(-2 \varrho-\vartheta  )(-%
		4 \varrho)
		=4\vartheta ^{2}- 16 \varrho \vartheta<0.
		\end{equation*}%
		Hence, the quadratic form is negative definite. We thus have, for $%
		0<\tilde{\varrho}< \varrho - \frac{\vartheta}{4}  $, the following perturbed quadratic form is still negative definite 
		\[
		-(\varrho - \tilde{\varrho})|\eta |^{2}-(\varrho - \tilde{\varrho})\frac{||\eta
			|^{2}-2v\cdot \eta |^{2}}{|\eta |^{2}}-\vartheta \{|\eta |^{2}-2v\cdot \eta \}  \leq 0.
		\]
		Therefore we conclude (\ref{k_vartheta_comparision}).
		\hide, for given $|v|\geq 1,$ we make another change of variable $\eta
		_{\shortparallel }=\{\eta \cdot \frac{v}{|v|}\}\frac{v}{|v|},$ and $\eta
		_{\perp }=\eta -\eta _{||}$ so that $|v\cdot \eta |=|v||\eta
		_{\shortparallel }|$ and $|v-v^{\prime }|\geq |\eta _{\perp }|.$ We can
		absorb $\{1+|\eta |^{2}\}^{|\beta |}$, $|\eta |\{1+|\eta |^{2}\}^{|\beta |}$
		by $e^{\frac{C_{\theta }}{4}|\eta |^{2}}$, and bound the integral in (\ref%
		{wk}) by (\ref{exponent}): 
		\begin{eqnarray*}
			&&C_{\beta }\int_{\mathbf{R}^{2}}(\frac{1}{|\eta _{_{\perp }}|}+1)e^{-\frac{%
					C_{\theta }}{4}|\eta |^{2}}\left\{ \int_{-\infty }^{\infty }e^{-C_{\theta
				}|v|\times |\eta _{||}|}d|\eta _{||}|\right\} d\eta _{\perp } \\
			&\leq &\frac{C_{\beta }}{|v|}\int_{\mathbf{R}^{2}}(\frac{1}{|\eta _{_{\perp
				}}|}+1)e^{-\frac{C_{\theta }}{4}|\eta _{\perp }|^{2}}\left\{ \int_{-\infty
			}^{\infty }e^{-C_{\theta }|y|}dy\right\} d\eta _{\perp }\text{ \ \ }%
			(y=|v|\times |\eta _{||}|).
		\end{eqnarray*}%
		We thus deduce our lemma since both integrals are finite.

		\unhide
	\end{proof}

	Recall $\kappa_\delta(x,v)$ in (\ref{Z_dyn}). Let us denote $f_{\delta}(t,x,v) 
	:=      \kappa_\delta (x,v) f(t,x,v)$. We assume that $f(s,x,v)=e^s f_0(x,v)$ for $s<0$. Then $
	\| f_{\delta} \|_{L^{2} (\mathbb{R} \times \Omega \times \mathbb{R}^{3})} 
	\lesssim   \| f \|_{L^{2} (\mathbb{R}_{+} \times \Omega \times \mathbb{R}^{3})}
	+ \| f_{0}\|_{L^{2} (\Omega \times \mathbb{R}^{3})}$, 
	$\| f_{\delta} \|_{L^{2} ( \mathbb{R} \times \gamma)}  \lesssim  \| f_{\gamma} \|_{L^{2} ( \mathbb{R}_{+} \times \gamma)} + \| f_{0} \|_{L^{2} (\gamma)}$.
	\hide
	Note that, at the boundary $(x,v) \in \gamma:=\partial\Omega \times \mathbb{R}^{3}$,  
	\begin{equation}\label{Z_support_dyn}
	f_{\delta}(t,x,v)|_{\gamma}\equiv 0, \ \  \text{for}   \ |n(x) \cdot v| \leq \delta \  \text{ or } \  |v|\geq \frac{1}{\delta}  .
	\end{equation}
	\unhide
	
	\hide

	The main goal of this section is the following:
	\begin{proposition}\label{prop_3} Assume $g \in L^{2} (\mathbb{R}_{+} \times \Omega \times \mathbb{R}^{3})$, $f_{0} \in L^{2} (\Omega \times \mathbb{R}^{3})$, and $f_{\gamma}\in L^{2} (  \mathbb{R}_{+}\times \gamma)$. Let $f \in L^{\infty}(  \mathbb{R}_{+}; L^{2} (\Omega \times \mathbb{R}^{3}))$ solves (\ref{linear_dyn}) in the sense of distribution and satisfies $f(t,x,v)   =   f_{\gamma}(t,x,v) $ on $ \mathbb{R}_{+} \times \gamma$ and $f(0,x,v)  = f_{0} (x,v)$ on $\Omega \times \mathbb{R}^{3}.$
		
		Then
		\begin{equation}\label{S1}
		\begin{split}
		|a  (t,x)| + |b(t,x)| + |c(t,x)|  
		\leq  \ \mathbf{S}_{1}f(t,x) + \mathbf{S}_{2} f(t,x),\\
		\mathbf{S}_{1} f(t,x) : =   \ 4 \int_{\mathbb{R}^{3}} 
		| f_{\delta} (t,x,v)|
		\langle v\rangle^{2} \sqrt{\mu(v)}\dd v,  \\
		\mathbf{S}_{s} f(t,x) : =    4 \int_{\mathbb{R}^{3}}
		| (\mathbf{I} - \mathbf{P}) f (t,x,v)| \langle v\rangle ^{2}\sqrt{\mu(v)} \dd v
		+ 2\chi(t) \int_{\mathbb{R}^{3}} |f_{0} (x,v)| \langle v \rangle^{2} \sqrt{\mu(v)} \dd v
		,
		\end{split}
		\end{equation} 
		where $f_{\delta}$ is defined in (\ref{Z_dyn}). 
		
		Moreover
		\begin{equation}\label{decom_Pf}
		\begin{split}
		\| \mathbf{S}_{1}f \|_{L^{3}_{x} L^{2}_{t}}   & \ \lesssim   \ \| w^{-1} f\|_{L^{2}_{t,x,v}}+  \| w^{-1} g\|_{L^{2}_{t,x,v} }+ \|   f\|_{L^{2}(\mathbb{R}_{+} \times\gamma)} ,\\
		\| \mathbf{S}_{2}f \|_{L^{2}_{t,x} }  & \  \lesssim   \  \| (\mathbf{I}- \mathbf{P}) f\|_{L^{2}_{t,x,v} }
		+ \| f_{0} \|_{L^{2}_{x,v}}
		,
		\end{split}
		\end{equation}
		for $w= e^{\beta |v|^{2}}$ with $0< \beta\ll 1$.
		
	\end{proposition}
	
	\bigskip
	
	We need several lemmas to prove Proposition \ref{prop_3}.

	\unhide
	
	\begin{lemma} \label{extension_dyn}Assume $\O$ is convex in (\ref{convexity_eta}) and $\sup_{0 \leq t \leq T}\|E(t)\|_{L^\infty (\O)} < \infty$. Let $\bar{E}(t,x) = \mathbf{1}_{\O}(x) E(t,x)$ for $x \in \R^3$. There exists $\bar{f}(t,x,v) \in L^{2}( \mathbb{R} \times   \mathbb{R}^{3} \times \mathbb{R}^{3} ; \mathbb R^2)$, an extension of $f_{\delta}$, such that 
		\begin{equation}\notag
		\bar{f}  |_{\Omega \times \mathbb{R}^{3}}\equiv f_{\delta}    \  \text{ and } \  \bar{f}  |_{\gamma}\equiv f_{ \delta} |_{\gamma}   \  \text{ and } \ \bar{f } |_{t=0} \equiv f_{\delta} |_{t=0}.
		\end{equation} Moreover, in the sense of distributions on $\mathbb{R} \times \mathbb{R}^{3} \times \mathbb{R}^{3} \to \mathbb R^2$,
		\begin{equation}\label{eq_barf_dyn}
		[\partial_{t} +  v\cdot \nabla_{x} + q \bar{E} \cdot \nabla_{v}]\bar{f} =  h  ,
		\end{equation}
		where
		\Be\begin{split}\label{barf_h}
			h_{} (t,x,v)=& \ 
			\kappa_\delta(x,v)    \mathbf{1}_{t \in [0,\infty)} 
			[ \partial_{t} + v\cdot \nabla_{x}   + q E
			\cdot \nabla_{v}  ] f\\
			&
			+ \kappa_\delta(x,v)  \mathbf{1}%
			_{t \in ( - \infty, 0 ]} e^t
			[1  + v\cdot \nabla_{x}   + q E
			\cdot \nabla_{v}] f_0 \kappa_\delta(x,v) \\
			& 
			+ f (t,x,v) [v\cdot \nabla_{x} + q_1 E
			\cdot \nabla_{v}] \kappa_\delta(x,v),\\
		\end{split}\Ee
		where $t_{\mathbf{b}}^{EX}, x_{\mathbf{b}}^{EX}, t_{\mathbf{f}}^{EX}, x_{\mathbf{f}}^{EX}$ are defined in (\ref{def_tb_EX}).
		
		Moreover,
		\Be\begin{split}\label{estimate_h1_h2}
			\| h_{} \|_{  L^{2}( \mathbb{R} \times   \mathbb{R}^{3} \times \mathbb{R}^{3})}  
			\lesssim  & \  \|  [ \partial_{t} + v\cdot \nabla_{x}   + q E
			\cdot \nabla_{v}  ] f\|_{L^{2}(
				\mathbb{R}_{+} \times 
				\Omega \times \mathbb{R}^{3})} 
			+  \| f \|_{L^{2} (\R \times \Omega \times \mathbb{R}^{3})} \\
			&  
			+ \| [ v\cdot \nabla_{x} + q E\cdot \nabla_{v} ] f_{0} \|_{L^{2 } (\Omega \times \mathbb{R}^{3})} .
		\end{split}\Ee
		
	\end{lemma}
	
	\begin{proof} In the sense of distributions  
		\begin{equation} \label{eqtn_f_delta}
		\partial_{t} f_{\delta}+ v\cdot \nabla_{x} f_{\delta} + q E
		\cdot \nabla_{v} f_{\delta} = h  \text{ in }  (\ref{barf_h}).
		\end{equation}
		Clearly $| [v\cdot \nabla_{x} + q_1 E
		\cdot \nabla_{v}] \kappa_\delta(x,v)| \lesssim_\delta 1$.\hide
		\begin{eqnarray}
		&&\Big|\{v\cdot \nabla_{x} + \e^{2} \Phi \cdot \nabla_{v}\} [1-\chi(\frac{%
			n(x) \cdot v}{\delta}) \chi \big( \frac{ \xi(x )}{\delta}\big) ]
		\chi(\delta|v|)\Big|  \label{der_chi} \\
		&=&\Big| - \frac{1}{\delta} \{v\cdot \nabla_{x} n(x) \cdot v + \e^{2} \Phi
		\cdot n(x) \} \chi^{\prime} \big(\frac{n(x) \cdot v}{\delta} \big) \chi %
		\big( \frac{ \xi(x )}{\delta}\big) \chi(\delta|v|)  \notag \\
		&& - \ \frac{1}{\delta} v\cdot \nabla_{x} \xi(x) \chi^{\prime} \big( \frac{%
			\xi(x)}{\delta}\big) \chi (\frac{n(x) \cdot v}{\delta}) \chi(\delta|v|) + \e%
		^{2}\delta \Phi \cdot \frac{v}{|v|} \chi^{\prime} (\delta|v|)[1-\chi(\frac{%
			n(x) \cdot v}{\delta}) \chi \big( \frac{\xi(x)}{\delta}\big) ] \Big|  \notag
		\\
		&\leq& \frac{4}{\delta}( |v|^{2}\|\xi\|_{C^2} + \e^{2}\|\Phi\|_\infty )
		\chi(\delta|v|) + \frac{C_{\Omega}}{\delta} |v|\chi(\delta|v|) + \e^{2}
		\delta\|\Phi\|_\infty \mathbf{1}_{|v| \leq {2}{\delta}^{-1}}  \notag \\
		&\lesssim & {\delta^{-3}} \mathbf{1}_{|v| \leq 2 \delta^{-1}}.  \notag
		\end{eqnarray}\unhide

		For $x \in \R^3 \backslash \bar{\O}$ we define
		\Be\begin{split}\label{def_tb_EX}
			\tb^{EX}(x,v) &:= \sup\{s \geq 0: x-\tau v \in \R^3 \backslash \bar{\O}
			\ \text{ for all } \ \tau \in (0,s)
			\},\\
			\tf^{EX}(x,v) &:= \sup\{s \geq 0: x+\tau v \in \R^3 \backslash \bar{\O}
			\ \text{ for all } \ \tau \in (0,s)
			\},
		\end{split}\Ee
		and $\xb^{EX}(x,v) = x- \tb^{EX}(t,x,v))v$, $\xf^{EX}(x,v) = x + \tf^{EX}(t,x,v))v$.
		
		We define, for $x \in \R^3 \backslash \bar{\O}$,
		\Be\label{def_f_E}
		\begin{split}
			f_E (t,x,v) =& \mathbf{1}_{\xb^{EX} (t,x,v) \in \p \O} f_\delta(t-\tb^{EX}(x,v), \xb^{EX} (x,v),v)\\
			+& \mathbf{1}_{\xf^{EX} (t,x,v) \in \p \O}f_\delta(t+\tf^{EX}(x,v), \xf^{EX} (x,v),v).
		\end{split}\Ee
		Recall that, from (\ref{Z_dyn}), $f_\delta\equiv0$ when $n(x) \cdot v = 0$, and hence $f_E\equiv0$ for $n(x) \cdot v = 0$. Since $\O$ is convex if $v\neq 0$ then $\{\xb^{EX} (x,v) \in \p \O\} \cap \{\xf^{EX} (x,v) \in \p \O\}= \emptyset$. Note that 
		\Be\label{no_jump_bdry}
		f_E(t,x,v) = f_\gamma(t,x,v) = f_\delta (t,x,v) \ \ \text{for }   x \in \p\O.
		\Ee
			And since for any $s>0$, 
			\[ \begin{split}
			&(t +s - \tb^{EX}(x+sv,v), \xb^{EX}(x+sv,v),v  ) = (t - \tb^{EX}(x,v),\xb^{EX}(x,v),v) ,
			\\ & (t +s + \tf^{EX}(x+sv,v), \xf^{EX}(x+sv,v),v  ) = (t - \tf^{EX}(x,v),\xf^{EX}(x,v),v),
			\end{split} \]
		so in the sense of distribution, in $\R \times [\R^3 \backslash \bar{\O}] \times \R^3$ 
		\Be
		\label{eqtn_f_E}
		\p_t f_E + v\cdot \nabla_x f_E = 0.
		\Ee

		We define 
		\Be\label{def_bar_f}
		\bar{f}(t,x,v) : = \mathbf{1}_{\O} (x)  f_\delta (t,x,v)
		+ \mathbf{1}_{\R^3 \backslash \bar{\O}} (x) f_E (t,x,v).
		\Ee

		From (\ref{eqtn_f_delta}), (\ref{no_jump_bdry}), and (\ref{eqtn_f_E}) we prove (\ref{eq_barf_dyn}). The estimates of (\ref{estimate_h1_h2}) are direct consequence of Lemma \ref{le:ukai}.
		\end{proof}

\textbf{Acknowledgements.} This paper is part of the author's thesis. He thanks his advisor Professor Chanwoo Kim for helpful discussions. This research is partially supported by NSF Grant No. 1501031.

\end{document}